\documentclass[10pt,a4paper,leqno]{article}
\usepackage{amsfonts,dsfont,amsmath,amsthm,enumerate,color}
\usepackage{graphicx}
\usepackage{fullpage}
\usepackage{authblk}

\usepackage{bold-extra} 

\usepackage{amssymb}

\definecolor{orange}{rgb}{1,0.25,0}
\definecolor{darkgreen}{rgb}{0.0, .5, 0.13}

\usepackage[colorlinks=true,linkcolor=blue,citecolor=red, pdfborder={0 0 0}]{hyperref}

\newcommand{\vertiii}[1]{{\left\vert\kern-0.25ex\left\vert\kern-0.25ex\left\vert #1 
		\right\vert\kern-0.25ex\right\vert\kern-0.25ex\right\vert}}

\renewcommand{\i}{{\rm i}}

\newcommand{\CC}{\mathbb{C}}
\newcommand{\NN}{\mathbb{N}}
\newcommand{\RR}{\mathbb{R}}
\newcommand{\ZZ}{\mathbb{Z}}
\newcommand\stwoscale{\overset{_2}{\rightarrow}} 

\theoremstyle{plain}
\newtheorem{theorem}{Theorem}[section]
\newtheorem{lemma}[theorem]{Lemma}
\newtheorem{proposition}[theorem]{Proposition}
\newtheorem{corollary}[theorem]{Corollary}

\theoremstyle{definition}
\newtheorem{definition}[theorem]{Definition}

\theoremstyle{plain}
\newtheorem{remark}[theorem]{Remark}

\numberwithin{equation}{section}

\newcommand\ep{\varepsilon}

\newcommand{\Ep}{\epsilon}
\title{\vspace{-1.8cm}  
Quantitative multiscale operator-type approximations for asymptotically degenerating spectral problems

}

\author[1]{\vspace{-.2cm}Shane Cooper}
\author[1]{Ilia V. Kamotski}
\author[1,2]{Valery P. Smyshlyaev\vspace{-.3cm}}

\affil[1]{\footnotesize Department of Mathematics, University College London, Gordon Street, London WC1E 6BT, UK.}
\affil[2]{\footnotesize Corresponding author. Email: v.smyshlyaev@ucl.ac.uk}

\renewcommand{\t}{\theta}
\renewcommand{\l}{\langle}
\renewcommand{\r}{\rangle}

\newcommand{\N}{\mathcal{N}_\t}

\newcommand{\M}{\mathcal{M}_\t}
\newcommand{\Vs}{V^\star}

\DeclareMathOperator{\ourN}{\mathtt{N}}

\newcommand{\be}{\begin{equation}}
\newcommand{\ee}{\end{equation}}
\begin{document}

\maketitle
\begin{abstract}

This work aims at developing 
approximations of two-scale type for 
a wide class of 
problems whose leading-order 
behaviour remains intrinsically 
two-scale (and hence when the 
classical homogenisation fails), 
and at obtaining tight 
operator-type error estimates for 
such approximations. 
We study 
an abstract family of asymptotically degenerating variational formulations, which 
are natural generalisations of those 
emerging upon application of  
Floquet-Bloch-Gelfand transform to 
high-contrast elliptic PDEs with 
periodic coefficients.   
A hierarchy of approximations is constructed 
with uniform 
error estimates under certain 
assumptions, 
satisfied by 
specific examples. 
We provide  approximations for the resolvents 
in terms of 
a 
`bivariate' operator which is 
an abstract generalisation of 
the two-scale limit operators.  
The resulting self-adjoint approximation 
is the bivariate operator's resolvent sandwiched 
by a 
connecting operator and its adjoint. 
For a broad class of periodic PDE 
problems, 
the connecting operator
is expressible in terms of 
a two-scale interpolation operator which is a 
two-scale modification 
of the classical Whittaker-Shannon interpolation. 
An explicit  description of the limit spectrum in the abstract setting is provided, and 
error estimates on the discrepancy between 
the original and limit spectra are established. 
For a key model example of high-contrast isolated periodic inclusions, we obtain new results on the rate of uniform 
convergence of the Floquet-Bloch spectra and of other spectral characteristics to those of 
the two-scale limit operator with 
 an explicitly characterised spectrum 
displaying band-gap opening 
near the inclusions' resonances. 
The 
general results are 
illustrated by numerous other 
examples, going beyond high-contrast PDEs,  
for which several new results are 
obtained. 

\end{abstract}

{\bf Keywords:} Asymptotically degenerating problems, two-scale operator approximations, spectral theory methods 


\section{Introduction}

\subsection{Motivation, background and discussion of main results for a model problem}

One of the main broader motivations for the present study comes from the desire for construction of tractable but accurate 
uniform operator approximations of two-scale type, with quantifiably 
small errors, for a broad class of mathematical models for which the leading-order 
asymptotic behaviour remains intrinsically two-scale (and hence where the 
classical homogenisation fails). 
An interest in such models comes from the fact that they are 
often capable of displaying certain non-standard and unusual physical effects in an asymptotically explicit way, which often clarifies the nature and the microscopic mechanism of the observable macroscopic effects. 

One class of such models includes two-scale Partial Differential Equations (PDEs) with high-contrast coefficients, or in other words asymptotically degenerating PDEs. 
Consider as an example a simple non-dimensionalised 
model of time-harmonic wave propagation in an $\ep$-periodic medium described by 
\begin{equation} 
\label{th-hc}
\nabla\cdot 
\big(a_{\ep,\delta}(x) \nabla u_\ep\big)  \,\,+\,\,  \rho\,\omega^2 u_\ep \,\,=\,\, 0, \quad \, 
x=(x_1,x_2,...,x_n)^T\,\in\mathbb{R}^n, \,\, n\ge 2.  
	\end{equation}
Here $\omega>0$ is angular frequency, ``density'' $\rho$ is (for simplicity) a positive constant, and 
``stiffness'' $a_{\ep,\delta}$ has a small value $\delta$ in the $\ep$-periodic isolated soft inclusions, and a fixed value normalised to unity 
in the connected stiff matrix. 
(I.e. 
 $a_{\ep,\delta}(x)=a_\delta\left(x/\ep\right)$ where  $a_\delta(y)=1-\chi(y)+\delta\chi(y)$ and $\chi$ is the characteristic function of a set of well-separated  
inclusions $1$-periodic with respect to each variable $y_j$, $j=1,2,...,n$.) 
\vspace{.08in} 

When both $\ep$ and $\delta$ are small, a critical 
scaling\footnote{
For the chosen wave propagation model \eqref{th-hc} 
with well separated periodic inclusions, the subcritical scaling $\ep^2\ll\delta\lesssim 1$ corresponds 
to classical homogenisation (which upon a rescaling corresponds to Bloch waves at low frequencies $\omega$, 
cf. Remark \ref{class-hom-validity}). 
The criticality at $\delta\sim\ep^2$ is 
caused by the frequency 
increasing 
and becoming comparable to the inclusions' eigenfrequencies. 
For the elliptic rather than spectral problems, 
i.e. when $\omega=0$ in \eqref{th-hc}, 
for recent new results on 
range of validity of (classical) homogenisation for broad classes of high-contrast {\it stochastic} models 
and with a different nature of the criticality 
due to percolation type effects 
see 
\cite{ArmKuu24}. 
} 
appears 
to be $\delta\sim\ep^2$, which in the context of the wave propagation model \eqref{th-hc} is 
a ``micro-resonant'' scaling: frequencies producing order-one wavelengths in the matrix material would produce order-$\ep$ wavelengths in the inclusion material i.e. 
those comparable with the inclusions' size.
In other words, such frequencies are comparable to the resonant frequencies of the inclusions, which (upon rescaling) appear to be 
the eigenvalues of the inclusions' Dirichlet Laplacian. 
This is reflected in the formal two-scale asymptotics of certain Bloch wave solutions to \eqref{th-hc}: 
$u_\ep(x)\sim u_0(x,x/\ep)$, where $u_0(x,y)=u(x)+v(x,y)$ 
with $v$ 1-periodic and supported in the inclusions in $y$; 
i.e. $u_0$ 
 is a 
function of 
only macroscopic variable $x$ in the matrix but is also a 
function of the microscopic variable $y$ in the resonating inclusions. 
That leads to a ``non-classical'' (in the sense of an asymptotic regime where the classical homogenisation limits fail) 
limit system for $u_0(x,y)$, which remains intrinsically two-scale. 
It is precisely this 
coupling of the micro and macro-scales in the limit system that explicitly 
displays such effects as band-gap opening near the resonances, as was probably first formally observed in a similar context in \cite{AurBonn85,Aur94}  
and 
was made rigorous by proving a convergence of the spectrum (although with an unknown rate) in \cite{HeLi} and \cite{Zhi2000,Zhi2005}. 
The $\delta\sim\ep^2$ scaling plays a similarly key role in the so-called double porosity type diffusion models, for related earlier studies see e.g. 
\cite{FeKh80,Ar90,Pa91,All,Sa99}. 
\vspace{.05in}

Continued intensive interest in 
such kind of asymptotically degenerating models is largely due to the fact that the 
related two-scale approximations indeed possess a wealth of interesting and unusual properties and effects. 
Without attempting here any comprehensive review, 
for an incomplete list of related 
works we mention 
\cite{Br,ChSZh06,VPS,Bell} for spatial non-locality and various 
macroscopic coupling effects for high-contrast and highly-anisotropic fiber-like or more general inter-connected structures (rather than just 
isolated inclusions); 
\cite{BoFe} for relations to artificial magnetism and negative materials; 
\cite{Av08,Ammari2009,Co,ZhiPa,Lipton2022} for (partial) band-gaps for high-contrast and partially degenerating elastic inclusions; 
\cite{ChCo-Maxw2015,BBF2017,Lipton2022b} for analogous effects in high-contrast electromagnetic media; 
\cite{BaKaSm} for ``doubly high contrast'' models; 
\cite{cherd,IVKVPS18} for localised modes due to defects in high-contrast periodic and \cite{ChCV23} stochastic  media; 
\cite{IVKVPS19} for band-gaps due to resonances in high-contrast elastic beam lattice materials 
and \cite{ChErKi,ChKiVeZu23} for dispersive effects in high-contrast micro-resonant media. 
Similar effects are observed in 
geometries containing split ring resonators, see e.g. \cite{LiptSchweiz2018}, 
and in 
thin 
structures, see e.g. recent publication 
\cite{
Ch-V-Z-2023}. 
For a recent review of a class of high contrast locally resonant materials from point of view of metamaterial 
modelling see e.g. 
\cite{ADHLY22} and further references therein. 
\vspace{.05in} 

What is in common in many 
of the above models 
is that the leading-order limit asymptotic behaviour remains 
two-scale, which is precisely the source of the diverse effects of interest. 
For a typical scenario, crudely, 
the exact solution $u_\ep(x)$ due to an input $F(x)$ is oscillatory for 
small $\ep$ and is approximated  by $u_0(x,x/\ep)$ 
where $u_0(x,y)$ is a solution of related two-scale limit problem with generally a two-scale associated input $f_0(x,y)$.  
One way of providing a rigorous relation between the original 
and the (two-scale) limit problems is via 
 the language of two-scale convergence, developed in \cite{Ng,All}. Namely, 
two-scale convergence as $\ep\to 0$ of inputs $F_\ep(x)$, denoted $F_\ep(x)\stwoscale f_0(x,y)$, typically  
implies that of the solutions: $u_\ep(x)\stwoscale u_0(x,y)$ where $u_0(x,y)$ is the solution of the two-scale limit problem 
with input $f_0(x,y)$. 
\vspace{.05in}

{\it Quantifying} 
such a convergence however, in a manner {\it uniform} with respect to all possible inputs $F(x)$, 
arguably 
requires applying and developing fundamentally new {
operator} tools, and this is one of the broader ultimate aims of 
the present work. 
Indeed, for any small but fixed $\ep>0$, for all the inputs $F(x)$ with associated solutions $u_\ep(x)$, one needs 
uniformly approximating 
the latter in terms of solutions $u(x,y)$ of the two-scale limit problem with appropriately chosen two-scale inputs  
$f(x,y)$. In other words one needs first recasting, in a suitable way, any $F(x)$ as a two-scale function $f(x,y)$ i.e. to 
construct a map or a connecting operator $\mathcal{J}_\ep: \, F(x)\mapsto f(x,y)$ with suitable properties. 
After the resulting two-scale limit problem is solved, its solution $u_0(x,y)$ has to be converted 
appropriately back into an accurate enough approximation $u_\ep^{\rm appr}(x)$ 
to the exact solution $u_\ep(x)$. 
For maintaining certain key 
properties of the resulting approximating operator 
(typically the self-adjointness, which would allow to use powerful tools of spectral theory of 
self-adjoint operators in Hilbert spaces), the latter step has 
to be 
achieved 
by an 
operator $\mathcal{J}_\ep^*$ adjoint to $\mathcal{J}_\ep$: $u_\ep^{\rm appr}(x)= \left(\mathcal{J}_\ep^*u_0\right)(x)$. 
If, additionally, $\mathcal{J}_\ep$ is chosen so that it possesses
certain asymptotic unitarity 
properties, this 
appears crucial for {\it quantifying a spectral convergence} 
 of the original operator to the two-scale limit one. 
In a nutshell, the present work aims at performing 
all that, for a wide class of problems (of operator resolvent and spectral type) 
with a two-scale type limit  asymptotic behaviour. 
\vspace{.04in}

To illustrate our 
general approach in more precise terms, let us return to the key motivating model \eqref{th-hc} 
with high-contrast periodic inclusions and highlight 
the new results which we are able to obtain for this 
specific example (Section \ref{e.dp}), putting them into a perspective. 
Problem \eqref{th-hc} is mathematically a spectral problem for unbounded non-negative self-adjoint operator 
$\mathcal{L}_\ep u=\,-\,\nabla\cdot\big(A_{\ep}(x)\nabla u\,\big)$ 
in Hilbert space $L^2\left(\RR^n\right)$, 
with $\ep$-periodic coefficients and 
with spectral parameter $\lambda=\rho\,\omega^2$.
(Here we regard $A_\ep(x):=a_{\ep,\,\ep^2}(x)$ i.e. we set for simplicity 
$\delta=\ep^2$.) 
Applying a rescaling or change of variables $y=x/\ep$, operator $\mathcal{L}_\ep$ is 
spectrally equivalent to 
$\mathcal{B}_\delta u=\,-\,\nabla_y\cdot\big(B_{\delta}(y)\nabla_y u\,\big)$ where 
$B_\delta(y)=\delta^{-1}a_\delta(y)$ is $1$-periodic function equal $1$ in the inclusions and 
$\delta^{-1}$ in the matrix. Its spectrum was analysed in \cite{HeLi} which proved via variational 
arguments applied to Floquet direct fiber-integral decomposition of $\mathcal{B}_\delta$ 
the convergence (with unspecified rate as $\delta\rightarrow 0$) of the Floquet-Bloch spectrum to the limit one, with the gaps 
opening at the (typical) resonances. 
\vspace{.04in} 

In \cite{Zhi2000} and \cite{Zhi2005}, operator $\mathcal{L}_\ep$ and its spectrum (for bounded domains and the whole 
space respectively) were analysed directly, via advanced 
therein techniques of two-scale resolvent operator convergence. 
Indeed, it appears that a 
 key for the analysis of the operator $\mathcal{L}_\ep$ and  
of its asymptotic 
properties as $\ep\to 0$ lies in analysing the related resolvent problem  
\begin{equation} 
\label{hom}
\big(\mathcal{L}_\ep \,+\,I\,\big)\,   u_\ep\, =\,
	-\,\,\nabla\cdot\big(A_{\ep} \nabla u_\ep\big)  \,\,+\,\,   u_\ep\,\, =\,\, F, \ \,\,\,\, F\in \,L^2\left(\RR^n\right). 
\end{equation}
It was shown in \cite{Zhi2005} that operator $\mathcal{L}_\ep$ has certain two-scale operator $\mathcal{L}_0$ as its limit in the 
sense of a (strong) {\it two-scale resolvent convergence}. 
Namely, if the right-hand sides $F=F_\ep(x)$ (strongly)  two-scale 
converge to $f_0(x,y)\in L^2(\mathbb{R}^n\times\square)$, 
where $\square$ is the $y$-periodicity cell (the unit size cube centred at the origin), 
then the solutions $u_\ep$ (strongly) two-scale converge to the solution $u_0(x,y)$ of the following two-scale (pseudo-)resolvent 
limit problem: 
\begin{equation} 
\label{hom-2sc}
	\mathcal{L}_0u_0  \,\,+\,\,   u_0\,\, =\,\, \mathcal{P}f_0. 
\end{equation}
In \eqref{hom-2sc} the two-scale limit operator $\mathcal{L}_0$ is non-negative self-adjoint in a ``bigger'' two-scale Hilbert space 
$\mathbb{H}_0$ which is 
a closed subspace of $L^2\left(\mathbb{R}^n\times \square\right)$, and 
 $\mathcal{P}$ is the orthogonal projection on  $\mathbb{H}_0$. 
So the above can 
be viewed as a strong two-scale convergence of the resolvent operators 
$\mathcal{R}_\ep=\left(\mathcal{L}_\ep+I\right)^{-1}$ to 
two-scale ``pseudo-resolvent'' 
$\mathcal{R}_0\mathcal{P}=\left(\mathcal{L}_0+I\right)^{-1}\mathcal{P}$: 
for any 
 $F_\ep\stwoscale f_0$, $\mathcal{R}_\ep F_\ep\stwoscale \mathcal{R}_0\mathcal{P}f_0$, however with no 
information on a rate of such convergence. 

The spectrum of the above two-scale limit operator $\mathcal{L}_0$ is explicitly described, and 
generally has a band-gap structure with infinitely many gaps opening at the resonances. 
It was then further 
shown in \cite{Zhi2005} that the above two-scale resolvent convergence, 
together with certain additional properties of two-scale 
compactness, implies a 
convergence of the spectra however again with an unknown 
rate. 
Robust estimates on the rate of convergence of the spectra near ``typical" resonances, 
explicit in terms of the inclusions' shapes and geometry, 
 were obtained in \cite{Lipton2017} via decomposition of 
quasi-periodic and ``electrostatic'' solution operators with
a subsequent analysis of related  resonances by tools of layer potential theory.
\vspace{.03in}

For {\it 
quantifying} the above operator convergence, i.e. establishing its 
rate, one 
needs obtaining for 
the solution $u_\ep(x)$ of the resolvent problem \eqref{hom} 
asymptotic approximation in terms of solutions $u_0(x,y)$ to the two-scale limit 
problem \eqref{hom-2sc} which would be {\it uniformly accurate} for all 
$F\in L^2\left(\mathbb{R}^n\right)$.  
As argued above (and made more precise in Remark \ref{ShannVsUnf}), any such approximation  
has to contain a 
tool like the above discussed 
connecting operator $\mathcal{J}_\ep$ 
for recasting any input $F(x)$ as that for the 
two-scale limit problem, which is generally a two-scale function $f(x,y)$. 
We will see that the general approach developed by us in the present work, 
and in particular  generic resolvent-type estimates in Theorem \ref{thm.bivariate}, 
when specialised to  \eqref{hom}   
imply  that the resolvent $\mathcal{R}_\ep$ is well approximated by 
$\mathcal{R}_0\mathcal{P}$ sandwiched by such 
an ($L^2$-isometric and ``asymptotically unitary") operator 
$\mathcal{J}_\ep: L^2\left(\RR^n\right)\rightarrow L^2\left(\RR^n\times\square\right)$ and its 
adjoint $\mathcal{J}_\ep^*$. 
Namely, new two-scale resolvent estimate 
\eqref{dpcompe3} holds, which 
can in a sense be viewed as that on the {\it rate} of an operator two-scale resolvent convergence and reads 
\begin{equation}
\label{2scresest}
\big\Vert\, \,\mathcal{R}_\ep F\,\,-\,\,\mathcal{J}_\ep^*\mathcal{R}_0\mathcal{P}\mathcal{J}_\ep F\,\big\Vert_{L^2\left(\mathbb{R}^n\right)}
\,\,\,\le\,\,\,C\,\ep\,\big\Vert\,F\,
\big\Vert_{L^2\left(\mathbb{R}^n\right)}, \ \ \ \ \forall F\in L^2\left(\mathbb{R}^n\right), 
\end{equation}
with a constant $C$ independent of both $\ep$ and $F$. 
Importantly, the approximating operator $\mathcal{J}_\ep^*\mathcal{R}_0\mathcal{P}\mathcal{J}_\ep$ in \eqref{2scresest}
 remains self-adjoint, 
and due to the asymptotic unitarity of $\mathcal{J}_\ep$ its 
spectrum can be shown to converge 
with a desired rate 
to that of the 
the two-scale limit (pseudo-)resolvent $\mathcal{R}_0\mathcal{P}$.  
More precisely, the $L^2$-isometry of $\mathcal{J}_\ep$ provided by the construction in Theorem \ref{p.unitaryequiv} 
 in combination with one more application of our generic scheme in Theorem \ref{t.collectivespec}, leads to 
a generic result on the rate of convergence of the spectra (estimate \eqref{specestgen} of Theorem \ref{bivariate.spec}). 
For the present example, the latter specialises 
to a new 
error estimate 
 \eqref{specest} on the uniform rate of convergence of the 
Floquet-Bloch spectrum of $\mathcal{L}_\ep$ 
(equivalently 
estimate \eqref{HempLien-est} for the above discussed operator 
$\mathcal{B}_\delta$ from \cite{HeLi}) 
to that explicitly described 
of the two-scale limit operator $\mathcal{L}_0$ 
and on the associated band gaps near the resonances, 
Theorem \ref{EVestDP} and Corollary \ref{HempLien-rate}. 
Our general spectral approach implies 
also certain explicit two-scale approximations with uniform estimates on the convergence rates for 
the Floquet-Bloch dispersion relations and for the related 
eigenfunctions (Bloch waves),   Remark \ref{EFestDP}. 
These 
estimates appear to 
also provide new results on the uniform asymptotics of the integrated density of states (Corollary \ref{denstates}), 
which may arguably 
be a more natural quantitative measure of the spectral convergence. 
They also provide 
some tight estimates 
on ranges of validity and failure of classical 
homogenisation's spectral approximations 
on approaching the critical scaling for the contrast, 
Remark \ref{class-hom-validity}. 
Our general approach allows obtaining  for high-contrast problems like \eqref{hom} not only $L^2$-type error estimates 
akin to \eqref{2scresest}, but 
also two-scale type approximations 
(with a slightly more subtle two-scale connecting operator $\mathcal{G}_\ep$ instead of $\mathcal{J}_\ep$, 
see Remark \ref{H1-2slp}) 
with estimates of an  
``asymptotically degenerate $H^1$'' (energy) type with an appropriate corrector: see abstract estimate \eqref{IKfinal3} of Theorem \ref{thm.IKunifest2} and its specialisation \eqref{dp.H1est} to the 
discussed example\footnote{After publication of previous arxiv versions of the present work, \cite{BonDueGlo25} 
 obtained some new quantitative (as well as qualitative) 
results for {\it stochastic} high-contrast models of type \eqref{hom}, with some interesting observations 
pertinent also to the present periodic scenario. 
In particular, Remark 1.7(b) of \cite{BonDueGlo25} 
develops an approximation in terms of an {\it $\ep$-dependent} operator (rather than the $\ep$-independent 
two-scale limit operator 
$\mathcal{L}_0$) capable of delivering $H^1$-type estimates 
akin to \eqref{dp.H1est}. Those approximations' main advantage is that they appear adjustable 
to the random setting. However, even in the periodic case, making such (further) approximations expressible in terms of $\mathcal{L}_0$ 
and ultimately 
applicable for obtaining error estimates on the limit spectrum would still require a number of non-trivial steps. For the latter, one way would be essentially to follow again through the key steps of 
our general approach in 
Sections \ref{section:discV}--\ref{s:resolv} of the present paper, although possibly 
in a somewhat simplified manner.}. 
\vspace{.05in} 


The two-scale connecting operator $\mathcal{J}_\ep$ appears in estimate \eqref{dpcompe3} of Theorem \ref{thm.2scOpRes} 
as a specialisation 
of our more general approach to periodic PDE (resolvent) problems,  which leads to its 
explicit construction 
as a composition of a problem-specific 
$L^2\left(\mathbb{R}^n\times\square\right)$-unitary ``translation operator'' $T_\ep$ with 
a generic ``two-scale interpolation operator'' $\mathcal{I}_\ep$: $\mathcal{J}_\ep=\,T_\ep\,\mathcal{I}_\ep$. 
Operator $T_\ep$ is specialisation of an abstract transfer operators $\mathcal{E}_\t$ (Lemma \ref{propeth}) 
and is usually 
naturally identified from the problem. It plays 
an important 
role of accounting 
for the actual asymptotic degeneracies,  
thereby ultimately assuring the desired quantifiable accuracy of the approximation. 
Indeed, the asymptotic degeneracies tend to 
translate into certain inherited degeneracies 
in 
the two-scale limit problem, 
and for the above model example \eqref{th-hc} of isolated degenerating inclusions 
this issue is dealt with by 
$\left(T_\ep f\right)(x,y)=f(x+\ep y, y)$ for $y$ in the inclusion. 
The generic operator 
$\mathcal{I}_\ep:L^2\left(\mathbb{R}^n\right)\to L^2\left(\mathbb{R}^n\times\square\right)$ 
plays the 
role of recasting any input $F(x)$ as a two-scale function $(\mathcal{I}_\ep F)(x, y)$. 
It was introduced, in an equivalent form, in \cite{Well2009} under the name of 
``periodic two-scale transform'' as a convenient tool for establishing various 
two-scale convergence 
results in periodic homogenisation. In our context its additional power appears to be that, in combination 
with the above translation operator $T_\ep$, it comes as a specialisation of a generic approach 
and most crucially is capable of 
{\it quantifying} the two-scale convergence in the sense of 
providing two-scale type operator approximations with tight error bounds. 
\vspace{.05in}

Operator $\mathcal{I}_\ep$ is a composition of a scaled Floquet-Bloch-Gelfand transform, an extension operator and 
a scaled (``semiclassical'') inverse Fourier 
transform, see \eqref{2ScInterp}, implying it is not only an ``asymptotically unitary'' $L^2$-isometry but 
has all kinds of desirable properties. 
In particular, 
$(\mathcal{I}_\ep F)(x, x/\ep)=
F(x)$ where 
$f(x,y):=(\mathcal{I}_\ep F)(x, y)$ is understood as 
$\square$-periodically extended on $\mathbb{R}^n$, 
and so $\mathcal{I}_\ep$ 
appears to be a ``two-scale interpolation operator''. 
Namely, for a fixed $\ep>0$ and $y\in\square$, for any $x$ with ``phase'' $y$ i.e.
  $x=\ep y+\ep m$ with an integer vector $m\in\mathbb{Z}^n$, $f(x,y)=F(x)$. This means that, for a given $y$, 
	$f(x,y)$ simply reads off the values of $F$ at 
	points $x$ with the phase $y$, but for other $x$ interpolates appropriately between those $\ep\,\square$-periodic points. 
Operator $\mathcal{I}_\ep$ has an explicit representation \eqref{7.54-1}, 
	which 
	appears to be a 
	two-scale analogue of classical
Whittaker-Shannon interpolation formula from signal processing, see Remark \ref{RemShann}. 
In fact, it is its two-scale nature which makes $\mathcal{I}_\ep$ an ($L^2$-isometric) {\it operator}, rather than just a ``formula'' classically. 
If the right hand side $F$ is itself a two-scale function, i.e. 
$F_\ep(x)=\Phi(x,x/\ep)$ where $\Phi(x,y)$ is 
$\square$-periodic in $y$, and its Fourier transform in $x$ is uniformly compactly supported with respect to $y$, 
then a natural two-scale analogue of classical Whittaker-Kotelnikov-Nyquist-Shannon sampling theorem implies that 
for sufficiently small $\ep$ simply $\big(\mathcal{I}_\ep F_\ep\big)(x,y)=\Phi(x, y)$, Remark \ref{RemShann}. 
Operators $\mathcal{I}_\ep$ and $\mathcal{J}_\ep$
 appear to resemble in some respects the periodic unfolding operator, see e.g. \cite{CDG}, but significantly 
differ from it 
in certain 
key aspects which are particularly important precisely for a wide class of degenerating models with a genuinely two-scale limit asymptotic behaviour.  
As we argue, our $\mathcal{J}_\ep$ based on $\mathcal{I}_\ep$ is the most natural connecting operator 
for such general classes of models (Remark \ref{ShannVsUnf}). 
\subsection{The approach and 
generalisations}
More generally, as we will see in this work too, asymptotic behaviour of a wide class of resolvent problems like \eqref{hom} plays crucial role  for properties of the related spectral problems as $\ep\to 0$. 
Moreover, the resolvent problems generally 
also hold keys for analysing related (two-scale) evolution problems, both parabolic and hyperbolic, cf. e.g. \cite{ZhP07,Pas05,IVKVPS13} for related results although only for two-scale convergences without rates. 
This all 
motivates efforts on establishing 
tractable but accurate leading-order 
operator approximations for a wider class of asymptotically degenerating problems akin to \eqref{hom}, 
with tight error estimates for their solutions for small $\ep$ uniformly with respect to $F$ (in various norms). 
\vspace{.05in}

In the context of non-degenerate PDEs or classical homogenisation (e.g. for problem \eqref{hom} corresponding to a fixed $\delta>0$ in 
\eqref{th-hc}), error estimates for the approximation given by the related homogenised equations are by now well-known, and various approaches exist to establish them. We shall 
not 
attempt to provide here a review of all these 
methods  
except to mention one 
of particular relevance to this work: the so-called {\it spectral method}, see 
for example \cite{Zh89,CoVa97,BiSu,ZhSpectr}, and also \cite{DGR23} and further references therein for 
recent developments of variants of the spectral approach for long-time homogenisation problems in both periodic and stochastic settings. 
The core of the spectral method in the context of periodic 
problems like \eqref{hom} is in applying a rescaled Floquet-Bloch-Gelfand 
transform and as a result reducing  \eqref{hom} to an equivalent family of problems on the periodicity cell (torus) 
$\square$, parametrised by 
quasiperiodicity variable (``quasimomentum'') $\t$ varying within the dual cell $\square^*=2\pi\square$. 
The latter problems have to be then meaningfully approximated, asymptotically for small $\ep$, uniformly with respect to $\t$.  
\vspace{.05in}

Degenerate problems however, like the above high-contrast problem \eqref{hom} with $\delta=\ep^2$, represent fundamentally new challenges 
for their treatment, including using the spectral method, 
precisely due to the fact that the classical homogenisation fails and the limit problem 
remains intrinsically two-scale.  
Within the spectral approach, 
this challenge is manifested in 
the associated 
 non-negative singular forms $a_\t$, see \eqref{pr} below, 
vanishing for every  $\t$ at 
non-trivial (in fact 
infinite-dimensional) subspaces $V_\t$ 
(near which the solutions $u_{\ep,\t}$ generally tend to be). 
To overcome these new challenges, we develop an approach which we 
believe bears 
fundamental novelties and allows to obtain new results for much wider classes of examples, both 
non-degenerate and degenerate,  see Section \ref{sec:examples}. 
Namely, 
as clarified below, we perform a robust uniform asymptotic analysis of families of generic 
variational problems \eqref{pr} 
and of corresponding self-adjoint operators near 
typical discontinuity points of $V_\t$ with respect to 
the parameter $\t$. 
This ultimately allows, in particular, to naturally arrive at 
a self-adjoint approximation to the exact solution operator in terms of abstract analogues of resolvent of the two-scale limit operator 
and of abstract connecting operators $A_\ep$ (Theorem \ref{thm.bivariate}). 
For the particular classes of two-scale periodic problems, the latter  
 appears to specialise to the above discussed two-scale connecting operators $\mathcal{J}_\ep$. 
The abstract 
nature of our approach implies that it can potentially be used 
even when the Floquet-Bloch transform may be not applicable or relevant 
(with one example in this direction given in Section \ref{sec:concpert}). 
\vspace{.08in} 

There has been some progress recently in obtaining approximations with uniform operator error estimates 
specifically for the high-contrast resolvent problems like \eqref{hom} 
with $\delta\sim\ep^2$.  
A leading-order approximation with $L^2$ error estimates 
in terms of an $\ep$-dependent two-scale type operator $\mathcal{L}_\ep^0$ 
(rather than the actual $\ep$-independent two-scale limit operator $\mathcal{L}_0$)
was established in \cite{ChCo}, 
using the spectral method as a basis. 
However 
the techniques employed therein appear insufficient for example for obtaining estimates on the rate of convergence of the spectrum 
to that of $\mathcal{L}_0$; moreover the methods developed in \cite{ChCo} are 
problem specific and not readily generalisable. 
 In  one-spatial dimension 
high-contrast 
models, leading-order approximations with error estimates were obtained in \cite{ChChCo} and \cite{ChKi} by different approaches. 
In \cite{ChErKi} and \cite{ChKiVeZu23}, in particular, approximations with improved error estimates, although still in terms of 
some $\ep$-dependent operators, were obtained for the scalar 
model \eqref{hom} and analogous linear elastic systems respectively via an 
asymptotic analysis of Dirichlet-to-Neumann maps under appropriate 
regularity assumptions on the 
coefficients and on the boundary of the inclusions. 
We re-emphasise here that upon applying our  general method to the above key example of multi-dimensional high-contrast model (Section \ref{e.dp}) we are able to obtain new operator estimates between the initial resolvent problem and that for a novel approximation via the ($\ep$-independent) two-scale limit operator $\mathcal{L}_0$, as well as the new results on the rates of uniform convergence of the spectra and of other spectral characteristics, 
all with no need for any regularity restrictions on the coefficients and less restrictive regularity of the inclusion's boundary. 
\vspace{.05in}




One of the 
wider aims of this article however is 
to demonstrate that a large class of problems of the above mentioned type (as well as many others) are all examples of  one particular 
generic  abstract family of asymptotically degenerating variational problems. As such, their leading-order asymptotics 
has a common structure reflecting the fact that these 
are all particular instances of an asymptotic 
approximation for that general variational problem. In this article we derive, under a range of abstract assumptions, a 
hierarchy of the leading-order asymptotics for this abstract problem with error estimates. We then specify the underlying abstract objects to provide  asymptotics (with operator-type error estimates) for various specific problems of interest.
\vspace{.05in}

As a way to motivate the general problem we  recall that the starting point in the spectral method, used in the above-mentioned 
$\ep$-periodic PDE setting 
\eqref{hom} where $A_\ep(x)=a_\delta(x/\ep)$,  
is following. 
Apply the rescaling $x \mapsto \ep y$ and then the Floquet-Bloch-Gelfand transform (see Section \ref{e.class}) to arrive at the family of problems on Sobolev 
space $H^1_{per}(\square)$ of $\square$-periodic functions, 
parametrised by the quasi-periodicity variable 
$\t$ on the dual cell $\square^*=[-\pi.\pi]^n$:  
\begin{equation}
	\left\{ \ \begin{aligned} & \text{For each $\t \in \square^*=[-\pi,\pi]^n$, find $u_{\ep,\theta} \in H^1_{per}(\square)$ such that} \\
		&	-\ep^{-2} \big( \nabla + \i \theta \big) \cdot a_\delta(y) \big(\nabla + \i \theta\big)  u_{\ep,\t}  \,\,+\,\,  u_{\ep, \t}\,\,=\,\, f, 
	\end{aligned}
	\right.\label{hom2}
\end{equation}
where $f$ and $u_{\ep,\t}$ are the transforms of (rescaled) $F$ and $u_\ep$ respectively. 
Next, we observe that the equivalent weak formulation of problem \eqref{hom2} is of the following more abstract variational form: 
\begin{equation}
	\left\{ \ \begin{aligned} & \text{For each $\ep >0$ and  $\t \in \Theta$, find $u_{\ep,\theta} \in H$ such that} \\
		&	\ep^{-2}
		a_\t\left(u_{\ep,\theta} ,\,\tilde u \right) \,\,+\,\, b_\t(u_{\ep, \theta},\tilde u) \,\,=\,\, \l  f,\tilde u\r,  \quad \forall \tilde u \in H,
	\end{aligned}
	\right.\label{pr}
\end{equation}
where $\Theta \subset \mathbb{R}^n$ is compact, $H$ is a complex Hilbert space, $f$ a bounded antilinear functional on  $H$, $a_\t$ and $b_\t$ are non-negative bounded sesquilinear forms such that $a_\t + b_\t$ is a family of uniformly equivalent inner products  on $H$, and  $a_\t$ is Lipschitz-continuous in $\t$ (see Section \ref{sec:pf} for the precise details). Notice that in many of our motivating examples the `singular' forms  
$a_\t$  have non-trivial degeneracy subspaces $V_\t = \left\{ u \in H \,\, | \,\, a_\t[u]:=
a_\t(u,u) =0\right\}$, 
that is why one can refer to such variational problems as { 
asymptotically  degenerating}.

Analysing the high-contrast problem \eqref{hom} stated in the transformed variational form \eqref{pr}, we observe that there are very 
few generic features of this abstract family of asymptotically degenerating variational problems for which 
two-scale type approximations with uniform error estimates can be constructed. 
This naturally leads to the 
idea of investigating \eqref{pr} under minimal abstract assumptions, thereby hoping to cover a 
wider class of interesting problems. 
This all is what ultimately allows us to successfully analyse and obtain new results for a wide class of asymptotically degenerate problems 
with relative ease. 
Indeed, we observe that 
the formulation \eqref{pr} does not just cover the above classical or high-contrast type settings \eqref{hom2} (for $\delta>0$ fixed and $\delta=\ep^2$ respectively, with corresponding $a_\t$, $b_\t$ and $V_\t$), but also a much wider class of interesting problems. 
In particular, 
it 
 includes 
models as diverse as the following. 
(In the list below, with reference to the examples in Section \ref{sec:examples}, 
we highlight 
some key features of these diverse models and of 
specific results we are able to obtain for those by applying our general method.) 
\vspace{.04in}

-- 
`Inverted' high-contrast problem (with stiff periodic inclusions in a soft matrix,  
Section \ref{e:idp}), with resulting approximation (accompanied by operator 
error bounds) of a different nature: 
by an infinite contrast `rigid inclusions' model 
rather than a two-scale one.  

-- 
Inclusions with `weakly bonded' imperfect interfaces (rather than with a high contrast), Section \ref{sec:impint}, where Hilbert space $H$ is not anymore 
$H^1_{\rm per}(\square)$ due to the interface discontinuities. 
The limit problem appears to be a coupled ``two-phase'' macroscopic (rather than a more general two-scale) one,  
 and displays a band gap at a single inclusion resonance, which all is accompanied by the tight 
 error bounds.

-- Problems with concentrated perturbations, which example in  Section \ref{sec:concpert} we expect to be particularly instructive for 
demonstrating the additional powers and potential brought in by the generality of the abstract approach. Indeed, in this case 
$\t$ is 
{\it not} anymore 
the Floquet-Bloch quasi-periodicity parameter, i.e. $\Theta\ne\square^*$ and contains an additional component accounting for the 
parameter of 
concentration $\delta$, which allows to obtain approximations with estimates 
uniform also with respect to $\delta$. 
This example is also instructive for demonstrating applicability of our more general approach even when the singular form $a_\t$ 
is not smooth in $\t$ at the underlying key degeneracy point ($\t_0=0$). 


-- Linear elasticity system with  inclusions which are `partially degenerating', i.e. whose elasticity tensor asymptotically degenerates on only some of 
its components (Section \ref{e.pdelast}).  

-- Schr\"{o}dinger operators 
with a strong periodic magnetic field (Section \ref{magnschrod}), with resulting 
shifted quasimomentum point $\t_0\ne 0$ of 
discontinuity 
of $V_\t$. 

-- Differential-difference equations (Section \ref{sec:nonloc}), with a genuinely non-quadratic (and not even polynomial) dependence of $a_\t$ on $\t$ due to the non-locality of the model. 

 -- Difference equations (Section \ref{sec:nonloc2}), where in the related two-scale limit problem the ``homogenised form'' $a^h_\t$ appears to 
specify an infinite-dimensional operator rather than just a (finite-dimensional) matrix.
\vspace{.04in} 

Further examples which are not covered in the present work but also fall into the abstract framework of \eqref{pr} include: 
a wide class of partially degenerating high-contrast PDE systems
(cf. \cite{IVKVPS13} where similar wide classes were studied although with results only on two-scale convergence rather 
than with any error estimates); 
 homogenisation problems on periodic quantum graphs and their generalisations, cf. e.g. \cite{IVKVPS19};  
problems in thin domains; 
problems on discrete periodic lattices; 
some higher-order differential and pseudo-differential operators. 
We emphasise here that we do not generally require in \eqref{pr} the forms $a_\t$ and $b_\t$ to  be generated by differential operators, nor do  we require $\t$ to 
necessarily be the  Floquet-Bloch parameter (cf. again the concentrated perturbation example of Section \ref{sec:concpert})  or even for $H$ to be a function space. This suggests possible far-reaching consequences of the present approach that 
can go  even further beyond the scope of the examples outlined  above. 


\subsection{Structure of the article and of the main results}

Let us now describe in a more specific way the structure of the article, the main technical ideas, and the main results. 
In Section \ref{sec:pf} we formulate the abstract problem and introduce our 
main assumption \eqref{KA} 
that can be regarded as a variant 
of a  
spectral gap 
condition. Namely, for every $\t\in\Theta$ the singular form $a_\t$ in \eqref{pr} is coercive 
(although generally 
non-uniformly in $\t$) on the orthogonal complement $W_\t$ of its null-space $V_\t$.  This condition is a far-reaching generalisation of a `key assumption' introduced in \cite{IVKVPS13}, found to be important in establishing the two-scale  convergence to two-scale homogenisation limits for a general class of partially degenerating elliptic PDE systems of type \eqref{hom} in general domains. 

In Section \ref{s:uniforma}, we show that if  the null-space $V_\t$ is Lipschitz continuous in 
$\t$ then $a_\t$ is uniformly coercive   
in $\t$ on $W_\t$,  
and as a result  the leading-order approximation  simply comes from  `projecting'  problem \eqref{pr} onto $V_\t$, 
see Theorem \ref{thm:contV}. This simple result not only forms the basis for further investigation, but 
appears applicable to certain physically relevant models, for example,  the stiff inclusion or { inverted } high contrast model of Example \ref{e:idp}  and in the study of certain polarisations of electromagnetic waves in photonic crystal fibers, cf. \cite{Cothesis,CoKaSmPCF}. 

In Section \ref{section:discV} we study the case of discontinuous $V_\t$. 
  This situation is typical in  
	examples  such as the above both 
	classical and 
high-contrast problems of type \eqref{hom}, and  corresponds to loss of the $\t$-uniformity of the spectral gap. 
This requires 
 a much more subtle asymptotic analysis 
near related 
singular points, 
ensuring 
certain 
almost-orthogonality via operator $\N$ which is in a sense an abstract version of the classical 
corrector, see \eqref{IliaN}.  
As a result, 
  in particular, in Theorem \ref{thm1.all} we construct a leading-order approximation to 
	\eqref{pr}  when the null-space $V_\t$ possesses an isolated singularity (say at 
	$\t=0$) that is removable in the following sense:  there exists a subspace $V_\star$ such that 
$	
	V^\star_\t  = \left\{ \begin{array}{cc}
		V_\t, & \t \neq 0, \\ V_\star, & \t = 0
	\end{array} \right.  
$ is Lipschitz continuous, see \eqref{contVs}. 
The resulting approximate problem \eqref{coupledbest} is on a ``sum'' of $V^\star_\t$ and a  
`defect subspace' $Z$ (describing the discontinuity gap between $V_\star$ and $V_0$), with  
$\M z=z+\N z$.   
The results of Section \ref{section:discV} are found 
useful for some 
applications, see e.g. the example with 
concentrated perturbations of Section \ref{sec:concpert}.
\vspace{.02in}

Approximating problem \eqref{coupledbest} 
is simpler than \eqref{pr}, 
but still depends on both $\ep$ and $\t$. In Section 
\ref{sec.2dif}, we provide a further approximation with even simpler self-similar $\ep$ and $\t$ dependencies via their 
ratio $\t/\ep=:\xi$. This is done by approximating 
 the forms $a_\t$ and $b_\t$ for small $\t$, which can be performed under  
additional 
$\t$-quadratic degeneracy condition for the spectral gap \eqref{distance} and 
mild regularity assumptions \eqref{H4} and \eqref{H5} on $a_\t$ and $b_\t$ at $\t=0$, 
that are readily observed in many (even if not all, cf. again the example in Section \ref{sec:concpert} mentioned above) 
examples. 

This leads us to one of our main abstract results, Theorem \ref{thm.IKunifest2}, that provides uniform approximations (in abstract analogues of both $L^2$ and $H^1$ ``energy'' norms) to the two-parameter solution $u_{\ep,\t}$ of \eqref{pr} in terms of solutions to a one-parameter 
family of variational problems 
\eqref{IKz3prob88} 
on the fixed smaller space $V_0=V_\star\dot{+}Z$ 
with sesquilinear forms 
$a^{\rm h}_\xi \,+\, b_0$, $\xi=\t/\ep \in \RR^n$. 
An important additional feature however, which plays a key role specifically for genuinely degenerate problems i.e. those with non-trivial $V_\t^\star$ (as is the case in high-contrast problems of type \eqref{hom} 
 with $\delta=\ep^2$ but 
{\it not} in classical homogenisation problems like \eqref{hom} with a fixed $\delta>0$), is the emergence both on the right-hand side of the approximating problem \eqref{IKz3prob88} and in the approximations of Theorem \ref{thm.IKunifest2} of an isometric ``transfer operator'' 
$\mathcal{E}_\t:V_\star\to V_\t^\star$. 
Operator $\mathcal{E}_\t$ 
accounts for $\t$-dependence of the ``regular'' form $b_\t$ restricted 
on $V_\t^\star$, and can be naturally identified in most of the relevant examples,  
with its existence generally assured by Lemma \ref{propeth}. 
It is the presence of $\mathcal{E}_\t$ which necessitates the above discussed translation operator $T_\ep$ in two-scale periodic problems. 
The non-negative abstract ``homogenised'' form $a^{\rm h}_\xi\left(z,\tilde z\right)$ 
is defined by \eqref{defhom.form} in terms of an abstract ``linearised'' corrector $N_\t$ solving \eqref{cell:prob2}, 
acts on the even smaller 
defect subspace  
$Z$, 
is a quadratic form  
in $\xi$ and is non-degenerate on $Z$. 
Form $a_\xi^{\rm h}$ appears to  generalise that for (symbols of) 
the homogenised operator 
in classical homogenisation problems, and $Z$ 
is found to be finite-dimensional under a stronger version \eqref{KA2} of the spectral gap condition \eqref{KA} which typically holds in many practical examples, see Section \ref{sec.newKA}. 
(Notice however a simple example of a difference equation in Section \ref{sec:nonloc2}, where \eqref{KA2} does not hold. 
Nevertheless \eqref{KA} still holds, and our general 
scheme is  applicable and despite infinite-dimensional $Z$ yields explicit approximations with error estimates, 
Theorem \ref{thm.2scOpResnonloc}.) 
\vspace{.1in} 

The significance of the dependence of (left hand side of) the approximating problem \eqref{IKz3prob88} only  on single parameter $\xi=\t/\ep$ manifests itself  in operator and spectral results of Section \ref{s:resolv} where 
approximations in terms of an abstract version of a `two-scale' limit operator, with principal symbol  $a^{\rm h}_\xi$,  are constructed.  
The section  
focuses on the associated abstract spectral problems in an ambient Hilbert space $\mathcal{H}\supset H$: 
\begin{equation}
	\left\{ \ \begin{aligned} & \text{For each $\ep >0$ and  $\t \in \Theta, \,\,\,$ find $\lambda_{\ep,\t}\in [0,\infty)\,$ with 
	$\,u_{\ep,\theta} \in H\backslash \{0\}$ such that} \\
		&	\ep^{-2}\,
		a_\t\left(u_{\ep,\theta},\,\tilde u \right) \,\,+\,\, b_\t\left(u_{\ep, \theta},\,\tilde u\right) \,\,=\,\, 
		\lambda_{\ep,\t}\,\, d_\t\left(u_{\ep,\t},\,\tilde u\right)  \quad \, \forall \tilde u \in H. 
	\end{aligned}
	\right.\label{prspec}
\end{equation}
Here 
$d_\t$ is an inner product in $\mathcal{H}$ and a 
compact sesquilinear form on $H$ (which is dense in $\mathcal{H}$) 
and such that $b_\t-d_\t$ is nonnegative on $H$. 
Then \eqref{prspec} is a spectral problem for associated 
positive self-adjoint operator $\mathcal{L}_{\ep,\t}$ in $\mathcal{H}$, 
with a spectrum ${\rm Sp}\,\mathcal{L}_{\ep,\t}$. 
We then show 
that 
condition \eqref{H6} of a unitary extensibility of the transfer operators $\mathcal{E}_\t$,  
that is typically observed in examples, implies that 
the spectrum of $\mathcal{L}_{\ep,\t}$ is uniformly approximated by that of 
operator $\mathbb{L}_{\t/\ep}$ generated by the `homogenised' form $\mathbb{S}_{\t/\ep}=a^{\rm h}_{\t/\ep}+b_0$, 
Theorems \ref{p.unitaryequiv} and \ref{ikthm2}. 
We then establish, by controlling ${\rm Sp}\,\mathbb{L}_{\xi}$ for large $\xi$, that the ``collective spectrum'' 
(the closure of the union of the spectra $\bigcup_{\t \in \Theta} {\rm Sp}\,\mathcal{L}_{\ep,\t}$)  
converges in appropriate sense with rate 
$\ep$, Theorem \ref{t.collectivespec} and Corollary \ref{c.collspec}. 
Moreover, in Theorem \ref{thm.bivariate} we approximate the 
inverse $\mathcal{L}_{\ep,\t}^{-1}$ 
in terms  of certain self-adjoint `bivariate'  operator $\mathcal{L}$, which is an abstract version of two-scale limit operator 
 and is related to form  $\mathbb{S}_\xi$ via  
(inverse) 
Fourier transform (i.e. $\xi \,\mapsto\, -\,\i\,\nabla_x$). 
An important role is played in the approximation 
by the already mentioned 
abstract $L^2$-isometric connecting operator $A_\ep$, which accounts in particular  for the effect of the above transfer operator $\mathcal{E}_\t$. 
Operator $A_\ep$ serves as an abstract prototype of the above discussed two-scale connecting operator $\mathcal{J}_\ep$ 
in the periodic PDE problems. 
We then show 
(Theorem \ref{bivariate.spec}) 
that the limit spectrum coincides with the spectrum of $\mathcal{L}$, and hence the collective spectrum 
converges to that of $\mathcal{L}$ with rate $\ep$. 
The key bivariate operator $\mathcal{L}$, specified by form \eqref{Q}, can be viewed as a second-order constant-coefficient differential operator acting in the Bochner space 
$L^2\left(\RR^n ;\, \overline{V_0} \right)$, where $\overline{V_0}$ is the closure of $V_0=V_\star\dot{+}Z$ in $\mathcal{H}$. 
It can be seen to generalise  the two-scale 
limit operators for various 
high-contrast models. 
The spectrum of the abstract bivariate operator $\mathcal{L}$ is 
explicitly 
characterised (Theorem \ref{bivariate.spec}) 
in terms of 
eigenvalues of certain operators on $\overline{V_0}$ and $\overline{V_\star}$, \eqref{valthm}--\eqref{valthm2}, 
or equivalently by \eqref{spLbeta} in terms of an 
operator-valued function $\beta(\lambda)$, \eqref{betaform}--\eqref{6.27-2}, 
generalising in some way the scalar Zhikov's $\beta$-function introduced in \cite{Zhi2000,Zhi2005} for the 
model high-contrast problem \eqref{hom}. This is turn provides an asymptotic characterisation, with error estimates, for gaps in the collective spectrum,  
which in particular leads  to new estimates for the gaps in the Floquet-Bloch spectrum in various specific examples 
(Section \ref{sec:examples}). 
\vspace{.08in} 

Extensive Section \ref{sec:examples}, already briefly reviewed above, aims at demonstrating the power and versatility of our abstract results by 
applying them 
to a diverse set of physically motivated examples. 
We thereby 
obtain a number of 
new results for various high-contrast and some other asymptotically 
 degenerating problems. Each of the problems is picked not only for their wider relevance, but also to demonstrate a particular feature and breadth of the article's main assumptions and results. 
While in some examples we go into fine details for demonstrating the full power of the developed general methods, 
in others we 
do not pursue a maximal generality 
but do quite the opposite: try to 
present a simpler example displaying a particular feature and effect leading to a specific result. 
\vspace{.08in} 

We hope that 
the proposed approach, in particular the 
generality of the abstract scheme and versatile features of the emerging accompanying tools,  
could have a significant potential for 
wide-ranging further developments and applications 
beyond those discussed here. 
Moreover, the presented general 
scheme retains sufficient flexibility for adapting it to various further models outside periodic high-contrast two-scale spectral homogenisation, in particular allowing to relax further the assumptions on the 
abstract analog $\Theta$ of the dual cell $\square^*$ and on the key spectral gap condition \eqref{KA}. 
Some of these avenues are already being pursued, and will be reported elsewhere. 

\section{Abstract problem formulation}
\label{sec:pf}
An abstract setup for the general class of problems under consideration in this article is as follows. 
Let  $H$ be a separable complex Hilbert space with a family of 
non-negative sesquilinear forms\footnote{
For a non-negative sesquilinear form 
$\mathfrak{b}:H\times H\to\CC$, $\mathfrak{b}[u]:= \mathfrak{b}(u, u)$ is non-negative real $\forall u\in H$. 
Then $\mathfrak{b}$ is complex-Hermitian, 
with Cauchy-Schwarz and triangle inequalities held, i.e. $\mathfrak{b}(u,\tilde u) = \overline{\mathfrak{b}(\tilde u,u)}$, 
$|\mathfrak{b}(u,\tilde u)|\le \mathfrak{b}^{1/2}[u]\,\mathfrak{b}^{1/2}[\tilde u]$, 
$\mathfrak{b}^{1/2}[u+\tilde u]\le \mathfrak{b}^{1/2}[u] + \mathfrak{b}^{1/2}[\tilde u], 
 \ \forall\, u,\tilde u \in H$. 
(Here $\mathfrak{b}^{1/2}[u]:=\left(\mathfrak{b}[u]\right)^{1/2}$.) 
We shall also occasionally use simple implications, ``squared'' triangle inequalities: 
$\mathfrak{b}[u_1+u_2]\le 2\,\mathfrak{b}[u_1] + 2\,\mathfrak{b}[u_2]$, 
$\mathfrak{b}[u_1+u_2+u_3]\le 3\,\mathfrak{b}[u_1] + 3\,\mathfrak{b}[u_2]+3\,\mathfrak{b}[u_3]$, $\forall\,u_1$, $u_2$, $u_3\in H$. 
}
$a_\t$ and $b_\t$ 
parametrised by $\t$ varying in a compact subset $\Theta$ of $\RR^n$, 
$n\geq 1$. 
We assume throughout that 
\begin{equation}\label{astructure}
(u,\tilde{u})_\t \,\,:=\,\, a_\t(u,\tilde{u}) + b_\t(u,\tilde{u}), \qquad u,\tilde{u} \in H,
\end{equation} 
form a family of uniformly equivalent 
inner products on $H$, i.e. for 
the norms  $\| u \|_\t : = (u,u)_\t^{1/2}$ 
\begin{equation}
\label{as.b1}
\text{there exists 
$K > 0$ such that} \ \|u\|_{\t_1} \,\,\le\,\, K_{}\, \|u\|_{\t_2}, \quad \forall u \in H,\ \forall\, \t_1,\t_2 \in \Theta.
\end{equation}
Furthermore, we assume that the forms $a_\t$ are Lipschitz continuous with respect to $\t$ in the following sense: there exists   $L_{a} > 0$ such that
\begin{align}
\label{ass.alip}
&\big| a_{\t_1}(u,\tilde{u}) - a_{\t_2}(u,\tilde{u}) \big| \,\,\le\,\, L_{a} \big| \t_1 - \t_2 \big|\, {\|u\|_{\t_1}}{\|\tilde{u}\|_{\t_1}}, \quad \forall u,\tilde{u} \in H,\,\,\, \forall \t_1,\t_2 \in \Theta.
\end{align}
We consider a general class of problems reducible to the following common abstract variational form. 
For any given  $0<\ep <1 $, $\t \in \Theta$, and $f \in H^*$, 
\begin{equation}
\label{p1}
\left\{ \ \begin{aligned} & \text{find $u_{\ep,\theta} \in H$ such that} \\
& \ep^{-2} a_\t\left(u_{\ep,\theta} ,\tilde{u}\right) \,\,+\,\, b_\t \left(u_{\ep,\theta},\tilde{u}\right) \,\,\,=\,\,\, \l f,\tilde{u}\r, \quad \forall \tilde{u} \in H.  
\end{aligned}
\right.
\end{equation}
 Here $H^*$ is the 
space of anti-linear continuous functionals on $H$ and $\l\cdot,\cdot\r$ denotes 
the duality pairing. 
For any fixed $\ep>0$, $\t \in \Theta$,
\begin{equation}\label{Aform}
A_{\ep,\t}(\cdot,\cdot) \,\,: =\,\, \ep^{-2} a_\t(\cdot,\cdot) \,+\, b_\t(\cdot,\cdot) 
\end{equation} is an equivalent inner product for $H$, and therefore problem \eqref{p1}  is well-posed. 
Our 
aim is to establish asymptotic approximations  of the solution $u_{\ep,\t}$ with respect to small $\ep$ that are uniform in an appropriate 
sense in both $\t$ and $f$. 

For each $\t$, we introduce the set of degeneracy or the kernel of the ``singular'' form $a_\t$ 
\begin{equation}
\label{spaceV}
V_\theta \,\,: =\,\, \big\{ v \in H \,\,\, \big| \,\,\, a_\t[v] = a_\t(v,v) = 0 \big\},
\end{equation}
denoting henceforth,  for a sesquilinear form $\mathfrak{b}$,   
$\mathfrak{b}[v]:=\mathfrak{b}(v,v)$. 
Notice that,  
as $a_\t$ is non-negative, 
\begin{equation}\label{Vkersesa}
a_\t(v,u) \,=\, a_\t(u,v)\, =\, 0, \quad \forall v \in V_\t,\,\,\, \forall u \in H.
\end{equation}
The boundedness of $a_\t$ and 
\eqref{Vkersesa} imply that $V_\t$ is a closed linear subspace of $H$. 
Let  $W_\theta$, another closed linear subspace of $H$, be the orthogonal complement of $V_\theta$ in $H$ with respect to the inner product $(\cdot,\cdot)_\t$: 
\begin{equation}
W_\t\,\,:=\,\,\big\{w\in H\,\,\, \big|\,\,\, (w,v)_\t=0, \,\, \forall v\in V_\t\big\}. 
\label{2.6-w}
\end{equation} 
The main assumption of the article is the following pointwise in $\t$ (spectral) {\it  gap condition:}
\begin{equation}\tag{H1}
\label{KA}
\ \begin{aligned}
& \text{$\forall \,\t \in \Theta$, $\,\, \exists\, \nu_\theta>0$ such that  $\forall w \in W_\theta$ the inequality 
$a_\t [w] \,\,\ge\,\, \nu_\t\, \|w\|_\t^2\,\,$ holds}.
\end{aligned} 
\end{equation}
We emphasise that the above gap condition is generally {\it non-uniform}: in fact, in most of the interesting examples (Section \ref{sec:examples}),
 $\inf_{\t\in\Theta}\,\nu_\t=0$. 
\begin{remark}
\label{spgap}
To see why \eqref{KA} can be interpreted a spectral gap condition, notice that 
for every $\t\in\Theta$ the form $a_\t$ defines a non-negative bounded self-adjoint operator in $H$ with say  
inner product $(\cdot,\cdot)_\t$. Condition \eqref{KA} together with \eqref{astructure} implies that the 
spectrum of this operator is contained in 
$\{0\}\cup [\nu_\t,1]$, in particular if both $V_\t$ and $W_\t$ are nontrivial then $(0,\nu_\t)$ is in the gap of the spectrum. 
\end{remark}
\begin{remark}
		\label{r.oldkafromnewka}  
		In a wide class of examples (see Section \ref{sec:examples}) one can 
		verify that the following further strengthening 
		(see Proposition \ref{prop.kaequiv}) of condition \eqref{KA} holds. 
		There exists $C >0 $ and a  non-negative sesquilinear form $c$, $\|\cdot\|_\t$-compact (see Section \ref{sec.newKA} for the precise definition) for all $\t\in \Theta$, such that
		\begin{equation}\tag{H1$^\prime$}
		\label{KA2.1}
		\|w\|_\t^2 \,\,\,\le\,\, C a_\t[w]\,\, +\,\, c[w], \quad \forall w \in W_\t, \; \forall\, \t \in \Theta.
		\end{equation}
In particular, we will see that  \eqref{KA2.1} is self-evident in the context of classical homogenisation problems, although 
already requires employing certain extension theorems for ``non-classical'' high-contrast models. 
In Section \ref{sec.newKA}, we shall see that \eqref{KA2.1} does not only imply 
\eqref{KA}\footnote{In fact, for implying \eqref{KA}, \eqref{KA2.1} can be slightly weakened by allowing both $C$ and $c$ to 
depend on $\theta$, although in our examples those appear $\t$-independent.} 
but has other important implications. 
Condition \eqref{KA2.1} can be re-stated as 
the forms $a_\t$ being 
(uniformly) coercive on $W_\t$ plus compact.
\end{remark}
\section{The case of a continuous $V_\theta$ (uniform spectral gap)}
\label{s:uniforma}
As we shall see, the asymptotics of the solution to \eqref{p1} crucially depends on certain continuity properties of 
the degeneracy subspace $V_\t$ with respect to $\t$. We begin with the  simple case of the spectral gap $\nu_\t$ being uniform in $\t$ and then we shall characterise this condition in terms of the continuity of $V_\t$. 
\subsection{The case of a $\t$-uniform  gap}
For a fixed $\t$ and small $\ep$, the solution $u_{\ep,\t}$ to \eqref{p1} is expected to be close to the 
null-space 
$V_\t$ of $a_\t$. So it is natural 
to seek an 
approximation 
by first restricting \eqref{p1} to $V_\t$, i.e. ($\ep$-independent) $v_\t\in V_\t$ 
solving 
\begin{equation}
\label{thmcontv.vprob}	b_\t\left(v_\theta,\tilde{v}\right) \,\,=\,\, \l f\,,\tilde{v}\r, \quad \forall \tilde{v} \in V_\theta.
\end{equation}
Then 
the ``error'' $w_{\ep,\t}:=u_{\ep,\t}-v_\t$ is orthogonal to $V_\t$ with respect to $A_{\ep,\t}$ and hence also 
with respect to $(\cdot,\cdot)_\t$,  cf. \eqref{Aform},  \eqref{astructure} and \eqref{Vkersesa}. 
In other words, $w_{\ep,\t}\in W_\t$ and because of the orthogonality 
\begin{equation}
\label{wprob} 
A_{\ep,\t}\left(w_{\ep,\t}\, , \widetilde{w}\right) \,=\, 
\ep^{-2}a_\t\left(w_{\ep,\t} \,, \widetilde{w}\right) \,+\,	b_\t\left(w_{\ep,\theta}\,,\widetilde{w}\right) 
\,\,=\,\, \l f\,,\widetilde{w}\r, \quad \forall \widetilde{w} \in W_\theta.
\end{equation}
Now, in a standard way, setting in \eqref{wprob}  $\widetilde{w}=w_{\ep,\t}$ and recalling the {spectral gap condition} \eqref{KA}, 
\[
A_{\ep,\t}[w_{\ep,\t}]\,=\, \l f,w_{\ep,\t}\r\,\,\le\,\, \| f\|_{* \t}\,\|w_{\ep,\t}\|_\t\,\,\le\,\, 
\| f\|_{* \t}\, \nu_\t^{-1/2}a_\t^{1/2}[w_{\ep,\t}]\,\,\le\,\,  \|f\|_{* \t}\,\, \nu_\t^{-1/2}\ep\, A_{\ep,\t}^{1/2}[w_{\ep,\t}]\,, 
\]
where
\begin{equation}
\label{fstar}
\| f \|_{* \t} \,\,\,: =\,\, \sup_{u \in H \backslash \{0 \}} \frac{| \l f,u \r |}{{\|u\|_\t}}\,.
\end{equation} 
As a result, for the approximation error $w_{\ep,\t}=u_{\ep,\t}-v_\t$, 
\begin{equation}\label{1}
A_{\ep,\t}[w_{\ep,\t}]\,=\,\ep^{-2} a_\t[w_{\ep,\theta}] + b_\t[w _{\ep,\theta}] \,\,\le\,\,    \ep^2\, \nu_\t^{-1}\| f\|_{* \t}^2, \quad \ \ \forall \ep>0\,. 
\end{equation}
 Moreover, another application of \eqref{KA} and \eqref{1} gives  
\begin{equation}\label{1-2}
 \|w _{\ep,\theta}\|_\t^2\,\,\le\,\, \nu_\t^{-1}\,a_\t[w _{\ep,\theta}]\,\,\le\,\, \nu_\t^{-1}\ep^2 A_{\ep,\t}[w _{\ep,\theta}]\,\,
 \le  \,\,  \ep^4\, \nu_\t^{-2}\,\| f\|_{* \t}^2\,.
\end{equation} 
If the spectral gap is uniform in $\t$, 
regarding \eqref{thmcontv.vprob} as an approximate problem,  
\eqref{1} and \eqref{1-2} immediately provide  the following simple error estimates. 
\begin{theorem}\label{thm:contV}
Assume that 
\begin{equation}
\label{bddspec}
\text{there exists } 
\nu >0\ \text{ such that }\ a_\t[w]\,\ge\,\,\nu \|w\|_\t^2, \quad \forall w \in W_\t, \, \forall \t \in \Theta.
\end{equation}
 Then 
 for $u_{\ep,\theta} \in H$ the  solution to \eqref{p1} and   $v_\theta\in V_\t$ the solution to \eqref{thmcontv.vprob}, 
\begin{align}
\label{errorcontinuouscase}
\ep^{-2} a_\t\left[u_{\ep,\theta} - v_\theta\right] \,+\, b_\t\left[u_{\ep,\theta} - v_\theta\right]  \,\,\,\le\,\,\,    \ep^2\, \nu^{-1}\,\| f\|_{* \t}^2, \\ 
\label{errorcontinuouscase2}
\left\| u_{\ep,\theta} - v_\theta \right\|^2_\t  \,\,\,\le\,\,\,    \ep^4\, \nu^{-2}\,\| f\|^2_{* \t}.
\end{align}
\end{theorem}
\begin{remark}\label{rem3.2}
Theorem \ref{thm:contV} clearly holds also `locally', i.e. with $\Theta$ replaced by any of its subsets $\Theta'$ such that 
assumption \eqref{bddspec} is satisfied only on $\Theta'$ rather than on the whole of $\Theta$. 
The theorem and its proof remain valid for all $\ep>0$ (i.e. not only for $0<\ep<1$, as assumed above). 
\end{remark}
\begin{remark}
	\label{r.normfindependentoft}
Note that while the right-hand-sides of \eqref{errorcontinuouscase} and \eqref{errorcontinuouscase2} formally depend on $\t$, 
this dependence is easily removed by \eqref{as.b1} :
$
\| f \|_{* \t_1} \le {K} \| f \|_{* \t_2}, \ \forall \t_1,\t_2 \in \Theta.
$ 
Therefore \eqref{errorcontinuouscase} and \eqref{errorcontinuouscase2} provide desired error estimates for small 
$\ep$, which are uniform in both $\t$ and $f$. 
\end{remark}
\subsection{A characterisation of  forms $a_\t$ with uniform gap condition}
\label{sec.contV}
In applications, the direct verification of \eqref{bddspec} can be complicated. An equivalent but often 
easier to verify condition relies on a notion of continuity of the degeneracy subspace $V_\t$ in $\t$ that we shall introduce now. 
Namely, we say that $V_\t$ is Lipschitz continuous with respect to $\t$ on $\Theta$ if  
\begin{equation}
\label{VtLip}
\exists L_V >0\ \text{ such that }\ 
\forall\, \t_1,\t_2 \in \Theta, \ \forall v_{1} \in V_{\t_1}, \quad 
\inf_{v_{2} \in V_{\t_2}} \left\|v_{1} - v_{2}\right\|_{\t_2} \,\,\le\,\, L_V \left| \t_1 - \t_2 \right|\, \|v_{1}\|_{\t_1}.
\end{equation}

As $\inf_{v_2 \in V_{\t_2}} \| v_1 - v_2\|_{\t_2} = 
\left\|P_{W_{\t_2}} v_{1}\right\|_{\t_2}$, where $P_{W_\t} : H \rightarrow W_\t$ is the orthogonal projection on $W_\t$ with respect to $( \cdot,\cdot)_\t$,  the inequality in \eqref{VtLip} is 
equivalent to 
\begin{equation}
\label{VtLip2}
\left\|P_{W_{\t_2}} v_{1}\right\|_{\t_2} \,\,\le\,\, L_V | \t_1 - \t_2 |\, \|v_1\|_{\t_1}, \quad \forall v_{1} \in V_{\t_1},  \ \forall\, \t_1,\t_2 \in \Theta.
\end{equation} 
The following result establishing, under assumption \eqref{KA},  the equivalence of the gap uniformity 
property \eqref{bddspec} and of the $V_\t$ continuity 
property \eqref{VtLip} holds\footnote{An intuition behind is that the $\t$-continuity property \eqref{ass.alip} of $a_\t$ implies certain regular behaviour 
 of the related spectra, cf. 
Remark \ref{spgap}. So, as long as the spectral gap remains uniformly positive, the zero eigenspace $V_\t$ can vary with $\t$ only continuously, while if the uniformity is 
violated on $\t$ approaching a point $\t_0$ this can be only be due to an instant addition of a non-trivial subspace to $V_\t$ at $\t=\t_0$.}.
\begin{theorem}\label{thm.contVequiv} Assume \eqref{KA}. Then  \eqref{bddspec} holds if and only if  \eqref{VtLip} holds.
\end{theorem}
\begin{proof}
			\emph{ Proof of \eqref{bddspec} $\hspace{-5pt}\implies\hspace{-5pt}$ \eqref{VtLip}}. 
Let $\t_1, \t_2 \in \Theta$ and 
$v_1 \in V_{\t_1}$. By \eqref{bddspec}, 
\eqref{Vkersesa} 
and \eqref{ass.alip} we obtain  
\[
\begin{aligned}
\| P_{W_{\t_2}} v_1 \|_{\t_2}^2\,\, & \le\,\,\, \nu^{-1} a_{\t_2} [P_{W_{\t_2}} v_1]\,\,=\,\, \nu^{-1} a_{\t_2} \left( P_{W_{\t_2}}v_1, v_1\right) \,\,=\,\, 
\nu^{-1}  \Big( a_{\t_2} \left( P_{W_{\t_2}}v_1, v_1\right) -  a_{\t_1} \left( P_{W_{\t_2}}v_1, v_1\right) \Big) \\
&\le\,\,\, \nu^{-1} L_a |\t_1 - \t_2 | \| P_{W_{\t_2}} v_1 \|_{\t_1}\| v_1 \|_{\t_1}.
\end{aligned}
\]
Hence, after an application of \eqref{as.b1}, \eqref{VtLip2} holds with $L_V =  \nu^{-1} L_aK$ and 
therefore so does \eqref{VtLip}. 

{\emph{Proof of \eqref{VtLip} $\hspace{-5pt}\implies\hspace{-5pt}$	\eqref{bddspec}}}. 
Suppose \eqref{bddspec} does not hold. Then there exists a convergent sequence $\t_n \in \Theta$ with limit $\t_0\in\Theta$, and a sequence $w_n \in W_{\t_n}$ such that $\| w_n \|_{\t_n} =1$ and $\lim_{n}a_{\t_n}[w_n] =0$. Now 
\begin{equation}\label{6520e}
1 \,=\,   \left\| w_n \right\|_{\t_n}^2 \,\,=\,\,  \left( P_{V_{\t_0} } w_n , w_n \right)_{\t_n}\,+\,\left( P_{W_{\t_0} } w_n , w_n \right)_{\t_n},
\end{equation}
where $P_{V_{\t_0}}$ is the orthogonal projector on $V_{\t_0}$ with respect to 
$(\cdot,\cdot)_{\t_0}$. For the contradiction, we will show that both terms on the right of \eqref{6520e} 
converge to zero. 
By \eqref{as.b1}
\[
\left| ( P_{W_{\t_0} } w_n , w_n )_{\t_n}\right| \,\,\le\,\, \left\| P_{W_{\t_0}} w_n \right\|_{\t_n} \left\| w_n \right\|_{\t_n}\,\, =\,\, 
\left\| P_{W_{\t_0}} w_n \right\|_{\t_n} \,\,\le\,\, K\,\left\| P_{W_{\t_0}} w_n \right\|_{\t_0},
\]
and we claim that $\lim_n \| P_{W_{\t_0}} w_n \|_{\t_0} =0$. Indeed, by  \eqref{KA} for $\t=\t_0$, \eqref{Vkersesa} and \eqref{ass.alip}, 
\[
\begin{aligned}
\left\| P_{W_{\t_0}} w_n \right\|_{\t_0}^2\,\,\, & \le\,\, \nu_0^{-1}\, a_{\t_0}\left[ P_{W_{\t_0}} w_n\right] \,\,=\,\, \nu_0^{-1} a_{\t_0}\left[ w_n\right]  
\,\,\le\,\, \nu_0^{-1} a_{\t_n}\left[ w_n\right]  \,+\, \nu_0^{-1} L_a\, |\t_n - \t_0 |\, \left\| w_n \right\|_{\t_n}^2 \\
&=\,\,  \nu_0^{-1} a_{\t_n}\left[ w_n\right]  \,+\, \nu_0^{-1} L_a\, \left|\t_n - \t_0 \right|\,\to\,0\, \mbox{ as } n\to\infty.
\end{aligned}
\]
Thus the last  term in \eqref{6520e} converges to zero. On the other hand, by 
\eqref{VtLip2},  
\[
\left|\left( P_{V_{\t_0}} w_n ,w_n\right)_{\t_n}\right|\,\,\,  
=\,\,
\left|\left(  P_{W_{\t_n}} P_{V_{\t_0}} w_n ,w_n\right)_{\t_n} \right| \,\,\le\,\, 
\left\|P_{W_{\t_n}} P_{V_{\t_0}} w_n\right\|_{\t_n} \,\, 
\le\,\, L_V |\t_n - \t_0|\,\left\|P_{V_{\t_0}} w_n \right\|_{\t_0}.
\]
Clearly, $\left\|P_{V_{\t_0}} w_n \right\|_{\t_0}  \le   \left\| w_n \right\|_{\t_0} \le K \left\| w_n \right\|_{\t_n}=K$, and therefore the first term on the right hand side of \eqref{6520e} also converges to zero. Whence, we arrive at the contradiction in \eqref{6520e}, and so \eqref{bddspec} holds.
%
\end{proof}
\begin{remark}\label{rem.merelycont}
The above proof demonstrates that 
Theorem \ref{thm.contVequiv} remains valid 
if we merely require both $a_\t$ and $V_\t$ to be say 
H\"{o}lder continuous (rather than Lipschitz continuous) in $\t$, with appropriate 
modification of \eqref{ass.alip} and \eqref{VtLip}. 
Also, a `local' analogue of Theorem \ref{thm.contVequiv} clearly holds, i.e. when in both \eqref{bddspec} and  \eqref{VtLip} $\Theta$ is replaced by its closed subset $\Theta'$. 
\end{remark}
\section{The case of discontinuous $V_\theta$ (non-uniform gap)}
\label{section:discV}
Typically, in applications (Section \ref{sec:examples}) the spaces $V_\t$  violate \eqref{VtLip} and have  isolated discontinuities. 
By Theorem \ref{thm.contVequiv}, near those discontinuities the gap constants $\nu_\t$ necessarily degenerate, 
and as a result Theorem \ref{thm:contV} becomes inapplicable and the approximations of \eqref{thmcontv.vprob} cease being uniformly accurate. 
Fortunately refined approximations,  
providing the desired accuracy,
 are possible near those discontinuity points. 
Henceforth, we consider this situation and begin with an analysis in the neighbourhood of a given 
point $\t_0\in \Theta$, 
without loss of generality\footnote{
The analysis and results in this section, that are local in nature, 
extend in a straightforward manner to the case when the discontinuity set  is an arbitrary set of isolated 
points. 
} 
 $\t_0=0$. 
\subsection{Local estimates}
\label{sect4.1}
While Theorem \ref{thm:contV} may be 
not anymore 
applicable, the 
key idea behind remains so. 
 The essence of the theorem  
was in first identifying 
$\ep$-independent subspaces $W_\t$ of $H$ on which $a_\t$ are uniformly coercive and then 
restricting 
problems \eqref{p1} to the 
orthogonal complements $V_\t$ of $W_\t$ with respect to $A_{\ep,\t}$. 
Regarding 
the former, we observe that as 
a consequence of \eqref{KA} held at $\t=\theta_0=0$ together with the continuity of $a_\t$ due to 
\eqref{ass.alip}, 
$a_\t$ remains uniformly coercive on $W_0$ 
in a small enough neighbourhood of $\theta_0$:
\begin{proposition}
	\label{Wt2coerc}  
Assume \eqref{KA}.
Then
	\begin{equation}
	\label{C2}
	\frac{\nu_0}{2K^2}  \big\| w_0 \big\|_\t^2\,\,\, \le\,\,\, a_\t \left[w_0\right], \quad \forall w_0 \in  W_0, 
	\ \forall \t  \in \Theta\ \text{such that}\ \, |\theta|\,\le\, \tfrac{1}{2}\nu_0L_a^{-1}.
	\end{equation}
\end{proposition}
\begin{proof}
For $w_0\in W_0$, 
as follows from  \eqref{KA}, 
\eqref{ass.alip}  
	and $ L_a  |\t| < \tfrac{1}{2}\nu_0$, 
	\[
	\nu_0 \left\| w_0 \right\|_0^2  \,\,\,\le\,\, a_0[w_0] \,\,\le \,\, a_\t\left[w_0\right] \,+ \,
	L_a|\t|\left\| w_0\right\|_0^2\,\,
	\le\,\, a_\t[w_0]\,+ \,\,
	\tfrac{1}{2}\, \nu_0 \left\| w_0\right\|_0^2\,,
	\]
	which implies 
	\begin{equation}
	\label{coercv0t}
	\tfrac{1}{2}\nu_0 \| w_0 \|_0^2  \,\, \le\,\, a_\t[w_0], \quad \forall w_0\in W_0. 
	\end{equation} 
The latter,  along with \eqref{as.b1}, implies \eqref{C2}.
\end{proof}	
Turning now to the orthogonality issue, 
recall that, at $\theta=0$, $V_0$ is the orthogonal complement of $W_0$ with respect to $A_{\ep,\t}$. However for $\t\neq 0$, in general, 
$V_0$ is not anymore orthogonal to $W_0$. 
Nevertheless, it is possible to partially rectify this as follows. The idea is, for small enough 
$\theta\neq 0$, to ``correct'' $V_0$ slightly to maintain the desired orthogonality only 
to the main order in small $\ep$, i.e. with regards to the singular   
part $a_\t$  
of $A_{\ep,\t}$. 
To that end, given $v_0\in V_0$, seek a ``corrector'' $\N v_0 \in W_0$ such that for $\M v_0:=v_0+\N v_0$, 
\begin{equation}\label{IliaN2}
a_\t\left( \M  v_0 , w_0 \right) \,\,=\,\, 0,  \qquad \forall v_0 \in V_0, \ \ \forall w_0 \in W_0,
\end{equation}
i.e. so that $\M V_0$ and $W_0$ are ``orthogonal with respect to $a_\t$''. 
Equivalently, we seek  
$\N v_0 \in W_0$ solving 
\begin{equation}\label{IliaN}
a_\t\left( \N  v_0 , \widetilde w_0 \right) \,\,= \,\,-\,\, a_\t\left(v_0,\widetilde w_0\right), \quad \forall\, \widetilde w_0 \in W_0. 
\end{equation}
Problem \eqref{IliaN} can be viewed as 
an abstract analog of the cell problem. It is well-posed for $|\theta|\leq \tfrac{1}{2} \nu_0 L_a^{-1}$ by Proposition \ref{Wt2coerc}, 
and determines a linear map $\N  : V_0 \rightarrow W_0$.  
Show that the following estimate holds: 
\begin{equation}\label{Nbound}
\left\| \N v_0\right\|_0\,\,\, \le\,\,\, 2{L_a}{\nu_0}^{-1} |\t|\, \left\| v_0\right\|_0, \quad \forall v_0 \in V_0, \ \    |\theta|\,\le\, \frac{1}{2} \nu_0\, L_a^{-1}. 
\end{equation} 
Indeed, 
from   \eqref{coercv0t}, \eqref{IliaN},  \eqref{ass.alip} and \eqref{Vkersesa}, 
\[
\tfrac{\nu_0}{2} \left\| \N v_0 \right\|_0^2 \,\le\, a_\t\big[\N v_0\big] \,=\, -\, a_\t\left(v_0,\,\N v_0\right) 
\,\le\, \big|a_0\left(v_0,\,\N v_0\right) \big|\, +\, L_a |\t| \left\|  v_0\right\|_0\, \left\| \N v_0\right\|_0 \,=\,  L_a |\t|\, \left\| v_0\right\|_0 \,\left\| \N v_0\right\|_0.
\]
Also, since 
$\N V_0 \subseteq W_0$, it readily follows that  $H = \M V_0 \dot{+}W_0$, i.e. $H$ is a direct 
sum 
 of $\M V_0$ and $W_0$. 
Further, from the orthogonality of $V_0$ and $W_0$, and \eqref{Nbound}, 
\begin{equation}\label{Mpositive}
\left\| v_0\right\|_0 \,\,\le\,\, \left\| \M v_0\right\|_0\,\,\le\,\,2\,\left\| v_0\right\|_0, \quad \forall\, v_0 \in V_0, 
\ \    |\theta|\,\le\, \frac{1}{2} \nu_0 L_a^{-1}. 
\end{equation}
It follows from \eqref{Mpositive} that $\M V_0$ is  closed in $H$. 
This all 
allow us, for sufficiently small $\t$, to construct a desirable approximation to the solution $u_{\ep,\t}$ of variational problem \eqref{p1} by restricting it 
to $\M V_0$ which is ``almost orthogonal'' to $W_0$. 
Indeed, the following subtle modification of Theorem \ref{thm:contV} holds.
\begin{theorem}
	\label{lmm1}
	Let \eqref{KA} hold,  $f \in H^*$, 
	$\t \in \Theta$, $|\t| <  \nu_0/ (2L_a )$. 
	Let $u_{\ep,\t}$ solve \eqref{p1}, 
	and $v_0 \in V_0$ solve
	\begin{equation}
	\label{w1prob0}
	\ep^{-2} a_\t\big(  \M v_0 \,,\, \M \tilde{v} \big) \,\,+\,\, b_\t\left( \M v_0, \M \tilde{v}\right) \,\,=\,\, 
	\left\langle f,\, \M \tilde{v} \right\rangle, \quad \forall \tilde{v} \in V_0.
	\end{equation}
	Then problem \eqref{w1prob0} is well-posed, and the following error estimates hold:
	\begin{align}
	\label{ik43000}
	\ep^{-2}a_\t\big[u_{\ep,\theta} - \M v_0 \big]\,\,+\,\,	
	b_\t\big[ u_{\ep,\theta} -  \M v_0 \big]\,\,\, \le \,\,\,   8K^2 \nu_0^{-1}  \ep^2\,\|f\|_{*\t}^2, \\
	\label{s4ep2bd}
	b_\t\big[ u_{\ep,\theta} - \M v_0 \big]\,\,\,\le\,\,\, 16K^4 \nu_0^{-2}  \ep^4\,\|f\|_{*\t}^2.
	\end{align}
\end{theorem}  
\begin{proof}
The well-posedness of \eqref{w1prob0} follows 
from its left-hand side specifying an equivalent inner product on $V_0$, as implied by \eqref{Mpositive}. 
For the difference $r : = u_{\ep,\t} - \M v_0$, the left-hand-side of \eqref{ik43000} equals $A_{\ep,\t}[r]$ (see \eqref{Aform}) and expanding this out (and dropping the subscripts in notation) gives $A[r] = A\left(u_{\ep,\t},\, r\right) - A\left(\M v_0 , r\right)$. 
Note $u_{\ep,\t} = \M v + w$ for some unique $v \in V_0$  and $w \in W_0$, and hence $r = \M v_r  + w$ where $v_r := v - v_0 \in V_0$, and so 
\[
A[r]\,\, = \,\, A\left(u_{\ep,\t},\, r\right) \,-\, A\left(\M v_0 ,\, \M v_r\right) \,-\, A\left(\M v_0, w\right).
\]
Now,  \eqref{p1} gives $A\left(u_{\ep,\t},r\right) = \langle f, r\rangle $, \eqref{w1prob0} gives $A\left(\M v_0 ,\M v_r\right) = \left\langle f, \M v_r \right\rangle $,  
and the 
almost-orthogonality due to \eqref{IliaN2} implies
 $ A\left(\M v_0, w\right)  = b_\t\left(\M v_0, w\right)$. Therefore,    
\begin{flalign*}
A[r]\, & \,=\,\,  \langle f, r \rangle \,-\, \left\l f, \M v_r\right\r -  b_\t(\M v_0, w)  \,\,=\,\, 
\langle f,  w \rangle \,-\, b_\t\left(\M v_0,  w\right) \,\,\le\,\, \Big( \| f\|_{*\t} \,+\, b_\t^{1/2}\left[ \M v_0\right] \Big) \|  w\|_\t,
\end{flalign*}
via \eqref{fstar}, 
Cauchy-Schwarz inequality, and 
\eqref{astructure}. Setting  $\tilde{v} = v_0$ in  \eqref{w1prob0} and recalling that $\ep <1$ implies 
$b_\t \left[\M v_0 \right] \le\, \left\| \M v_0\right\|^2_\t\, \le 
A\left[\M v_0\right] \,\le\, \| f\|_{*\t}\,\left\| \M v_0\right\|_\t \le 
\| f\|^2_{*\t}$ 
and as a result $ A[r] \le 2 \| f\|_{*\t} \|  w\|_\t$. 
So for proving \eqref{ik43000} one needs to bound $\| w\|_\t$ in terms of $\ep A^{1/2}[r]$. 
Now  \eqref{C2} gives $\tfrac{\nu_0}{2K^2} \|  w \|_\t^2 \le a_\t[ w]$, and noticing that  $a_\t(\M v_r,w) = 0$ 
by \eqref{IliaN2} and $r =\M v_r + w$ implies 
\begin{equation}\label{IliaN4.1}
 a_\t[w]\,\, \le\,\, a_\t[r]. 
\end{equation}
Therefore, $
\tfrac{\nu_0}{2K^2} \|  w \|_\t^2   \le  a_\t[r]  \le \ep^2 A[r], 
$ and as a result 
\begin{equation}\label{IliaN5}
\|  w\|_\t^2 \,\,\,\le\,\,\,  {2K^2}{\nu_0}^{-1} \ep^2 A[r], 
\end{equation}
implying \eqref{ik43000}.  \\
It remains to prove  \eqref{s4ep2bd}, whose left-hand-side equals $b_\t[ r]$. Then, from \eqref{IliaN5} and \eqref{ik43000} it suffices to show that 
\begin{equation}\label{IliaN6}
b_\t[r ]\,\,\le\,\, \| w\|_\t^2\,.
\end{equation} 
Since \eqref{w1prob0} is a restriction of \eqref{p1} to $\M V_0$, $r$ is orthogonal to $\M V_0$ with respect to $A$, 
i.e. $A\left(r,\M \tilde{v}\right) =0$ for any $\tilde{v} \in V_0$. 
Consequently, since $w = r - \M v_r$, one infers $A[r] \le A[w]$. This inequality along with \eqref{IliaN4.1} yields $b_\t[r] \le b_\t [w]$ implying 
\eqref{IliaN6}. 
The proof is complete. 
 \end{proof}
We finish this subsection with a comparison between Theorem \ref{lmm1} and Theorem \ref{thm:contV}.
 In the continuous case we restrict variational problem \eqref{p1} to the subspace $V_\t$, but  for the general (possibly discontinuous) case 
this may be not anymore sufficient for maintaining the same order of the approximation's accuracy, and 
we restrict instead  (locally near $\theta=\theta_0=0$) to $\M V_0$. 
 From this
 observation one may  expect that $V_\t$ is a subset of  $\M V_0$.  
Indeed, the following proposition holds. 
\begin{proposition}
\label{lemVtsubMV0} 
Assume \eqref{KA}. Let $\t \in \Theta$, $\,|\t| \le\, \nu_0 / (2L_a)$. Then
\begin{gather}\label{250520e2}
V_\t \,\,\subseteq\,\, \M V_0.   \\
\text{In fact, 
} \hspace{\textwidth} \nonumber \\
v_\t \,\,=\,\, \M P_{V_0} v_\t, \quad \forall v_\t \in V_\t.\label{250520e1}
\end{gather}
\end{proposition}
\begin{proof}
For \eqref{250520e2}, for any fixed $v_\t\in V_\t$ we need to find $v_0\in V_0$ such that $v_\t=\M v_0=v_0+\N v_0$. 
As $\N v_0\in W_0$, necessarily, $v_0=P_{V_0} v_\t$. Hence, for both \eqref{250520e2} and \eqref{250520e1}, it remains 
to show that $v_\t-P_{V_0} v_\t=\N P_{V_0} v_\t$. Clearly $v_\t-P_{V_0} v_\t=P_{W_0} v_\t\in W_0$, and also 
via \eqref{Vkersesa} 
\[
\,a_\t\big(v_\t-P_{V_0} v_\t,\, \widetilde w_0\big)\,\,=\,\,-\,\,
a_\t\left(P_{V_0} v_\t, \widetilde w_0\right), 
\quad \forall \widetilde w_0\in W_0. 
\]
Hence, by \eqref{IliaN}, $v_\t-P_{V_0} v_\t  
=\N P_{V_0} v_\t$, which completes the proof. 
\end{proof}
\begin{remark}\label{rem.s3.1}
	Observe that the proofs in this subsection only require that \eqref{KA} holds at 
	$\t=0$, that \eqref{ass.alip} holds for $\t_1=0$, 
	that \eqref{as.b1} holds in a neighbourhood of $\t =0$, 
	and (with appropriate change of the exact constants) that $W_0$ is 
	the orthogonal complement with respect to any equivalent inner product. 
\end{remark}
\subsection{On the class of $V_\t$ with  a removable singularity}\label{ssRSing}
{ So far we have made no assumptions on the nature of the singularity of $V_\t$, and so 
	one could apply Theorem \ref{lmm1} 
	for any singularity (and even for any non-singular point). 
However, in a large class of relevant examples 
(Section \ref{sec:examples}), the singularities of $V_\t$ 
are ``removable'' ones as defined below. 
As we will see, this additional feature allows to simplify the approximating problem \eqref{w1prob0} 
further by restricting the singular form $a_\t$ from $V_0$ to its ``defect subspace'' $Z$ 
accounting for the discontinuity gap in $V_\t$ at $\t=0$. 
In the remainder of the article we mostly focus on developing Theorem \ref{lmm1} further  for such singularities. 
Namely, we assume that $\t_0=0$ is not an isolated point 
of $\Theta$ and 
$V_\t$ has a {\it removable singularity} at $\t_0$ in the following sense:
there exists a closed subspace $V_\star
$ of $H$ and constant $L_\star\ge0$ such that (cf. \eqref{VtLip} )
\begin{equation}\tag{H2}
\label{contVs}
\Vs_\t \,\,:=\, \left\{ \begin{array}{cc}
V_\t & \t \neq 0, \\[.5em] V_\star & \t = 0,
\end{array} \right.   \quad \text{satisfies} \quad  
\inf_{v_2 \in \Vs_{\t_2}} \left\| v_1 - v_2 \right\|_{\t_2} \,\,\le\,\, L_\star | \t_1 - \t_2 |\, 
\left\| v_1 \right\|_{\t_1}, \quad \forall v_1 \in \Vs_{\t_1}, \ \, \forall\, \t_1, \t_2 \in \Theta,
\end{equation} 
or equivalently satisfies
\begin{equation}
\label{contVs2}
\| P_{W^*_{\theta_2}}v_1  \|_{\t_2} \,\,\le\,\, L_\star | \t_1 - \t_2 |\, \left\| v_1 \right\|_{\t_1}, \quad \forall\, v_1 \in \Vs_{\t_1}, \ \, \forall\, \t_1,\, \t_2 \in \Theta, 
\end{equation}  
where $W^*_{\theta}$ is the orthogonal complement of $V^*_\theta$ in $H$ with respect to $(\cdot,\cdot)_\theta$. 
Note that \eqref{contVs} formally includes 
the case without singularity when $V_\star=V_0$. 
\begin{remark}\label{constV}
	For a wide class of examples (cf. Section \ref{sec:examples}) $V_\theta$  is independent of $\theta$ away from the discontinuity, i.e. $V_\theta = V$ for $\theta \neq 0$. In this situation \eqref{contVs} trivially holds with $V_\star = V$ and $L_\star = 0$.
\end{remark}
First 
observe that $
V_\star \subset V_0.
$
Indeed, for $v_\star\in V_\star$ and $\t\in\Theta$, $\t\neq 0$ and $\t\to 0$, by \eqref{ass.alip}, 
\eqref{astructure} and  \eqref{contVs2},  
\[
a_0[v_\star] \,=\, \lim_{\t \rightarrow 0} a_\t[v_\star] \,=\, \lim_{\t \rightarrow 0} a_\t\left[P_{W_\t^\star} v_\star\right] \,\,\le\,\,
  \lim_{\t \rightarrow 0}\, \left\|P_{W_\t^\star} v_\star\right\|_\t^2 \,\,\le\,\, \lim_{\t \rightarrow 0} L_\star^2  |\t|^2  \left\|  v_\star\right\|_0^2 \,\,=\,\, 0.
\]
A key role in our subsequent constructions will be played by a defect subspace $Z$ of $V_0$, characterising the 
discontinuity gap between $V_\star$ and $V_0$. Namely, let $Z$ be 
a closed linear subspace of $V_0$, such that 
\begin{gather}\label{spaceZ}
V_0 \,\,=\,\, V_\star\, \dot{+}\,Z, \\
\textrm{with the direct sum obeying a transversality condition: for some constant $0\leq K_Z<1$, } \hspace{\textwidth} \nonumber\\ 
\label{VZorth}
\big|(v_\star, z)_0 \big| \,\,\le\,\, K_Z \left\| v_\star \right\|_0 \left\| z\right\|_0, \qquad \forall v_\star \in V_\star, \,\, \forall z \in Z, \\ 
\textrm{or equivalently} \hspace{\textwidth} \nonumber\\ 
\label{VZorth2}
{\left(1-K_Z^2\right)}^{1/2}  \left\| v_\star\right\|_0  \,\,\le\,\, \left\| v_\star +z \right\|_0, \qquad \forall v_\star \in V_\star, \,\, \forall z \in Z.
\end{gather} 
\begin{remark}
\label{zorth}
Note that such $Z$ always exist, 
in particular $Z$  could  be the orthogonal complement of $V_\star$ in the Hilbert space $\big(V_0,\,(\cdot,\cdot)_0\big)$, 
in which case $K_Z=0$. 
In the regular case  $V_\star=V_0$, trivially $Z=\{0\}$. 
\end{remark}
Now we are ready to provide an alternative representation of $\M V_0$ in terms of  $V_\t^\star$ and $Z$, which is useful for a further simplification of the 
approximating problem \eqref{w1prob0} as the singular form $a_\t$ vanishes on $V_\t^\star$.  
The following technical 
lemma, important for our consequent constructions, holds. 

\begin{lemma}\label{EquivdefcontVs}
	Assume \eqref{KA} and \eqref{contVs}. If  $ \t\in \Theta$ satisfies $KL_\star |\t| < \tfrac{1}{3}\left(1- K_Z\right)$ then 
\begin{gather}\label{V+ZCoercive}
\big\| v_\t^\star \,+\, z \,+\, w_0  \big\|_0^2 \,\,\,\ge\,\, \frac{1-K_Z}{3}
\Big( \big\| v_\t^\star \big\|_0^2 \,\,+\, \big\|  z\big\|_0^2 \,\,+\, \big\| w_0\big\|_0^2 \Big) \,, \quad \forall v_\t^\star \in V_\t^\star,\,\, \forall z\in Z, \,\, \forall w_0 \in W_0\,;\\ 
\text{and if additionally $| \t | <   \nu_0/(2L_a)$ then } \hspace{\linewidth} \nonumber \\
	\label{frame}
	\M V_0 \,\,=\,\, V_\t^\star \,\,\,\dot{+}\,\, \M Z.  
	\end{gather}
\end{lemma}
\begin{proof}[Proof of \eqref{V+ZCoercive}]
Set $\kappa_0 := 1-K_Z$, so $0<\kappa_0\leq 1$. Assumption \eqref{VZorth} and 
the orthogonality of $W_0$ and 
$V_0 = V_\star \dot{+}Z$ imply  
\begin{equation}\label{kappa0}
 \big\|v_\star+z+w_0\big\|_0^2\,\,\ge\,\, \kappa_0\Big(\|v_\star\|_0^2\,\,+\,\|z\|_0^2\,\,+\, \| w_0 \|_0^2\Big), \quad  \forall v_\star \in V_\star,\, \,\forall z\in Z, \, \,\forall w_0 \in W_0.
\end{equation} 
Thus, for any $v_\t^\star \in V_\t^\star,\,  z\in Z$ and $ w_0 \in W_0$, with $W_*:=W_0^*$, 
\begin{flalign*}
 \big\|v_\t^\star+z+w_0\big\|_0^2\,\,\,&\ge\,\, \tfrac{1}{2}\big\| P_{V_\star}v_\t^\star+z+w_0\big\|_0^2\,-
\big\|P_{W_\star}v_\t^\star  \big\|_0^2 \,\,\ge\,\, \tfrac{\kappa_0}{2}\left( \big\|  P_{V_\star}v_\t^\star\big\|_0^2\,+\big\|z\big\|_0^2\,+\big\|w_0\big\|_0^2\right)\,-\,
\big\|P_{W_\star}v_\t^\star  \big\|_0^2\\
 &=\,\,\tfrac{\kappa_0}{2} \Big(\|v_\t^\star\|_0^2\,+\|z\|_0^2 \,+ \| w_0\|_0^2\Big)\,-\,
\big(\tfrac{\kappa_0}{2}\,+\,1\big)\big\|P_{W_\star}v_\t^\star  \big\|_0^2\,.
\end{flalign*}
Now, \eqref{contVs2}, 
\eqref{as.b1}, and the assumption on $|\t|$ give  
\[
\left(\tfrac{\kappa_0}{2}+1\right)\left\| P_{W_\star}v_\t^\star\right\|_0^2  \,\,\,\,\le\,\,\,\, 
\left(\tfrac{\kappa_0}{2}+1\right)\big(KL_\star|\t|\big)^2\, \left\| v_\t^\star\right\|_0^2\,\,\,\,\le\,\,\,\, 
\left(\tfrac{\kappa_0}{2}+1\right)\left(\tfrac{\kappa_0}{3}\right)^2\,\left\| v_\t^\star\right\|_0^2 \,\,\,\,\le\,\,\, \tfrac{1}{6}\kappa_0\, \left\| v_\t^\star\right\|_0^2,
\]
and \eqref{V+ZCoercive} follows.

{\it Proof of \eqref{frame}}. The inclusion $V_\t^\star \subseteq V_\t$ and  \eqref{250520e2} show $V_\t^\star + \M Z \subseteq \M V_0$. Furthermore, \eqref{V+ZCoercive} for $w_0 = \mathcal{N}_\t z$ together with the closedness of $\M Z$ (following e.g. 
from \eqref{Mpositive} ) 
 implies that this sum is a direct sum and closed. 

	It remains to show that $V_\t^\star \,\dot{+}\, \M Z$ is not a proper subset of $\M V_0$.  If it were, there would exist a non-zero $v_0 =v_\star + z$, $v_\star \in V_\star$, $z \in Z$, such that $\M v_0 $ is orthogonal (with respect to $(\cdot,\cdot)_0$) to  $V_\t^\star\, \dot{+}\, \M Z$. 
Seeking a contradiction to this orthogonality, a natural choice of an element of $V_\t^\star \,\dot{+}\, \M Z$ expected to be close to 
$\M v_0 =\M v_\star + \M z$ is $u=P_{V_\t^\star} v_\star + \M z$. Since   \eqref{250520e1} gives
$P_{V_\t^\star}  v_\star  =\M P_{V_0}P_{V_\t^\star}  v_\star=\M P_{V_0}v_\star-\M P_{V_0}P_{W_\t^\star}  v_\star= 
\M v_\star-\M P_{V_0}P_{W_\t^\star}  v_\star$ (as $v_\star\in V_0$), 
we conclude that $\M v_0 $ is orthogonal to $u=\M v_\star-\M P_{V_0}P_{W_\t^\star}  v_\star +\M z=
\M v_0-\M  P_{V_0}P_{W_\t^\star}v_\star$.
Now, by \eqref{Mpositive} and the latter orthogonality, 
\[
\left\| v_0 \right\|_0 \,\,\le\,\, \left\| \M v_0 \right\|_0  \,\,\le\,\, \left\| \M v_0- u\right\|_0
\,\,=\,\,  \left\| \M P_{V_0} P_{W_\t^\star} v_\star\right\|_0.
\]
Further, by the second inequality in  \eqref{Mpositive}, 
the properties of $P_{V_0}$, \eqref{as.b1}  and \eqref{contVs2}, 
\[
\left\| \M P_{V_0}P_{W_\t^\star}  v_\star \right\|_0  \,\,\le\,\, 2\, \left\| P_{V_0} P_{W_\t^\star}  v_\star \right\|_0 \,\,\le\,\,  
2\, \left\| P_{W_\t^\star}  v_\star \right\|_0
\,\,\le\,\, 2KL_\star |\t|\,  \left\| v_\star \right\|_0. 
\]
Now \eqref{VZorth2} gives $\left\|v_\star\right\|_0 \le \left(1- K_Z^2\right)^{-1/2} \left\| v_0 \right\|_0$, 
and hence 
  $\left\| v_0 \right\|_0 \le 2 K L_\star |\t|\,  \left(1-K_Z^2\right)^{-1/2} \left\| v_0 \right\|_0$. 
This along with the assumed restriction on $\t$ and the inequality 
$\left(1-K_Z^2\right)^{-1/2} \le \left(1-K_Z\right)^{-1}$ lead to $\left\|v_0\right\|_0 =0$, which is a contradiction.
\end{proof}
We now present a global approximation  for the case of $V_\t$ with a removable singularity at $\t=0$. Let 
\begin{equation}\label{r0}
\text{ $r_0 =\nu_0 / (2L_a)$ if $L_\star = 0$ or $r_0 = \min\Big\{\,\nu_0 / \left(2L_a\right)\,,\, \left(1-K_Z\right)/ \left(3KL_\star\right)\,\Big\}$ otherwise,}
\end{equation}
and fix positive $r_1 \le r_0$. The direct sum representation \eqref{frame} 
and Theorem \ref{lmm1} show that, for $|\t | < r_1$, the solution $u_{\ep,\t}$  to \eqref{p1} is approximated in terms of the solution $v_0$ of the simplified problem \eqref{w1prob0} by 
$\mathcal{M}_\t v_0=v_{\ep,\t} \,+\, \mathcal{M}_\t z_{\ep,\t}$, with unique   
$v_{\ep,\t}\in V_\t^\star$ and $z_{\ep,\t}\in Z$. Recalling \eqref{Vkersesa}, we conclude that 
$\left(v_{\ep,\t}, \,z_{\ep,\t}\right)$ is the unique solution to 
\begin{equation}\label{coupledbest}
\ep^{-2} a_\t\big(  \M\, z_{\ep,\t},\, \M \tilde{z}\big) \,\,+\,\, b_\t\big(v_{\ep,\t} + \M\, z_{\ep,\t},\, \tilde{v} + \M \tilde{z}\big) \,\,=\,\, 
\big\langle f,\, \tilde{v} + \M \tilde{z}\, \big\rangle, \quad \forall \tilde{v} \in V^\star_\t,\, \,\forall \tilde{z}\in Z. 
\end{equation}
Furthermore, the solution $v_{\t}$ to \eqref{thmcontv.vprob} approximates $u_{\ep,\t}$ outside this neighbourhood of the origin. Indeed \eqref{contVs} implies that $V_\t$ is Lipschitz continuous on the compact set $\Theta_{r_1} = \{ \t \in \Theta \, : \,  | \t | \ge r_1 \}$. Therefore Theorem \ref{thm.contVequiv} 
(applied for $\Theta$ replaced by $\Theta_{r_1}$, cf. Remark \ref{rem.merelycont}) implies that the 
assumption \eqref{bddspec} 
of Theorem \ref{thm:contV} holds on $\Theta_{r_1}$ with a  positive constant
$
\nu(r_1) = \inf_{\t \in \Theta_{r_1}} \nu_\t.
$ More precisely, we have proved the following result:
\begin{theorem}\label{thm1.all}
	Assume \eqref{KA}--\eqref{contVs} and let $0 < r_1 \le r_0$ (see \eqref{r0}). Consider  $f \in H^*$,	$u_{\ep,\t}$ the solution to   \eqref{p1}  and the approximation 
 	\[
 u_{\ep,\t}^{(0)} : = \left\{ \begin{array}{ll}
 v_{\ep,\t} + \M\, z_{\ep,\t} & |\t| < r_1, \\
 v_{\t} & |\t| \ge r_1,
 \end{array} \right.
 \] 
where 
 the pair $\left(v_{\ep,\t},\,  z_{\ep,\t}\right)\in V_\t^\star \times Z$ solves \eqref{coupledbest} and 
$v_\t \in V_\t=V_\t^\star$ solves \eqref{thmcontv.vprob}.
Then, the following error estimates hold for all $\t\in\Theta$ and $0<\ep<1$:
	\begin{align}
	\label{ik430}
	\ep^{-2}a_\t\left[u_{\ep,\theta} - u_{\ep,\t}^{(0)}\right]\,\,+\,\,	
	b_\t\left[ u_{\ep,\theta} -u_{\ep,\t}^{(0)}\right] \,\,\,\le \,\,\,   C_1\ep^2\,\|f\|_{*\t}^2, \qquad C_1= \max\left\{8K^2 \nu_0^{-1},\, 1/\nu(r_1)\right\},\\
	\label{ik430.ep2}
	b_\t\left[ u_{\ep,\theta} -  u_{\ep,\t}^{(0)}\right] \,\,\,\le \,\,\,   C_2\ep^4\|f\|_{*\t}^2, \qquad C_2= \max\left\{16K^4 \nu_0^{-2},\,1/\nu^2(r_1)\right\}.
	\end{align}
\end{theorem} }
We emphasise that the ``inner'' approximate problem \eqref{coupledbest} provides a significant further 
simplification compared to \eqref{w1prob0}, as its singular part $a_\t$ is now a form on the defect subspace $Z$ only. 
\section{Further refinements of Theorem \ref{thm1.all} }
\label{sec.2dif}
The results of Section \ref{section:discV} approximate the original problem by simpler ones, for example Theorem \ref{thm1.all} reduces problem \eqref{p1} on $H$ to those 
on the generally smaller subspaces $V_\t^\star$ and $Z$. However, the price to pay 
is the more complicated dependence on the parameter $\t$, in particular through the operator $\M$. The purpose of this section is to simplify the approximating problems further, especially their 
dependence on $\t$, as much as possible under certain readily verifiable additional assumptions. 
This ultimately allows to approximate the original problem by one which in turn leads in Section \ref{s:resolv} to construction of an abstract version of a two-scale limit operator, possessing certain important properties as is  
illustrated by new results for a number of examples in Section \ref{sec:examples}. 

{
We begin by noting that \eqref{V+ZCoercive} implies that $V_\t^\star$ and $Z$ form a closed direct sum in $H$ for small enough $\t$. Furthermore, one can see that $b_\t$ generates an equivalent norm on $V_\t^\star \,\dot{+}\,Z$. Indeed, since  $Z \subseteq V_0$, we use   \eqref{ass.alip} and  \eqref{V+ZCoercive} to obtain 
$a_\t[z] \,\le\, \tfrac{3L_a}{1-K_Z} |\t|\, \left\|  v_\t^\star + z\right\|_0^2$, and consequently, 
via \eqref{astructure} and \eqref{as.b1}, 
\begin{equation}\label{btnormonV+Z}
\left\| v_\t^\star + z\right\|_\t^2 \,\,\le\,\, 2\, b_\t\left[v_\t^\star + z\right], \quad \forall v_\t^\star \in V_\t^\star, \,\,\forall z \in Z,\, \ \forall \t \in \Theta, \,\,\,
|\t| \,\le\, r_1\,  :=\, \min \left\{ r_0,\,  \tfrac{1-K_Z}{6K^2L_a}\, \right\}, 
\end{equation}
with $r_0$ given by \eqref{r0}. 
\subsection{Case of quadratically degenerating spectral gap width}\label{sec.quadnu}
Here, we additionally assume that the spectral gap degenerates near $\t_0=0$ at most quadratically. 
Namely, 
there exists  $\gamma >0$ such that, for   $\nu_\t$ defined in \eqref{KA}, 
\begin{equation}
\tag{H3}\label{distance}
	\nu_{\t} \,\,\ge\,\, \gamma\, |\t|^2, \quad \forall\, \t \in \Theta, \quad \mbox{i.e. } \ a_\t[w]\,\,\ge\,\,\gamma\,|\t|^2\,\|w\|_\t^2, \quad \forall w\in W_\t, \,\,\forall\, \t \in \Theta. 
\end{equation}
This is a generalisation of the 
property of 
quadratic degeneracy of the spectral gap in classical homogenisation, 
inherited by numerous asymptotically degenerating non-classical models (Section \ref{sec:examples}). 
Condition \eqref{distance} allows 
us to control the degeneracy of $a_\t\left[\M\,\, \cdot\,\right]$ on $Z$, and as a result to 
further simplify approximate problem \eqref{coupledbest} by removing $\M$ 
(i.e replacing it by the identity operator) in both 
$b_\t$ and $f$ terms (but not in the $a_\t$ term). 
\begin{proposition}\label{p.nondegZ}
	Assume \eqref{KA}--\eqref{distance} and $\t \in \Theta$, $|\t| \le r_0$, for  $r_0$  as in 
	Theorem \ref{thm1.all}, 
	see \eqref{r0}. Then 
\begin{equation}\label{150620.e1}
a_\t\left[\M z\right] \,\,\ge\,\,  \nu_\star|\t|^2\, \left\| z \right\|_0^2, \quad \forall z \in Z, \quad 
\mbox{ with } \, \nu_\star \,=\,  \tfrac{\gamma (1-K_Z)}{3K^2}.
\end{equation}
\end{proposition}
\begin{proof}
 It is enough to consider the case $\t \neq 0$. 	Then $W_\t = W_\t^\star$ and \eqref{KA} implies 
$
a_\t\left[\M z\right] = a_\t\left[P_{W^\star_\t}\M z\right] \ge 
 \nu_\t \left\| P_{W^\star_\t}\M z\right\|_\t^2.
$ Moreover,     \eqref{V+ZCoercive} for $w_0 = \N z$ and $v_\t^\star = -\,P_{V_\t^\star} \M z$ implies 
$
\left\| P_{W^\star_\t}\M z\right\|_0^2 \ge \tfrac{1-K_Z}{3} \| z \|_0^2.
$ 
Now \eqref{150620.e1} readily follows upon recalling \eqref{as.b1} and \eqref{distance}.
\end{proof}
\begin{corollary} By \eqref{as.b1}, \eqref{Nbound} and \eqref{150620.e1} one has 
\begin{equation} \label{Zcoerc}
\| \N  {z}\|_\t^2 \,\,\le\,\, \kappa_1^2 a_\t[\M {z}], \quad \forall {z} \in Z, \quad
\kappa_1\,=\,  2\, {K L_a} \nu_0^{-1} \nu_{\star}^{-1/2}.
\end{equation}
\end{corollary}
This leads us to the following theorem. 
 \begin{theorem}\label{thm.trunc1}
	Assume \eqref{KA}--\eqref{distance}, and consider the objects as in Theorem \ref{thm1.all} with $r_1$ as in \eqref{btnormonV+Z}. Then, for each $\t \in \Theta,$	 $|\t| < r_1 $, there exists a unique solution  $v+z \in V_\t^\star \dot{+} Z$ to
\begin{equation}
\label{p.trunc1}
\ep^{-2} a_\t\big(  \M z,\, \M \tilde{z}\big) \,+\, b_\t\big(v + z, \tilde{v} + \tilde{z}\big) \,\,=\,\, \big\langle f,\, \tilde{v} + \tilde{z} \big\rangle, \quad \forall\,\, \tilde{v} +\tilde{z} \in V_\t^\star \dot{+}  Z.
\end{equation}
Furthermore, the following error estimates hold \footnote{The point in estimate \eqref{trunc1.e3} is in bounding the 
$b_\t$-term without $\M$ in the approximation. The $a_\t$-term, already bounded by \eqref{trunc1.e2}, is added to 
\eqref{trunc1.e3} for a convenience of future use. Notice that, compared to \eqref{ik430.ep2}, 
estimate \eqref{trunc1.e3} provides only $O(\ep^2)$ rather than $O(\ep^4)$ error, even for the $b_\t$-term. 
To maintain the higher accuracy, one would need to replace the simplified approximate problem \eqref{p.trunc1} by a slightly 
more elaborate one which we do not pursue here.}:
	\begin{align}
	\label{trunc1.e2}
	\ep^{-2}a_\t\big[u_{\ep,\theta} -(v+ \M z)\big]\,+\,	b_\t\big[ u_{\ep,\theta} -(v + \M z)\big] \,\,\, \le\,\,    C_3\ep^2\|f\|_{*\t}^2, \qquad C_3\, = \,2 C_1 + 12 \kappa_1^2.
	\\
	\label{trunc1.e3}
	\ep^{-2}a_\t\big[u_{\ep,\theta} -(v+ \M z)\big]\,+\,	b_\t\big[ u_{\ep,\theta} - (v+ z)\big] \,\,\le\,\,    C_4\ep^2\|f\|_{*\t}^2, \qquad 
C_4\,=\, 2C_3 +\kappa_1^2, 
	\end{align}
	where $C_1$ is given by \eqref{ik430} and $\kappa_1$ is given by \eqref{Zcoerc}. 
\end{theorem} 
\begin{proof} For any $\ep>0$ and $\t\in\Theta$, $|\t|\le r_1$, the sesquilinear form 
\begin{equation}\label{Bform}
B(v+z,\tilde{v}+\tilde{z}) : = \ep^{-2} a_\t\big(  \M z, \M \tilde{z}\big) + b_\t(v + z, \tilde{v} + \tilde{z}\big), \quad  v, \tilde{v} \in V_\t^\star,\ z, \tilde{z} \in Z,
\end{equation} is bounded and coercive  in the Hilbert space $(V_\t^\star \dot{+} Z, (\cdot,\cdot)_\t)$ 
(see  \eqref{Mpositive}, \eqref{V+ZCoercive} and \eqref{btnormonV+Z}), and so problem \eqref{p.trunc1} is well-posed.
Furthermore, setting $\tilde{v} = v$ and $\tilde{z} = z$ in \eqref{p.trunc1} and utilising \eqref{btnormonV+Z} gives $B[v+z]
 \le  \| f\|_{*\t} \sqrt{2 b_\t[v + z]}$ from which we can readily deduce
\begin{equation}\label{18.09.20e1}
B[v+z]
 \le 2  \| f\|_{*\t}^2 , \quad \text{and} \quad \ep^{-2} a_\t[ \M z] \, \le \,
\|f\|_{*\t}\sqrt{2 b_\t[v + z]}\,\,-\,\,b_\t[v+z] \,\le \,
\tfrac{1}{2}\| f\|_{*\t}^2.
\end{equation}

{\it Proof of \eqref{trunc1.e2}.} Notice that the left-hand-side of \eqref{trunc1.e2} equals 
$A_{\ep,\t}\big[u_{\ep,\t} - (v+\M z)\big]$ (see \eqref{Aform}) and that Theorem \ref{thm1.all}, 
see \eqref{ik430}, states $A_{\ep,\t}\big[u_{\ep,\t} - (v_{\ep,\t} + \M z_{\ep,\t})\big] \le C_1 \ep^2  \| f\|_{*\t}^2$  for $v_{\ep,\t} + z_{\ep,\t}$ the solution to \eqref{coupledbest}. Thus,  it 
remains to bound $A_{\ep,\t}[ r_v + \M r_z ]$ where $r_v := v_{\ep,\t} - v$ and $r_z := z_{\ep,\t} - z$.
Subtracting  \eqref{p.trunc1} from \eqref{coupledbest} for $\tilde{v} = r_v$ and $\tilde{z}=r_z$ gives 
\[
\ep^{-2}a_\t\left[\M r_z\right]\,+\,
b_\t\big(v_{\ep,\t}+\M z_{\ep,\t},\,\, r_v\,+\, \M r_z\big)\,-\,b_\t\left(v+z,\, r_v+r_z\right)
\,=\,\left\langle f,\, \N r_z \right\rangle, 
\] 
which upon further direct calculation (and noticing $a_\t[\M r_z]=a_\t[r_v+\M r_z]$) yields
\[
A_{\ep,\t}\big[ r_v + \M r_z \big] \,=\, \left\langle f,\, \N r_z \right\rangle \,-\,
b_\t\big(\N z,\, r_v + \M r_z\big) \,-\, b_\t\big(v+ z,\, \N r_z\big).
\]
From this identity, along with  \eqref{Zcoerc}, Cauchy-Schwarz inequality, and \eqref{astructure}, we obtain 
\begin{flalign*}
A_{\ep,\t}\big[ r_v + \M r_z \big] & \le\,\,   \kappa_1 \left( \| f\|_{*\t}\, 
a_\t^{1/2}\left[ \M r_z\right] \,\,+\,\, a_\t^{1/2}\left[\M z\right]\,
b_\t^{1/2}\left[ r_v + \M r_z\right]  \,\,+\,\,
b_\t^{1/2}[ v+z]\, a_\t^{1/2}\left[\M r_z\right] \right)  \\
& \le\,\, \ep\,  \kappa_1 \left( \| f\|_{*\t}\, A_{\ep,\t}^{1/2}\left[ r_v + \M r_z \right] \,\,+\,\, 
B^{1/2}[v+z]\, A_{\ep,\t}^{1/2}\left[ r_v + \M r_z \right] \right).  
\end{flalign*}
(In the last inequality, along with the definitions \eqref{Aform} and \eqref{Bform} for $A_{\ep,\t}$ and $B$ respectively and the fact that 
$a_\t\left[r_v\right]=0$, 
discrete Cauchy-Schwarz inequality was also used.) 
This along with the first inequality in \eqref{18.09.20e1} gives 
$A_{\ep,\t}[ r_v + \M r_z ] \le   6 \kappa_1^2  \ep^2\| f\|_{*\t}^2$,  
and \eqref{trunc1.e2} follows via \eqref{ik430} and the squared triangle inequality. 

{\it Proof of \eqref{trunc1.e3}}. From \eqref{trunc1.e2} we only need showing 
$b_\t[\N z]\,\le \, \tfrac{1}{2} \ep^2 \kappa_1^2 \| f\|_{*\t}^2$ 
and this follows from \eqref{Zcoerc} and 
 the second inequality in \eqref{18.09.20e1}. 
\end{proof}}
\subsection{Case of $a_\t$ with  additional regularity}\label{sec.atreg}
While \eqref{distance} was sufficient for removing $\M$ from $b_\t$ and the right-hand-side (cf. problems \eqref{coupledbest} and \eqref{p.trunc1}), one cannot in general remove $\M$  from $a_\t$. 
However in the majority of examples, Section \ref{sec:examples}, $a_\t$ has an additional regularity in $\t$ which allows one to approximate $a_\t(\M\,z,\,\M\,\tilde z)$ up to quadratic terms in small $\t$ and thereby further simplify problem  \eqref{p.trunc1}.

So we additionally assume here 
that $\theta_0=0$ is an interior point of $\Theta$, and 
$a_\t$ 
satisfies the following ``differentiability'' properties with respect to $\t$ at $\t=0$. 
There exist sesquilinear maps $a'_0: V_0 \times H \rightarrow \CC^n$ and $a_0'': V_0 \times V_0 \rightarrow \mathbb{C}^{n\times n}$, 
i.e. vector-valued and matrix-valued maps respectively,
such that 
\begin{equation}\tag{H4}\label{H4}
\left\{ \hspace{.5em} \begin{aligned}
&\big| a_\t(v,u) - a'_0(v,u)\cdot \t \big| \,\,\le\,\, K_{a'} | \t|^2\, \|v\|_0\,\|u\|_0,  \qquad \forall v  \in V_0, \, \,\forall u\in H,\, \,\forall \t\in\Theta;\\
&\big|a_\t(v,\tilde{v}) - a''_0(v,\tilde{v})\t\cdot \t \big| \,\,\le\,\, K_{a''} |\t|^3\, \|v\|_0\,\|\tilde{v}\|_0,  \qquad \forall v, \tilde{v} \in V_0,\,\,\, \forall \t\in\Theta,
\end{aligned} \right.
\end{equation}
for some 
non-negative 
constants $K_{a'}, K_{a''}$. Notice that \eqref{H4} and \eqref{ass.alip} gives  
$| a'_0(v,u) \cdot \theta| \le \left(L_a | \t|+K_{a'}|\t|^2\right) \|v\|_0\|u\|_0$ for all $\t\in\Theta$. 
Considering 
$\t\to 0$ yields
\begin{equation}\label{a'a''cont}
| a'_0(v,u) \cdot \theta| \le L_a | \t| \|v\|_0\|u\|_0, \\
\qquad \forall v \in V_0,\, \forall u \in H,\, \forall \t\in\mathbb{R}^n.
\end{equation}
Notice that the non-negativity of $a_\t$ implies 
$a'_0(v,\tilde v)=0$, $\forall v,\tilde v\in V_0$. \\
We now demonstrate that assertion \eqref{H4} allows us to approximate $\N$ near $\t=0$ by some $N_\t$ which is linear in $\t$. 
To that end, in problem \eqref{IliaN} defining $\N$, approximate $a_\t$ on its left hand side by $a_0$ and $a_\t$ on 
the right hand side according to \eqref{H4} by $a_0'(v_0,\widetilde w_0)\cdot\t$. As a result, for each $\t\in \mathbb{R}^n$, 
we define  $N_\t: V_0 \rightarrow W_0$ so that $N_\t v$ for $v\in V_0$ is  a solution to 
``linearised'' abstract cell problem 
\begin{equation}
\label{cell:prob2}
a_0( N_\t v , \widetilde{w}_0)\, =\, -\,\,a'_{0}(v, \widetilde{w}_0) \cdot \t, \qquad \forall \widetilde{w}_0 \in W_0.
\end{equation}
The unique solvability of \eqref{cell:prob2} is ensured by \eqref{KA}  and \eqref{a'a''cont}; in particular, one has  
\begin{equation}\label{Nbdd}
\| N_\t v \|_0\, \le \, {L_a}{\nu_0}^{-1}|\t| \,\, \| v\|_0, \quad \forall v \in V_0.
\end{equation}
As the right hand side of \eqref{cell:prob2} is linear in $\t$, $N_\theta v = \theta\cdot N v$, 
where  $N : V_0 \rightarrow [W_0]^n$ is a bounded linear mapping. 
The following proposition establishes closeness of $N_\t$ to $\N$ for small $\t$. 
\begin{proposition}\label{truncN}
Assume \eqref{KA}, \eqref{H4} and  $\t \in \Theta,\,\, |\t| \le \nu_0/(2 L_a)$. Then, the following inequality holds:
\begin{equation}\label{curlNtrunc}
\big\| \N v\,- \, N_\t v\big\|_0 \,\, \le\,\, \kappa_2\, |\t|^2\, \| v\|_0, \quad \forall v\in V_0, \quad 
\mbox{ with } \, \kappa_2 = 
\nu_0^{-1} \left( 2L_a^2 \nu_0^{-1}  + K_{a'}\right).
\end{equation}

\end{proposition}
\begin{proof} For $R =\N v - N_\t v \in W_0$, by \eqref{IliaN} and \eqref{cell:prob2} we obtain 
\[
a_0[R]  
\,\,=\,\, a_0\left( \N v, R\right) - a_0\left( N_\t v, R\right)
\,\,=\,\,   a_0\left( \N v, R\right) - a_\t\left( \N v, R\right)    -  a_\t( v, R) \,+\,  a'_{0}(v, R) \cdot \t.
\]
Now, \eqref{ass.alip} and \eqref{Nbound} give  $|   a_0( \N v, R) - a_\t( \N v, R) | \le  2 L_a^2 \nu_0^{-1} |\t|^2 \|  v\|_0 \| R\|_0 $, 
and the first inequality in \eqref{H4}  gives $\big|   a_\t( v, R) -  a'_{0}(v, R) \cdot \t\big| \le  K_{a'} |\t|^2\, \| v\|_0 \|R\|_0$. Therefore $a_0[R] \le \nu_0 \kappa_2 |\t|^2\| v\|_0 \|R\|_0$
which along with \eqref{KA} gives \eqref{curlNtrunc}. 
%
\end{proof}
Now, we are in a position  to further approximate $a_\t\left( \M v,\M \tilde v \right)$ as entering the approximations in e.g. 
Theorem \ref{thm.trunc1}, see \eqref{p.trunc1}. 
To that end, recalling first that $\M = I + \N$ and applying \eqref{IliaN2} 
we observe  that 
\begin{equation}\label{amnorth}
a_\t\left(\M v,\, \M \tilde{v}\right) \,\,=\,\,  
a_\t( v,  \tilde{v})\,-\, a_\t\left( \N v,\, \N  \tilde{v}\right).
\end{equation}
Now, according to \eqref{H4} and \eqref{curlNtrunc} approximate $a_\t( v,  \tilde{v})$ by $a''_0(v,\tilde{v}) \t \cdot \t$, and  
$a_\t( \N v, \N  \tilde{v})$ by $a_0( N_\t v, N_\t \tilde{v})$. 
As a result, $a_\t(\M v, \M \tilde{v})$ is approximated by  the   sesquilinear form 
$a^{\rm h}_\t: V_0 \times V_0 \rightarrow \CC$, $\t \in \mathbb{R}^n$,  given by
\begin{equation}\label{defhom.form}
a^{\rm h}_\t(v,\tilde{v})\,\, : =\,\, 
a''_0\left(v,\tilde{v}\right) \t \cdot \t \,-\, a_0\left(  N_\t v, N_\t \tilde{v}\right) \,\,=\,\, 
a''_0\left(v,\tilde{v}\right) \t \cdot \t \,\,-\, a_0\big( \t \cdot Nv,\, \t \cdot N \tilde{v}\big), \quad \forall v, \tilde{v} \in V_0,
\end{equation}
which is a quadratic form in $\t$. We call $a_\t^h$ a ``homogenised'' form, for reasons to become clear later.  
The following proposition establishes the closeness of this approximation, and a $|\t|^2$-coercivity of $a^h_\t$ on $Z$. 
\begin{proposition}\label{prop.ahom}
		Assume \eqref{KA}--\eqref{H4}. Then, the following inequalities hold:
\begin{gather}
\label{atrunc.maine1}
\Big| a_\t\left(\M v , \M \tilde{v} \right) \,-\, a^h_\t\left(v,\tilde{v}\right) \Big| \,\,\,\le\,\,\, 
\kappa_3\, |\t|^3\, \| v\|_0\, \| \tilde{v} \|_0, \quad \forall v, \tilde{v} \in V_0, \,\, \,\forall \t \in \Theta,\,\, |\t| <  \nu_0/\left(2 L_a\right);\\
\label{ahcoercive}
a^h_\t[z] \,\,\,\ge\,\,\, \nu_\star |\t|^2\, \| z \|_0^2, \quad \forall z \in Z, \ \ \forall \t \in \mathbb{R}^n, 
\end{gather}
where $\kappa_3 =K_{a''} + {\nu_0}^{-1}{L_a}K_{a'} + L_a \kappa_2$, and $\nu_\star$ is given in  \eqref{150620.e1}. 
\end{proposition}
\begin{proof}[Proof of \eqref{atrunc.maine1}.]
From \eqref{amnorth} and \eqref{defhom.form}, 
\begin{flalign*}
a_\t\left(\M v , \M \tilde{v} \right) - a^h_\t\left(v,\tilde{v}\right)\,\,=\, 
\big[\,a_\t\left(v,\tilde v\right)-a''_0\left(v,\tilde{v}\right)\t\cdot\t \,\big] \,\,+ \,\, 
\big[\, a_0\left(  N_\t v, N_\t \tilde{v}\right)-a_\t\left( \N v, \N  \tilde{v}\right) \,\big]. 
\end{flalign*}
Note that the second inequality in \eqref{H4} provides a desired estimate for  the first bracketed term on the right. Let us consider the second term: applying  \eqref{cell:prob2} and \eqref{IliaN}, 
\[
a_0\left(  N_\t v, N_\t \tilde{v}\right)\,-\,a_\t\left( \N v, \N  \tilde{v}\right)\,=\, 
-\,\, a_0'\left(v , N_\t \tilde{v}\right)\cdot\t \,\,+\,\, a_\t\left(v , \N \tilde{v}\right) \,\, = 
\]
\[
\ \ \ \ \ \ \ \ \ \ 
\Big\{a_\t\left(v , N_\t \tilde{v}\right)\,-\,a_0'\left(v , N_\t \tilde{v}\right)\cdot\t \Big\} \,\, + \,\,  
a_\t\big(v , \N \tilde{v}-N_\t \tilde{v}\big). 
\]
By the first inequality in \eqref{H4} and \eqref{Nbdd} we obtain 
\[
\Big\vert a_\t\left(v , N_\t \tilde{v}\right) \,-\, a_0'\left(v , N_\t \tilde{v}\right)\cdot \t \Big\vert \,\,\le\,\,
 K_{a'} |\t|^2 \,\|  v\|_0\, \left\| N_\t \tilde{v} \right\|_0 
\,\le\,\,  {L_a}\,{\nu_0}^{-1} K_{a'} \, | \t|^3 \,\| v\|_0\, \| \tilde{v}\|_0\,.
\]
Also, by \eqref{ass.alip} and \eqref{curlNtrunc} we deduce that 
\[ 
\big\vert a_\t(v ,\, \N \tilde{v}-N_\t \tilde{v})\big\vert \,\,\le\,\,
 L_a |\t|\, \| v\|_0\, 
\big\|  \N \tilde{v} \,-\, N_\t \tilde{v}\big\|_0 \,\,\le\,\,\, L_a\, \kappa_2\, |\t|^3\, \| v\|_0\, \| \tilde{v} \|_0. 
\] 
Combining the above estimates yields \eqref{atrunc.maine1}. \\
{\it Proof of \eqref{ahcoercive}.}	  For fixed $z\neq 0$, $a^{h}_\t[z]$ as defined by \eqref{defhom.form} 
is  quadratic in $\t$, and so for each fixed ``direction'' $\t|\t|^{-1}$ ($\t\neq 0$) 
  the ratio $ a^{h}_\t[z] /  \left(|\t|^2 \| z\|_0^2\right)$ is independent of $|\t|$. Moreover, we recall that $0$ in an interior point  of $\Theta$. So to prove \eqref{ahcoercive}, for a chosen $\t\neq 0$ we bound the ratio via passing to the limit as $|\t|\to 0$ along the corresponding direction and 
successively using  \eqref{atrunc.maine1} and   \eqref{150620.e1},  as follows:  
\[
\frac{a_{\t}^{h}[ z]}{|{\t}|^2\| z\|_0^2}  \,\,=\,\,
\lim_{|\t| \rightarrow 0}\frac{a_{\t}^{h}[ z]}{|{\t}|^2\| z\|_0^2}  
\,\,= \,\, \lim_{|\t|\rightarrow 0}  \frac{a_{\t}[\M z]}{|{\t}|^2\| z\|_0^2} \,\, \ge \,\, \lim_{|\t|\rightarrow 0}  \nu_\star\,=\, \nu_\star.\qedhere
\]
\end{proof}
Now, we are ready to further simplify approximate problem \eqref{p.trunc1}. 
\begin{theorem}\label{thm.maindiscthm}
	Assume \eqref{KA}--\eqref{H4}  and let $\t \in \Theta$, $\,|\t| < r_1$ for $r_1$ as in \eqref{btnormonV+Z}. Then, there exists a unique solution  $v^h+z^h \in V_\t^\star \dot{+} Z$ to
	\begin{equation}
	\label{z3prob}
	\ep^{-2} a^{\rm h}_{\t}\left(z^h,\, \tilde{z}\right) \,\,+\,\, b_\t\left(v^h+z^h,\, \tilde{v}+ \tilde{z}\right)\,\,\, =\,\,\, 
	\left\langle f,\, \tilde{v} +\tilde{z} \right\rangle, \quad \,\, \forall\, \tilde{v} +\tilde{z} \in V_\t^\star \dot{+}  Z.
	\end{equation}
	Furthermore, $v^h + (I + N_\t) z^h$ approximates $u_{\ep,\t}$, the solution to \eqref{p1}, in the following sense: 
	\begin{equation}
	\label{final1}
	\ep^{-2}a_\t\big[u_{\ep,\theta} - \left(v^h + (I+ N_\t) z^h\right)\big]\,+\,	
	b_\t\big[ u_{\ep,\theta} -\left(v^h + (I+N_\t)z^h\right)\big] \,\,\,\le\,\,\,    C_5\,\ep^2\,\|f\|_{*\t}^2,
	\end{equation}
	for $C_5 = 3 C_4 +3K^4 \kappa_2^2 \nu_\star^{-2} + \tfrac{3}{2}K^2 L_a^2 \nu_0^{-2}  \nu_\star^{-1}  + 3K^2 \kappa_3^2 \nu_\star^{-3}$.	 
	Moreover, $v^h +z^h$ approximates $u_{\ep,\t}$ as follows:
	\begin{equation}
	b_\t\big[ u_{\ep,\theta} -(v^h +z^h)\big] \,\,\,\le\,\,\,    C_6\,\ep^2\,\|f\|_{*\t}^2, \quad C_6 = 2C_4 + 2K^2 \kappa_3^2 \nu_\star^{-3}.
 \label{final2}
	\end{equation}
\end{theorem}
\begin{proof}
Since $a^{\rm h}_\t$ is  bounded (e.g. via \eqref{defhom.form}, \eqref{H4} and \eqref{Nbdd}) 
and non-negative  on $Z$ (see Proposition \ref{prop.ahom}), by arguing as in the beginning of the proof of Theorem \ref{thm.trunc1}, it follows that \eqref{z3prob} is well-posed and
\begin{equation}\label{zhgoodnear0}
	\ep^{-2} a^{\rm h}_\t[z^h] \, \le \, \tfrac{1}{2} \| f\|_{*\t}^2.
\end{equation}
As a further preparation, we need a more refined estimate for $z^h$. 
To that end, we 
set in \eqref{z3prob} $\tilde{z}=z^h$ and  $\tilde{v}= -\,P_{V_\t^\star} z^h$, and note that 
$\tilde{v} + \tilde{z} = P_{W_\t^\star} z^h$  and 
(as for $v \in V_\t^\star$ and $w\in W_\t^\star$, $b_\t(v,w)= (v,w)_\t = 0$) 
that $b_\t \left(v^h + z^h, P_{W_\t^\star} z^h\right) = b_\t\left[P_{W_\t^\star} z^h\right]$. This  gives
\[
\ep^{-2} a^h_\t[z^h] + b_\t[P_{W_\t^\star} z^h] = \langle f, P_{W_\t^\star}z^h \rangle \,\le\, \| f\|_{*\t} \| P_{W_\t^\star}z^h \|_\t 
 \,\le\, \| f\|_{*\t}\, \|z^h \|_\t \,\le\, K  \| f\|_{*\t}\, \| z^h \|_0. 
\]
Along with \eqref{ahcoercive}  this yields  
\begin{equation}\label{100720.e2}\ep^{-2} |\t|^2 \| z^h \|_0 \,\le\,  {K}{\nu_\star}^{-1}  \| f\|_{*\t}.
\end{equation}

{\it Proof of \eqref{final1}.} With the aim of exploiting \eqref{trunc1.e3}, 
decompose the argument in the square brackets on left-hand-side of \eqref{final1} in two slightly different ways as follows:
\begin{eqnarray*}
u_{\ep,\theta} - \left(v^h + (I+ N_\t) z^h\right)&=&\big[u_{\ep,\theta} - (v + \M  z)\big] \,+ \,
\left[\left(v-v^h\right) +\M \left(z-z^h\right)\right] \,+\, \left( \N z^h - N_\t z^h\right)= \nonumber \\
&&\big[u_{\ep,\theta} - (v +   z)\big] \,\,+ \,\,
\left[\left(v-v^h\right) + \left(z-z^h\right)\right]\,\, -\,\, N_\t z^h. 
\end{eqnarray*}
Applying e.g. a squared triangle inequality to the first of the above decompositions for the $a_\t$-term on the 
left-hand-side of \eqref{final1} and to the second decomposition for the $b_\t$-term results in bounding the whole 
left-hand side of \eqref{final1} 
from above by
\[
3\Bigl(\ep^{-2}a_\t\big[u_{\ep,\theta} -(v+ \M z)\big]+	b_\t\big[ u_{\ep,\theta} - (v+z)\big] \Bigr) 
+  3B\left[\left(v-v^h\right) +\left(z-z^h\right)\right]
+3\ep^{-2} a_\t\big[\N z^h - N_\t z^h \big]+3b_\t[N_\t z^h],
\]
where $v+z$ solves \eqref{p.trunc1} and $B$ is given by \eqref{Bform}. By \eqref{trunc1.e3} the first term is bounded by  
$3C_4 \ep^2 \| f\|_{*\t}^2$.
 By \eqref{as.b1}, \eqref{curlNtrunc} and \eqref{100720.e2}, the third term is  bounded by $3K^4 \kappa_2^2 \nu_\star^{-2} \ep^2 \|f \|_{*\t}^2$. By \eqref{Nbdd}, \eqref{ahcoercive} and \eqref{zhgoodnear0}, the last term is bounded by 
$\tfrac{3}{2}K^2 L_a^2 \nu_0^{-2}  \nu_\star^{-1}  \ep^2 \| f\|_{*\t}^2$. 
  So it remains to bound the second term. By subtracting \eqref{z3prob} from \eqref{p.trunc1} 
	(both with $\tilde{v}=v-v^h$ and $\tilde{z}=z-z^h$)  we deduce that
\[
B\left[v-v^h +z-z^h\right]
= \ep^{-2} a^h_\t(z^h,z-z^h)  - \ep^{-2} a_\t(\M z^h , \M (z-z^h) ).
\]
Now, by sequentially applying  \eqref{atrunc.maine1}, \eqref{150620.e1} and \eqref{100720.e2} we obtain  
\[
\begin{aligned}
B\left[v-v^h +z-z^h\right] & \,\le\, \ep^{-2}  | \t|^3  \kappa_3 \|z^h\|_0 \,\| z-z^h\|_0 \,  \le \, \ep^{-2}| \t|^2  \kappa_3 \|z^h\|_0 \,   \nu_\star^{-1/2 }a_\t^{1/2}\left[\M(z-z^h)\right] \\
& \le \, K \nu_\star^{-3/2} \kappa_3 \| f\|_{*\t}\,  a_\t^{1/2}\left[\M(z-z^h)\right] \,\le\,  \ep K \nu_\star^{-3/2} \kappa_3 \| f\|_{*\t}  B^{1/2}[v-v^h +z-z^h].
\end{aligned}
\] 
Thus 
\begin{equation}\label{280920e1}
 B\left[v-v^h +z-z^h\right] \, \le \, K^2 \kappa_3^2 \nu_\star^{-3} \ep^2 \| f\|_{*\t}^2
\end{equation}
 and \eqref{final1} follows by combining the above bounds.

{\it Proof of \eqref{final2}.} Since the left-hand-side of \eqref{final2} is bounded by 
$2b_\t\big[u_{\ep,\t} - (v+z)\big] + 2b_\t\big[v+z - (v^h+z^h)\big]$, 
the desired inequality immediately follows from \eqref{trunc1.e3} and \eqref{280920e1}. 
\end{proof}
\subsection{Case of continuous $b_\t$}
The last in our hierarchy of approximating 
problems, problem \eqref{z3prob}, has two main advantages: restriction to a smaller 
subspace $V_\t^\star\dot{+}Z$, and replacement of the singular form $a_\t$ by the ``homogenised form'' $a_\t^h$  which is a quadratic form in 
$\t$ (as well as is restricted further to the defect subspace $Z$ only). The dependence of $b_\t$ on 
$\t$ remains so far unspecified, however as we will later see, if it were possible to approximate it for small $\t$ by a 
$\t$-independent $b_0$ (as well as to replace 
$V_\t^\star$ by $\t$-independent $V_\star$) that would provide significant additional benefits for properties of the approximate problem. 
In particular, as we will see in Section \ref{s:resolv}, 
the resulting approximation can be expressed in terms of a solution to an 
abstract version of a two-scale limit operator with important further implications.

To that end, we make here the following additional assumption: $\Theta$ is connected, and $b_\t$ is Lipschitz continuous at $\t=0$ i.e. there exists $L_b\ge 0$ such that
\begin{equation}\label{H5}\tag{H5}
\big|b_\t(v, \tilde{v}) \,-\,b_0( v, \tilde{v} )\big| \,\, \le \,\, L_b |\t|\, \| v\|_0\, \| \tilde{v} \|_\t, \quad \forall v, \,\tilde{v} \in V_0,\,\, \forall\, \t \in \Theta.
\end{equation}

First, we observe that \eqref{H5} implies 
existence of a transfer operator $\mathcal{E}_\t$ which plays an important role by allowing to state the forthcoming approximate problem on $\t$-independent subspace 
$V_\star\dot{+}Z$. 
\begin{lemma}\label{propeth}
Conditions \eqref{contVs} and \eqref{H5} 
imply $\forall\t\in\Theta$ 
 existence of a bijection $\mathcal{E}_\t : V_\star \rightarrow V_\t^\star$ such that 
\begin{gather}
b_\t\left(\mathcal{E}_\t v,\mathcal{E}_\t\tilde{v}\right)\,\,=\,\,b_0( v,\tilde{v}), \quad \forall v,\tilde{v}\in V_\star,  \label{Eprop1} \\
\text{and} \hspace{\textwidth} \nonumber\\
\big\vert b_\t(\mathcal{E}_\t v,  z)\,-\,b_0( v,z)\big\vert \,\, \le \,\, K_b\, |\t| \, \| v\|_0 \, \| z \|_0, \quad \forall v \in V_\star, \forall z \in Z, \label{Eprop2}
\end{gather}
for some constant $K_b \ge0$ independent of $\t$. 
\end{lemma} 
In most of the relevant examples (Section \ref{sec:examples}),  
operator $\mathcal{E}_\t$ 
will be naturally identified. 
In its abstract form, a proof of the lemma is given in the Appendix A.  Notice that \eqref{Eprop2} is equivalent to 
\begin{equation}\label{Eprop3}
\big\vert b_\t(z, v')-\,b_0\left( z,\mathcal{E}_\t^{-1}v'\right)\big\vert\,\,\le\,\, K_b\, |\t|\, \| v'\|_\t \| z \|_0,\quad \forall v' \in V_\t^\star, \,\forall z \in Z.
\end{equation}
Our aim is to simplify further the last approximate problem \eqref{z3prob} by stating it 
on the $\t$-independent subspace $V_\star \dot{+} Z$ instead of $V_\t^\star \dot{+} Z$, as 
well as approximating $b_\t$ by $b_0$. The former can be achieved via the above transfer 
operator $\mathcal{E}_\t$ by replacing in \eqref{z3prob} $v^h$ and $\tilde v$ 
(both in  $V_\t^\star$) by 
$\mathcal{E}_\t v$ and $\mathcal{E}_\t\tilde v$, with 
$v$ and 
$\tilde v$ now in $V_\star$. For the latter, the hope is to  use \eqref{H5}. 
Properties \eqref{Eprop1} and \eqref{Eprop2} 
allow for 
$\mathcal{E}_\t$ 
to 
be dropped from the $b$-term, but not from the right-hand side. 
As a result, 
the following important theorem providing an approximation to the original problem \eqref{p1} by the new simplified problem holds. 
\begin{theorem}\label{thm.IKunifest}
	Assume \eqref{KA}--\eqref{H5} and consider  $f \in H^*$, $\,\t \in \Theta$, 
	$\,|\t| <r_1$ for $r_1$ as in \eqref{btnormonV+Z}, and	$u_{\ep,\t}$ the solution to   \eqref{p1}, 
	and let $\mathcal{E}_\t: V_\star \to V_\t^\star$ be as in Lemma \ref{propeth} i.e. such that \eqref{Eprop1} and \eqref{Eprop2} hold. 
	Then, there exists a unique  solution $v + z \,\in\, V_\star \dot{+} Z$  to 
\begin{equation}
	\label{IKz3prob88}
	a^{\rm h}_{\t/\ep} 
	(z, \tilde{z}) \,\,+\,\, b_0(v+z, \tilde{v}+ \tilde{z}) \,\,\,=\,\,\, 
	\left\langle f,\, \mathcal{E}_\t\tilde{v} +\tilde{z} \right\rangle, \quad \forall\, \tilde{v} + \tilde{z}\, \in\, V_\star \dot{+}  Z\,,
	\end{equation}
and, there exist  constants $C_7$ and $C_8$, independent of $\ep$, 
$\t$ and $f$, such that
	\begin{eqnarray}
			\label{IKfinal1}
\ep^{-2} a_\t\big[u_{\ep,\theta}  - \big(\mathcal{E}_\t v + (I+N_\t)z\big)\big] \,\,+\,\,
b_\t\big[u_{\ep,\theta}  - \big( \mathcal{E}_\t v + (I+ N_\t) z\big)\big] &\, \le \,& C_7\, \ep^2 \| f\|_{*\t}^2; \\
		b_\t\big[\, u_{\ep,\theta} \,-\,\left(\mathcal{E}_\t v +z\right)\,\big] &\, \le \,&    
		C_8\,\ep^2\|f\|_{*\t}^2.
		\label{IKfinal2}
	\end{eqnarray}
\end{theorem}
\begin{proof} 
Some parts of the 
	proof are quite similar to those of Theorem \ref{thm.maindiscthm}, so we will be slightly less detailed then. 
	Notice first that since $a_\t^{\rm h}$ given by \eqref{defhom.form} is a quadratic form in $\t$, 
	$\ep^{-2}a_\t^{\rm h}=a^{\rm h}_{\t/\ep}$. 	
Further, since $b_0[\cdot] = \| \cdot \|_0^2$ on $V_\star \dot{+}Z$, 
and 
$a^{\rm h}_\t$ is bounded and non-negative 
on $Z$, 
the sesquilinear form given by the left-hand-side of \eqref{IKz3prob88}  is bounded and coercive on $V_\star \dot{+}  Z$ and hence, as $\mathcal{E}_\t$ is  clearly bounded, problem \eqref{IKz3prob88} is well-posed. Furthermore, taking in  
\eqref{IKz3prob88} $\tilde v=v$ and $\tilde z=z$, and using \eqref{ahcoercive} and \eqref{VZorth2}, 
and occasionally denoting by $C$ a positive constant independent of $\t, \ep$ and $f$ whose precise value may change from line to line, we first obtain 
$\left(\ep^{-2} |\t|^2 + 1\right) \| z\|_0^2+\|v\|^2_0 \le C \|f\|_{*\t}\left(\|z\|_0+\|v\|_0\right)$. 
This immediately bounds $\|v\|_0^2$ and  $\|z\|_0^2$ by $C \,\|f\|_{*\t}^2$, and as a result so also  $\ep^{-2} |\t|^2\|z\|_0^2$. 
We also obtain a refined 
estimate for $z$ analogous to \eqref{100720.e2} by taking in \eqref{IKz3prob88} 
 $\tilde{z}=z$ and  $\tilde{v}= -\,P_{V_\star} z$, concluding that $\ep^{-4} |\t|^4\|z\|_0^2$ is also  
bounded by $C \,\|f\|_{*\t}^2$. 
Combining the above estimates, we obtain 
\begin{equation} \label{est1234} 
\Big(\ep^{-4} |\t|^4 + \ep^{-2} |\t|^2 + 1\Big) \| z\|_0^2\,+\,\|v\|^2_0\,\,\le \,\,C\,
\|f\|_{*\t}^2.
\end{equation}
%
By Theorem \ref{thm.maindiscthm} to prove both \eqref{IKfinal1} and \eqref{IKfinal2} we only need bounding related difference terms:  
\begin{equation}\label{18082e2}
\ep^{-2}a_{\t}\big[ N_\t(z^h-z)\big] \,+\,  b_\t\big[ N_\t(z^h-z)\big]  \,+\,
 \ep^{-2}a_{\t}\big[z^h-z\big] \,+\,  b_\t\big[ v^h+z^h -(\mathcal{E}_\t v +z)\big],  
\end{equation}
where $v^h+z^h$ is the solution to \eqref{z3prob}. 
Now, (recalling $\ep <1$) by \eqref{Nbdd} and \eqref{ahcoercive}, 
\[
\ep^{-2}a_{\t}\big[  N_\t (z^h-z)\big] \,+\, b_{\t}[ N_\t (z^h-z)]  \,\,\le \,\, 
C\,
\ep^{-2}  a^{\rm h}_\t\big[z^h-z\big].
\]
Next, \eqref{amnorth}, 
\eqref{atrunc.maine1} and \eqref{Nbound} first show that  
$\ep^{-2} a_\t\big[z^h-z\big] = \ep^{-2} a_\t\big[\M (z^h-z)\big] + 
\ep^{-2} a_\t\big[\N (z^h-z)\big]$ is bounded by a multiple of 
$\ep^{-2}a^{\rm h}_\t\big[z^h-z\big] +\ep^{-2} |\t|^3 \left\| z^h-z\right\|_0^2+
\ep^{-2} |\t|^2 \left\| z^h-z\right\|_0^2$. 
Then, as $|\t|^3\le \left(|\t|^2+|\t|^4\right)/2$ and bounding the resulting $|\t|^2$-terms via the $a^{\rm h}_\t$-term using 
\eqref{ahcoercive}, we conclude that $\ep^{-2} a_\t\big[z^h-z\big]$  is bounded by a multiple of 
$ 
\ep^{-2}a^{\rm h}_\t\big[z^h-z\big] \,+\, \ep^{-2} |\t|^4 
\Big( \| z^h\|_0^2 \,+\, \|z\|_0^2\Big)$. 
Consequently, via \eqref{100720.e2} and \eqref{est1234}, we appropriately bound \eqref{18082e2} if we bound 
\begin{equation}\label{18.08.20e1}
\ep^{-2} a^h_\t\big[z^h - z\big] \,+\, b_\t\big[ v^h+z^h -\left(\mathcal{E}_\t v +z\right)\big].
\end{equation}
To this end, we shall demonstrate that replacing in \eqref{z3prob} $v^h+z^h$ by 
$\mathcal{E}_\t v +z$ produces a small error on the right-hand side. For $\tilde{v} \in V_\t^\star$ and $\tilde{z} \in Z$, utilising \eqref{Eprop1}, 
 we deduce 
\[
b_\t\big(\mathcal{E}_\t v+z, \tilde{v}+ \tilde{z}\big) \,\,= 
\,\,  
b_0\big( v,\,\mathcal{E}_\t^{-1}\tilde{v}\big) +b_\t\big(\mathcal{E}_\t v,  \tilde{z}\big)+
b_\t\big(z, \tilde{v}+ \tilde{z}\big)
\,\,=\,\,
b_0( v+z,\, \mathcal{E}_\t^{-1}\tilde{v}+\tilde{z})\,\, +\,\, J,
\]
where	$J=\big(b_\t(\mathcal{E}_\t v,  \tilde{z})-b_0( v,\tilde{z}) \big)+\big(b_\t(z, \tilde{v})-b_0( z,\mathcal{E}_\t^{-1}\tilde{v})\big)+\big( b_\t(z,  \tilde{z})
-b_0( z,\tilde{z})\big)$. Thus, via \eqref{IKz3prob88}, 
\begin{equation}
\label{IKz3prob89}
\ep^{-2} a^{\rm h}_{\t}(z, \tilde{z}) \,+\, b_\t(\mathcal{E}_\t v+z, \tilde{v}+ \tilde{z})
\,\, =\,\, \langle f,   \tilde{v} +\tilde{z} \rangle \,+J, \,\,\, \quad \forall\, (\tilde{v}, \tilde{z}) \in V_\t^\star \times  Z.
\end{equation}
We show that $J$ is small. Indeed, \eqref{Eprop2}, \eqref{Eprop3}, \eqref{H5} and estimates \eqref{est1234} provide the following bound: 
	\begin{equation}\label{Jbound}
	|J|\,\le \, C\,\ep \|f\|_{*\t}\Big(\ep^{-2} |\t|^2 \|\tilde{z}\|_0^2\,+\,\|\tilde{v}\|_0^2\Big)^{1/2}, \quad \,\, \forall \tilde{v} \in V_\t^\star,\, \, \forall \tilde{z} \in Z,
	\end{equation}
for some $C>0$ independent of $\t,\, \ep$ and $f$. 
Comparing \eqref{IKz3prob89}  with \eqref{z3prob}	we conclude that
\begin{equation}
\label{IKz3prob899}
\ep^{-2} a^{\rm h}_{\t}\big(z-z^h, \tilde{z}\big) \,+\, 
b_\t\big(\mathcal{E}_\t v+z-(v^h+z^h),\, \tilde{v}+ \tilde{z}\big) \,\, =\,\,J,\quad \  \forall\, (\tilde{v}, \tilde{z}) \in V_\t^\star \times  Z.
\end{equation}
Finally, to bound \eqref{18.08.20e1}, we can set in \eqref{IKz3prob899} $\tilde z=z-z^h$ and $\tilde v= \mathcal{E}_\t v-v^h$, 
and then use \eqref{Jbound} followed by \eqref{ahcoercive}, \eqref{V+ZCoercive} and \eqref{btnormonV+Z}. 
\end{proof}

We end this subsection 
by first recalling that one can produce a global in $\t$  approximation to $u_{\ep,\t}$  by combining  an approximation for $|\t | <r_1$, given by Theorem \ref{thm.trunc1}, \ref{thm.maindiscthm} or \ref{thm.IKunifest}, with the approximation $v_\t$ for $|\t| \ge r_1$ given by Theorem  \ref{thm1.all} (with $1/\nu(r_1)  \le  \gamma^{-1} r_1^{-2}$, cf. \eqref{distance}). 
However, it turns out that the solution $v+z$ to \eqref{IKz3prob88} is well-defined  also 
when $|\t| \ge r_1$, and can be seen to still approximate $u_{\ep,\t}$ up to leading order. 
Such global approximations 
will play a vital role for some of our subsequent constructions in Section \ref{s:resolv}, as well as in various examples of Section \ref{sec:examples}. 
The following important theorem holds. 
\begin{theorem}\label{thm.IKunifest2}
	Assume \eqref{KA}--\eqref{H5} and consider  $f \in H^*$, $\t \in \Theta$,	$u_{\ep,\t}$ the solution to   \eqref{p1}. Then, there exists a unique  solution $v + z \in V_\star \dot{+} Z$  to 	\eqref{IKz3prob88}, 
	and there exist constants $C_{9}$ and $C_{10}$, independent of $\ep$, $\t$ and $f$, such that
	\begin{eqnarray}
	\label{IKfinal3}
		\ep^{-2} a_\t\big[u_{\ep,\theta}  \,-\, \big(\mathcal{E}_\t v + (I+N_\t)z\big)\big] 
		\,\,+\,\, b_\t\big[u_{\ep,\theta} \, -\, \big( \mathcal{E}_\t v + (I+N_\t) z\big)\big]  
		\,&\le& \, C_{9}\, \ep^2\, \| f\|_{*\t}^2,  \\ 
	\label{IKfinal3-2}	
		b_\t\big[ u_{\ep,\theta} \,-\, \left(\mathcal{E}_\t v +z\right)\big]
		\,&\le& \, C_{10}\, \ep^2 \,\| f\|_{*\t}^2.		
	\end{eqnarray}
\end{theorem}
\begin{proof} Due to Theorem \ref{thm.IKunifest} we only need to consider the case $|\t| \ge r_1$. Note that the well-posedness of \eqref{IKz3prob88} and estimates \eqref{est1234} presented in the proof of Theorem \ref{thm.IKunifest} remain valid for all $\t \in 
\Theta$. 
In particular, for $|\t| \ge r_1$, 
\eqref{est1234} implies 
\begin{equation}\label{281020e1}
\| z\|_0^2 \le C 
r_1^{-4} \ep^4 \| f\|_{*\t}^2\,,
\end{equation}
and so to prove \eqref{IKfinal3} and \eqref{IKfinal3-2} we only need bounding the difference 
$
	\ep^{-2} a_\t[u_{\ep,\theta}  - \mathcal{E}_\t v] + b_\t[u_{\ep,\theta}  - \mathcal{E}_\t v].
$
Now by recalling Theorem \ref{thm:contV} we see that it remains to bound the difference
 $ 
 \ep^{-2} a_\t[ v_\t - \mathcal{E}_\t v] \, + \, b_\t[v_\t  - \mathcal{E}_\t v] \, =\, b_\t[v_\t  - \mathcal{E}_\t v]$, 
where $v_\t\in V_\t=V_\t^*$ solves \eqref{thmcontv.vprob}. 
 Setting $\tilde{z} = 0$ in \eqref{IKz3prob88}  and utilising \eqref{Eprop1} implies that $\mathcal{E}_\t v \in V_\t^\star$ solves 
 $ 
 b_\t(\mathcal{E}_\t v , \tilde{v}) = \langle f, \tilde{v} \rangle - b_0(z, \mathcal{E}_\t^{-1} \tilde{v})$, 
$\forall \tilde{v} \in V_\t^\star$. 
 Comparing this to  \eqref{thmcontv.vprob} and using \eqref{281020e1} implies  
$ b_\t[v_\t  - \mathcal{E}_\t v]  \le C r_1^{-4} \ep^4 \| f\|_{*\t}^2$, completing the proof.
\end{proof}

\subsection{A strengthening of condition \eqref{KA} 
}
\label{sec.newKA}
In conclusion of this section, we provide a sufficient condition for \eqref{KA} which on the one hand is often simpler to  
verify for a broad class of examples, and on the other hand assures an important additional property of finite dimensionality 
of the defect subspace $Z$.  The latter provides a substantial further simplification, as 
  the singular form $a^h_\t[z]$ that appears in the approximate problems \eqref{z3prob}  and \eqref{IKz3prob88} 
	can then be represented as a finite dimensional 
	matrix (the homogenised tensor). 

Recall \eqref{KA2.1} from Remark \ref{r.oldkafromnewka}, that is  there exists $C >0 $ and a  non-negative sesquilinear form $c$ on $H$, $\|\cdot\|_\t$-compact\footnote{A form $c$ is $\|\cdot\|_\t$-compact if every sequence $\{u_n\}$ in $H$, bounded in $\|\cdot\|_\t$, has a convergent subsequence $\left\{u_{n_k}\right\}$ with respect to $c$, i.e. $c\left[u_{n_k}-u \right]\to 0$ as $k\to \infty$ for some $u$. Notice that, by the uniform equivalence \eqref{as.b1} of the norms $\| \cdot \|_\t$, the  compactness needs only be established for one $\t$ to hold for all $\t$.} for some 
$\t\in \Theta$, such that
\begin{equation}\tag{H1$^\prime$}
\label{KA2}
\|w\|_\t^2\,\,\, \le\,\,\, C a_\t[w] \,\,+\,\, c[w], \quad \forall w \in W_\t, \; \, \forall \t \in \Theta.
\end{equation}
Inequality \eqref{KA2} can be interpreted as condition that the forms $a_\t$ are ``uniformly coercive on $W_\t$ plus 
compact'': $a_\t=\mathfrak{a}_\t+k$, where $\mathfrak{a}_\t[w]:=a_\t[w]+C^{-1}c[w]\ge C^{-1}\|w\|_\t^2$, 
$\forall \t\in W_\t$, are uniformly coercive and 
$k[w]:=-\,C^{-1}c[w]$ is compact. 
The next result follows from standard compactness arguments that we present here for the reader's convenience.
\begin{proposition}
	\label{prop.kaequiv}
Assertion \eqref{KA2} implies \eqref{KA}.
\end{proposition}
\begin{proof} Suppose  \eqref{KA} does not hold for some $\t \in \Theta$. Then there exists a sequence $w_n \in W_\theta$  such that $a_\t[w_n] < \tfrac{1}{n} \| w_n \|_\t^2$. 
Notice that \eqref{KA2} implies $c[w_n]>0$ for $n>C$, so we can assume $c[w_n]=1$. 
Then   \eqref{KA2} implies  $w_n$ is bounded in $H$. Consequently, 
up to a discarded subsequence, $\lim_n c[w_n-u]=0$ for some $u\in H$. 
Moreover (possibly up to another subsequence) $w_n$ weakly converges to some $w \in H$. 
The $\|\cdot \|_\t$-compactness of $c$ implies that $c$ is bounded in $H$ (i.e. $c[u] \le C'\|u\|^2_\t$, $\forall u\in H$,  
for some $C'>0$). 
Hence, $\forall\, \tilde u\in H$, $c(u-w, \tilde u)= \lim_n c(w_n-w,\tilde u)- \lim_n c(w_n-u,\tilde u)=0$, and so $c[u-w]=0$. 
Therefore $c[w] = c[u] = \lim_n c[w_n] = 1$. 
We  now demonstrate that $w\in W_\t \cap V_\t = \{0\}$. 
Clearly $w \in W_\t$ since $W_\theta$ is weakly closed (as an orthogonal complement). On the other hand, since $a_{\t}$ is non-negative and bounded in $H$ 
it is 
weakly lower semi-continuous, and therefore, $a_\t[w] \le \liminf_n a_\t[w_n]$ =0, i.e. $w \in V_\theta$.
\end{proof}
One advantage of \eqref{KA2} is that 
it provides a direct means to verify \eqref{KA}. 
We finally turn to the other important implication of \eqref{KA2}, the finite dimensionality of the defect subspace $Z$.  
\begin{proposition}\label{prop:Zfinite}
	Assume \eqref{KA2} and \eqref{contVs}. Then any 
	$Z$ satisfying \eqref{spaceZ}--\eqref{VZorth} is finite dimensional.
\end{proposition}
\begin{proof}
	Show first that since $c$ is $\|\cdot\|_0$-compact it is sufficient to prove that for 
	some $0\neq \t \in \Theta$ and $\kappa >0$, 
	\begin{equation}
	\label{Zfinite}
	\|z\|_0^2 \,\,\,\le \,\,\, \kappa\, c\big[ P_{W_\t^\star}z\big], \quad \forall z \in Z.
	\end{equation}
Indeed, for a bounded sequence $\{z_n\}$ in $Z$, $\{w_n\}:=\left\{P_{W_\t^\star}z_n\right\}$ is also bounded. 
Hence (up to a subsequence) 
 $c[w_n-u]\to 0$ for some $u\in H$, and so $c[w_n-w_m]\to 0$ as $m,n\to\infty$. 
Then \eqref{Zfinite} implies $\{z_n\}$ is a Cauchy sequence and hence 
 $z_n\to z\in Z$. 
So every bounded sequence in $Z$ has a convergent subsequence and hence $Z$ must be finite-dimensional. 

Let us now prove \eqref{Zfinite}. Fixing $z \in Z$, for small enough $\t\in\Theta\backslash\{0\}$, 
by \eqref{V+ZCoercive} for $v_\t^\star = - P_{V_\t^\star} z$ and $ w_0 =0$, \eqref{as.b1} and \eqref{KA2},  we obtain  
	\begin{equation}
	\label{Zfiniteaddede1}
\tfrac{1}{3}(1-K_Z) \| z\|_0^2 \,\,\le\,\,  \left\| P_{W_\t^\star} z \right\|_0^2\,\,\le\,\, K^2 \left\| P_{W_\t^\star}z \right\|_\t^2 
\,\,\le\,\, K^2\, \big( C a_\t[z] \,+\, c[P_{W_\t^\star}z]\big).
	\end{equation}
	Now 
	note that
 \eqref{ass.alip} gives $a_\t[z] \le L_a |\t|\, \| z\|_0^2$ and hence, for small enough $\t$, 
\eqref{Zfiniteaddede1} implies \eqref{Zfinite}.
\end{proof}

\section{Uniform approximations 
for related operators and 
spectra} 
\label{s:resolv}
In this section, we develop certain approximations for general classes of self-adjoint operators generated by the forms 
$A_{\ep,\t}\,=\,\ep^{-2} a_\t \,+\, b_\t$ and for their spectra, with uniform in $\t\in \Theta$ error estimates as $\ep\to 0$. \\
An abstract setup for wide classes of examples, see Section \ref{sec:examples}, is as follows. 
Let $\mathcal{H}$ be a complex separable Hilbert space with a family of uniformly equivalent inner products $d_\t$  
for each $\t \in \Theta$, i.e. 
\begin{equation}
\label{vs61}
d_{\t_1}[u] \,\,\le\,\, K_d\, d_{\t_2}[u], \quad \forall u \in \mathcal{H},\, \ \ \forall \t_1, \t_2 \in \Theta,  
\quad \text{ for some } K_d>0. 
\end{equation} 
Assume that 
$H$ is a compactly embedded\footnote{ 
The compact embedding assumption, which holds for most of the examples, is introduced for simplifying the exposition. 
In principle, it can be relaxed, cf. Example \ref{sec:nonloc2}. } 
dense subset of $\mathcal{H}$. Furthermore, we assume that 
\begin{equation}
\label{ik2}
d_\t[u] \,\,\le\,\, b_\t[u], \quad \forall u \in H,\, \ \forall \t \in \Theta.
\end{equation} 
Consider the self-adjoint operator $\mathcal{L}_{\ep,\theta}$  in $\mathcal{H}$ with inner product $d_\t$, generated 
according to the standard Friedrichs extension procedure by the (non-negative, closed, densely-defined) sesquilinear form $A_{\ep,\t}$ with the form domain $H$. 
Note that $\mathcal{L}_{\ep,\theta}$ has compact resolvent and therefore has a discrete spectrum which consists of 
the sequence of positive real eigenvalues $\{\lambda_{\ep,\theta}^{(k)}\}_{k\in \NN}$ (which may accumulate only at infinity 
assuming $\mathcal{H}$ is infinite-dimensional) labelled in ascending order and repeated according to multiplicity:
\[
1 \le \lambda_{\ep,\theta}^{(1)} \le \lambda_{\ep,\theta}^{(2)} \le \ldots \le \lambda_{\ep,\theta}^{(k)} \le \lambda_{\ep,\theta}^{(k+1)} \le \ldots
\]
In this section we provide approximations to the operators $\mathcal{L}_{\ep,\theta}$ with corresponding quantitative asymptotic approximations for the eigenvalues $\lambda^{(k)}_{\ep,\t}$, 
for small $\ep >0$, that are uniform in $\t\in \Theta$. 

Throughout this section we use the following notation. For  a linear subset $\mathcal{U}$  of $ \mathcal{H}$,  $\overline{\mathcal{U}}$ denotes the closure of $\mathcal{U}$ in $\mathcal{H}$  and $\mathcal{P}_{\overline{\mathcal{U}}}^\t$ is the orthogonal  projection onto $\overline{\mathcal{U}}$ with respect to  $d_\t$.  
Where appropriate, we use the notation $(\mathcal{V},d
)$ to denote the Hilbert space formed by equipping the vector space 
$\mathcal{V}$ with the inner product $d
$. We denote the spectrum of a linear operator $\mathbf{L}$ by $\text{Sp}\, \mathbf{L}$. 

\subsection{The case of continuous $V_\t$}\label{s.spcontV}
In this subsection we suppose that the assumptions of Theorem \ref{thm:contV} holds.  Consider original problem \eqref{p1}  with the functional $f$ given by 
\begin{equation}
\label{ik3}
\left\l f,\tilde u\right\r  \,\,:=\,\,d_\t\left(g,\tilde u\right), \quad  \forall \tilde u \in H,
\end{equation}
for any  $g\in \mathcal{H}$. 
Notice that by \eqref{ik2} $f\in H^*$, and the solution $u_{\ep,\t}$ to \eqref{p1} is in the domain $ \text{dom}\,\mathcal{L}_{\ep,\theta}\subset H$ of operator $\mathcal{L}_{\ep,\theta}$ 
and $\mathcal{L}_{\ep,\theta}u_{\ep,\t}=g$. 
Then Theorem \ref{thm:contV}, in particular \eqref{errorcontinuouscase2}, along with  \eqref{ik2}, \eqref{ik3} and \eqref{fstar}, for the solution $v_\t$ of the approximate problem  
\eqref{thmcontv.vprob},  imply 
\begin{equation}
\label{ik4}
d_\t[u_{\ep,\t}-v_\t] \,\,\le\,\, \left\| u_{\ep,\t}-v_\t\right\|^2_\t 
\,\,\le\,\, \ep^4\nu^{-2} \| f\|_{*\t}^2 \,\,\le\,\,  \ep^4\,\nu^{-2} d_\t[g], \quad \, \forall \t \in \Theta, 
\ \ \forall g\in\mathcal{H}.
\end{equation}
Let us give an operator-theoretic interpretation of \eqref{ik4}. 
Let $\mathbf{B}_\t$ be the self-adjoint operator in Hilbert space 
$(\overline{V_\t},d_\t)$,  
 generated by the closed positive sesquilinear form $b_\t$ with form domain $V_\theta$.
In particular,    
\begin{equation}
\label{ikddd} d_\t\left(\mathbf{B}_\t v,\tilde{v}\right)\,\,=\,\,b_\t\left( v,\tilde{v}\right),\quad  \forall v\in 
\text{dom}\, \mathbf{B}_\t \subset V_\t, \ \ \forall \tilde{v}\in V_\t, 
\end{equation} 
where $\text{dom}\, \mathbf{B}_\t
$ is the domain of $\mathbf{B}_\t$. 
Then 
the approximate problem \eqref{thmcontv.vprob} 
can be rewritten as 
$ \mathbf{B}_\t v_\t=\mathcal{P}_{\overline{V_\t}}^\t g$,  and so \eqref{ik4} can  
equivalently be re-stated as the following norm-operator estimate: 
\begin{equation}
\label{ik7}	\left\|\,\mathcal{L}_{\ep,\theta}^{-1}\,\,-\,\, \mathbf{B}_\t^{-1}  \mathcal{P}_{\overline{V_\t}}^\t\, \right\|_{(\mathcal{H},d_\t)\rightarrow (\mathcal{H},d_\t)}\,\,\,\le\,\, \,\ep^2\,\nu^{-1}.
\end{equation}
Next we observe that $\mathbf{B}_\t^{-1}  \mathcal{P}_{\overline{V_\t}}^\t $ is compact, non-negative and self-adjoint in $(\mathcal{H},d_\t)$.
Therefore, the 
spectrum of $\mathbf{B}_\t^{-1}  \mathcal{P}_{\overline{V_\t}}^\t $  consists of real non-negative eigenvalues with only a possible accumulation point at zero. Let us put these 
eigenvalues in descending  order:
\[
\alpha_{\theta}^{(1)} \geq\alpha_{\theta}^{(2)} \geq  \alpha_{\theta}^{(3)} \geq \ldots 
\]
Now, the key standard step is in noticing that the operator estimate \eqref{ik7} implies uniform 
estimates for the spectra via the 
 min-max principle (see e.g. \cite{ReeSim}). Namely, uniformly for all 
 $k\in \mathbb{N}$ and $\t \in \Theta$, 
\begin{equation}
\label{ik8}
\bigl| 1/\lambda_{\ep,\theta}^{(k)} \, - \, \alpha_{\theta}^{(k)} \bigr| \,\,\,\le\,\,\, \ep^2\nu^{-1}.
\end{equation}
Finally, we notice that all non-zero eigenvalues of $\mathbf{B}_\t^{-1}  \mathcal{P}_{\overline{V_\t}}^\t$ are the inverses of the eigenvalues of  $\mathbf{B}_\t$ and vice versa. 
Therefore we have the following relations between the eigenvalues $\{\mu_{\theta}^{(k)}\}$  of $\mathbf{B}_\t$ and  the eigenvalues 
$\lambda_{\ep,\theta}^{(k)}$ of $\mathcal{L}_{\ep,\theta}$:
\be\Big| 1/\lambda_{\ep,\theta}^{(k)} \,-\,  1/\mu_{\theta}^{(k)} \Big| 
\,\,\le\,\, \ep^2\nu^{-1}, \quad \ \forall k\in \NN, \ \ \forall \t \in \Theta, \qquad \text{if $ \ \dim  V_\theta\ =\infty$,}\label{ik889.0}
\end{equation}
 or
\begin{equation}\label{ik889}
\Big| 1/\lambda_{\ep,\theta}^{(k)} -  1/\mu_{\theta}^{(k)} \Big| \,\,\le\,\, \ep^2\nu^{-1},\quad 
\Big| 1/\lambda_{\ep,\theta}^{(p)} \Big| \,\,\le\,\, \ep^2\nu^{-1} ,\quad \forall k \le N ,\  \forall p\ge N+1, \ \forall \t \in \Theta, \ \text{if $\dim V_\theta\ =N$. }
\end{equation}
Inequalities \eqref{ik7}, \eqref{ik889.0} and \eqref{ik889} provide the desired 
estimates on the closeness 
of the ``resolvents'' and of the spectra of the exact and approximate operators, $\mathcal{L}_{\ep,\theta}$ and 
$\mathbf{B}_\t$ respectively, uniform in $\t$.

\subsection{The case of discontinuous $V_\theta$ with removable singularities}\label{spectralasymptoticsDiscV}
Here, we suppose that the assumptions of Theorem \ref{thm.maindiscthm} hold, establishing the closeness of the solution $u_{\ep,\t}$ to the original problem \eqref{p1} and of the solution 
$v^h+z^h\in \Vs_\t\,\dot{+}\, Z$ to the approximate problem \eqref{z3prob}. 
We shall follow the pattern of the previous subsection, aiming first at recasting \eqref{z3prob} 
in an operator form. 
To that end, for $\t \in \Theta,$ $\,| \t | < r_1$, $0<\ep<1$,  define on $\Vs_\t\,\dot{+}\, Z$ an inner product $s$ 
by
\begin{equation}
\label{ikspectr200} s(v+z,\tilde{v}+\tilde{z}) \,\,: =\,\,\ep^{-2} a_\t^{\rm h}(z,\tilde{z}) \,+\, b_\t(v+z,\tilde{v}+\tilde{z}), \quad   \forall v,\,\tilde{v}\in \Vs_\t,\, \  \, \forall z,\tilde{z}\in Z.
\end{equation}
As follows from the proof of Theorem \ref{thm.maindiscthm}, 
$V_\t^\star\,\dot{+}\,Z$ endowed with the inner product $s$ is a Hilbert space that is continuously embedded in $H$ and therefore compactly embedded in $\mathcal{H}$. 
Set $\mathcal{H}_\t : = \left(\overline{\Vs_\t\,\dot{+}\, Z}, \,d_\t\right)$  and  
let     $\mathbf{L}_{\ep,\theta}: {\rm dom}\, \mathbf{L}_{\ep,\theta} 
\rightarrow \mathcal{H}_\t$ be the self-adjoint operator in $\mathcal{H}_\t$ 
generated by the closed positive sesquilinear form $s$ with the form domain  $\Vs_\t\,\dot{+}\, Z$. The spectrum  of $\textbf{L}_{\ep,\theta}$ consists of positive isolated eigenvalues (which may only accumulate at infinity if $V_\t^\star\,\dot{+}\,Z$ is infinite-dimensional).  

Consider problem \eqref{p1} with  functional $f$ given by \eqref{ik3}. Then, for the solution to 
the approximate problem \eqref{z3prob}, 
 $v^h+z^h \,=\, \mathbf{L}_{\ep,\theta}^{-1}\,\mathcal{P}_{\mathcal{H}_\t}^\t g$. 
By Theorem \ref{thm.maindiscthm} (see \eqref{final2}) and \eqref{ik2} one has 
\[
d_\t\left[u_{\ep,\t}-(v^h+z^h)\right] \,\,\le \,\, C_6\, \ep^2 d_\t[g], 
\]
	which can be rewritten in the operator form as 
	\begin{equation}
	\label{ik37}	\left\|\,\mathcal{L}_{\ep,\theta}^{-1}\,-\, \mathbf{L}_{\ep,\theta}^{-1}\mathcal{P}_{\mathcal{H}_\t}^\t\,\right\|_{(\mathcal{H},d_\t)\rightarrow (\mathcal{H},d_\t)}\,\,\,\le\,\,\, C_6^{1/2}\ep,  
	\end{equation}
	where $\mathbf{L}_{\ep,\theta}^{-1}\,\mathcal{P}_{\mathcal{H}_\t}^\t$ is a self-adjoint operator in $(\mathcal{H},d_\t)$. 
Arguing then as in the previous subsection we have: 
\begin{theorem}\label{ikthm}
Assume \eqref{KA}--\eqref{H4}. Let $\{ \lambda^{(k)}_{\ep,\theta}\}_{k\in \NN}$ and $\{\Lambda^{(k)}_{\ep,\theta}\}_{k\in \NN}$ be the eigenvalues of the operators 
$\mathcal{L}_{\ep,\theta}$ and $\mathbf{L}_{\ep,\theta}$ respectively. 
Then, for $r_1$ given by \eqref{btnormonV+Z}, 
\begin{gather*}
\Big| 1/\lambda_{\ep,\theta}^{(k)} \,- \, 1/\Lambda_{\ep,\theta}^{(k)} \Big| \,\,\le\,\,  C_6^{1/2}\ep, \quad \forall k\in \NN,\  \,\forall\, \t \in \Theta, \ |\t| < r_1, \quad 
\text{ if $\,\,\dim  \left(\Vs_\t \,\dot{+}\,Z\right)\ =\infty$, } \\
\text{or} \hspace{\textwidth}\\
\Big| 1/\lambda_{\ep,\theta}^{(k)} -  1/\Lambda_{\ep,\theta}^{(k)} \Big| \,\le\, C_6^{1/2} \ep,\ \  
\Big| 1/\lambda_{\ep,\theta}^{(p)} \Big| \le C_6^{1/2} \ep, \ \ \forall k\le N, \, \forall   p\ge N+1,\, \forall \t \in \Theta,\, |\t| < r_1, \ \text{if $\dim \left( \Vs_\t\dot{+} Z\right) =N.$}
\end{gather*}
\end{theorem}
\begin{remark}\label{r.ikth77}
The approximations for the eigenvalues of $\mathcal{L}_{\ep,\t}$ given by Theorem \ref{ikthm} 
for $|\t|<r_2\le r_1$ can be combined with 
the results of Section \ref{s.spcontV} for $|\t|\ge r_2$. 
Indeed, under \eqref{contVs}--\eqref{distance},  the estimates \eqref{ik889.0} and \eqref{ik889} hold  for $|\t| \ge r_2 >0$ with $\nu$ replaced by $\gamma r_2^2$ (as seen directly from 
Theorem  \ref{thm1.all} with 
$\nu(r_2) = \gamma r_2^2$ as implied by \eqref{distance}).	
	
\end{remark}	
\subsection{The case of Lipschitz continuous $b_\t$}\label{s.spbt}
{
Let us now suppose  the assumptions of Theorem \ref{thm.IKunifest2} hold, establishing the closeness 
of the solution $u_{\ep,\t}$ of the original problem \eqref{p1} to the approximations based on the solution 
$v+z\in V_\star\,\dot{+}\,Z=V_0$ of the simplified problem \eqref{IKz3prob88}. 
With the aim of rewriting 
\eqref{IKz3prob88} and the resulting estimate 
\eqref{IKfinal3-2} in an operator form, notice first that the left-hand side of 
\eqref{IKz3prob88} 
 has the following important self-similarity property: 
it 
depends on $\ep$ and $\t$ only via $\t/\ep=:\xi$. 
  For each $\xi \in \mathbb{R}^n$, let  $\mathbb{L}_\xi$ be the  self-adjoint operator in 
	$\mathcal{H}_0 = \left(\overline{V_\star \,\dot{+\,}Z},\,d_0\right)$ generated by the following inner product on $V_\star \,\dot{+}\,Z$: 
\begin{equation}\label{Sform}
\mathbb{S}_\xi(v+z,\tilde{v}+\tilde{z}) \,\,: = \,\, a_\xi^{\rm h}(z,\tilde{z}) \,\,+\,\, b_0(v+z,\tilde{v}+\tilde{z}), \quad   \forall v +z, \ \tilde{v}+ \tilde{z}\in V_\star \,\dot{+}\, Z.
\end{equation}
Similarly to the previous subsections, for any $\xi\in\mathbb{R}^n$, $\mathbb{L}_\xi$ has a compact resolvent and hence a discrete positive spectrum which can only accumulate at infinity. 

If the right-hand-sides of \eqref{p1} is 
given by \eqref{ik3}, 
\eqref{IKfinal3-2} can be recast in operator form as follows. 
Notice that for the right-hand side of \eqref{IKz3prob88}, 
$\big\langle f,\, \mathcal{E}_\t \tilde v+\tilde z\big\rangle = d_\t\big(g,\, \mathcal{E}_\t \tilde v+\tilde z\big) = 
d_\t\big(g,\, G_\t\left(\tilde v+\tilde z\right)\big)$ where the bounded 
operator 
$G_\t:  \mathcal{H}_0 \to \big(\mathcal{H},\,d_\t\big)$
is given by 
$G_\t(v+z)=\mathcal{E}_\t v+z$. 
Then $\big\langle f,\, \mathcal{E}_\t \tilde v+\tilde z\big\rangle =
d_0\left(G_\t^*\,
g, \, \tilde v+\tilde z\right)$ where $G_\t^*: \big(\mathcal{H},\,d_\t\big)\to \mathcal{H}_0$ is the adjoint of $G_\t$ 
and so for the solution of \eqref{IKz3prob88}, 
$v+z=\mathbb{L}^{-1}_{\t/\ep}\,G_\t^*\,
g$. 
For the approximation in \eqref{IKfinal3-2}, $u^{\rm appr}_{\ep,\t}=\mathcal{E}_\t v+z=G_\t(v+z)$, which implies the following 
operator estimate: 
\be
\label{g-theta}
\left\|\, \mathcal{L}_{\ep,\t}^{-1} \,\, -\,\, 
G_\t\, \mathbb{L}_{\t / \ep}^{-1}
G_\t^*
\,\right\|_{(\mathcal{H},d_\t) \rightarrow (\mathcal{H},d_\t)} \,\,\le\,\, C_{10}^{1/2} \,\ep. 
\ee
The spectrum of $\mathcal{L}_{\ep,\t}^{-1}$ is then approximated by that of the self-adjoint operator
$\mathcal{R}_{\ep,\t}^{\rm appr}=G_\t\, \mathbb{L}_{\t / \ep}^{-1} G_\t^*$. 
The latter however is 
quite inexplicit due to the presence of (generally non-isometric) 
operator $G_\t$, 
and can be 
not close to the spectrum of $\mathbb{L}_{\t / \ep}^{-1}$. 
In principle, under an additional mild continuity assumption on $d_\t$ at $\t=0$ similar to \eqref{H5}, 
the approximation can be made slightly more explicit as of a ``matched asymptotics'' type 
using Remark \ref{r.ikth77} with optimally chosen small $r_2(\ep)$, 
although with a degraded rate in $\ep$ (c.f. \cite{ChCo}). 
We avoid here pursuing this further in generality, 
and introduce instead an 
additional assumption which, as we will see, 
often naturally
holds in examples (Section \ref{sec:examples}). 
This assumption makes the spectrum of approximation $\mathcal{R}_{\ep,\t}^{\rm appr}$ 
close to that of $\mathbb{L}_{\t / \ep}^{-1}\mathcal{P}_{\mathcal{H}_0}^0$ 
(which is much more explicit), 
and at the same time maintains the order of the accuracy. 
This is seen by correcting $\mathcal{R}_{\ep,\t}^{\rm appr}$, with a small error, 
and replacing it as a result with an operator approximation which is 
unitarily equivalent to $\mathbb{L}_{\t / \ep}^{-1}\mathcal{P}_{\mathcal{H}_0}^0$. 
Namely, let us suppose  that $\mathcal{E}_\t$ (that satisfies \eqref{Eprop1}--\eqref{Eprop2}) and $d_\t$ also 
satisfy\footnote{It can be seen that the 
existence of such extended $\mathcal{E}_\t$   
would necessitate 
existence $\forall \t \in \Theta$ of such bijections 
$\mathcal{E}_\t:V_\star\to V^\star_\t$ for which, apart from \eqref{Eprop1}--\eqref{Eprop2}, 
$d_\t\left(\mathcal{E}_\t v,\mathcal{E}_\t\tilde v\right) = d_0( v,  \tilde v)$, 
$\forall v,\tilde v \in V_\star$. 
The latter appears 
equivalent to the requirement that the spectrum of $\mathbf{B}_\t$ defined for 
$\t\in\Theta\backslash\{0\}$ by \eqref{ikddd} is independent of $\t$ and coincides with that of $\mathbf{B}_\star$ 
defined by \eqref{ikddd2} below. Together 
with appropriate continuity of the corresponding 
eigenvectors in $\t$, this suffices for \eqref{H6}.}   
\begin{equation}
\tag{H6}
\label{H6}
\begin{aligned}
&\text{ $\mathcal{E}_\t$ extends 
 to a bijection  in $ \mathcal{H}$ such that 
$d_\t\left(\mathcal{E}_\t u,\mathcal{E}_\t\tilde u\right) = d_0( u,  \tilde u), \quad 
\forall u,\tilde u \in \mathcal{H}, \, \,\t \in \Theta,$}\\
& \mathcal{E}_0=I\,\, 
\text{and $\big\|\, \mathcal{E}_{\t_1} \,-\, \mathcal{E}_{\t_2} \,\big\|_{(\mathcal{H},d_0) \rightarrow (\mathcal{H},d_{\t_1})} \,\,\le\,\, K_e\, |\t_1-\t_2|, \quad \forall \t_1,\t_2 \in \Theta, \ \ $ for some $K_e \ge 0$.}
\end{aligned}
\end{equation}

With the help of \eqref{H6}, the issue with the non-isometry of $G_\t$ in \eqref{g-theta} can be 
rectified 
by replacing $\tilde z$ on the right-hand side of \eqref{IKz3prob88} by $\mathcal{E}_\t\tilde z$ as well as $z$ in \eqref{IKfinal3-2} by 
$\mathcal{E}_\t z$. On the one hand, as we will see, this introduces a small additional error in \eqref{IKfinal3-2}, but on the other hand allows to express the amended approximation $\mathcal{E}_\t (v+ z)$ 
in a 
suitable 
operator form. Indeed, $v+z$ now solves the amended \eqref{IKz3prob88} which via \eqref{Sform} and \eqref{H6} reads 
$\mathbb{S}_{\t/\ep}(v+z,\tilde v+\tilde z)=
d_\t\big(g,\,\mathcal{E}_\t(\tilde v+\tilde z)\big)= 
d_0\big(\mathcal{E}_\t^{-1}g,\,\tilde v+\tilde z\big)$. 
Hence $\mathcal{E}_\t (v+ z)=
\mathcal{E}_\t \mathbb{L}_{\t / \ep}^{-1}\mathcal{P}_{\mathcal{H}_0}^0\mathcal{E}_\t^{-1}g$. 
Notice that \eqref{H6} implies that $\mathcal{E}^{-1}_\t$ is the adjoint of $\mathcal{E}_\t$ 
which is a unitary 
map from   $(\mathcal{H}, d_0)$ to $(\mathcal{H},d_\t)$, so the emerging 
approximating operator $\mathcal{E}_\t \mathbb{L}_{\t / \ep}^{-1}\mathcal{P}_{\mathcal{H}_0}^0\mathcal{E}_\t^{-1}$ is 
self-adjoint and non-negative in $(\mathcal{H}, d_\t)$. 
As a result, the following theorem holds.
\begin{theorem}\label{p.unitaryequiv} Assume \eqref{KA}--\eqref{H6}. For all $\t \in \Theta$ and $0<\ep <1$ one has:
	\begin{equation*}\label{splimSe.3}
\left\|\, \mathcal{L}_{\ep,\t}^{-1} \,\, -\,\, 
\mathcal{E}_\t\, \mathbb{L}_{\t / \ep}^{-1}\mathcal{P}_{\mathcal{H}_0}^0\mathcal{E}_\t^{-1}\,\right\|_{(\mathcal{H},d_\t) \rightarrow (\mathcal{H},d_\t)} \,\,\le\,\, C_{11} \,\ep,
 \quad 
C_{11} \,=\, C_{10}^{1/2} \,+\, \frac{1}{2} K_e \nu_\star^{-1/2} \left(\, 1 \,+\, 
\left(\tfrac{ 1+K^2}{2(1-K_Z)}\right)^{1/2} \,\right).
\end{equation*}
\end{theorem}
\begin{proof}
Estimate \eqref{IKfinal3-2} of	Theorem \ref{thm.IKunifest2} informs us via \eqref{ik2} that 
$d_\t\big[ \mathcal{L}_{\ep,\t}^{-1}g - (\mathcal{E}_\t v + z)\big]\le C_{10} \ep^2 d_\t[ g]$, where $v+z$ is the solution to \eqref{IKz3prob88} with functional \eqref{ik3}.  So it remains to bound $  \big(\mathcal{E}_\t v + z\big) - \mathcal{E}_\t \mathbb{L}_{\t / \ep}^{-1}\mathcal{P}_{\mathcal{H}_0}^0\mathcal{E}_\t^{-1}g$. 
	Let  $ v_1+z_1 =\mathbb{L}_{\t / \ep}^{-1}\mathcal{P}_{\mathcal{H}_0}^0 \mathcal{E}^{-1}_\t g$, that is (see  \eqref{Sform} and \eqref{H6}) $v_1+z_1 \in V_\star \dot{+} Z$ solves 
	\begin{equation}\label{v1z1prob}
	\ep^{-2} a_\t^{\rm h}(z_1,\tilde{z}) \,+\, b_0(v_1+z_1,\tilde{v}+\tilde{z}) \,=\, 
	d_\t\big(g,\, \mathcal{E}_\t(\tilde{v} + \tilde{z})\,\big),  
	\quad \forall \, \  \tilde{v} + \tilde{z} \in V_\star \dot{+} Z.
	\end{equation}
	Then 
	\begin{equation}\label{15.09.20e1}
	\big(\mathcal {E}_\t v +z\big) \,-\, \mathcal{E}_\t \mathbb{L}_{\t / \ep}^{-1}\mathcal{P}_{\mathcal{H}_0}^0\mathcal{E}_\t^{-1}g\,=\, \mathcal {E}_\t v +z - \mathcal{E}_\t(v_1+z_1) =   (I- \mathcal{E}_\t ) z +  \mathcal{E}_\t  \big(  v+z -(v_1 +  z_1) \big).
	\end{equation}
	Let us bound $d_\t\big[(I- \mathcal{E}_\t ) z\big] $.  By \eqref{VZorth} one has 
	$2\left(1-K_Z\right)\big(b_0[v]+ b_0[z]\big) \le b_0[v+z]$ which combined with 
	\eqref{IKz3prob88}, \eqref{ahcoercive},	\eqref{ik2}, 
 \eqref{as.b1}   and 
	\eqref{H6} 
	gives 
	\begin{flalign*}
	&\nu_\star\ep^{-2}|\t|^2d_0[z] +2\left(1-K_Z\right) \left(d_0[v]+ K^{-2} d_\t[z]\right) \,\le\, \ep^{-2} a^{\rm h}_\t[z] + b_0[v+z] \,=\, d_\t(g,\mathcal{E}_\t v + z) \\
	&\le\,\, d_\t^{1/2}[g] \left(d_\t^{1/2}[\mathcal{E}_\t v] + d_\t^{1/2}[z]\right)\,=\, 
	d_\t^{1/2}[g] \left(d_0^{1/2}[v] + d_\t^{1/2}[z]\right)\\
	& \le \, \tfrac{1}{8}\left(1-K_Z\right)^{-1} \left(1+  K^{2} \right)d_\t[g]  \,+ \,
	2\left(1-K_Z\right) \Big(d_0[v] \,+\, K^{-2}d_\t[z] \Big).
	\end{flalign*}
	Thus $\nu_\star \ep^{-2} |\t|^2 d_0[z]  \le \tfrac{1}{8}\left(1-K_Z\right)^{-1}(1+K^2) 
	d_\t[g] $  and therefore, via \eqref{H6}, one has  
	\[
	d_\t[(I-\mathcal{E}_\t) z] \,\le\,  K_e^2 |\t|^2 d_0[z] \, \le\,\tfrac{1}{8} K_e^2 \nu_\star^{-1}(1-K_Z)^{-1} (1+K^2) \ep^2 d_\t[g].
	\] 
	It remains to  bound the last term in \eqref{15.09.20e1}. 
	By \eqref{H6} it is equivalent to bounding  $d_0\left[v+z -(v_1 +  z_1)\right]$. 
	Subtracting \eqref{v1z1prob}  from \eqref{IKz3prob88} (with 
	$\tilde{v} = v-v_1$ and $\tilde{z} = z-z_1$) and using \eqref{H6}, \eqref{ik2} and \eqref{ahcoercive} gives 
	\[
	\ep^{-2} a^{\rm h}_\t \left[z-z_1\right] + b_0\left[v +z -( v_1 + z_1)\right] \,=\,
	d_\t\bigl(g, (I-\mathcal{E}_\t) (z-z_1) \bigr) \,\,\le 
	\]
	\[
	\ \ \ \ \ \ \quad \ \  d^{1/2}_\t[g]\,  K_e|\t|\,d^{1/2}_0\left[z-z_1\right] \,\,\,\le \,\,\, 
	K_e \,\nu_\star^{-1/2}\,  d^{1/2}_\t[g]
	\left(a^{\rm h}_\t[z-z_1]\right)^{1/2}.
	\]
	Therefore  $
	b_0\left[ v + z -(v_1 +z_1) \right] \le\tfrac{1}{4} \nu_\star^{-1}\, K_e^2\, \ep^2 \,d_\t[g]$ and so 
 $d_0\left[ v + z -(v_1 +z_1) \right] \le\tfrac{1}{4} \nu_\star^{-1}\, K_e^2 \,\ep^2 d_\t[g].$ Applying finally the triangle inequality completes the proof.  
\end{proof}
Theorem \ref{p.unitaryequiv}, together with the fact that 
$\mathcal{E}_\t: (\mathcal{H}, d_0)\rightarrow (\mathcal{H},d_\t)$ 
is unitary by \eqref{H6}, 
provides the following analogue of Theorem \ref{ikthm}:
\begin{theorem}\label{ikthm2}
	Assume \eqref{KA}--\eqref{H6}. Let $\{ \lambda^{(k)}_{\ep,\theta}\}_{k\in \NN}$ and $\{\lambda^{(k)}_{\xi}\}_{k\in \NN}$ be the eigenvalues of the operators 
	$\mathcal{L}_{\ep,\theta}$ and $\mathbb{L}_{\xi}$ respectively. Then, 
	for some $C_{11}>0$ independent of $\ep$,  $\t$ and $k$, 
	\begin{equation}\label{ik88}
	\left| 1/\lambda_{\ep,\theta}^{(k)} \,- \, 1/\lambda_{\theta/\ep}^{(k)} \right| \,\,\le\,\,  C_{11}\,\ep, \quad \forall k\in \NN,\ \ \forall \t \in \Theta, \quad 
	\text{ if $\ \ \ \dim  \left(V_\star \,\dot{+}\,Z\right)\ =\infty$, }
	\end{equation}
	or
	\begin{equation}\label{ik89}
	\left| 1/\lambda_{\ep,\theta}^{(k)} -  1/\lambda_{\theta/\ep}^{(k)} \right| \,\le\, C_{11}\ep,\ \  \left| 1/\lambda_{\ep,\theta}^{(p)} \right| \le  C_{11}\ep, \quad \forall k\le N , \, \forall   p\ge N+1,\, \forall \t \in \Theta, \ \text{	if $\ \dim  \left(V_\star\dot{+} Z\right)\ =N.$}
	\end{equation}
 \end{theorem}}
We next aim at approximating the collective spectrum 
$\bigcup_{\t \in \Theta} {\rm Sp}\, \mathcal{L}_{\ep,\t}=:{\rm Sp}_\ep$. 
The importance of ${\rm Sp}_\ep$ is due to the fact that in many examples 
(Section \ref{sec:examples}) operators $\mathcal{L}_{\ep,\t}$, $\t\in\Theta$, serve as fibers in a decomposition 
of a 
(transformed) original operator whose spectrum 
is (the closure of)  
${\rm Sp}_\ep$. 
Theorem \ref{ikthm2} provides 
an approximation to ${\rm Sp}_\ep$, however,  this approximation still depends on  
$\ep$. We will now rectify this 
to produce an important {\it $\ep$-independent approximation} of the above collective spectrum. 
\begin{theorem}\label{t.collectivespec}
Assume \eqref{KA}--\eqref{H6}. Then 
\begin{equation}\label{symdist.e1}
{\rm dist}_H \Big(  \bigcup_{\t \in \Theta} {\rm Sp}\, \mathcal{L}_{\ep,\t}^{-1} \cup \{ 0\} , \bigcup_{\xi \in \mathbb{R}^n} {\rm Sp}\, \mathbb{L}^{-1}_\xi \cup \{ 0 \}  \Big)  \,\,\le\,\, C_{12} \,\ep,
\end{equation}
where for non-empty $X,Y\subset\mathbb{R}$, 
${\rm dist}_H(X,Y):=\max\bigl(\sup_{x\in X}{\rm dist}(x,Y),\, \sup_{y\in Y}{\rm dist}(y,X)\bigr)$ is the symmetric Hausdorff distance and $C_{12}$ is a positive constant independent of $\ep$. 
\end{theorem}
Before proving the theorem, we state its corollary providing 
an $\ep$-independent approximation of the set $\bigcup_{\t \in \Theta} {\rm Sp}\, \mathcal{L}_{\ep,\t}$ in any finite interval.
\begin{corollary}\label{c.collspec}
For every interval $ [a,b] \subset (-\infty,\infty)$   one has 
\begin{equation}
\label{cor67}
{\rm dist}_{[a,b]}\Big( \, \overline{ \bigcup_{\theta \in \Theta} {\rm Sp}\, \mathcal{L}_{\ep,\t} }, \ \  \overline{\bigcup_{\xi \in \mathbb{R}^n} {\rm Sp}\, \mathbb{L}_\xi } \, \Big) \,\,\le\,\, C_{b}\, \ep,
\quad \forall\, \  0<\ep<1,\\
\end{equation}
with a constant $C_b$ independent of $\ep$ and $a$ and 
$C_b\le C\left(1+b^2\right)$ with a $b$-independent $C$. 
In \eqref{cor67},  
${\rm dist}_{[a,b]}(X,Y) := \max\big({\rm dist}([a,b]\cap X,Y), {\rm dist}([a,b]\cap Y,X)\big)$ 
where ${\rm dist}(X,Y):=\sup_{x\in X}{\rm dist}(x,Y)$ is the (non-symmetric) distance, 
and we adopt the convention that ${\rm dist}(\emptyset,A)={\rm dist}(A,\emptyset)
= 0$ for any set $A$. 
In particular, this 
can be interpreted as that 
$\overline{ \bigcup_{\theta \in \Theta} {\rm Sp}\, \mathcal{L}_{\ep,\t} }$ converges when $\ep\to 0$ to $\overline{\bigcup_{\xi \in \mathbb{R}^n} {\rm Sp}\, \mathbb{L}_\xi }$ in the Fell topology (cf. e.g. \cite[p. 142 Corollary 5.1.7]{Be}), with a ``rate'' specified by 
\eqref{cor67}.
\end{corollary}
\begin{proof}
Since, for any $\t\in\Theta$ and $\xi\in \mathbb{R}^n$, 
${\rm Sp}\, \mathcal{L}_{\ep,\t}\cup {\rm Sp}\, \mathbb{L}_\xi \subset [1,\infty)$, 
we can set $C_b=0$ for $b< 1$. 
Let $b\ge 1$, and let for some $\t\in\Theta$, $0<\ep<1$ and $k\in\mathbb{N}$, 
$1\le \lambda:=\lambda^{(k)}_{\ep,\t}\in [a,b]\cap {\rm Sp}\, \mathcal{L}_{\ep,\t}$. 
(The case of 
$\lambda:=\lambda^{(k)}_\xi\in [a,b]\cap {\rm Sp}\, \mathbb{L}_\xi$, $\xi\in\mathbb{R}^n$, is considered in a similar way.) 
Then, by \eqref{symdist.e1}, either $1/\lambda\le C_{12}\,\ep$ or for some $\xi\in\mathbb{R}^n$ and 
$l\in\mathbb{N}$, $\big|1/\lambda-1/\mu\big|\le 2\,C_{12}\,\ep$ where $\mu:=\lambda^{(l)}_{\xi}\ge 1$. 
Assuming first the latter, 
$ 
|\lambda-\mu|\,=\,\left|\lambda^{-1}-\mu^{-1}\right|\lambda\mu\,\le\,
2\,C_{12}\,\ep\, b \,\big(|\lambda-\mu|\,+\,b\big)$. 
Therefore, if $\ep<\min\left\{1,\,(4 C_{12}b)^{-1}\right\}=:\ep_b$ then 
$2\,C_{12}\,\ep b< 1/2$ and it follows that $|\lambda-\mu|\le 4 C_{12}b^2\ep$. 
Notice that, as $\lambda\le b$, the former case ($1/\lambda\le C_{12}\ep$) is not possible for $\ep<\ep_b$. 
On the other hand, if $\ep\ge\ep_b$, we notice that as $V_0\neq\{0\}$, $\lambda^{(1)}_0<+\infty$. 
(Similarly, from the variational principle, $\lambda^{(1)}_{\ep,0}\le\lambda^{(1)}_0<+\infty$.) 
So, taking instead $\mu=\lambda^{(1)}_0$ (similarly, $\mu=\lambda^{(1)}_{\ep,0}$), 
$
|\lambda-\mu|\,\le\,\max\left\{b,\lambda^{(1)}_0\right\}\,\le\max\left\{b,\,\lambda^{(1)}_0\right\}\ep_b^{-1}\ep.
$ 
So it would suffice, for $b\ge 1$,  to take 
$C_b=\max\left\{b,\,\lambda^{(1)}_0\right\}\ep_b^{-1}=
\max\left\{b,\lambda_0^{(1)}\right\}\max\left\{4C_{12}b, 1\right\}\ge 4 C_{12}b^2$. 
\end{proof}
\begin{proof}[Proof of Theorem \ref{t.collectivespec}] 
Note that \eqref{symdist.e1} follows if we establish that: 
\begin{gather}
\label{syd.e2} \text{for every } l_\ep \in \bigcup_{\t \in \Theta} {\rm Sp}\, \mathcal{L}_{\ep,\t}^{-1} \cup  \{ 0 \}  \text{ there exists } l_0 \in \bigcup_{\xi \in \mathbb{R}^n} {\rm Sp}\, \mathbb{L}^{-1}_\xi \cup  \{ 0 \}  \ \text{such that}\ | l_\ep - l_0 | \le C_{12} \ep; \\
\label{syd.e3} \text{for every }  l_0 \in \bigcup_{\xi \in \mathbb{R}^n} {\rm Sp}\, \mathbb{L}^{-1}_\xi \cup  \{ 0 \}  \text{ there exists }  l_\ep \in \bigcup_{\t \in \Theta} {\rm Sp}\, \mathcal{L}_{\ep,\t}^{-1} \cup  \{ 0 \}  \ \text{such that}\ | l_0 - l_\ep | \le C_{12} \ep.
\end{gather}
Theorem \ref{ikthm2} implies \eqref{syd.e2}, and also 
implies 
\eqref{syd.e3} for all 
$l_0 \in  \bigcup_{\xi \in \ep^{-1} \Theta} {\rm Sp}\, \mathbb{L}^{-1}_\xi  \cup \{ 0 \}$. 
Therefore, in order to establish \eqref{syd.e3} it remains to consider arbitrary  $l_0 \in {\rm Sp}\, \mathbb{L}^{-1}_\xi$,  $\xi \in \mathbb{R}^n \backslash \ep^{-1} \Theta$. 
Recalling that $\theta=0$ is assumed to be an interior point of $\Theta$, let 
 $0<\ep<1$ and let $\xi\in \mathbb{R}^n\backslash \ep^{-1}\Theta$ i.e. $|\xi|>\ep^{-1}R$ where $R>0$ is radius of the largest closed ball centred at the origin that is contained in $\Theta$. 

So, for small $\ep$, we are interested in approximating the spectrum 
${\rm Sp}\, \mathbb{L}^{-1}_\xi$ for large $|\xi|$. 
The idea is to regard $\tilde\ep=|\xi|^{-1}$ as a new small parameter in  $\mathbb{L}^{-1}_\xi$, 
and to use previous methods and 
results (namely those of Theorem \ref{thm:contV} and Section \ref{s.spcontV}). 
As a result, we will see that the spectrum of $\mathbb{L}^{-1}_\xi$ is approximated by that of an 
operator whose spectrum is in turn 
approximating that of 
$\mathcal{L}_{\ep,\t}^{-1}$. 


To that end, for any 
$\xi\in\mathbb{R}^n\backslash\{0\}$,  
introduce ``wiggled'' objects as follows. Let $\tilde\ep:=|\xi|^{-1}$ and $\tilde\theta:= \xi/|\xi|$, and so 
$\xi=\tilde\t/\tilde\ep$ with $\tilde\ep>0$ and 
$\tilde\t\in\tilde\Theta= {S}^{n-1}$ 
the unit sphere 
in $\mathbb{R}^n$ centered at the origin. 
We also set 
\begin{equation}
\label{wigprobl}
\begin{aligned}
\widetilde{H} = V_\star \dot{+} Z, \quad 
\quad \widetilde{A}_{\tilde{\ep},\tilde{\theta}}[v+z]\,: = \,
\tilde{\ep}^{\,-2} \,a^{\rm h}_{\tilde{\theta}}[z] + b_0[v+z]\,=\,
\mathbb{S}_{\tilde\t/\tilde\ep}[v+z]\,=\,\mathbb{S}_\xi[v+z], \quad \tilde{\theta} \in \widetilde{\Theta},  
\end{aligned}
\end{equation}
i.e., for any $\tilde\t\in\widetilde{\Theta}$, we set 
$\tilde a_{\tilde\t}(v+z,\tilde v+ \tilde z) =a^{\rm h}_{\tilde{\theta}}(z, \tilde z)$,  
$\,\tilde b_{\tilde\t}(v+z,\tilde v+ \tilde z)=b_0(v+z,\tilde v+ \tilde z)$, with respective inner product  
$(v+z,\tilde v+ \tilde z)_{\tilde\t}= a^{\rm h}_{\tilde{\theta}}(z, \tilde z)+b_0(v+z,\tilde v+ \tilde z)$. 
Then $\widetilde V_{\tilde\t}=V_\star$ and 
$\widetilde W_{\tilde\t}=\bigl\{v+z\in V_\star \dot{+} Z \,\,\big\vert \,\, b_0(v+z, \tilde v)=0, \ \forall \tilde v\in V_\star\bigr\}$. 
Now we notice that for the above wiggled objects all 
the assumptions of Theorem \ref{thm:contV} 
hold. In particular, 
\eqref{bddspec} can be seen to hold as follows. 
For $\tilde\t\in\widetilde{\Theta}$ and $w=v+z\in \widetilde W_{\tilde\t}$, via \eqref{ahcoercive}, 
\[
a^{\rm h}_{\tilde{\theta}}[z]\,\,\ge\,\, \nu_\star \,\|z\|_0^2
\,\,=\,\,\nu_\star\,b_0[z]\,\,=\,\,\nu_\star\,b_0\big[(v+z)-v\big]
\,\,\ge\,\,
\nu_\star\,b_0[v+z].
\]
(In the last inequality we used that $w=v+z$ 
and $v$ are orthogonal with respect to $b_0$.) 
As a result, for any $0<\tilde\nu<1$, 
$
\tilde a_{\tilde\t}[w]\,\,=\,\,a^{\rm h}_{\tilde{\theta}}[z]\,\,\ge\,\,
\tilde\nu\,a^{\rm h}_{\tilde{\theta}}[z]\,+\,
(1-\tilde\nu)\,\nu_\star \,b_0[v+z].  
$ 
Hence choosing 
$\tilde\nu:=\nu_\star/(1+\nu_\star)$, implies 
$
\tilde a_{\tilde\t}[w] 
\,\,\ge\,\,\tilde\nu\,\Big( a^{\rm h}_{\tilde{\t}}[z]+b_0[v+z]\,\Big)\,\,=\,\,
\tilde\nu\, \|w\|_{\tilde\t}^2, 
$ 
and therefore \eqref{bddspec} holds with $\nu=\tilde\nu$. 

Next, for the assumptions of Section \ref{s.spcontV}, we set 
$\widetilde{\mathcal{H}}:=\overline{V_\star \dot{+} Z}$ with $d_{\tilde\t}=d_0$, $\forall\tilde\t\in\widetilde\Theta$. 
So  $\widetilde{\mathcal{L}}_{\tilde{\ep},\tilde{\theta}}$ is the self-adjoint operator in  
$\left(\overline{V_\star \dot{+} Z}, d_0\right)$ 
generated by  $\tilde{A}_{\tilde{\ep},\tilde{\theta}}$ with form domain  $V_\star\dot{+} Z$.
Further, 
$\mathbf{\widetilde B}_{\tilde\t}$ becomes in this setting  the $\tilde\t$-independent 
self-adjoint operator $ \mathbf{B}_\star$ in $\left(\overline{V_\star},d_0\right)$ generated by $b_0$ with form domain $V_\star$, i.e. (cf. \eqref{ikddd}), 
\begin{equation}
\label{ikddd2} 
d_0({\mathbf{B}_\star} v,\tilde{v})\,\,=\,\,b_0( v,\tilde{v}),\quad  \forall v\in \text{dom}\, {\mathbf{B}_\star}\subset V_\star , \ \ \forall \tilde{v}\in V_\star. 
\end{equation} 
Recall that Theorem \ref{thm:contV} holds for all $\ep>0$ (see Remark \ref{rem3.2}), 
and hence so are the results of Section \ref{s.spcontV}. 
Consequently,  via \eqref{ik889.0}--\eqref{ik889}, we have 
\begin{equation}\label{iliabc}
{\rm dist}_H \left( {\rm Sp}\, \widetilde{\mathcal{L}}^{-1}_{\tilde{\ep},\tilde{\theta}}\cup \{ 0 \} \,,\, {\rm Sp}\,  \mathbf{B}_\star^{-1}\cup \{ 0 \}\right)\,\, \le \,\,\tilde{\ep}^{2}\, \tilde\nu^{-1}, \quad  \forall\,\,\tilde{\ep}>0, \, \, \forall \,\tilde{\theta} \in {S}^{n-1}.
\end{equation}  
 Notice that, see \eqref{wigprobl},  
$\mathbb{L}^{-1}_{ \xi} = \widetilde{\mathcal{L}}^{-1}_{\tilde\ep,\tilde\t}$, and therefore \eqref{iliabc} implies 
\begin{equation}\label{bdrylimspec}
{\rm dist}_H\left( {\rm Sp}\, \mathbb{L}^{-1}_{ \xi} \cup \{ 0 \} \,,\, {\rm Sp}\,  
\mathbf{B}_\star^{-1}\cup \{ 0 \}\right) \,\,\le\,\,  |\xi|^{-2} \tilde\nu^{-1}, \quad   \forall \xi \in \mathbb{R}^n \backslash  \{ 0 \}. 
\end{equation}
In particular, for every  $l_0 \in {\rm Sp}\, \mathbb{L}^{-1}_\xi$, 
$\xi \in \mathbb{R}^n \backslash \ep^{-1} \Theta$ (hence $|\xi|>\ep^{-1}R$), 
there exists $\mu \in {\rm Sp}\,  \mathbf{B}_\star^{-1}\cup \{ 0 \}$ such that 
\begin{equation}
\label{syd.e4}
\left|\, l_0 \,-\, \mu\, \right| \,\, \le \,\,|\xi|^{-2}\, \tilde\nu^{-1}\,\, < \,\,
 \ep^2\,R^{-2}\, \tilde\nu^{-1}.  
\end{equation}

	On the other hand, from \eqref{ikddd} and \eqref{ikddd2} 
	via  
	\eqref{Eprop1} and \eqref{H6}, 
	for every $\t \in \Theta \backslash \{0\}$,  
	$\mathbf{B}_\t^{-1} = \mathcal{E}_\t  \mathbf{B}_\star^{-1} \mathcal{E}_\t^{-1}$ and thus 
	${\rm Sp}\, \mathbf{B}_\t^{-1} ={\rm Sp}\,  \mathbf{B}_\star^{-1}$.  Consequently, Remark \ref{r.ikth77} implies that
for the above $\mu \in {\rm Sp}\,  \mathbf{B}_\star^{-1}\cup \{ 0 \}={\rm Sp}\,  \mathbf{B}_\t^{-1}\cup \{ 0 \}$ with any chosen $\t\in\Theta$ with $|\t| = R$, 
there exists  $l_\ep \in {\rm Sp}\, \mathcal{L}^{-1}_{\ep,\t} \cup \{0\}$,   such that  
\begin{equation}
\label{syd.e5}
\left| \,\mu \,-\, l_\ep \,\right|  \,\,\le\,\,  \ep^2 \gamma^{-1} R^{-2}.
\end{equation}
Thus, \eqref{syd.e4} and \eqref{syd.e5} imply that \eqref{syd.e3} holds for $l_0 \in \bigcup_{\xi \notin \ep^{-1} \Theta} {\rm Sp}\, \mathbb{L}^{-1}_\xi $ and the proof is complete. 
 \end{proof}

\subsection{Explicit characterisation of the limit collective spectrum 
}
\label{s.charlimsp}
As Corollary \ref{c.collspec} provides an approximation of the collective spectrum 
 for the original problem $\overline{\bigcup_{\t \in \Theta} {\rm Sp}\, \mathcal{L}_{\ep,\t}}$ in terms 
of the ``limit collective spectrum'' 
${\rm Sp}_0:=\overline{{\bigcup_{\xi \in \mathbb{R}^n} {\rm Sp}\, \mathbb{L}_\xi}  }$, 
we aim here at characterising the latter limit spectrum. 
Recall that if $\lambda$ is an eigenvalue of $\mathbb L_{\xi}$ with 
eigenvector $0\neq v+z\in V_\star \dot{+}Z$, then by \eqref{Sform} 
 	\begin{equation}
	\label{lxispec}
	\mathbb{S}_\xi(v+z,\tilde v+\tilde{z}) \,  =\, 
 	a_\xi^{\rm h}(z,\tilde{z}) \,+\, b_0(v+z,\tilde v+\tilde{z}) \,  =\, \lambda\,\, d_0(v+z,\tilde{v} +\tilde{z}), \quad \forall\, \tilde{v} + \tilde{z} \in V_\star \dot{+} Z.  
 	\end{equation} 
	We start with a simple explicit characterisation of ${\rm Sp}_0$ in terms of the spectra of 
	$\mathbf{B}_0$ 
	and $\mathbf{B}_\star$, defined respectively by \eqref{ikddd} for $\t=0$ and \eqref{ikddd2}. 
	Notice that $\mathbf{B}_0$ coincides with $\mathbb{L}_0$ i.e. with $\mathbb{L}_\xi$ for $\xi=0$.  
	\begin{theorem}
	\label{limspecsimple} 
	Let $\lambda_0^{(k)}$ and $\lambda_\star^{(k)}$, $k\ge 1$, be respectively the eigenvalues of $\mathbf{B}_0$ 
	(equivalently of $\mathbb{L}_0$) and $\mathbf{B}_\star$, listed in ascending order accounting for 
	multiplicities. 
	Then $\lambda_0^{(k)}\le\lambda_\star^{(k)}$, 
	$\forall k$ if ${\rm dim} V_\star=\infty$ and 	$\forall k\le N$ if $0\le N:={\rm dim} \,V_\star<\infty$, and 
	(assuming $Z$ nontrivial) 
	\begin{eqnarray}
	\label{valthm}
	{\rm Sp}_0\,\,:=\,\,\overline{{\bigcup_{\xi \in \mathbb{R}^n} {\rm Sp}\, \mathbb{L}_\xi}  }\,\,=\,\,
	\bigcup_{k=1}^\infty \big[ \,\lambda_0^{(k)},\,\lambda_\star^{(k)}\,\big], \ \ \ \ \mbox{if } \,{\rm dim} V_\star=\infty, 
	\ \ \ \ \mbox{or} \\ 
	\label{valthm2}
		{\rm Sp}_0\,\,=\,\,		
		\bigcup_{k=1}^N \big[ \,\lambda_0^{(k)},\,\lambda_\star^{(k)}\,\big]\,
		\bigcup \big[ \,\lambda_0^{(N+1)},\,+\infty\,\big) \ \ \ \ \mbox{if } \,0\le N:={\rm dim} \,V_\star<\infty. 
	\end{eqnarray}
	\end{theorem}
	\begin{proof}
	For any $\xi\in\mathbb{R}^n$, for the eigenvalue $\lambda_\xi^{(k)}$ of $\mathbb{L}_\xi$, 
	$\lambda_0^{(k)}\le\lambda_\xi^{(k)}$ from the non-negativity of $a^{\rm h}_\xi$ in \eqref{lxispec}. 
	On the other hand, if the domain $V_\star\dot{+}Z$ of form $\mathbb{S}_\xi$ is restricted to $V_\star$ 
	then the form coincides with that determining $\mathbf{B}_\star$. Hence, by the min-max arguments, 
		$\lambda_\xi^{(k)}\le\lambda_\star^{(k)}$, $\forall\xi$ (for $k\le N$ if $N:={\rm dim}\, V_\star<\infty$). 
Notice 
that $\lambda_\xi^{(k)}$ 
continuously depend on $\xi$. 
(This directly follows 
from \eqref{lxispec}, continuous dependence of $a^h_\xi$ on $\xi$ and the min-max arguments.)  
So it remains to show that for large $\xi$, $\lambda_\xi^{(k)}$ is close to $\lambda_\star^{(k)}$ and is 
unbounded for $k=N+1$ when $N={\rm dim} \,V_\star<\infty$. 
This immediately follows from the argument leading to \eqref{bdrylimspec}. 
Namely, \eqref{ik889.0}--\eqref{ik889} implies in that context that 
$\left\vert 1/\lambda_{\xi}^{(k)} -  1/\lambda_\star^{(k)} \right\vert \,\le\, |\xi|^{-2}\tilde\nu^{-1}$ and 
$\left\vert 1/\lambda_{\xi}^{(N+1)} \right\vert \,\le\, |\xi|^{-2}\tilde\nu^{-1}$. 	
	\end{proof}
	Theorem \ref{limspecsimple} implies in particular that the limit collective spectrum has gaps whenever 
	$\lambda_\star^{(k)}<\lambda_0^{(k+1)}$. 
	We next aim at deriving some 
	more explicit 
	expressions for the eigenvalues and eigenvectors of $\mathbb{L}_\xi$. 
	This is useful for both approximating the ``dispersion relations'' $\lambda^{(k)}_{\ep,\t}$ via 
	Theorem \ref{ikthm2}, the associated eigenvectors, and for 
	characterising any gaps in the spectrum more explicitly. 
	For that, 
	it would help if it were possible to 
	eliminate $v$ from \eqref{lxispec} by expressing it in terms of a self-adjoint operator acting on $z$. 
	This is prevented by the coupling term $b_0(z,\tilde v)$, and 
	we aim at overcoming this by first showing that the defect subspace $Z$ (satisfying   \eqref{spaceZ} and \eqref{VZorth}) can always be chosen so that
 	\begin{equation}\label{zvbd}
 	b_0(z, v_\star)\, =\, d_0(z,v_\star), \quad \forall  z \in Z, \ \ 
	\forall v_\star \in V_\star.
 	\end{equation}	
	In most of the specific examples, Section \ref{sec:examples}, \eqref{zvbd} naturally holds. 
	In our 
	abstract setting, 
	\eqref{zvbd} can always be achieved by selecting $Z$ appropriately. 
	To see this, 
	assume without loss of generality that 
 	\begin{equation}
	\label{kdass}
 	d_0[v_0] \, \le\, M_d\, b_0[v_0], \quad \forall v_0 \in V_0=V_\star\dot{+} Z, \quad 0<M_d <1.
 	\end{equation} 
	(This follows from 
	\eqref{ik2} for $\t=0$ 
	by re-defining\footnote{Equivalently, relying instead of \eqref{kdass} on 	\eqref{ik2}, we can modify the sought 
	identity \eqref{zvbd} to 
	$b_0(z, v_\star)\, =\,M_d d_0(z,v_\star)$, $\forall  z \in Z$, 
	$\forall v_\star \in V_\star$, for some $0<M_d<1$; with obvious minor modifications in all the subsequent arguments and formulas.} 
$d_\t$ via multiplying it by any positive constant 
	smaller than 1 which maintains all the previous assumptions, with 
	corresponding rescaling of the spectrum.) 
	
 Aiming at \eqref{zvbd}, let 
$Z'$ 
be the orthogonal complement of $V_\star$ in $V_0$ 
with respect to $b_0=(\cdot,\cdot)_0$ on $V_0$, cf.  
Remark \ref{zorth}. So, for any 
	$z'\in Z'$ and $v_\star\in V_\star$, $b_0(z',v_\star)=0$ and 
	the idea is to construct $Z$ by ``correcting'' $Z'$ via 
adding to $z'$ some $Tz'\in V_\star$, i.e. to seek $Z\ni z=z'+Tz'$. Then \eqref{zvbd} reads 
$b_0(z' +Tz', v_\star)\, =\, d_0(z'+Tz',v_\star)$, which can be restated as a problem for  
$T : Z'\rightarrow V_\star$ as follows: 
\begin{equation}
	\label{tuprobl}
 	b_0(Tz',\tilde{v}) \,-\, d_0(Tz',\tilde{v}) \,\,=\,\, d_0(z',\tilde{v}) \,-\, b_0(z',\tilde{v})
	\,=\,\, d_0(z',\tilde{v}), \quad \forall \tilde{v} \in V_\star.
 	\end{equation}
Notice that problem \eqref{tuprobl} is well-posed on $V_\star$ due to the coercivity 
implied by 
\eqref{kdass}, and so uniquely 
determines	a $\|\cdot\|_0$-bounded injective linear operator 
$T : Z'\rightarrow V_\star$,  
 	and $Z:= (I+ T) Z'$ satisfies  \eqref{zvbd}. 
	Check 
	that $V_0 = V_\star\dot{+}Z$. For any $v_0\in V_0$, seek $v_\star\in V_\star$ and $z\in Z$ i.e. 
$z'\in Z'$ such that for $z=z'+Tz'$, 
$v_0=v_\star+z=\left(v_\star+T z'\right)+ z'$. 
This uniquely determines $z'=P_{Z'} v_0$, $v_\star=P_{V_\star} v_0-T P_{Z'} v_0$ 
and $z=v_0-v_\star$, so $V_0 = V_\star\,\dot{+}\,Z$. 
Furthermore, via \eqref{zvbd} and \eqref{kdass}, for all $z \in Z$ and $v_\star \in V_\star$, 
 	\[
 	\left|(v_\star,z)_0\right|\,=\,|b_0(v_\star,z) |\,=\, 
	|d_0(v_\star,z) |\,\,\le\,\, d_0^{1/2}[v_\star]\, d_0^{1/2}[z]\,\,\le \,\,  
	M_d\,  b_0^{1/2}[v_\star]\, b_0^{1/2}[z]
\,\,=\,\,M_d\,\|v_\star\|_0\,\|z\|_0. 
 	\]	
 	Therefore 
	\eqref{VZorth} hold  with $K_Z=M_d$. 
	Notice that $Z$ is closed (and therefore also weakly closed) linear subspace, as follows e.g. from 
	the closedness of $Z'$ and the boundedness of $T$. 
	\vspace{.1in}
 	
 	Now, by \eqref{zvbd}, the spectral problem \eqref{lxispec} can be rewritten as 
 	\begin{flalign*}
 	a_\xi^{\rm h}(z,\tilde{z})\,+\,b_0(z,\tilde{z})\,+\,b_0(v,\tilde{v})\,+\,d_0(z,\tilde{v})\,+
	\,d_0(v,\tilde{z})\,\,=\,\,\lambda\,\, d_0(v+z,\tilde{v} +\tilde{z}), 
	\quad \forall\, \tilde{v} + \tilde{z} \in V_\star \dot{+} Z, 
 	\end{flalign*} or equivalently,
 	\begin{gather}
 	\label{03.09.20e1} b_0(v,\tilde{v})  \,-\, \lambda d_0(v,\tilde{v})\, =\, (\lambda-1)\, d_0(z,\tilde{v} ), \quad \forall\, \tilde{v}\in V_\star; \\
 	a_\xi^{\rm h}(z,\tilde{z}) \, =\, -\, b_0(z,\tilde{z}) \,+\, \lambda\, d_0(z,\tilde{z}) \,+\, 
	(\lambda-1)\, d_0(v, \tilde{z}), \quad \forall\, \tilde{z} \in  Z. 
 	\label{03.09.20e2}
 	\end{gather}
 	It follows from \eqref{03.09.20e1}  that $v\in {\rm dom}\, \mathbf{B}_\star$ where 
	$\mathbf{B}_\star$ is defined by 
	\eqref{ikddd2}, 
	i.e. is self-adjoint operator in 
	$\overline{V_\star}$, 
	generated by $b_0$ with form domain $V_\star$. 
	Furthermore,  if $\lambda \notin {\rm Sp}\, \mathbf{B}_\star$,  $v=v_\lambda(z)$ is uniquely found from $z$: 
	\begin{equation} 
	\label{vfromz}
	v\,\,=\,\,v_\lambda(z)\,=\, (\lambda-1) \big(\mathbf{B}_\star-\lambda I\big)^{-1} \mathcal{P}_{\overline{V_\star}}^0\,z. 
	\end{equation}
	For such $\lambda$,  \eqref{03.09.20e2} implies that $z\in Z\backslash \{0\}$ solves
 	\begin{equation}
	\label{finallimitspectralproblem}
 	a_\xi^{\rm h}(z,\tilde{z}) \,\,=\,\, \beta_\lambda (z,\tilde{z}),  \quad \forall  \tilde{z} \in Z,
 	\end{equation}
 where, for $\lambda\notin {\rm Sp}\, \mathbf{B}_\star$, $\beta_\lambda : Z\times Z \rightarrow \CC$	is the sesquilinear form
\begin{equation}
\label{betadef}
\beta_\lambda(z,\tilde z)\,\,:=\,\,-\,\,b_0(z,\tilde z)\,+\,\lambda\, d_0(z,\tilde z)\,+\,
(\lambda-1)d_0\left(v_\lambda(z),\tilde z\right), 
\end{equation} 
where $v_\lambda(z)$ is the unique solution of \eqref{03.09.20e1}. 
Show that, for real $\lambda\notin {\rm Sp}\, \mathbf{B}_\star$, 
form $\beta_\lambda$ is 
Hermitian, 
i.e.  $\beta_\lambda(\tilde z,z)=\overline{\beta_\lambda(z,\tilde z)}$ for all $z,\tilde z\in Z$; 
in particular $\beta_\lambda[z]:=\beta_\lambda(z,z)$ is real-valued. 
Indeed, given $z,\tilde z\in Z$, 
by setting in \eqref{03.09.20e1} $z=\tilde z$ and $\tilde v=v_\lambda(z)$ and combing with \eqref{betadef} 
and using the symmetry of $b_0$ and $d_0$, 
\[
\beta_\lambda(z,\tilde z)\,=\,-\,b_0(z,\tilde z)\,+\,\lambda\, d_0(z,\tilde z)\,+\,
b_0\big(v_\lambda(z),v_\lambda(\tilde z)\big)\,-\,
\lambda\,\, d_0\big(v_\lambda(z),v_\lambda(\tilde z)\big)\,\,=\,\,
\overline{\beta_\lambda(\tilde z,z)}.  
\]
For an equivalent operator interpretation of $\beta_\lambda$, from \eqref{betadef} and \eqref{vfromz}, 
\begin{equation}\label{betaform}
\beta_\lambda (z,\tilde{z}) \,\,=\,\, -\,\, b_0(z,\tilde{z}) \,+\,\, d_0\big(\beta(\lambda)z,\,\tilde{z}\big)
\end{equation} 
 for $\beta(\lambda) : \overline{Z} \rightarrow \overline{Z}$ the bounded linear operator
\begin{equation}
\label{6.27-2}
	\beta(\lambda)  \,\,:=\,\,\lambda\, I \,+\, (\lambda-1)^2\, \mathcal{P}_{\overline{Z}}^0\,
	(\mathbf{B}_\star-\lambda I)^{-1} \mathcal{P}_{\overline{V_\star}}^0\,.
\end{equation}

 Now notice that, since $a^{\rm h}_\xi[z]$ is non-negative real for all $\xi$ and $z$,  \eqref{finallimitspectralproblem} immediately implies the 
inclusion
\begin{equation}\label{easyspinc}
\bigcup_{\xi \in \mathbb{R}^n} {\rm Sp}\, \mathbb{L}_\xi \,\subseteq\, \Bigl\{ \lambda \notin {\rm Sp}\, \mathbf{B}_\star : \beta_{\lambda }[z] \ge 0 \text{ for some $0\neq z \in Z$} \Bigr\} \cup {\rm Sp}\, \mathbf{B}_\star.
\end{equation}
In fact we have a stronger assertion, providing the following important characterisation of the limit spectrum in terms of the form $\beta_\lambda$. 
\begin{theorem}\label{thm.limspecrep} The following characterisation of the limit collective spectrum holds:
\begin{equation*}
{\rm Sp}_0\,:=\,
\overline{\bigcup_{\xi \in \mathbb{R}^n} {\rm Sp}\, \mathbb{L}_\xi} \,\, =\,\, \Bigl\{ \lambda \notin {\rm Sp}\, \mathbf{B}_\star : \beta_{\lambda }[z] \ge 0 \text{ for some $0\neq z \in Z$} \Bigr\} \cup {\rm Sp}\, \mathbf{B}_\star.
\end{equation*}
%
\end{theorem}
\begin{proof}

{{\it Step 1.} 
	Here we prove
	\begin{equation*}
	\overline{\bigcup_{\xi \in \mathbb{R}^n} {\rm Sp}\, \mathbb{L}_\xi }\,\subseteq\, 
	\mathcal{C} \,: =\,\Bigl\{ \lambda \notin {\rm Sp}\, \mathbf{B}_\star : \beta_{\lambda }[z] \ge 0 \text{ for some $0\neq z \in Z$} \Bigr\} \cup {\rm Sp}\, \mathbf{B}_\star.
	\end{equation*}
By \eqref{easyspinc} it is sufficient to show that	
	$\mathcal{C}$ is closed. Recalling that ${\rm Sp}\, \mathbf{B}_\star$ is closed, let $\lambda_n \in \mathcal{C}$ be such that 
	$\lim_n \lambda_n = \lambda \notin {\rm Sp}\, \mathbf{B}_\star$, and consider $n$ large enough so that  
	${\rm dist} ( \lambda_n , {\rm Sp}\, \mathbf{B}_\star) \ge \delta >0$. Then 
	$\left\| (B-\lambda_n I)^{-1}\right\|_{(V_\star,d_0) \rightarrow (V_\star,d_0)}\leq \delta^{-1}$, and so 
	by \eqref{6.27-2} 
	$\| \beta(\lambda_n) \|_{(\mathcal{H},d_0) \rightarrow (\mathcal{H},d_0)} $ is bounded,  and there exists $z_n \in Z$, $d_0[z_n] =1$ and $\beta_{\lambda_n}[z_n] \ge 0$. 
Via \eqref{betaform} these assertions imply that $b_0[z_n]$ is bounded, and 
so (up to a subsequence) $z_n$ converges weakly  in $H$, and strongly in $\mathcal{H}$, to some $z \in Z$, $d_0[z]=1$ and by weak lower semi-continuity of $b_0$, 
 $b_0[z] \le {\liminf}_n b_0[z_n]=:\underline{\lim}_n b_0[z_n]$.  We show that all these assertions imply  that 
$\beta_\lambda[z] \ge 0$. Indeed, for $\mathcal{R}_{\mu} :=(\mathbf{B}_\star - \mu I)^{-1}$, we calculate via 
\eqref{betadef} and \eqref{vfromz} 
\[
\beta_{\lambda_n}[z_n]  \,= \,-\,\, b_0[z_n] \,+\,\lambda_n\, d_0[z_n]  \,+\, (\lambda_n -1)^2\,
d_0\left( \,\mathcal{R}_{\lambda_n} \mathcal{P}_{\overline{V_\star}}^0 z_n\,,\,z_n\,\right) \,\,=
\]
\[
\ \ \ \ \ \ \ \ 
\, -\,\, b_0[z_n]\, +\,\,\lambda_n\, d_0[z_n] \, +\, (\lambda_n -1)^2\,d_0\left( \,\mathcal{R}_{\lambda} \mathcal{P}_{\overline{V_\star}}^0 z_n\,,\,z_n\,\right) \,+\, (\lambda_n-\lambda)(\lambda_n -1)^2 
d_0\left(\, \mathcal{R}_{\lambda_n} \mathcal{R}_{\lambda} \mathcal{P}_{\overline{V_\star}}^0 z_n\,,\,z_n\,\right),  
\] 
having used in the last equality the standard resolvent identity 
$\mathcal{R}_{\lambda_n} - \mathcal{R}_{\lambda}=(\lambda_n-\lambda)\mathcal{R}_{\lambda_n} \mathcal{R}_{\lambda}$. 
By taking the limit superior (and recalling $\beta_{\lambda_n}[z_n] \ge 0$)  we obtain  
\[
0 \,\le\, \overline{\lim}_n\, \beta_{\lambda_n}[z_n] \,\le \,-\,\, \underline{\lim}_n b_0[z_n] \,+\,
\lambda d_0[z]  \,+\, (\lambda -1)^2d_0
\left(\, \mathcal{R}_{\lambda} \mathcal{P}_{\overline{V_\star}}^0\, z\,,\,z\,\right) \,\, \le 
\]
\[
\ \  \ \ \ \ \ \ \ \ \ \ \ \quad 
- \,\, b_0[z] \,+\,\lambda\, d_0[z] \, +\, (\lambda -1)^2
d_0\left(\, \mathcal{R}_{\lambda} \mathcal{P}_{\overline{V_\star}}^0 \,z\,,\,z\,\right) 
\,=\,\, \beta_\lambda[z].
\]
That is $\beta_\lambda[z]\ge 0$ and since  $0\neq z \in Z$ it follows that $\lambda \in \mathcal{C}$. Hence $\mathcal{C}$ is closed, as required. 
\vspace{.04in} 

{\it Step 2.} Let us now prove
\begin{equation*}
\mathcal{C}
\subseteq \overline{\bigcup_{\xi \in \mathbb{R}^n} {\rm Sp}\, \mathbb{L}_\xi }.
\end{equation*}
We shall consider two cases. First, let 
$\lambda \in {\rm Sp}\, \mathbf{B}_\star\subset [1,+\infty)$. Notice that, by \eqref{bdrylimspec},  for every $\lambda \in {\rm Sp}\, \mathbf{B}_\star={\rm Sp}\, \widetilde{\mathbf{B}}$ and for any $\xi\in\mathbb{R}^n$ such that 
$|\xi|>(\lambda/\tilde\nu)^{1/2}$ there is a $\lambda_\xi \in {\rm Sp}\, \mathbb{L}_{\xi} $  such that
$
\left| \lambda^{-1} - \lambda_\xi^{-1} \right| \le |\xi |^{-2} \tilde\nu^{-1}.
$
Consequently $\lim_{|\xi| \rightarrow \infty} \lambda_\xi = \lambda$, i.e. $\lambda \in \overline{\bigcup_{\xi \in \mathbb{R}^n} {\rm Sp}\, \mathbb{L}_\xi }.$

Now let 
$\lambda \in \mathcal{C} \backslash  {\rm Sp}\, \mathbf{B}_\star$. Then for some 
$0\neq z' \in Z$, 
$\beta_\lambda[z'] \ge 0$. 
Fix $\eta \in \mathbb{R}^n$, $|\eta|=1$, and set 
\[
k \,:=\, \sup_{z\in Z \backslash \{ 0\}} \,\frac{\beta_\lambda[z]}{a^{\rm h}_\eta [z]}\,.
\]
By 
definition \eqref{betaform}--\eqref{6.27-2} of the form $\beta_\lambda$ implying that $\beta_\lambda[z]/\|z\|_0^2$ is bounded, and by  coercivity of $a^{\rm h}_\eta$, 
see \eqref{ahcoercive}, we see that $k$ is finite and, as $\beta_\lambda[z']\ge0$,  $k$ is non-negative. 
We aim at showing that the above supremum is attained by some point $z_0\in Z\backslash\{0\}$. If $k=0$ we can set 
$z_0=z'$. Let $k>0$, and let $z_n \in Z$, $d_0[z_n]=1$, be such that ${\beta_\lambda[z_n]}/{a^{\rm h}_\eta [z_n]}$ 
converges to $k$. Then, for large enough $n$, $\beta_\lambda[z_n]\ge 0$, and similarly to Step 1 we conclude from 
\eqref{betaform}--\eqref{6.27-2} that $b_0[z_n]$ is bounded. 
Then, by arguing further similarly to Step 1, we see that (up to a subsequence) $z_n$ weakly converges in $H$ to some $0\neq z_0 \in Z$ and that
$\beta_\lambda[z_0] \ge \overline{\lim}_n \beta_\lambda[z_n]$.
 Furthermore, since the sesquilinear form $a^{\rm h}_\eta$ is bounded and positive (see Proposition \ref{prop.ahom}) 
it is weakly lower semi-continuous and so one has  $a^{\rm h}_\eta[z_0] \le \underline{\lim}_n  a^{\rm h}_\eta [z_n]$. Thus we obtain 
 \[
k \,\,\,\ge\,\,\, \frac{\beta_\lambda[z_0]}{a^{\rm h}_\eta [z_0]} \,\,\,\ge\,\,\, \frac{ \overline{\lim}_n \beta_\lambda[z_n]}{\underline{\lim}_n  a^{\rm h}_\eta [z_n]}  \,\,\, \ge\,\,\, \overline{\lim}_n   \frac{ \beta_\lambda[z_n]}{ a^{\rm h}_\eta [z_n]} \,\, = \,\, k.
 \]
Therefore $k$ is attained by $z_0$, as desired. 
So the sesquilinear form $\widehat{A}(z,\tilde z):=k a^h_{\eta}(z,\tilde{z})-\beta_\lambda(z,\tilde{z})$ 
is non-negative on $Z$ (i.e. $\widehat{A}[z]\ge 0$, $\forall z\in Z$) and vanishes at $z_0\neq 0$, $A[z_0]=0$. 
Therefore, cf. \eqref{Vkersesa}, $\widehat{A}(z_0,\tilde z)=0$ for all $\tilde z\in Z$, i.e. 
$
\beta_\lambda(z_0,\tilde{z})\, =\, k\, a^{\rm h}_{\eta}(z_0,\tilde{z})\,=\, a^{\rm h}_{k^{1/2}\eta}(z_0,\tilde{z}), \quad \forall \tilde{z} \in Z.
$ 
Hence, cf. \eqref{finallimitspectralproblem}, 
$\lambda \in {\rm Sp}\, \mathbb{L}_{\xi}$ for $\xi={k^{1/2}} \eta$ with, according to \eqref{vfromz},  non-zero eigenvector 
$v_0\,+\,z_0\,=\,(\lambda-1)(\mathbf{B}_\star-\lambda I)^{-1} \mathcal{P}_{\overline{V_\star}}^0 z_0 + z_0$. 
}
\end{proof}
\begin{remark}
\label{remdisprel}
Theorem \ref{thm.limspecrep} assures that $\lambda\in\mathbb{R}\,\backslash\, {\rm Sp}\,\mathbf{B}_\star$ is not in the limit collective 
spectrum ${\rm Sp}_0$, in particular is in its gap, 
if $\beta_\lambda$ is negative-definite i.e. 
$\beta_\lambda[z]<0$ for all $0\ne z\in Z$. 
Notice that 
the above 
also  provides 
a way for describing the limit dispersive relations $\lambda^{(k)}_\xi$ and associated eigenvectors. 
Indeed, fixing $\lambda\notin{\rm Sp}\,\mathbf{B}_\star  $               
and $\eta\in S^{n-1}$, for $\xi =t\eta$ with $t\ge 0$ 
\eqref{finallimitspectralproblem} reads: for some $0\ne z\in Z$,  
$t^2 a_\eta^{\rm h}(z,\tilde{z}) \,=\, \beta_\lambda (z,\tilde{z})$,  $\forall  \tilde{z} \in Z$. 
If for simplicity $Z$ is finite-dimensional, the forms 
$\beta_\lambda$ and $a_\eta^{\rm h}$ are realised by Hermitian matrices 
$B_\lambda$ and $A_\eta^{\rm h}$ 
respectively, and 
so $B_\lambda z= t^2A_\eta^{\rm h}z$. 
As $A_\eta^{\rm h}$ is positive, this is restated as an eigenvalue 
problem: 
$\mathcal{B}_{\lambda,\eta}\hat z:=\left(A_\eta^{\rm h}\right)^{-1/2}B_\lambda\left(A_\eta^{\rm h}\right)^{-1/2} \hat z= t^2\hat z$ for eigenvalues $t^2\ge 0$ and eigenvectors $\hat z=\left(A_\eta^{\rm h}\right)^{1/2}z$. 
Hence, for every non-negative eigenvalue $k\ge 0$ of 
Hermitian matrix 
$\mathcal{B}_{\lambda,\eta}$ with associated eigenvector $\hat z$, 
$\lambda$ is an eigenvalue of $\mathbb{L}_{k^{1/2}\eta}$ with associated eigenvector 
$z_0+v_\lambda(z_0)$ where $z_0=\left(A_\eta^{\rm h}\right)^{-1/2}\hat z$. 
\end{remark}
\subsection{An approximation by a bivariate operator}\label{s.bivariate}

Here we will provide 
representations to the 
approximations in Theorem \ref{splimSe.3} for 
the resolvents $\mathcal{L}_{\ep,\t}^{-1}$  and in Corollary \ref{c.collspec} for the 
collective spectrum of the original operators $\mathcal{L}_{\ep,\t}$ 
in terms of those 
of an 
operator defined on the Bochner space $L^2(\RR^n ; \mathcal{H}_0)$, 
i.e. on a separable Hilbert space-valued functional space with 
$\mathcal{H}_0 = \left(\overline{V_\star \dot{+}Z},\,d_0\right)$. 
This operator, as examples in Section \ref{sec:examples} will illustrate, can be viewed as an abstract version 
of a two-scale limit operator. 
Some basic facts about the 
Bochner spaces, see e.g. \cite{Hytonen}, which are relevant to our purposes 
are collected in Appendix B. 
With the right-hand-side given by \eqref{ik3} for any $g\in\mathcal{H}$, according to Theorem  \ref{splimSe.3} 
the solution $u_{\ep,\t}=\mathcal{L}_{\ep,\t}^{-1}g$ to the original problem \eqref{p1} is approximated by 
${u}^a_{\ep,\t}= 
\mathcal{E}_\t\, \mathbb{L}_{\t / \ep}^{-1}\mathcal{P}_{\mathcal{H}_0}^0\mathcal{E}_\t^{-1}g$. 
Fixing $h:=\mathcal{P}_{\mathcal{H}_0}^0\mathcal{E}_\t^{-1}g\in \mathcal{H}_0$, 
one recalls 
that $v+z:=\mathbb{L}_{\t / \ep}^{-1}h\in \text{dom }\mathbb{L}_{\t / \ep}\subset V_\star\dot{+}Z=V_0$ depends on $\t$ and $\ep$ only via their ratio $\xi:=\theta/\ep\in\mathbb{R}^n$ where according to 
\eqref{defhom.form}  the dependence of $\mathbb{L}_\xi$ on $\xi$ is quadratic. 
The idea is to 
try and represent it via an appropriate (inverse) Fourier transformed problem with a transformed variable $x\in\mathbb{R}^n$ of $\xi$, 
or in other words to regard $\mathbb{L}_\xi$ as a symbol of ($\mathcal{H}_0$-valued) differential operator in $x$.  
To that end, first recall that according to \eqref{Sform} the above $v+z\in V_0$ is the solution to 
\begin{equation}
\label{Linvprobl}
a^{\rm h}_{\xi}(z,\tilde z)\,\,+\,\,b_0(v+z, \tilde v+\tilde z)\,\,=\,\,d_0(\,h\,,\, \tilde v+\tilde z), \quad 
\forall\,\, \tilde v+\tilde z\in V_\star\dot{+}Z; \quad \xi:=\t/\ep, \quad \t\in\Theta. 
\end{equation} 
Regard now $h=h(\t)$, $\t\in\Theta$, as belonging to the Bochner space $L^2(\Theta; \mathcal{H}_0)$ with the standard 
$\RR^n$-Lebesgue measure induced on $\Theta$. By extending $h(\t)$ by zero outside $\Theta$ for the whole of $\RR^n$,
we can assume $h\in L^2(\RR^n; \mathcal{H}_0)$. 
The quadratic dependence of 
$a^h_\xi$ on $\xi$ in \eqref{defhom.form} can be represented as follows: 
\begin{equation}
\label{ahgrad}
a^{\rm h}_\xi(z, \tilde z)\,\,=\,\,\sum_{j,k=1}^n a^{\rm h}_{jk}(z, \tilde z)\,\xi_j\,\xi_k\,\,=\,\,
\sum_{j,k=1}^n a^{\rm h}_{jk}\left(\xi_j z, \xi_k\tilde z\right), 
\ \ \ \ 
\forall z, \tilde z\in Z,\,\, \xi\in\mathbb{R}^n. 
\end{equation} 
Here 
\begin{equation}
\label{ahij}
a^{\rm h}_{jk} \big( {z}, {\tilde{z}} \big)\,\,:=\,\, a_{0\, jk}''( z,\tilde z )\,-\, 
a_{0} ( N^j {z}, N^k {\tilde{z}}), 
\end{equation} 
where 
$N^j:=N_{e^j}=e^j\cdot N$ and  
$e^1,...,e^n$ is the canonical basis in $\mathbb{R}^n$, i.e. $N^j=N_\t$ with $\t=e^j$, cf. \eqref{cell:prob2}. 

Let us now, given $0<\ep<1$, make in \eqref{Linvprobl} a change of variable $\t \rightarrow \xi=\t/\ep\in\RR^n$, and 
 with the aim of formally integrating \eqref{Linvprobl} over $\RR^n$ in $\xi$ and recalling \eqref{ik2} 
assume $v,\tilde v\in L^2\left(\RR^n; V_\star\right)$. (Remind that we regard $V_\star$ and $Z$ as Hilbert spaces with norm $\|\cdot\|^2_0=b_0[\cdot]$.) 
Regarding $h\in L^2(\RR^n; \mathcal{H}_0)$ as arbitrary, and bearing in mind the 
boundedness of $a^{\rm h}_\xi$ in $z, \tilde z\in Z$ as well as its 
quadratic growth 
in $\xi$, we need to take 
$z(\xi),\tilde z(\xi)\in L^2\left(\RR^n; Z\right)$ so that also 
$\xi_jz(\xi),\xi_k\tilde z(\xi)\in L^2\left(\RR^n; Z\right)$, $\forall j,k=1,...,n$. 
In other words, $z$ and $\tilde z$ can be said to belong to weighted Bochner space 
$L^2\left(\RR^n, {\langle\xi\rangle^2}{\rm d}\xi; Z\right)$ with the weight $\langle\xi\rangle^2:=1+|\xi|^2$. 
The resulting problem is to find $v\in L^2\left(\RR^n; V_\star\right)$ and 
$z\in L^2\big(\RR^n, {\langle\xi\rangle^2}{\rm d}\xi; Z\big)$, such that 
\begin{equation}
\label{Lproblint}
\int_{\RR^n}a^{\rm h}_{\xi}\left(z,\tilde z\right){\rm d}\xi\,\,+\,\int_{\RR^n}\,\,b_0(v+z, \tilde v+\tilde z){\rm d}\xi\,= \int_{\RR^n}d_0\big(\,h(\ep\xi), \tilde v+\tilde z\big){\rm d}\xi, \ \ \ 
\forall \tilde v+\tilde z\in L^2\big(\RR^n; V_\star\big)\dot{+}L^2\big(\RR^n, {\langle\xi\rangle^2}{\rm d}\xi; Z\big). 
\end{equation} 
(Notice that the above is obviously a direct sum, since so is $V_\star\dot{+}Z$.) 

We next argue 
that, for arbitrary $h\in \mathbb{H}_0:=L^2(\RR^n; \mathcal{H}_0)$, problem \eqref{Lproblint} is  well-posed on its own right, and the form $\mathbb{A}$ on its left-hand side generates 
a self-adjoint operator $\mathbb{L}$ in Bochner (Hilbert) 
space $\mathbb{H}_0$. 
Indeed, the form is non-negative and has domain 
$\mathbb{D}=L^2\left(\RR^n; V_\star\right)\dot{+}L^2\left(\RR^n, {\langle\xi\rangle^2}{\rm d}\xi;\, Z\right)$ which, see  Proposition \ref{propb1} of Appendix B, is dense in 
$\mathbb{H}_0$ 
and is closed with respect to the form (Proposition \ref{propb2}). 
Since, as implied by \eqref{ik2}, ${\rm Sp}\, \mathbb{L}\subset [1,+\infty)$, for any  
$h\in \mathbb{H}_0$ problem \eqref{Lproblint} 
has a unique 
solution $v+z\in\mathbb{D}$. 
Moreover, given $0<\ep<1$, as shown in Proposition \ref{propb3}, \eqref{Lproblint} holds 
if and only if 
\eqref{Linvprobl} holds for a.e. $\t\in\Theta$. 
Therefore (see Definition \ref{defb4} and Proposition \ref{propb5}), the latter immediately implies that the newly defined operator $\mathbb{L}$ is a direct integral of $\mathbb{L}_\xi$  
which serve as 
its fibers: $\mathbb{L}=\int_{\RR^n}^\oplus \mathbb{L}_\xi \, {\rm d}\xi$.  
Similarly, for the resolvents, $\mathbb{L}^{-1}=\int_{\RR^n}^\oplus \mathbb{L}^{-1}_\xi \, {\rm d}\xi$. 

We now aim at equivalently reformulating problem \eqref{Lproblint} in a Fourier transformed 
setting,  
i.e. for $\left(\check v+u\right):=\mathcal{F}^{-1}(v+z)$ where $v+z$ is the 
solution to \eqref{Lproblint} 
with $h(\ep\xi)$ replaced for a moment by arbitrary $f\in \,\mathbb{H}_0$, 
i.e. 
$v+z=\mathbb{L}^{-1}f$, 
 and $\mathcal{F}^{-1}$ being  
the inverse Fourier transform in the sense of Definition \ref{defb6} for 
$\mathbb{H}=L^2\left(\RR^n; \left(\mathcal{H}, d_0\right)\,\right)$. 
As 
follows from  \eqref{ftdef}, $\mathcal{F}$ restricted to the closed subspace 
$\mathbb{H}_0$ of $\mathbb{H}$ coincides with the 
Fourier transform as given by Definition \ref{defb6} directly for 
$\mathbb{H}_0$. 
Then, 
$\check v+u=\mathcal{F}^{-1}\mathbb{L}^{-1}f= 
\mathcal{F}^{-1}\mathbb{L}^{-1}\mathcal{F} F= \mathcal{L}^{-1}F$, where 
$F:=\mathcal{F}^{-1}f\in \mathbb{H}_0$ and $\mathcal{L}:=\mathcal{F}^{-1}\mathbb{L}\,\mathcal{F}$. 
We will show that $\mathcal{L}$, 
a self-adjoint operator in $\mathbb{H}_0$, is generated by 
a form which is a formal (inverse) Fourier transform of the one in \eqref{Lproblint}. 

To that end, let $u,\tilde u\in H^1\big(\RR^n ; \left(Z, b_0\right)\big)$, see Definition \ref{defb7}, 
and introduce $a^{\rm h}\big( \nabla  u(x), \nabla\tilde{u}(x)\big)$ 
by formally replacing ``symbols'' $\xi_j$ on the right-hand side of \eqref{ahgrad}  by operators 
 $-i\,\partial_{x_j}$, i.e. by their Fourier-transformed counterparts: 
$
a^{\rm h}\big( \nabla  u(x), \nabla\tilde{u}(x)\big)\,:=\,a^h_{-i\nabla}(u,\tilde u)\,:=\,
\sum_{j,k=1}^n a^{\rm h}_{jk}\left(\partial_{x_j} u, \partial_{x_k}\tilde u\right). 
$ 
This motivates considering the 
subspace 
$\check{\mathbb{D}}=H^1(\RR^n ; Z) \dot{+} L^2(\RR^n ;V_\star)$ of 
$\mathbb{H}_0$, 
on which we define the bivariate form
\begin{equation}\label{Q}
Q\big(u+v,\tilde{u} + \tilde{v}\big) \,\,: =\,\,\int_{\RR^n} a^{\rm h}\big( \nabla  u(x), \nabla\tilde{u}(x)\big)  \, {\rm d}x 
\,\,+\,\,  \int_{\RR^n} b_0\Big(u(x) + v(x),\, \tilde{u}(x) + \tilde{v}(x) \Big) \, {\rm d} x, 
\end{equation}
where $u,\tilde{u} \in H^1\big(\RR^n;\left(Z,b_0\right)\big)$ and 
$ v, \tilde{v} \in L^2\big(\RR^n;\left(V_\star,b_0\right)\big).$ 

Lemma \ref{lemft} establishes that the form $Q$ specifies a self-adjoint ``bivariate'' operator $\mathcal{L}$ 
in Hilbert space $\mathbb{H}_0$,  which is a Fourier-transformed counterpart of $\mathbb{L}$, namely 
$\mathcal{L}=\mathcal{F}^{-1}\mathbb{L}\,\mathcal{F}$. 
\begin{remark}
We shall see in the examples that $\mathcal{L}$ often coincides with the  
two-scale limit  operator, e.g. 
for elliptic differential operators with 
high-contrast periodic coefficients, 
Section \ref{e.dp}. 
\end{remark}


Aiming at restating Theorem \ref{p.unitaryequiv} in terms of operator $\mathcal{L}$, 
we first observe that from  Proposition \ref{propb5} (see Definition \ref{defb4}) 
and Lemma \ref{lemft} 
\begin{equation}
\label{lft}
\mathbb{L}_\xi^{-1} f(\xi) \,\,=\,\,  
\big(\mathbb{L}^{-1} f\big)(\xi) \,\,=\,\,
\big(\mathcal{F}\mathcal{L}^{-1}\mathcal{F}^{-1} f \big)(\xi) \quad 
for \,\,a.e.\ \xi, \quad f \in \mathbb{H}_0=L^2(\RR^n ; \mathcal{H}_0).
\end{equation}
Relation \eqref{lft} signifies the important fact that, while the bivariate resolvent $\mathcal{L}^{-1}$ is 
generally not decomposable into a direct integral, its Fourier transformed counterpart 
$\mathcal{F}\mathcal{L}^{-1}\mathcal{F}^{-1}$ is. 
Further, we observe that the orthogonal projectors $\mathcal{P}_{\mathcal{H}_0}^0:(\mathcal{H},d_0) \rightarrow \mathcal{H}_0$ and  $\mathcal{P}:L^2(\RR^n ; (\mathcal{H},d_0)) \rightarrow L^2(\RR^n ; \mathcal{H}_0)$ are related by the identity 
(Proposition \ref{propb9}) 
\begin{equation}
\label{projft}
\mathcal{P}_{\mathcal{H}_0}^0f(\xi) = \big(\mathcal{F}  \mathcal{P}\mathcal{F}^{-1} f \big)(\xi) \quad 
for \,\,a.e.\  \xi, \quad f \in \mathbb{H}=L^2(\RR^n; (\mathcal{H},d_0)). 
\end{equation}
Next, introduce in $\mathbb{H}$ a normalised rescaling operator 
$\Gamma_{\ep} : \mathbb{H} \rightarrow \mathbb{H} $ and its inverse $\Gamma_{\ep}^{-1}$ by 
\begin{equation}
\label{gammaep}
\big(\Gamma_{\ep}f\big)(x)\,\,:=\,\,\ep^{n/2}f(\ep x), \ \ \ \ \ \ \ 
\big(\Gamma^{-1}_{\ep}f\big)(x)\,\,=\,\,\ep^{-n/2}f\left(\ep^{-1} x\right). 
\end{equation} 
Notice that $\Gamma_{\ep}$ and $\Gamma^{-1}_{\ep}$ are 
 unitary operators in $\mathbb{H}$. Then, via \eqref{lft} and \eqref{projft},  
\begin{equation}
\label{lpft}
 \mathbb{L}_{\t/\ep}^{-1} \mathcal{P}_{\mathcal{H}_0}^0 f (\t)\, =\, 
\Big( \Gamma_{\ep}^{-1}  \mathbb{L}^{-1} \mathcal{P}_{\mathcal{H}_0}^0  \Gamma_{\ep} f \Big)(\t)
\, =\, 
\Big( \Gamma_{\ep}^{-1} \mathcal{F} \mathcal{L}^{-1} \mathcal{P} \mathcal{F}^{-1} \Gamma_{\ep} f \Big)(\t), \ \ \ for\,\,  a.e. \  \t, \quad f \in 
L^2\big(\RR^n;(\mathcal{H},d_0)\big).
\end{equation} 
Finally, for reformulating Theorem \ref{p.unitaryequiv}, we recall that it approximates the exact solution 
$u_{\ep,\t}=\mathcal{L}_{\ep,\t}^{-1}g$, where according to \eqref{ik3} $g\in\mathcal{H}$ for any $\t\in\Theta$. 
We now assume $g\in L^2\big(\Theta;\left(\mathcal{H},d_0\right)\big)$, and 
comparing with the approximation 
$u^a_{\ep,\t}=\mathcal{E}_\t \mathbb{L}_{\t / \ep}^{-1}\mathcal{P}_{\mathcal{H}_0}^0\mathcal{E}_\t^{-1}g$ of 
Theorem \ref{p.unitaryequiv}  suggests taking in \eqref{lpft} $f=\chi\mathcal{E}^{-1}g$, where 
$\mathcal{E}$ 
 is given by $f(\t) \mapsto \mathcal{E}_\t f(\t),$ for a.e. $ \t \in \Theta$, and 
$\chi: L^2(\Theta ; (\mathcal{H},d_0)) \rightarrow  L^2(\RR^n ; (\mathcal{H},d_0))  $ 
is the 
operator of extension by zero outside $\Theta$. 
We notice that bounded operator $\mathcal{E}$ acts from $\mathbb{H}_\Theta:=L^2\big(\Theta ; (\mathcal{H},d_0)\big)$ into itself 
as, due to \eqref{H6}, $\mathcal{E}_\t$ is continuous in $\t$ and so is $\t$-(weakly) measurable. 
Further, in combination with \eqref{vs61}, \eqref{H6} assures that $d_\t$ is also continuous in $\t$ and hence $(u,\tilde u)=\int_\Theta d_\t\big(u(\t), \tilde u(\t)\big) {\rm d}\t$ forms an equivalent inner product in $\mathbb{H}_\Theta$. When endowed with such an inner product, we conveniently denote this 
space by $L^2\big(\Theta ; (\mathcal{H},d_\t)\big)$, and notice that because of \eqref{H6} operator $\mathcal{E}$ is 
unitary when viewed as acting from $L^2\big(\Theta;\left(\mathcal{H},d_0\right)\big)$ to 
$L^2\big(\Theta;\left(\mathcal{H},d_\t\right)\big)$. 
Assembling all this together and also noticing that 
$\Gamma_{\ep}^{-1} \mathcal{F}= \mathcal{F}\Gamma_{\ep}$ and $\mathcal{F}^{-1}\Gamma_{\ep}=\Gamma_{\ep}^{-1} \mathcal{F}^{-1}$, 
for the above approximation,  
$
u^a_{\ep,\t}\,=\,\left(\mathcal{E}\, \chi^*\, \mathcal{F}\Gamma_{\ep}\,
\mathcal{L}^{-1}\mathcal{P}\, 
\Gamma_{\ep}^{-1} \mathcal{F}^{-1}\, \chi\,\mathcal{E}^{-1}g\right)(\t), 
$ 
where $\chi^*: L^2(\RR^n ; (\mathcal{H},d_0)) \rightarrow L^2(\Theta ; (\mathcal{H},d_0)) $  
is the operator of restriction from $\RR^n$ to $\Theta$ which is the adjoint of $\chi$. 
As a result, Theorem \ref{p.unitaryequiv} implies the following.
\begin{theorem}\label{thm.bivariate} 
Assume \eqref{KA}--\eqref{H6}. Then, for $0<\ep <1$, one has
	\[
 d_\t^{1/2}\Big[\,\mathcal{L}_{\ep,\t}^{-1}\, g(\t) \,\,-\,\, \big(A_\ep^* \mathcal{L}^{-1}\mathcal{P} A_\ep g\big)(\t)\Big] 
\,\, \le\,\,  C_{11}\,\ep\,  d_\t^{1/2}\big[g(\t)\big] ,	\quad \forall g \in L^2\big(\Theta; (\mathcal{H},d_\t)\big), \quad 
for \,\,a.e.\ \t \in \Theta,
	\]
where  ``connecting operator'' $A_\ep : L^2\big(\Theta ; (\mathcal{H},d_\t)\big) \rightarrow L^2\big(\RR^n ; (\mathcal{H},d_0)\big) $ is the composition 
$A_\ep : = \Gamma_\ep^{-1}\mathcal{F}^{-1} \, \chi\,\mathcal{E}^{-1} $, 
and 
$A_\ep^* : L^2\big(\RR^n ; (\mathcal{H},d_0)\big) \rightarrow L^2\big(\Theta ; (\mathcal{H},d_\t)\big) $ is its adjoint given by  
$A_\ep^* : = \mathcal{E}\, \chi^*\,  \mathcal{F}\Gamma_\ep$. 
Furthermore, 
$A_\ep$ is an $L^2$-isometry and 
the following identities hold:
\begin{equation}
\label{abident}
A_\ep^* A_\ep  = I \quad \text{and } \quad A_\ep A_\ep^* = \Gamma_\ep^{-1}\mathcal{F}^{-1}  \chi_\Theta \mathcal{F}\Gamma_\ep = 
\mathcal{F}^{-1}\Gamma_\ep  \chi_\Theta \Gamma_\ep^{-1}\mathcal{F},
\end{equation} 
where $\chi_\Theta$ is operator of multiplication by the characteristic function of $\Theta$ in 
$L^2\big(\RR^n ; (\mathcal{H},d_0)\big)$. 
\end{theorem}
(Identities \eqref{abident} immediately follow via obvious relations $\chi^*\chi=I$ and 
$\chi\chi^*=\chi_\Theta$.) 
Notice that, by \eqref{abident}, $\,A_\ep A_\ep^*$ is the operator of projection onto 
$
{\ep^{-1}\Theta}$ in the Fourier space. 
Notice also that all the above implies that the approximating operator $A_\ep^* \mathcal{L}^{-1}\mathcal{P} A_\ep$ is self-adjoint in 
$L^2\big(\Theta ; (\mathcal{H},d_\t)\big)$, and is in 
the form of a direct integral along its fibers as given for a.e. $\t\in\Theta$ by Theorem \ref{p.unitaryequiv}. 
\vspace{.08in}

Turning now to approximation of the collective spectrum of $\mathcal{L}_{\ep,\t}$ in terms of that of the 
bivariate operator $\mathcal{L}$, we first observe that because of the unitary equivalence of $\mathcal{L}$ and 
$\mathbb{L}$ (Lemma \ref{lemft}) the spectra of the latter two coincide. 
On the other hand, recall that $\mathbb{L}$ is the direct integral of  $\mathbb{L}_\xi$, $\xi\in\RR^n$, 
and that the eigenvalues $\lambda_k(\xi)$ of $\mathbb{L}_\xi$ 
continuously depend on $\xi$ (cf. the proof of Theorem \ref{limspecsimple}). 
Then, by e.g. Theorem XIII.85 (d) of \cite{ReeSim}, the spectrum of $\mathbb{L}$ is the closure 
of the union of the spectra of $\mathbb{L}_\xi$. As a result, 
$
{\rm Sp}\, \mathcal{L} = \overline{\bigcup_{\xi \in \RR^n} {\rm Sp} \, \mathbb{L}_\xi}
$
and Theorem \ref{limspecsimple}, Theorem \ref{thm.limspecrep} and Corollary \ref{c.collspec} give together the following result.
\begin{theorem}
\label{bivariate.spec}
The spectrum of the bivariate operator $\mathcal{L}$ is described by \eqref{valthm}--\eqref{valthm2}. 
	One also has 
	\be
	\label{spLbeta}
	{\rm Sp}\, \mathcal{L} \,\,=\,\, \Big\{ \lambda \notin {\rm Sp}\, \mathbf{B}_\star :\,\, \beta_{\lambda }[z] \ge 0 \text{ for some $0\neq z \in Z$} \Big\} \cup {\rm Sp}\, \mathbf{B}_\star.
	\ee
Furthermore, for every interval $[a,b] \subset (-\infty,\infty)$ one has 
\be
\label{specestgen}
{\rm dist}_{[a,b]} \Big(\, \overline{\bigcup_{\theta \in \Theta} {\rm Sp}\, \mathcal{L}_{\ep,\t}}  \,,\,\,{\rm Sp}\, \mathcal{L}\Big) \,\,\le\,\, C_b\,\ep, \ \ \ \forall\, 0<\ep<1, 
\ee
with $C_b$ as given in Corollary \ref{c.collspec}. In particular, if $(a,b)$ is a gap in the spectrum of $\mathcal{L}$, i.e. $(a,b) \cap {\rm Sp}\, \mathcal{L} = \emptyset$ then $[a +C_b \ep,b-C_b\ep]$ is in a gap of the collective spectrum $\overline{\bigcup_{\theta \in \Theta} {\rm Sp}\, \mathcal{L}_{\ep,\t}}$ 
when $\ep < (b-a)/(2C_b)$.
\end{theorem}
\begin{remark}
Under additional assumptions on the regularity of $b_\t$ 
in $\t$, one can identify $\mathcal{L}_{\ep,\t}$ and  $\overline{\bigcup_{\theta \in \Theta} {\rm Sp}\, \mathcal{L}_{\ep,\t} }$ respectively as the fibres and  spectrum of a decomposable self-adjoint operator 
$\mathcal{L}_\ep=\int_\Theta^\oplus \mathcal{L}_{\ep,\t} \, {\rm d}\theta$ 
acting in the space $L^2\left(\Theta; \left(\mathcal{H},d_\t
\right)\right)$, see the examples section below.
\end{remark}
We end this section by discussing the possibility of gaps in the  spectrum ${\rm Sp}\, \mathcal{L}$. We know 
from Theorem \ref{bivariate.spec} that an interval $I$ is in a gap of ${\rm Sp}\, \mathcal{L}$ if, and only if, ${\rm Sp}\, \mathbf{B}_\star \cap I = \emptyset$ and for every $\lambda \in I$ the form $\beta_\lambda$ is negative-definite on $Z$. When $Z$ is one-dimensional it is straightforward to verify the existence of such intervals.  Indeed, for such $Z$ and $\lambda\notin{\rm Sp}\, \mathbf{B}_\star$, one has via \eqref{betadef} and \eqref{vfromz} 
 the representation 
\begin{equation} \label{iksign}
\beta_\lambda(z,\tilde{z})\,=\,-\,\, b_0(z,\tilde{z}) \,+\,\lambda\, d_0(z,\tilde{z}) \,+\,(\lambda-1)^{2}\sum_{k} 
\,\frac{d_0(z,\varphi^{(k)})d_0(\varphi^{(k)},\tilde{z})}{\lambda^{(k)}_\star-\lambda}\,, \quad \forall \,z,\tilde{z} \in {Z},
\end{equation}
where 
the eigenvectors $ \varphi^{(k)}$ corresponding to the eigenvalues 
$ \lambda^{(k)}_\star$ 
of $\mathbf{B}_\star$ form 
an orthonormal basis in $\left(\overline{V_\star},\, d_0\right)$. 
Now we can see that  if 
$\lambda^{(n)}_\star$ is an  
eigenvalue with 
$\varphi^{(n)}$ 
not orthogonal to $Z$ then, for $0\ne z\in Z$, $\beta_\lambda[z]$ is positive just to the left, and negative just to the right, of $\lambda^{(n)}_\star$. Consequently, there is some interval to the left (respectively the right) of $\lambda^{(n)}_\star$  in ${\rm Sp}\, \mathcal{L}$ (respectively the  gap).  In general, there maybe infinitely many such eigenvalues and thus there are infinitely many spectral gaps.
This situation occurs, for example, in the high-contrast 
problem studied by V. Zhikov in \cite{Zhi2000,Zhi2005}, wherein $\mathcal{L}^{-1}$ coincides with the homogenised two-scale limit operator resolvent  $(\mathcal{L}_0 + I)^{-1}$ and $\beta_{\lambda-1}$ coincides with the Zhikov $\beta$-function, see Example \ref{e.dp}. 

The situation is more complicated when $Z$ is not one-dimensional. Whilst \eqref{iksign} still holds and 
intervals 
just to the left of 
$\lambda^{(n)}_\star$ with 
$\varphi^{(n)}$ not orthogonal to $Z$ 
are in ${\rm Sp}\, \mathcal{L}$,  
as there will always exist $0\ne z\in Z$ such that $d_0\left(z,\varphi^{(n)}\right)= 0$ it is not necessarily the case that the right of this point is in the gap. 
There may even be no gaps. For example, if ${\rm dim}\, Z >1$, ${\rm dim}\, V_\star =1$  and say $b_0$ coincides with $d_0$ on $Z$, then one can always find a $0\neq z \in Z$ such that  $\beta_\lambda[z]\ge 0$ for all $\lambda \ge 1$; thus ${\rm Sp}\, \mathcal{L} = [1,\infty)$. 
\section{Examples}
\label{sec:examples}
Here we aim at demonstrating the power and versatility of our 
approach 
by applying it to a diverse set of problems. 
We provide examples of 
models that can be reformulated as problems of the type \eqref{p1} and satisfy (some of)  the main assumptions \eqref{KA}--\eqref{H6}. We 
begin 
with the classical and 
high-contrast (two-scale) homogenisation problems, where the 
approach already allows us to obtain some 
substantially 
new results. 
Then we 
study (and obtain more of new results for) various other physically-relevant models, 
each chosen to showcase the relevance and utility of the  main abstract assumptions and results.
\subsection{Uniformly elliptic PDEs with rapidly oscillating periodic coefficients}
\label{e.class}
We begin 
with the classical homogenisation problem for elliptic PDEs 
with 
periodic coefficients. 
Apart from being a natural warm-up for seeing how our general scheme 
works, this example provides us with certain constructions 
useful in some subsequent examples. 
We consider scalar PDEs  
but 
the 
exposition readily extends 
for systems. 
For our purposes in the present example, we restrict 
the application of the above developed general scheme up to 
Section \ref{sec.2dif}, based on assumption \eqref{KA}--\eqref{H4}. 


Consider the following resolvent problem in the whole of $\mathbb{R}^n$: 
\begin{equation}
\label{differentialPDE}
\left\{ \begin{aligned}
& \text{Find $u_\ep \in H^1(\RR^n)$ such that} \\
&-\,\nabla\cdot\Big( A \left(\tfrac{x}{\ep} \right)\nabla u_\ep (x) \Big)\,\,+\,\,  u_\ep (x) \,\,=\,\, F(x), 
\qquad  \text{for \,\,a.e.}\  x \in \RR^n,
\end{aligned} \right.
\end{equation}
for a given $0<\ep<1$, $\,F \in L^2(\RR^n)$, and measurable possibly complex-valued $n\times n$ matrix $A$ that satisfies the following standard conditions of Hermitian symmetry, uniform ellipticity and boundedness: 
\be
\begin{aligned} \label{IKcond}
 \quad A = A^*:=\overline{A^T}, &   \quad  & \gamma_0^{-1} |\eta|^2\,\le\, A(y)\eta \cdot \overline
{\eta}\,\le\, \gamma_0 |\eta|^2 \quad \text{a.e. } y \in \mathbb{R}^n, \, \forall \eta \in \CC^{n}, \ \text{for some constant $\gamma_0 \ge 1$.}
\end{aligned}
\ee
We assume that $A(y)$ is periodic with period $1$ with respect to each variable  $y_j$, $j=1,2,...,n$, 
that is $\Box=[-\frac{1}{2},\frac{1}{2}]^n$ is  the periodicity cell and $\Box^*=[-\pi,\pi]^n$ is the associated 
Bloch-dual cell (the Brillouin zone). Our goal is to construct approximations of $u_\ep $ with ``operator-type'' error bounds (in $L^2(\RR^n)$ and $H^1(\RR^n)$ 
norms) of order $\ep$, i.e. those that linearly depend on $\| F \|_{L^2(\RR^n)}$.  

In this and several 
subsequent examples a key role will be played by Floquet-Bloch-Gelfand transform,  
 which reduces 
a problem like \eqref{differentialPDE} to an equivalent one of the type \eqref{p1} 
 via a decomposition into quasi-periodic functions. 
Namely, 
see e.g. \cite{Gel,Ku,ZhSpectr}, 
Floquet-Bloch-Gelfand transform (which we call for short Gelfand or 
Floquet-Bloch transform) $U :L^2(\RR^n) \rightarrow L^2(\square^* \times \square)$ and its inverse are unitary operators 
that can be defined, for example, as the continuous extensions of the ($L^2$-isometric) mappings
\be
\label{gt1}
UF(\theta,y) \,: =\, (2\pi)^{-n/2}\sum_{m\, \in\, \ZZ^n} F(y+m) e^{-\i\, \theta \cdot (y+m)}, \qquad F \in C^\infty_0(\RR^n),
\ee
\be
\label{gt2}
U^{-1} g (x) \,=\,(2\pi)^{-n/2} \int_{\square^*} g(\t ,\,\{x\}   )\, e^{\i\, \t \cdot x}\, {\rm d}\t , \qquad 
g \in C \big( \square^* ;\, C_{per} (\square)\big). 
\ee
($C_{per} (\square)$ is here the space of functions on $\square$ which allow a $\square$-periodic continuous extension on $\mathbb{R}^n$; $\{x\}\in\square$ denotes a ``fractional part'' of $x\in\mathbb{R}^n$: 
e.g. $\{x\}:=x-m$ for the unique $m\in\mathbb{Z}^n$ such that $x-m\in [-1/2,1/2)^n\subset\square$.)

Fix $0<\ep <1$ and  
apply to \eqref{differentialPDE} 
first 
the normalised (unitary)  rescaling operator $\Gamma_\ep : L^2(\RR^n) \longrightarrow L^2(\RR^n)$,  
$(\Gamma_\ep F)(y) :=  \ep^{n/2}F(\ep y)$, and then the 
Gelfand transform $U$. 
Then, via the properties of Gelfand transform, 
for a.e. $\theta \in \Box^*$,  
$u_{\ep,\theta}(\cdot) := U\Gamma_\ep u_\ep(\theta, \cdot)$ belongs to $H^1_{per}(\Box)$ the 
Hilbert space of functions in $H^1(\Box)$ admitting a locally-$H^1$  $\square$-periodic extension on $\mathbb{R}^n$, 
 and solves
\begin{equation}
\label{differentialPDE1}
-\,e^{-\i\, \theta \cdot y}\,\ep^{-2}\,\nabla\cdot\Big( A \left( y\right)\nabla\big( e^{\i \theta \cdot y} u_{\ep,\theta}(y)\,\big)\, \Big) \,\,+\,\, u_{\ep,\theta}(y) \,\,=\,\, F_{\ep,\theta}(y), \qquad  \text{a.e.}\  y \in \Box,
\end{equation}
where $F_{\ep,\theta}(\cdot) : = U\Gamma_\ep F(\theta, \cdot) \in L^2(\Box)$. The standard 
equivalent weak formulation of \eqref{differentialPDE1} is: 
\begin{equation}
\label{difformIK}
\left\{ \begin{aligned}
& \text{ For a.e. $\t\in\square^*$, find $u_{\ep,\theta} \in H^1_{per}(\Box)$ the solution to} \\
&\ep^{-2} \int_\Box A(\nabla+\i\t)u_{\ep,\theta}\cdot \overline{(\nabla+\i\t)\tilde{u}} \,\,+\,\int_\Box u_{\ep,\theta}\,\overline{\tilde{u}} \,\,=\,\, \int_\Box F_{\ep,\theta}\,\overline{\tilde{u}}\,, \qquad \forall \tilde{u} \in H^1_{per}(\Box).
\end{aligned} \right.
\end{equation}
 Problem \eqref{difformIK}  is of the type \eqref{p1}\footnote{Strictly speaking, given $F\in L^2\left(\RR^n\right)$, 
\eqref{difformIK} is for {\it almost every}  
$\t\in\square^*$ while \eqref{p1} is {\it for all} $\t\in\Theta$. 
There is no issue here, as the forms $a_\t$ and $b_\t$ are well-defined in \eqref{difform} {\it for all} $\t$. 
It is also not to matter for our purposes that we regard the periodicity and dual cell's $\square$ and $\square^*$ 
as closed cubes rather than tori.}. 
Indeed, choosing Hilbert space $H=H^1_{per}(\Box)$,  it can be restated as: 
\begin{equation}
\label{difform}
\mbox{Find } u_{\ep,\t}\in H\,\,\mbox{ such that }
\ep^{-2}\, a_\t\left(u_{\ep,\theta}\,,\tilde{u}\right) \,+\, b_\t\left(u_{\ep,\theta}\,,\tilde{u}\right)
\,\, =\,\,\l f,\, \tilde{u}\r, \qquad \forall \tilde{u} \in H, \,\, a.e.\, \t \in \Theta,
\end{equation}
for  $\Theta=\Box^*$,  
$\,\,\left\l f, \tilde u \right\r \,: =\,   \int_\Box F_{\ep,\theta}\, \overline{\tilde u}\,$, 
\begin{equation}
\label{cforma}
\begin{aligned}
a_\t(u,\tilde{u}) \,:=\, \int_\Box A(\nabla+\i\t)u\cdot \overline{(\nabla+\i\t)\tilde{u}}\,,
 \qquad \text{and} \qquad  b_\t(u,\tilde{u})\,:=\, \int_\Box u\,\,\overline{\tilde{u}}\,
,  \qquad u,\tilde{u} \in H^1_{per}(\Box).
\end{aligned}
\end{equation}
Recall that, according to \eqref{astructure}, for the inner products in $H$, 
$(u,\tilde u)_\t:=a_\t (u,\tilde u)+b_\t (u,\tilde u)$. 
Assumption \eqref{as.b1} then easily follows: 
for $u\in H$, 
via assumptions \eqref{IKcond} on the coefficients $A$, 
\begin{equation}
\label{2.1classhom}
a_{\t_1}[u]\le\gamma_0\left\|\left(\nabla+\i\t_1\right)u\right\|^2_{L^2(\Box)}\le
2\gamma_0\left\|\left(\nabla+\i\t_2\right)u\right\|^2_{L^2(\Box)}+
2\gamma_0\left|\t_1-\t_2\right|^2\|u\|^2_{L^2(\Box)}\le 
2\gamma_0^2a_{\t_2}[u]+8\pi^2 n\gamma_0\|u\|^2_{L^2(\Box)}, 
\end{equation}
and 
\eqref{as.b1} holds e.g. with $K=\left(2\gamma_0^2+8\pi^2n\gamma_0+1\right)^{1/2}$. 
 It is also straightforward to show 
\eqref{ass.alip}. 
%

Let us next determine the spaces $V_\theta$ and $W_\theta$, see \eqref{spaceV} and \eqref{2.6-w}. 
Notice that for the form $a_\t$ 
\be\label{IKaest}
a_\t[u] \,\ge\, \gamma_0^{-1}\int_\Box \big| (\nabla+\i\t) u\big|^2\,\,\ge\,\, \gamma_0^{-1}|\t|^2 \int_\Box | u|^2, \qquad \forall u \in H^1_{per}(\Box),\ \forall \t  \in \Box^*. 
\ee
(The last inequality follows 
via expansion into the Fourier series on $\Box$: 
for $u\in H$, 
$u(y)=\sum_{m\in\ZZ^n}c(m)e^{2\pi\i\, m\cdot y}$, and $|2\pi m +\t|\ge|\t|$, $\forall\t\in\Theta$, 
$m\in\ZZ^n$.) Therefore, from \eqref{spaceV} and \eqref{2.6-w} via \eqref{IKaest} and \eqref{cforma},  
\be
\label{vwclasshom}
\begin{aligned}
V_\theta = \left\{
\begin{array}{lr}
\{ 0 \}, & \theta \neq 0, \\[5pt]
{\rm Span} (\mathbf{e}), & \theta = 0,
\end{array} 
\right. & \qquad & W_\theta = \left\{
\begin{array}{lr}
H^1_{per}(\Box), & \theta \neq 0, \\[5pt]
H^1_{per, 0} \,: =\,\big\{ w\in H^1_{per}(\Box) \,\, \big| \,\, \int_\Box w = 0 \big\}, & \theta = 0,  
\end{array} 
\right. 
\end{aligned}
\ee
where $\mathbf{e}=1$ is the constant unity function on $\Box$. 
We 
see that $V_\t$ is discontinuous with respect to $\t$  
at the origin,  so 
 we are in the context of Sections \ref{section:discV} and \ref{sec.2dif}. Now  let us demonstrate that the 
related main assumptions \eqref{KA}--\eqref{H4} 
hold. 

\textbullet\, Hypothesis 
\eqref{KA} follows from 
noticing  that the stronger assertion \eqref{KA2.1} (see Proposition \ref{prop.kaequiv}) trivially holds with $C=1$ upon choosing  ${c}= b_\t$ (notice 
that $b_\t$ is here 
$\t$-independent, see \eqref{cforma}, and $b_\t$ is 
$\|\cdot\|_\t$-compact by the Rellich compactness theorem).

\textbullet \, Assumption  \eqref{contVs} holds trivially for $V_\star = \{ 0\}$ with $L_\star = 0$;  see  Remark \ref{constV}.

\textbullet\, Let us show that hypothesis \eqref{distance} holds with $\gamma = \left(n\pi^2+\gamma_0\right)^{-1}$. 
Indeed, for $\t\ne 0$ from \eqref{IKaest}, 
$ 
\|u\|_\t^2
\,=\,a_\t[u] \,+\,b_\t[u]\,
\le\, \big(1+ \gamma_0|\t|^{-2}\big)\,a_\t[u]\,, 
$ 
and consequently (noticing that $|\t|^2\le n\pi^2$ )
\[\nu_\theta\,\,:=\, \inf_{w \in W_\t \backslash \{ 0 \}}\, \frac{a_\t[w]}{\|w\|_\t^2}\,\,\ge\,\, 
\big(1+ \gamma_0|\t|^{-2}\big)^{-1}\,\,=\,|\t|^2\left(|\t|^2+\gamma_0\right)^{-1}\,\,\ge\,\,|\t|^2\left(n\pi^2+\gamma_0\right)^{-1}.  
\]
Notice also that, for any $r>0$, Theorem \ref{thm:contV} (with $\Theta=\Box^*$ replaced by 
$ \Theta \cap \{ |\t | \ge r \}$, cf. Remark \ref{rem3.2}) implies 
(as $V_\t = \{ 0 \}$ for $\t \neq 0$; 
and from \eqref{fstar} and \eqref{cforma} 
we have $\|f\|_{*\t}\le \big\|U\Gamma_\ep F(\theta, \cdot) \big\|_{L^2(\Box)}$) 
\begin{equation}\label{coutside}
\ep^{-2} \gamma_0^{-1} \big\| (\nabla + \i \t) u_{\ep,\t} \big\|^2_{L^2(\Box)} \,+\, 
\big\| u_{\ep,\t}\big\|^2_{L^2(\Box)} \,\le\, \ep^2 \left(n \pi^2 + \gamma_0\right) |r|^{-2} 
\big\| U\Gamma_\ep F(\theta, \cdot)\big\|^2_{L^2(\Box)}, \quad \mbox{a.e. }\,  \t \in \Box^*,\, |\t| \ge |r|.
\end{equation}

\textbullet \, Assumption \eqref{H4} is obviously satisfied with $K_{a'} = \gamma_0$, $K_{a''}=0$ and 
\begin{equation}
\label{IKa'}
\begin{aligned}
& a'_{0}(v,  u) \cdot \t: =\i\int_\square   A\, \t v \cdot \overline{ \nabla  u} 
,\qquad a''_{0}(v,  \tilde{v})\t \cdot \t:= \int_\square A\,\t\cdot\t\, v\, \overline{ \tilde{v}}.
\end{aligned}
\end{equation}


As \eqref{KA}--\eqref{H4} hold we can apply our general theory and, in particular, 
Theorem \ref{thm.maindiscthm}. 
Let us 
specify the objects appearing therein. In the present setting $V^\star_\t = \{ 0 \}=:V_\star$ for all 
$\theta\in\Box$,  and 
the space $Z$ in \eqref{spaceZ} is simply the one-dimensional space  $V_0 = {\rm Span} (\mathbf{e})$. 
Therefore, in the notation of Theorem \ref{thm.maindiscthm}, $v^h =0$,  $z^h=z_{\ep,\t}\textbf{e}$, where $z_{\ep,\t}\in \CC$,  and  \eqref{z3prob} becomes  a simple 
algebraic equation
\begin{equation}
	\label{IKz3prob}
	\ep^{-2}\, a^{\rm h}_{\t}(\textbf{e}, \textbf{e})z_{\ep,\t} \,\,+\,\,
	b_\t(\textbf{e}, \textbf{e})\, z_{\ep,\t} \,\,=\,\, \l f, \textbf{e} \r\,.
	\end{equation}
Let us rewrite the coefficients of this equation in more explicit terms. Clearly, $\langle f, \mathbf{e} 
\rangle = \int_\square U\Gamma_\ep F(\theta, y)  \, {\rm d}y$ and $b_\t(\textbf{e}, \textbf{e})=1$. 
As for the first coefficient, recalling \eqref{defhom.form} and \eqref{cell:prob2}:
 \begin{equation}
\label{ahom7.1}
a^{\rm h}_{\t}(\textbf{e}, \textbf{e})\,\,=\,\,a''_0(\textbf{e},\textbf{e})\, \t \cdot \t \,-\, 
a_0( N_\t\textbf{e} , N_\t \textbf{e})\,\,
 =\,\,a''_0(\textbf{e},\textbf{e})\, \t \cdot \t\,+\,a_0'\left( \textbf{e},\,N_\t\textbf{e}\right)\cdot\t,
\end{equation} 
where $N_\t\textbf{e}\in H^1_{per,0}(\Box)$ solves (via \eqref{cell:prob2} and \eqref{IKa'}, and recalling 
$A=A^*$)
\begin{equation}
\label{cell:IKprob22}
a_0( N_\t\textbf{e} , w) \,=\, -\,\,
a_0'\left( \textbf{e},\,w\right)\cdot\t\,=\,-\,
\i\,\t\cdot \int_\Box \overline{\,A\nabla w}\,, \qquad \forall w \in H^1_{per,0}(\Box), \,\,\, \forall \t \in \RR^n.
\end{equation}
As a result, $N_\t\textbf{e}=\i\,\t\cdot\ourN$ where $\ourN\in H^1_{per,0}\left(\Box;\,\CC^n \right)$ solves 
\begin{equation}
\label{IKclasscor1}
\int_{\Box} A\Big( \nabla (\t\cdot \ourN) + \t \Big) \cdot \overline{\nabla w} \,=\, 0\,, \qquad \forall w \in H^1_{per,0}(\Box),\ \, \forall \t \in \RR^n.
\end{equation}
Thus $\ourN$ is the classical corrector, see e.g. \cite{JKO}.
\begin{remark} If we introduce the components of $\ourN$, $\ourN=( \ourN^1, \ldots ,  \ourN^n)^T$, then \eqref{IKclasscor1} can be equivalently rewritten in a more traditional form: denoting $e^1,\ldots,e^n$ is the canonical basis in $\mathbb{R}^n$, 
\begin{equation}
\label{IKclasscor}
\left\{ \begin{aligned}
& \text{  For $j=1,\ldots, n$, find $\ourN^{j} \in H^1_{per,0}(\Box)$ such that} \\
&\int_{\Box} A\bigl( \nabla \ourN^{j} + e^j \bigr) \cdot \overline{\nabla w }= 0 , \qquad \forall w \in H^1_{per,0}(\Box). 
\end{aligned} \right.
\end{equation}
\end{remark}
Now we express  $a^{\rm h}_{\t}(\textbf{e}, \textbf{e})$ in terms of $\ourN$: via \eqref{ahom7.1} and \eqref{IKa'} and 
using $\mbox{Im}\left(a^{\rm h}_{\t}(\textbf{e}, \textbf{e})\right)=0$, 
\[
 a^{\rm h}_{\t}(\textbf{e}, \textbf{e})\,=\, a''_0(\textbf{e},\textbf{e})\, \t \cdot \t\,+\,
a_0'(\textbf{e}, N_\t\textbf{e})\cdot\t\,=\, 
\int_\Box  A \,  \theta\cdot   \theta \,+\,\i\int_\Box  A\,\t\cdot\overline{\nabla N_\t\textbf{e}}\,=\,
\int_\Box  A \,  \theta\cdot   \theta \,+\int_\Box  A\,\nabla (\theta \cdot \ourN)\cdot   \theta. 
\] 
Thus we can represent $a^{\rm h}_{\t}(\textbf{e}, \textbf{e})$ as 
\be \label{IKahom55} a^{\rm h}_{\t}(\textbf{e}, \textbf{e})\,\,=\,\,  A^{\rm hom} \theta \cdot \theta,
\ee
 where $A^{\rm hom}$ is the classical homogenised matrix with components
\be\label{IKcoef}
 A^{\rm hom}_{ij}\,\, :=\,\, \int_\Box\Big( A_{ij}+\sum_{k=1}^n A_{ik}  \partial_{k}\ourN^{j}\Big)\,, \qquad  i,j \in \{1,\ldots,n\}.
 \ee
Matrix $A^{\rm hom}$ is well-known to be positive definite (which can also directly be seen from \eqref{IKahom55}, \eqref{ahcoercive}) and Hermitian (which can be checked using \eqref{IKcoef} and \eqref{IKclasscor}). 

Putting all of this together, we conclude from \eqref{IKz3prob}
\begin{equation}\label{classicalzsol}
z_{\ep,\t} 
 \,\,=\,\, \frac{\int_\square U\Gamma_\ep F(\theta, y)  \, {\rm d}y}{\ep^{-2}A^{\rm hom}\t\cdot\t \,\,+\, 1}\,.
\end{equation}
Now from Theorem \ref{thm.maindiscthm} and \eqref{coutside}, 
we readily deduce the following result 
(abbreviating $z_{\ep,\t}\,\mathbf{e}$ to $z_{\ep,\t}$). 
\begin{proposition}
\label{thm.IKeclassexample} 
Let $u_{\ep,\theta}\in H^1_{per}(\Box)$ solve \eqref{differentialPDE1} and   $z_{\ep,\t}$ be given by 
\eqref{classicalzsol}.  
Then for some $r_1>0$ 
and $\chi$ the characteristic function for the ball of radius $r_1$, 
\begin{gather}
	\label{eq.IKmaindiscthm}
\begin{aligned}
\ep^{-2}\Bigl\| (\nabla+\i \t) \Bigl( u_{\ep,\theta} \, -\, \chi(\t) \left(1 + \i\,\t \cdot \ourN\right) z_{\ep,\t}\Bigr)\Bigr\|^2_{L^2(\Box)}  \,\, +\,\,	\bigl\| u_{\ep,\theta}  \,-\,\chi(\t) \left(1 + \i\,\t \cdot \ourN\right) 	z_{\ep,\t}\bigr\|^2_{L^2(\Box)} \hspace{1.7 cm}\\  \le\,\,  \ep^2\,  c_0^2\,\bigl\|U\Gamma_\ep F(\theta, \cdot) \bigr\|^2_{L^2(\Box)}, 
\end{aligned} \\
\label{IKL2}
\big\| u_{\ep,\theta} \, -\,   \chi(\t)	z_{\ep,\t} \big\|_{L^2(\Box)} \,\le\,\,  \ep\, c_1\,\big\| U\Gamma_\ep F(\theta, \cdot) \big\|_{L^2(\Box)},
\end{gather}
for some positive constants $c_0, c_1$ independent of $\ep$, $\t$ and $F$. 
\end{proposition}
We now show that, via the inverse Gelfand and scaling transforms, 
\eqref{eq.IKmaindiscthm} and \eqref{IKL2} 
provide respectively the desired  $H^1$ and $L^2$ 
estimates for certain approximations of  $u_\ep$, the solution to \eqref{differentialPDE}. 
To this end, recall that according to \eqref{IKL2} $\chi\, z_{\ep,\theta}$ serves as an approximation to the transformation 
$u_{\ep,\t}$ of the original solution $u_\ep$, 
where $u_\ep=\Gamma^{-1}_\ep U^{-1}u_{\ep,\t}$. 
So set the (inverse) transformed approximation 
\be
\label{IKappr45} 
u_\ep^{(0)} \,\,:=\,\, \Gamma^{-1}_\ep U^{-1} \,\chi\, z_{\ep,\theta}.   
\ee
As $\chi(\t)z_{\ep,\t}$ does not depend on $y$, by \eqref{gt2} and \eqref{classicalzsol}, 
$u_\ep^{(0)}\in C^\infty(\RR^n)\cap H^1(\RR^n)$ and $\nabla u_\ep^{(0)}\in L^\infty(\RR^n)$. 
Next, for the ``corrector'' term in \eqref{eq.IKmaindiscthm}, by the properties of the Gelfand transform, cf. \eqref{gt2}, 
\begin{eqnarray}
\label{7.19-2}
\Gamma^{-1}_\ep U^{-1}\chi\, \i\,\t \cdot \ourN  	\, z_{\ep,\t}\,\,=\,\, 
\Gamma^{-1}_\ep\left(\ourN\cdot U^{-1}\,\i\,\t\,\chi\, z_{\ep,\t}\right)\,\,=\,\,
\Gamma^{-1}_\ep\left(\ourN\cdot \nabla\,U^{-1}\,\chi\, z_{\ep,\t}\right)\,\,= \nonumber \\ 
\,\, \ \ \ \ \ \ 
 \,\,(\tilde\Gamma^{-1}_\ep\ourN)\cdot   \left(\Gamma^{-1}_\ep \nabla U^{-1}\chi 	z_{\ep,\t}\right)
\,\,=\,\, 
\ep\,(\tilde\Gamma^{-1}_\ep\ourN)\cdot  \nabla \left(\Gamma^{-1}_\ep U^{-1}\chi 	z_{\ep,\t}\right)
\,=\, \ep\left(\tilde\Gamma^{-1}_\ep\ourN\right)\cdot  \nabla u^{(0)}_\ep,
\end{eqnarray} 
where $\left(\tilde\Gamma_\ep^{-1} f\right)(x):=f(x/\ep)$ denotes ``ordinary'' rescaling. 
[In \eqref{7.19-2} we have used sequentially that: $\ourN$ does not depend on $\t$ and 
$U^{-1}\big(f(y)g\big)=f(y)U^{-1}g$,  
$\,U^{-1}\big(i\t\,f(\t)\big)=\nabla\left(U^{-1}f\right)$, 
$\,\Gamma_\ep^{-1}(fg)=\left(\tilde\Gamma_\ep^{-1}f\right)\Gamma_\ep^{-1}g$, 
and $\Gamma_\ep^{-1}(\nabla f)=\ep\nabla\left(\Gamma_\ep^{-1}f\right)$.]
As a result \eqref{eq.IKmaindiscthm} and \eqref{IKL2}, upon application of the $L^2$-unitary inverse 
Gelfand transform $U^{-1}$ and inverse rescaling $\Gamma_\ep^{-1}$ 
(and noticing that 
$\Gamma_\ep^{-1}U^{-1}\big(\,(\nabla+i\t)\,f\big)=\ep\,\nabla\left(\Gamma_\ep^{-1}U^{-1}f\right)\,$), 
lead to the following.
\begin{theorem}
\label{IK77}
Let $u_\ep$ solve \eqref{differentialPDE} and $u_\ep^{(0)}$
be given by  \eqref{IKappr45} with $z_{\ep,\t}$ specified by \eqref{classicalzsol}. Then 

\begin{gather}
\label{IKH1est}
\Bigl\Vert u_\ep \,-\, \Bigl(u_\ep^{(0)}+ \ep\big(\tilde\Gamma^{-1}_\ep\ourN\big)\cdot  \nabla u^{(0)}_\ep\Bigr)\,  \Bigr\Vert_{H^1(\mathbb{R}^n)} \,\le\,\, \ep\, c_0\, \Vert F \Vert_{L^2(\mathbb{R}^n)}, \\
\label{IKL2est}
\big\Vert u_\ep - u_\ep^{(0)}  \big\Vert_{L^2(\mathbb{R}^n)} \,\,\le\,\, \ep\, c_1\,\Vert F \Vert_{L^2(\mathbb{R}^n)}.
\end{gather}
\end{theorem}
The above theorem already provides constructive approximations of the solution to \eqref{differentialPDE}, however it is customary to relate these 
to the solution of the 
corresponding 
{\it homogenised equation}. We now provide this link. 
For the homogenised differential operator applied to $u^{(0)}_\ep$, 
by the standard properties of the scaling and Gelfand transformations 
together with  \eqref{IKappr45} and the fact that $\chi z_{\ep,\t}$ is independent of $y$,
\[
-\,\nabla\cdot\bigl( A^{\rm hom} \nabla u^{(0)}_\ep\bigr) \,=\, \Gamma_\ep^{-1} U^{-1} 
\Big(\,-\, \ep^{-2} (\nabla + \i\, \t) \cdot  A^{\rm hom} (\nabla + \i\, \t) \Big)U  \Gamma_\ep u^{(0)}_\ep\,=\,
\Gamma_\ep^{-1} U^{-1} 
\left(\ep^{-2}\t \cdot  A^{\rm hom} \t \right)\chi z_{\ep,\t}. 
\]
%
%
This together with 
\eqref{classicalzsol} 
implies that $u^{(0)}_\ep$ solves 
$ 
-\,\nabla\cdot\big( A^{\rm hom} \nabla {u^{(0)}_\ep} \big) + u^{(0)}_\ep\,\,=\,\,  
\Gamma_\ep^{-1} U^{-1}  \chi  
\int_\square U\,\Gamma_\ep F(\theta, y)  \, {\rm d}y.
$ 
Now notice that the standard Fourier transform\footnote{The conventional Fourier transform in $L^2(\RR^n)$  
is here specified by
	\[
	\left(\mathcal{F}g\right)(\theta) : = (2\pi)^{-n/2}\int_{\mathbb{R}^n} e^{-\i \theta \cdot y}g(y)\,dy, \qquad g \in 
	L^2(\RR^n)\cap L^1(\RR^n).
	\]} $\mathcal{F}$ in $L^2\left(\RR^n\right)$ and the Gelfand transform, as 
	follows from \eqref{gt1} 
	and \eqref{gt2}, are related by: 
	$\int_\square U g(\theta, y)  \, {\rm d}y = \mathcal{F} g(\t)$, $g \in L^2(\RR^n)$, $\t \in \square^*$, and $U^{-1} (h \otimes \mathbf{e}) = \mathcal{F}^{-1} h$ for $h \in L^2(\RR^n)$ with support in $\square^*$.
Consequently, we determine that $u^{(0)}_\ep$ is the solution to 
\begin{equation}\label{Sep}
-\,\nabla\cdot\big( A^{\rm hom} \nabla {u^{(0)}_\ep} \big) \,+\, u^{(0)}_\ep\,=\, \mathcal{S}_\ep F,
\end{equation}
for the smoothing operator
 $\mathcal{S}_\ep : L^2(\RR^n) \rightarrow C^\infty(\RR^n) \cap H^1(\RR^n) \cap W^{1,\infty}(\RR^n)$ given by
\be
\label{7.22-2}
 \mathcal{S}_\ep F\,\,\, =\,\,\, 
 \Gamma^{-1}_\ep \mathcal{F}^{-1} \chi \mathcal{F}\Gamma_\ep \,F\,\, =\,\, 
\mathcal{F}^{-1} ( \tilde\Gamma_\ep \chi ) \mathcal{F}\,F.  
\ee 
(In the latter equality we have used that $\Gamma^{-1}_\ep \mathcal{F}^{-1}= \mathcal{F}^{-1}\Gamma_\ep$, 
$\mathcal{F}\Gamma_\ep=\Gamma_\ep^{-1}\mathcal{F}$ and 
$\Gamma_\ep \chi \Gamma_\ep^{-1}g=( \tilde\Gamma_\ep \chi )g\,$.) 
%
%
%
%
Let ${u}\in H^2(\RR^n)$ be the solution to the classical homogenised equation
\begin{equation}
\label{homeq}
-\nabla\cdot\big( A^{\rm hom} \nabla {u}(x) \big) \,+\, {u}(x) \,=\, F(x), \qquad x \in \RR^n. 
\end{equation}
Applying $\mathcal{S}_\ep$ to \eqref{homeq} and using the standard properties of 
Fourier transform, it is easy to see 
that $u^{(0)}_\ep$ solving \eqref{Sep} and ${u}$ 
are related by 
$u^{(0)}_\ep = \mathcal{S}_\ep u$. 
Further, let us show that one has the inequality
\be
\label{7.23-2}
\big\| u^{(0)}_\ep - u \big\|_{H^1(\RR^n)} \,\,\le\,\,\, \ep\, r^{-1}_1 \gamma_0\, \| F\|_{L^2(\RR^n)}.
\ee
Indeed, by the Plancherel identity, \eqref{7.22-2},  and \eqref{homeq}, 
\be
\label{7.23-3}
\big\| u^{(0)}_\ep - u \big\|_{H^1(\RR^n)} \,=\, 
\Big\| (1+|\xi|^2)^{1/2}\big(\tilde\Gamma_\ep \chi-1\big)\mathcal{F}u(\xi)\Big\|_{L^2(\RR^n)}\,=\, 
\Big\| \frac{(1+|\xi|^2)^{1/2}}{A^{\rm hom}\xi\cdot\xi+1}
\left(\tilde\Gamma_\ep \chi-1\right)\mathcal{F}F(\xi)\Big\|_{L^2(\RR^n)}. 
\ee
Noticing that $\left\vert\tilde\Gamma_\ep \chi-1\right\vert$ vanishes for $|\xi|<\ep^{-1}r_1$ and equals -1 otherwise, and recalling 
that \eqref{IKcond} implies  $A^{\rm hom}\xi\cdot\xi\ge \gamma_0^{-1}|\xi|^2$ (see e.g. \cite{JKO}) 
leads to \eqref{7.23-2}. 

Combining \eqref{7.23-2} with inequalities \eqref{IKH1est} and \eqref{IKL2est} provides the following result.
\begin{proposition}
\label{homeqest}
 Let $u_\ep$ solve \eqref{differentialPDE} and $u$ solve \eqref{homeq}. Then 
	\begin{gather}
	\label{IKH1est...}
	\Bigl\Vert\, u_\ep \,-\, \Bigl(u+ \ep\big(\tilde\Gamma^{-1}_\ep\ourN\big)\cdot  \nabla 
	\mathcal{S}_\ep u\Bigr)\,\Bigr\Vert_{H^1(\mathbb{R}^n)} \,\,\le\,\,\, \ep\, \bigl(c_0+ r_1^{-1} \gamma_0\bigr)\, \Vert F \Vert_{L^2(\mathbb{R}^n)}, \\
	\label{IKL2est...}
	\big\Vert u_\ep - u  \big\Vert_{L^2(\mathbb{R}^n)} \,\,\le\,\,\, \ep\, 
	\bigl(c_1+ r_1^{-1} \gamma_0\bigr)\,\Vert F \Vert_{L^2(\mathbb{R}^n)}.
	\end{gather}
\end{proposition}

The result 
\eqref{IKL2est...} was obtained 
in \cite{BiSu}, 
\cite{ZhL2}. The estimate of type 
\eqref{IKH1est...} was  obtained in \cite{ZhH1}, \cite{ZhPasH1} although with different smoothing operator and in \cite{BSh1} with the same smoothing operator $\mathcal{S}_\ep$. 
Finally note 
that the  operator and spectral 
results of Section \ref{s:resolv} are also 
applicable for this example, which we do not pursue here.  
In contrast, in the following high-contrast example 
the 
results of Section \ref{s:resolv} 
play key 
role. 

\subsection{High-contrast elliptic PDE with 
periodic coefficients}\label{e.dp}
Here we  
apply full powers of our method and  obtain as an outcome 
what we expect to be 
new results for 
 high-contrast  elliptic PDEs with periodic coefficients: two-scale $L^2$-resolvent estimates,  
$H^1$ energy estimates, and estimates on convergence of the spectra. 
Indeed the present example 
formed one of key motivations for the general approach developed in this work. 
We comment 
that the approach below could be extended to a wider class of 
PDE systems, cf. \cite{IVKVPS13}, in particular would be essentially the same for the analogous 
high-contrast problems of linear 
elasticity thereby impoving some results of \cite{ChKiVeZu23}\footnote{The only substantive difference  in the case of linear elasticity is that we would need to replace the extension Proposition \ref{prp.zhiext} with analogous extension property in linear elasticity, see Proposition \ref{prp.zhiextelast} below.}.
We assume here $n > 1$, 
and following the example of \cite{Zhi2005} focus on the simplest 
model, with resolvent problem: 
\begin{equation}
\label{dp.differentialPDEdp}
\left\{ \begin{aligned}
& \text{Find $u_\ep \in H^1(\RR^n)$ such that} \\
&-\,\nabla\cdot\bigl( A_\ep \left(\tfrac{x}{\ep} \right) \nabla u_\ep(x)\, \bigr) \,\,+\,\, u_\ep(x) \,\,=\,\, F(x), \qquad x \in \RR^n,
\end{aligned} \right.
\end{equation}
for a given $F \in L^2(\RR^n)$ and $\Box$-periodic coefficients $A_\ep$ of the form
\be
\label{aepdp}
A_\ep(y) = \left\{ 
\begin{array}{lr}
1  & y \in \Box \backslash B, \\[5pt]
\ep^2  & y \in B.
\end{array} \right.
\ee
The boundary of inclusion 
$B$ is assumed 
Lipschitz, 
$\overline{B} \subset \left(-\tfrac{1}{2},\tfrac{1}{2}\right)^n$ for simplicity, and hence the ``matrix'' (periodically extended $\square \backslash B$) is connected. 
We emphasise that our approach works without change for general measurable (i.e. with no regularity assumptions) 
complex Hermitian matrix-valued $A_\ep(y)=A_1(y)+\ep^2A_2(y)$ where 
$A_1$ and $A_2$ are 
supported in $\square\backslash B$ and $\overline{B}$ respectively and obey \eqref{IKcond}.

Following the steps in Example \ref{e.class}, 
for $0<\ep <1$ after 
application of the rescaling $\Gamma_\ep$ and the Gelfand transform ${U}$, we observe 
that $u_{\ep, \t} (\cdot) : = U \Gamma_\ep u_\ep(\theta, \cdot)\in H^1_{per}(\Box)$ for a.e. $\t \in \Theta : = \Box^*$, and solves 
\begin{equation}\label{dp.pde}
-\,e^{-\i \theta \cdot y}\ep^{-2}\nabla\cdot\Bigl( A_\ep \left( y\right)\nabla\left( e^{\i \theta \cdot y} u_{\ep,\theta}(y)\,\right)\,
\Bigr) \,\,+\,\, u_{\ep,\theta}(y) \,\,=\,\, U\Gamma_\ep F(\t,y), \qquad  \text{a.e.}\  y \in \Box. 
\end{equation}
This has the equivalent weak formulation
\begin{equation}
\label{dp.varp}
\begin{aligned}
\ep^{-2} \int_{\Box \backslash B} (\nabla + \i \t) u_{\ep,\t} \cdot\, \overline{(\nabla + \i \t) \tilde{u}} \,\,\,\,+\, 
\int_{B}  (\nabla + \i \t) u_{\ep,\t} \cdot\, \overline{(\nabla + \i \t) \tilde{u}} \,\,\,+ 
\int_{\Box} u_{\ep,\t}\, \overline{\tilde{u}} \,\,\, =\,\,\, \l f, \tilde{u}\r ,  \hspace{1.7cm}\\ 
\forall \tilde{u} \in H^1_{per}(\Box), \hspace{0.7cm}  \text{ where }
\end{aligned}
\end{equation}
\begin{equation}
\label{fdp}
\ \l f, \tilde{u}\r \,\,:=\,\, \int_\Box \,  U\Gamma_\ep F(\t,y)\,\, \overline{\tilde{u}(y)} \,\, {\rm d}y.
\end{equation} 
 We note \eqref{dp.varp}  is a problem of the form \eqref{p1} with $\ H :=H^1_{per}(\Box)$, $ \ \Theta : = \Box^*$,  
\begin{equation}
\label{abcdp}
\begin{aligned} 
&a_\t(u,\tilde{u}) := \int_{\Box \backslash B} (\nabla + \i \t) u \cdot \overline{(\nabla + \i \t) \tilde{u}}\,, \quad  \text{and} \quad     b_\t(u,\tilde{u}) : =  
\int_{B} (\nabla + \i \t) u \cdot \overline{(\nabla + \i \t) \tilde{u}}\,\,+ \int_\square u\,\,\overline{\tilde{u}}\,. 
\end{aligned}
\end{equation}
Then the same argument as in Example \ref{e.class}, cf. \eqref{2.1classhom}, assures 
that assumptions \eqref{as.b1} and \eqref{ass.alip} hold.
Next, for determining the subspaces $V_\t$ but also for some later purposes, we notice that 
 the assumptions on the `soft' phase $B$ ensure the following extension result (see e.g. \cite[Proposition 4.3]{Zhi2005}).
\begin{proposition}\label{prp.zhiext}
	There exists an extension operator $E : H^1(\Box \backslash B) \rightarrow H^1(\Box)$ 
i.e. such that 
$Eu|_{\Box \backslash B} = u$, and such that with a constant $C_E>0$, 	
	$\| Eu \|_{H^1(\Box)} \,\,\,\le\,\,\, C_E \| u \|_{H^1(\Box \backslash B)}$ 
for all $\,u \in H^1(\Box \backslash B)$, 
	and
	\begin{equation}\label{ZhExtension}
	\int_\Box |\nabla Eu|^2 \,\,\,\,\le\,\,\,\, C_E^2 \int_{\Box \backslash B} |\nabla u|^2, \quad 
	\forall \,u \in H^1(\Box \backslash B). 
	\end{equation}	
\end{proposition}
Using Proposition \ref{prp.zhiext} we obtain 
\begin{align}
a_\t[u]\,=\,
 \int_{\Box \backslash B} \big\vert (\nabla + \i \t) u\big\vert^2  & =  \int_{\Box \backslash B} 
\left\vert \nabla (e^{\i \t \cdot y}  u)\right\vert^2 \,\ge\, 
C_E^{-2} \int_\Box \bigl\vert \nabla E \left(e^{\i \t \cdot y}  u\right)\bigr\vert^2 \,\ge\, 
C_E^{-2}|\, \t |^2 \int_\Box \left\vert E \left(e^{\i \t \cdot y}  u\right)\right\vert^2 
 \nonumber \\ \label{dpH1} & \ge\,\, C_E^{-2}|\, \t |^2 \int_{\Box \backslash B}|u|^2\,, \qquad  \forall u \in H^1_{per}(\Box \backslash B), \,\,\forall \t \in \Theta,
\end{align}
where the first inequality holds by \eqref{ZhExtension}, the second (c.f. \eqref{IKaest}) by  expanding 
$e^{-\i \t \cdot y}E \left(e^{\i \t \cdot y}  u\right)\in H^1_{per}(\Box)$ in 
Fourier series in $\Box$, 
and the last from the extension property $E \vert_{\Box \backslash B} = I$. Therefore, 
via \eqref{spaceV},
\begin{equation}\label{hcV}
V_\theta = \left\{
\begin{array}{lr}
H^1_0(B), & \theta \neq 0, \\[5pt]
\Big\{ v \in H^1_{per}(\Box) \, \Big\vert \, \text{$v$ is constant in $\Box \backslash B$} \Big\} \,\,=\,\, \CC\overset{\cdot}{+} H^1_0(B), & \theta = 0, 
\end{array} 
\right.  \end{equation}
where elements of $H^1_0(B)$ are understood as those of $H^1_{per}(\Box)$ which are identically zero in 
$\Box \backslash B$. 
Clearly, $V_\t$ is discontinuous with respect to $\t$ 
at the origin. 
Spaces $W_\t$ can be determined via \eqref{2.6-w} and \eqref{astructure} with \eqref{abcdp}, 
although their explicit form is not of importance for our purposes. 

Now, proceeding as in Example \ref{e.class}, let us demonstrate that the main assumptions \eqref{KA}--\eqref{H6} hold. 

\textbullet\, To prove  \eqref{KA} we shall demonstrate that the stronger assertion  \eqref{KA2} holds for $C=C_E^{2} $  and 
$c(u,\tilde{u}) : = C_E^2 \int_{\Box \backslash B} u \overline{\tilde{u}}$, with the latter clearly being $\|\cdot\|_\t$-compact by the Rellich theorem. 
For fixed $\t \in \Theta$ and $u\in H$, we find that $v:= u - e^{-\i\t\cdot y}E\left(e^{\i\t\cdot y}u\right) \in H^1_0(B) \subseteq V_\t$ satisfies
\[
\| u-v\|_\t^2\,=\,
\| u-v\|_{L^2(\Box )}^2\,+\,\| (\nabla +\i \t)(u-v) \|_{L^2(\Box)}^2  
\]
\begin{equation}
\label{h1prdp}
 \ \ \ \ \ \ \ \ \  
=\,\,\, \left\| E \left(e^{\i \t \cdot y} u\right) \right\|_{H^1(\Box)}^2 \,\,\,\le\,\,\, C_E^2\,\,  \| e^{\i \t \cdot y} u \|_{H^1(\Box \backslash B)}^2 
\,\,=\,\, C_E^2  \int_{\Box \backslash B} \Big(\big| (\nabla + \i \t) u \big|^2 \,+\, |u|^2\Big).
\end{equation}
Since, for any $w\in W_\t$ and $v\in V_\t$, 
$\|w\|_{\t} \le \| w- v\|_{\t}$,   
the above inequality with $u=w$ implies that
\begin{equation}\label{dpNewKA}
\begin{split}
\| w\|_\t^2  \,\,\le\,\, C_E^2\bigg( a_\t[w] + \int_{\Box \backslash B} |w|^2 \bigg), \qquad \forall w\in W_\t, \ \forall \t \in \Theta, 
\end{split}
\end{equation}
which establishes \eqref{KA2}. 

\textbullet \, The validity of  \eqref{contVs} is immediate for $V_\star = H^1_0(B)$ and $L_\star = 0$;  see  Remark \ref{constV}. Furthermore, one can choose the defect subspace $Z = {\rm Span}\, \{ \mathbf{e}\}$ where $ \mathbf{e} \in H^1_{per}(\square)$ is 
the constant unity: indeed, $V_0 = H^1_0(B) \dot{+} \CC$ implying \eqref{spaceZ}, and 
 \eqref{VZorth} holds with $K_Z = |B|^{1/2} <1$: for $\phi \in H^1_0(B)$, 
\begin{equation}
\label{4.17-dp}
\left|(\phi,\mathbf{e})_0\right| \,=\, \left|\int_B \phi\right| \,\,\le\,\,   | B |^{1/2} \left( \int_B |\phi|^2 \right)^{1/2} \le 
\,\,| B |^{1/2} \|\phi\|_0 \,\,=\,\, |B|^{1/2} \| \phi \|_0 \| \mathbf{e} \|_0\,. 
\end{equation}

\textbullet\, Assumption \eqref{distance} holds with $\gamma = C_E^{-2}(\pi^2 n + C_E^2)^{-1}$. Indeed, \eqref{dpNewKA} and \eqref{dpH1} imply for $w\in W_\t$ and $\t\neq 0$, 
$
\| w \|_\t^2 
\le C_E^2 \big(1 + C_E^2 |\t|^{-2}\big) a_\t[w],
$
and consequently
\[
\nu_\theta\,\,=\,\, \inf_{w \in W_\t \backslash \{ 0 \}} \frac{a_\t[w]}{\|w\|_\t^2}\,\,\,\,\ge\,\,\,  
C_E^{-2} \big(1+ C_E^2|\t|^{-2}\big)^{-1}\,\,\,\ge\,\,\, |\t|^2\, C_E^{-2}\big(\pi^2 n + C_E^2\big)^{-1}.
\]

\textbullet \, Assumption \eqref{H4} is obviously satisfied, cf. \eqref{IKa'}, with  $K_{a'} =1$ , $K_{a''}=0$ and 
\begin{equation}
\label{dp.rega}
\begin{aligned}
& a'_{0}(v, u) \cdot \t\,\,: =\,\, \i  \int_{\Box \backslash B}   \t v \cdot \overline{ \nabla {u}} ,\qquad 
a''_{0}(v, \tilde{v})\,\t \cdot \t\,\,:=\,\, |\t|^2 \int_{\Box \backslash B} v\, \overline{\tilde{v}}\,.
\end{aligned}
\end{equation}
Now let us calculate $a^{\rm h}_\t[\mathbf{e}]$. Recalling \eqref{defhom.form}, \eqref{cell:prob2} and \eqref{dp.rega}, 
we obtain: 
\begin{equation}\label{dp.a'2}
a^{\rm h}_{\t}[\mathbf{e}]\,\,=\,\,a''_0[\mathbf{e}] \t \cdot \t \,+\, a_0'\left( \mathbf{e}, N_\t\mathbf{e}\right) \cdot \t 
\,\,=\, \int_{\Box \backslash B} |\t|^2 \,+\, \i \int_{\Box \backslash B} \t \cdot\overline{\nabla \left( N_\t \mathbf{e}\right)},
\end{equation} 
where $N_\t\mathbf{e}\in W_0$ solves (see \eqref{cell:prob2})
\begin{equation}
\label{dp.cell1}
\int_{\Box \backslash B }  \nabla\left(N_\t\textbf{e}\right) \cdot \overline{\nabla w} \,\,=\,\, -\,\i \int_{\Box \backslash B} \t\cdot \overline{ \nabla w} , \qquad \forall w \in W_0, \,\,\, \forall \t \in \RR^n.
\end{equation}
It is clear (since $V_0 = \CC \dot{+} H^1_0(B)$ and $H^1_{per}(\square) = V_0 \oplus W_0$) that the above equality holds 
in fact for test functions $\phi \in H^1_{per}(\square)$. Therefore, 
\begin{equation}
\label{dp.correctorproblem}\begin{aligned}
& N_\t\mathbf{e}\,\,=\,\, \i \,\t \cdot \big(\ourN_1^{\rm pd},\ldots,\ourN_n^{{\rm pd}}\big)^T \ \ \text{in $\,\Box\backslash B$, where real-valued  } 
\text{  $\ourN_j^{{\rm pd}} \in H^1_{per}(\Box \backslash B)$, $j=1,...,n$, solve}  \\ 
& 
\int_{\Box \backslash B}  \nabla \ourN_j^{{\rm pd}}\cdot\, \overline{\nabla \phi }\,\,=\, - \int_{\Box \backslash B}  e^j \cdot \overline{\nabla \phi } , \qquad \forall \phi \in H^1_{per}(\Box \backslash B),
\end{aligned} 
\end{equation}
where $e^1,\ldots,e^n$ is the canonical basis in $\mathbb{R}^n$. Thus 
$\ourN^{\rm pd}  = (\ourN_1^{\rm pd},\ldots,\ourN_n^{\rm pd})$ is (up to a 
constant) the perforated domain corrector, see e.g. \cite[Section 3.1]{JKO}.
As a result, via  \eqref{dp.a'2} and \eqref{dp.correctorproblem}, $a^{\rm h}_{\t}[\mathbf{e}]$ is 
\be \label{dp.ahom} 
a^{\rm h}_{\t}[\mathbf{e}]\,\,=\,\,  A^{\rm hom}_{\rm pd} \theta \cdot \theta,
\ee
where $A^{\rm hom}_{\rm pd}$ is the perforated domain homogenised matrix with components
\be\label{dp.coef}
\{A^{\rm hom}_{\rm pd}\}_{ij} \,\,=\,\, \int_{\Box \backslash B} \delta_{ij} \,+ \int_{\Box \backslash B}  \partial_{i}\ourN_{j}^{\rm pd}, \qquad  i,j \in \{1,\ldots,n\}.
\ee
Matrix $A^{\rm hom}_{\rm pd}$ is well-known to be positive definite  and symmetric  (see e.g. \cite[Section 3.1]{JKO}); these can 
also be seen directly, respectively via \eqref{dp.ahom} and \eqref{ahcoercive}, and  \eqref{dp.coef} and \eqref{dp.correctorproblem} 
with $\phi=\ourN_{i}$.

\textbullet \, Assumption \eqref{H5} is immediate with 
$L_b = 1$: indeed, via \eqref{abcdp}, 
\[
\bigl| b_\t(u, \tilde{u}) - b_0 (u, \tilde{u})  \bigr|\,\,=\,\,  
\left|  \int_{B}  \i\, \t u \cdot \overline{(\nabla + \i \t) \tilde{u}} \,\,\,+   \int_{B} \nabla  u \cdot \overline{ \i\, \t  \tilde{u}}\, \right|  
\]
\[
\ \ \ \ \ \ \ \ 
\,\,\le\,\,
|\t| \Big( \| u \|_{L^2(B)} \| (\nabla + \i \t)\tilde{u} \|_{L^2(B)} \,+\,  \| \nabla u \|_{L^2(B)}  \| \tilde{u} \|_{L^2(B)} \Big)
\,\,\le\,\, |\t|\, \|u\|_0 \,\|\tilde{u}\|_\t, \qquad \forall u, \tilde{u} \in H^1_{per}(\square).
\]
We now need to construct operator $\mathcal{E}_\t: H^1_0(B)\rightarrow H^1_0(B)$ satisfying the conditions in Lemma \ref{propeth}. 
For \eqref{Eprop1}, notice that for $\phi, \tilde{\phi} \in H^1_0(B)$,
\[
b_0(  \phi,  \tilde{\phi})\,:=\,
\int_B \nabla \phi \cdot \overline  {\nabla \tilde{\phi}} + \int_\square  \phi \overline{ \tilde{\phi}} = 
\int_B (\nabla + \i \t) e^{-\i \t \cdot y}\phi \,\cdot\, \overline{(\nabla + \i \t) e^{-\i \t \cdot y}\tilde{\phi} } \,+ \int_\square e^{-\i \t \cdot y} \phi \overline{e^{-\i \t \cdot y}\tilde{\phi}}= 
 b_\t\left(e^{-\i \t \cdot y}\phi,e^{-\i \t \cdot y}\tilde{\phi}\right).
\]
So to satisfy \eqref{Eprop1}  
we can take as  $\mathcal{E}_\t$ multiplication by $e^{-\i\, \t \cdot y}$ in $V_\star=H^1_0(B)$.  
One can then also readily verify \eqref{Eprop2} with $K_b = \sqrt{|B|(1+n/4)}$ as follows:  
\[
\Big| b_\t\big(e^{-\i\,\t\cdot y} \phi,\mathbf{e}\big)  \,-\, b_0(\phi,\mathbf{e}) \Big| \,\,=\,\,
  \left| \,\int_B e^{-\i\,\t\cdot y}\nabla  \phi\cdot \overline{ \i\, \t } \,\, + \int_B  \big(e^{-\i\, \t \cdot y} -1\big) \phi \right| 
	\ \ 
	\le\,\,\, |\t| \,\| \nabla \phi\|_{L^2(B)} |B|^{1/2} \,\,\,+ 
	\]
	\[
	\ \ \ \ \ 
	\sqrt{|B|n/4}\,|\t|\|\, \phi\|_{L^2(B)}  
\,\, \le\,\, \sqrt{|B|(1+n/4)}\, |\t|\,  \|\phi\|_0 \,\,=\,\, \sqrt{|B|(1+n/4)}\, |\t| \, \|\phi\|_0 \| \mathbf{e} \|_0, \quad \  \forall \phi \in H^1_0(B).
\]

\textbullet\, Finally, for 
\eqref{H6} set 
$\mathcal{H} = L^2(\square)$ with $\t$-independent $d_\t=d_0$ the standard $L^2(\square)$ inner product. 
There are several possibilities for extending $\mathcal{E}_\t$ from $H_0^1(B)$ to $L^2(\square)$, which we discuss 
in the following remark. 
\begin{remark}
\label{ecalextns}
To maintain the isometry 
in \eqref{H6}, the sought extension of $\mathcal{E}_\t$ from $H_0^1(B)$ to $L^2(\square)$  has to remain the multiplication by $e^{-\i \t \cdot y}$ 
for functions from $L^2(B)$ i.e. supported in the inclusion $B$. 
Hence the extension can be specified by a 
unitary map on the orthogonal complement, i.e. 
$\mathcal{E}_\t: L^2(\square\backslash B)\to L^2(\square\backslash B)$ 
which are $\t$-Lipschitz with $\mathcal{E}_0=I$ to obey \eqref{H6}. 
This can be done in numerous ways, and 
all of the choices would make our relevant abstract results from Section \ref{s:resolv} applicable. 
Of these choices, the simplest 
one seems to extend $\mathcal{E}_\t$ 
as the multiplication by $e^{-\i \t \cdot y}$ throughout: 
$\mathcal{E}^{(1)}_\t u = e^{-\i \t \cdot y}u$, $\forall u\in L^2(\square)$. 
However another natural possibility appears to 
set $\mathcal{E}_\t$ as an identity 
(i.e. multiplication by unity) in $\square\backslash B$: 
$\mathcal{E}^{(2)}_\t u = (1-\chi_B)u+ \chi_B e^{-\i \t \cdot y}u$, where 
$\chi_B$ denotes the characteristic function of $B$.  
In Section \ref{interpol77} below, we will compare the results for both of these choices. 
Notice that for both choices, \eqref{H6} readily holds. 
Indeed, $H=H^1_{per}(\square)$ is compactly embedded into and dense in $\mathcal{H}$ and \eqref{ik2} is satisfied, 
and the inequality in \eqref{H6} clearly 
holds with $K_e=\sqrt{n}$.
\end{remark}

%
As \eqref{KA}--\eqref{H6} are fulfilled we can apply 
the results of Sections \ref{section:discV} - \ref{s:resolv}, and we 
detail below implications of the relevant approximation theorems for the present example. 
\subsubsection{Application of Theorem \ref{thm.IKunifest2}}
We begin with specifying the approximations given in Theorem \ref{thm.IKunifest2}. Therein,  $V_\star = H^1_0(B)$ and $Z ={\rm Span}\, \{ \mathbf{e}\}$ 
(recalling $\mathbf{e}(y) \equiv 1$) and consequently, $v = v_{\ep,\t}\in H^1_0(B)$,  $z = c_{\ep,\t} \mathbf{e}$, $c_{\ep,\t} \in \CC$, and problem \eqref{IKz3prob88}, 
via \eqref{dp.ahom}, \eqref{abcdp} and \eqref{fdp}, specialises to 
	\begin{equation}\label{dplimp1}
\begin{aligned}
\ep^{-2} \left(A^{\rm hom}_{\rm pd}\t \cdot \t\right) c_{\ep,\t}  \overline{\tilde{c}}\,\, + 
\int_B \nabla v_{\ep,\t}\cdot \overline{\nabla \phi} \,\,+ \int_\square (v_{\ep,\t}+c_{\ep,\t} ) \overline{(\phi + \tilde{c})}  
\,\, =\, \int_\square U \Gamma_\ep F(\t ,y) 
\overline{\bigl(e^{-\i \t \cdot y} \phi(y) + \tilde{c}\bigr)} \, {\rm d}y, \hspace{.5cm} \\ \qquad \forall\, \phi \in H^1_0(B) , \,\, \ \forall \tilde{c} \in \CC.
\end{aligned}
\end{equation} 
This can equivalently be re-written as 
\begin{equation}\label{zdp}
\left\{ \ \begin{aligned}
 \Big( \ep^{-2} A^{\rm hom}_{\rm pd}\t \cdot \t  +1\Big) c_{\ep,\t} \,\, + \int_B v_{\ep,\t}(y) \, {\rm d}y\,\,  & = \int_\square U \Gamma_\ep F(\t ,y) \, {\rm d}y, \qquad \t \in \square^*;\\
 -\,\Delta v_{\ep,\t}(y)\,  + v_{\ep,\t}(y)\,+ c_{\ep,\t}\,\, &=\,\, e^{\i \t \cdot y} \,U\Gamma_\ep F(\t,y), \qquad y \in B, 
\ \ \t \in \square^*.
\end{aligned} \right.
\end{equation}
Applying Theorem \ref{thm.IKunifest2} 
and noticing that, by \eqref{fstar} and \eqref{fdp}, \eqref{abcdp},  
$\|f\|_{*\t}\le \big\|U\Gamma_\ep F(\t,\cdot)\big\|_{L^2(\square)}$, 
we conclude that  inequalities \eqref{IKfinal3} and 
\eqref{IKfinal3-2} imply the following. 
\begin{proposition}
	\label{prop:dp}
	Let $u_{\ep,\t}$ solve  \eqref{dp.pde} and $c_{\ep,\t}$, $v_{\ep,\t}$ solve \eqref{zdp}.	Then 
	\be
	\label{dp.mainest1}
	\begin{aligned}
		\ep^{-2}\int_{\Box \backslash B}  \Big\vert\big(\nabla+\i \t\big) \Big( u_{\ep,\t}(y)  - 
		\bigl(1 + \i\,\t \cdot \ourN^{\rm pd}(y) \bigr) c_{\ep,\t}\Big)\Big\vert^2 {\rm d}y  \,\,\,  +\,\int_{\Box \backslash B}  
		\bigl|u_{\ep,\theta}(y) - 
		\big(1 &+ \i\,\t  \cdot \ourN^{\rm pd}(y) \big)  c_{\ep,\t}\bigr|^2  {\rm d}y\hspace{.05\linewidth} \\&  
		\le\,\,\,  C_9\,\ep^2 \int_\square \big|U\Gamma_\Ep F(\t,y)\big|^2  {\rm d}y,
	\end{aligned}
	\ee
	\be\label{dp.mainest2}
	\int_{B}  \bigl\vert\big(\nabla+\i \t\big) \bigl( u_{\ep,\t}(y)  - 
		\bigl(	c_{\ep,\t} +e^{-\i \t \cdot y}v_{\ep,\t} \bigr) \bigr)\bigr\vert^2 {\rm d}y  \,  + 
	\int_\square  \bigl| u_{\ep,\theta}(y)  -\, \left( 	c_{\ep,\t} +e^{-\i \t \cdot y}v_{\ep,\t}\right) {\rm d}y  \bigr|^2  \,\,\le\,  
	\,C_{10}\,\ep^2\int_\square |U\Gamma_\Ep F(\t,y)|^2  {\rm d}y. \hspace{.35\linewidth}
	\ee
\end{proposition} 
Inequalities  \eqref{dp.mainest1} and \eqref{dp.mainest2} provide   $L^2$ estimates for the corresponding approximations of $u_\ep$, the solution  to \eqref{dp.differentialPDEdp}, and its gradient. 
Indeed, 
applying the inverse transforms to the approximation $c_{\ep,\theta} \mathbf{e}+e^{-\i \t \cdot y}v_{\ep,\t}$ to 
$u_{\ep,\theta}=U\Gamma_\ep u_\ep$ in \eqref{dp.mainest2}, set 
\be\label{dp.app1} 
u_\ep^{(0)} \,\,:=\,\, \Gamma^{-1}_\ep U^{-1} c_{\ep,\theta} \mathbf{e}  , \qquad 
v_{\ep}^{(0)} \,\,: =\,\, \Gamma_\ep^{-1}  U^{-1} e^{-\i \t \cdot y} v_{\ep,\t}. 
\ee
Note that, cf. \eqref{IKappr45} and \eqref{7.19-2}, 
as $c_{\ep,\t}\mathbf{e}$ is $y$-independent $u_\ep^{(0)}$ is smooth and 
\[
\Gamma^{-1}_\ep U^{-1}  \big(1 + \i\,\t \cdot \ourN^{\rm pd}\big) c_{\ep,\t} \mathbf{e}
\,=\,\,u_\ep^{(0)}+ \ep\left(\tilde\Gamma^{-1}_\ep\ourN^{\rm pd}\right)\cdot  \nabla u^{(0)}_\ep \quad 
\text{in $\,\RR^n\, \backslash\,  \overline{B_\ep}$, where  $ B_\ep \,:= \bigcup_{m\in \ZZ^n} \ep (B+m)$.} 
\]
Then  inequalities \eqref{dp.mainest1}, \eqref{dp.mainest2}, 
via the $L_2$-unitarity of the above inverse transform $\Gamma_\ep^{-1}U^{-1}$, 
  lead to the following theorem. 
\begin{theorem}
	\label{IK77-2}
	Let $u_\ep$ solve \eqref{dp.differentialPDEdp} and $u_\ep^{(0)}$,  $v_\ep^{(0)}$ be as in \eqref{dp.app1} 
	where $c_{\ep,\t}$, $v_{\ep,\t}$ solve \eqref{zdp}. Then 
	there exist  positive constants $c_0$ and $c_1$ independent of $\ep$  and of $F\in L^2(\mathbb{R}^n)$, such that
	\begin{gather}
	\label{dp.H1est}
	\bigl\Vert u_\ep \,-\,\bigl( u_\ep^{(0)}\,+\,\ep\,\ourN^{\rm pd}\left(\tfrac{\cdot}{\ep}\right)\cdot  \nabla u^{(0)}_\ep\bigr)   \bigr\Vert_{H^1(\RR^n\backslash \overline{B_\ep} )}
	\,\,+\,\,
	\ep \bigl\Vert u_\ep \,-\,\bigl(u_\ep^{(0)} \,+\, v^{(0)}_\ep
	\,\bigr)   \bigr\Vert_{H^1({B_\ep} )}
	\,\,\,\le\,\,\, c_0\,\ep\, \Vert F \Vert_{L^2(\mathbb{R}^n)}, \\
	\label{dp.L2est}
	\bigl\Vert u_\ep \,-\,\bigl(u_\ep^{(0)} \,+\, v^{(0)}_\ep
	\,\bigr)   \bigr\Vert_{L^2(\mathbb{R}^n)} \,\,\,\le\,\,\, c_1\,\ep\, \Vert F \Vert_{L^2(\mathbb{R}^n)}.
\end{gather}
\end{theorem}
An estimate resembling 
the $L^2$-estimate \eqref{dp.L2est} was derived first in \cite{ChCo}, by different means. 
However it was based (in our notation) 
on an earlier abstract estimate \eqref{final2} with 
the approximating problem \eqref{z3prob} 
(for $\langle f, \phi \rangle = \int_\square f \overline{\phi}$), 
 i.e. with $\t$-dependent form $b_\t$. 
The latter prevents expressing the approximation in terms of a solution of the explicit 
$\ep$-independent two-scale limit problem. 
In contrast, our approximation in both \eqref{dp.L2est} and the new (energy norm) $H^1$-estimate \eqref{dp.H1est} 
is expressible in terms of an appropriate solution to the ($\ep$-independent) 
two-scale limit problem, 
see Remark \ref{H1-2slp} below. 
Moreover, our $L^2$ approximations like in 
\eqref{dp.L2est} can be refined further (see the next subsection) 
to have an improved operator structure. 
That is in terms of 
self-adjoint approximating operators  \eqref{dpcompe3} containing the resolvent of the two-scale limit problem 
sandwiched by $L^2$-isometric and ``asymptotically unitary''  two-scale connecting operators 
(which are in turn key for new 
spectral estimates in Section \ref{specestdp}). 

Remark also that most recent work \cite{BonDueGlo25}, which obtained some new quantitative (as well as qualitative) 
results for {\it stochastic} high-contrast models of the present type, has made some interesting observations 
pertinent also to the periodic scenario considered here. In particular, Remark 1.7(b) of \cite{BonDueGlo25} 
develops an approximation in terms of an {\it $\ep$-dependent} operator capable of delivering $H^1$-type estimates 
akin to \eqref{dp.H1est}. Those approximations' main advantage is that they do not rely on the Floquet-Bloch transform,  
which makes them adjustable 
to the random setting. However, even in the periodic case, making such approximations 
expressible in terms of the ($\ep$-independent) two-scale limit problem and ultimately 
applicable for error estimates on the limit spectrum would still require a number of non-trivial steps. For the latter, one way would be essentially to follow again 
our general approach in 
Sections \ref{section:discV}--\ref{s:resolv}, although possibly 
in a somewhat simplified manner. 

\subsubsection{Explicit approximation via the two-scale limit operator and an associated two-scale 
connecting operator}
\label{interpol77}
We recall, see e.g. \cite{Zhi2005,IVKVPS13},  that for problem \eqref{dp.differentialPDEdp} the following 
property of 
two-scale (pseudo-)resolvent convergence is held. 
If $F_\ep\in L^2(\RR^n)$ weakly or strongly two-scale converges 
to $f_0\in L^2(\RR^n \times \square)$ then $u_\ep$ the solution of \eqref{dp.differentialPDEdp} 
(with $F= F_\ep$) respectively 
weakly or strongly two-scale converges to 
$u_0(x,y) = u(x) + v(x,y)$  the solution to the two-scale 
limit resolvent problem  $\left(\mathcal{L}_0+I\right) u_0= \mathcal{P} f_0$. 
Here $\mathcal{L}_0$ is a self-adjoint  \textit{two-scale operator} in Hilbert space 
$L^2\big(\RR^n ;\, \mathbb{C}\,\dot{+}\,L^2(B)\big)\,=\,L^2(\RR^n) \,\dot{+}\, L^2\big(\RR^n ; L^2(B)\big)$ 
which is a closed subspace of 
$L^2\big(\RR^n ; L^2(\square)\big)=L^2\big(\RR^n\times \square\big)$. 
Operator $\mathcal{L}_0$ is 
generated by the \textit{two-scale  form}
\be
\label{q02sc}
\begin{aligned}
Q_0(u+v,\,\phi+\psi) \,\,=\,\,	\int_{\RR^n} A^{\rm hom}_{\rm pd} \nabla u(x) \cdot \overline{\nabla \phi(x)} \, {\rm d}x \,+\, \int_{\RR^n}\int_B \nabla_y v(x,y) \cdot \overline{\nabla_y \psi(x,y)} \,\, {\rm d}y\,{\rm d}x, 
\end{aligned}
\ee
for $u,\phi \in H^1(\RR^n), v,\psi \in L^2\left(\RR^n ; H^1_0(B)\right),$
with the dense form domain $H^1\left(\RR^n\right) \,\dot{+}\, L^2\big(\RR^n ; H^1_0(B)\big)$, and  
$\mathcal{P}: L^2\left(\RR^n\times \square\right) \rightarrow L^2\left(\RR^n\right) \,\dot{+}\, 
L^2\big(\RR^n ; L^2(B)\big)$ is the orthogonal projection or simply
\be
\label{p-cal-dp}
\mathcal{P}g(x,y) =  \left\{ 
\begin{array}{lcr}
	g(x,y) & & x\in \RR^n, y \in B \\
	{|\square \backslash B|^{-1}} \int_{\square \backslash B} g(x,y') \, {\rm d}y'  & &  x\in \RR^n, y \in \square \backslash B.
\end{array}
\right.
\ee
Now we observe that the above objects are precisely those that appeared in Section \ref{s.bivariate} when specialised to the present example. 
Indeed, recall that 
$\mathcal{H}=L^2(\square)$  with inner product $d_\t(u,{\tilde{u}})  = \int_\square u \,\overline{\tilde{u}}$, 
and notice  via \eqref{abcdp} that \eqref{zvbd} holds. 
Further, according to Section \ref{s.bivariate}, 
$\mathcal{H}_0 :=\overline{Z\dot{+}V_\star} =  \CC \,\dot{+}\, L^2(B)$, and
\begin{align*}
&\mathbb{H}\,=\,L^2\big(\RR^n ; (\mathcal{H},d_0)\big) = L^2\left(\RR^n \times \square\right), \qquad 
\mathbb{H}_0\,=\,L^2(\RR^n ; \mathcal{H}_0) = L^2(\RR^n) \,\dot{+}\, L^2(\RR^n ; L^2(B)), \\
&\text{and}  \quad
\check{\mathbb{D}}\,=\,H^1(\RR^n;Z) \,\dot{+}\, L^2(\RR^n;V_\star) \,=\, H^1(\RR^n) \,\dot{+}\, L^2\big(\RR^n ; H^1_0(B)\big),
\end{align*}
 all equipped with the standard norms. Therefore, comparing the above two-scale form $Q_0$ and the  bivariate form $Q$ (see \eqref{Q}) and recalling \eqref{dp.ahom} we find that simply 
$Q(u+v,\,\phi+\psi) \,=\, Q_0(u+v,\,\phi+\psi)  \,+\, (u+v,\,\phi+\psi)_{L^2(\RR^n \times \square)}$. 
So  for  the  abstract bivariate operator $\mathcal{L}$ generated by $Q$ 
as introduced in Section \ref{s.bivariate}, $\mathcal{L}=\mathcal{L}_0+I $. 
We notice also from \eqref{q02sc} that 
for $g\in\mathbb{H}$ the two-scale limit problem 
$\left(\mathcal{L}_0+I\right)u_0=\mathcal{L}u_0=\mathcal{P}g$ 
can be written as quite an explicit system: 
find such 
$u_0=u+v\in\check{\mathbb{D}}$ that 
\be 
\label{2sclimprob}
\left\{ \ \begin{aligned}
-\,\,\text{div}_x\left(\,A^{\rm hom}_{\rm pd}\,\nabla_x u(x)\right)\,\,+\,\,u(x)\,\,+\,\,\int_\square v(x,y)\,{\rm d}y
  \,\,\,& =\,\,\, \int_\square g(x,y)\,{\rm d}y, \qquad x \in \RR^n;\\
 -\,\,\Delta_y v(x,y)\,\,  +\,\, u(x)\,\,+\,\,v(x,y) \,\,\, &=\,\,\, g(x,y), \qquad y \in B.
\end{aligned} \right.
\ee
As a result, Theorem \ref{thm.bivariate} via a routine specialisation to the present setting 
(for both choices of $\mathcal{E}_\t$ discussed in Remark \ref{ecalextns}) 
yields the following. 
\begin{theorem}
	\label{thm.bivariatehc} 
For $0<\ep <1$ 
and self-adjoint operators $\mathcal{L}_{\ep,\t}$ specified in $\mathcal{H}=L^2(\square)$ by forms in \eqref{dp.varp}, 
	\[
	\Big\Vert\mathcal{L}_{\ep,\t}^{-1} g(\t) \,-\, \Big(A_\ep^* \left(\mathcal{L}_0+I\right)^{-1}\mathcal{P} A_\ep g\Big)(\t)\Big\Vert_{L^2(\square)}  \,\,\le\,\, 
	C_{11}\,\ep\,
	  \big\Vert g(\t)\big\Vert_{L^2(\square)} \, ,	\quad \forall g \in L^2(\square^* \times \square), \quad a.e.\ \t \in \square^*,
	\]
where 
	$A_\ep : L^2(\square^* \times \square) \rightarrow L^2(\RR^n \times \square) $, $A_\ep=\Gamma_\ep^{-1}\mathcal{F}^{-1} \, \chi\,\mathcal{E}^{-1}$,   
	and its adjoint 
	$A_\ep^* : L^2(\RR^n \times \square ) \rightarrow L^2(\square^*\times\square )$, 
	$A_\ep^*=
	\mathcal{E}\, \chi^*\, \mathcal{F}\,\Gamma_\ep$, 
 are given by the continuous extensions of: for $y\in B$,  
\begin{equation}
\label{Aep}
 A_\ep g(x,y) 
 \,\,=\,\, (2\pi)^{-n/2} \ep^{-\,n/2} \int_{\square^*}  e^{\i\, \t \cdot y}g(\t,y) 
e^{\i \,\t\, \cdot\, \frac{x}{\ep}}\, {\rm d} \t, \quad x\in \mathbb{R}^n, \ y\in B,  
 \end{equation}
 \begin{equation}
\label{Aepstar}
 A_\ep^* h(\t,y) \,\,=\,\,(2\pi)^{-\,n/2} \ep^{-\,n/2} e^{-\i\, \t \cdot y}\int_{\RR^n} h(x,y) 
e^{-\i \,\frac{\t}{\ep}\, \cdot\, x} \, {\rm d}x, \quad \t\in \square^*, \ y\in B;  
 \end{equation} 
for the choice $\mathcal{E}_\t=\mathcal{E}^{(1)}_\t$ 
(see Remark \ref{ecalextns}),  
\eqref{Aep} and \eqref{Aepstar} hold for $y\in\square$;  
if $\mathcal{E}_\t=\mathcal{E}^{(2)}_\t$, 
then for $y\in\square\backslash B$ 
the exponential factors $e^{\i\,\t\cdot y }$ and 
$e^{-\i\,\t\cdot y }$ in \eqref{Aep} and \eqref{Aepstar}  respectively have to be dropped. 
\end{theorem}

On the basis of the above theorem we can construct 
an approximation for 
the original resolvent problem 
\eqref{dp.differentialPDEdp} in terms of that for the two-scale limit problem \eqref{2sclimprob},  
as follows.   
Denote by 
$\mathcal{L}_\ep\,=\,-\,\nabla\cdot\bigl( A_\ep \left(\tfrac{x}{\ep} \right) \nabla \cdot \bigr) $ the 
non-negative self-adjoint operator defined in a standard way in Hilbert space $L^2\left(\mathbb{R}^n\right)$, so
 for the solution of  \eqref{dp.differentialPDEdp}  $u_\ep=\left(\mathcal{L}_\ep+I\right)^{-1}F$. 
The following theorem holds. 
\begin{theorem}\label{thm.2scOpRes}
Let $\mathcal{L}_\ep$ and $\mathcal{L}_0$ be respectively the original and the two-scale limit operators as described 
above, and $\mathcal{P}$ be the projector given by \eqref{p-cal-dp}. 
Then with some constant $C$, 
	for all $0<\ep<1$ one has 
	\begin{equation}
		\label{dpcompe3}
		\bigl\Vert \left(\mathcal{L}_\ep+I\right)^{-1} \,-\,  
		\mathcal{J}_\ep^* \left(\mathcal{L}_0+I\right)^{-1} \mathcal{P} \mathcal{J}_\ep \bigr\Vert_{L^2(\RR^n) \rightarrow L^2(\RR^n)} \,\,\le\,\, C\, \ep.  
	\end{equation}
Here $\mathcal{J}_\ep: L^2\left(\RR^n\right) \rightarrow L^2\left(\RR^n \times \square\right)$ is an $L^2$-isometry, and 
	for both choices of the transfer operator $\mathcal{E}_\t$ (Remark \ref{ecalextns}) 
 $\mathcal{J}_\ep=T_\ep\,\mathcal{I}_\ep$ where 
	$T_\ep: L^2\left(\RR^n \times \square\right) \rightarrow L^2\left(\RR^n \times \square\right)$ is a unitary ``translation'' operator  
	such that for $y\in B$, $T_\ep f(x,y)=f(x+\ep y, y)$. 
	Operator		$\mathcal{I}_\ep: L^2(\RR^n) \rightarrow L^2(\RR^n \times \square)$,  
 which we call 
	``the two-scale interpolation operator'' (see Remark \ref{RemShann} below),
	is a bounded operator given by the composition 
	\be
	\label{2ScInterp}
	\mathcal{I}_\ep:\,\,=\,\,\Gamma_\ep^{-1}\mathcal{F}^{-1} \, \chi\,
	U\, \Gamma_\ep\, \,\,=\,\, T_\ep^{-1}A_\ep\, U\, \Gamma_\ep.
	\ee
	In \eqref{2ScInterp}, $\Gamma_\ep: \, F(x)\mapsto \ep^{n/2}F(\ep x)$ is the $L^2$-unitary rescaling operator and 
	$\Gamma_\ep^{-1}: \, f(x,y)\mapsto \ep^{-n/2}f\left(\ep^{-1} x,\,y\right)$ is its inverse in $x$; 
	$\,\,U:F(x)\mapsto g(\t,y)$ is the Floquet-Bloch-Gelfand transform, see \eqref{gt1};  
	$\,\,\chi:L^2\left(\square^*\times\square\right)\rightarrow L^2\left(\RR^n\times\square\right)$ is the extension by zero outside 
	$\square^*$ in the first variable, and 
	$\mathcal{F}^{-1}:g(\xi,y)\mapsto f(x,y)$ is the inverse Fourier transform also in the first variable. 
	
	$\mathcal{I}_\ep$ is 
	an $L^2$-isometry and the continuous extension of 
	\begin{equation}
	\label{7.54-1}
	\left(\mathcal{I}_\ep F\right)(x,y) \,\,=\,\,  \sum_{m\in \ZZ^n} F\big(\ep y + \ep m\big)\,\,  
	{\rm Sinc}\left( \frac{x}{\ep}  - m - y\right) \,, \ \ \ F\in C_0^\infty\left(\mathbb{R}^n\right), 
	\end{equation}
	where ${\rm Sinc}(z)$, $z\in\mathbb{R}^n$,  is the ($n$-dimensional normalised) sinc-function: 
\[
{\rm Sinc}(z)\,\,:=\,\,\prod_{j=1}^n\,\left\{\begin{array}{cll}\frac{\sin\left(\pi z_j\right)}{\pi z_j}&, \ &z_j\neq 0 \\
1&, \ \ &z_j=0, 
\end{array}\right.
\ \ \ z\in \mathbb{R}^n. 
\]
$\mathcal{J}_\ep^*=\mathcal{I}_\ep^*\,T_\ep^{-1}$ is the adjoint of $\mathcal{J}_\ep$, where 
$\mathcal{I}_\ep^*: L^2(\RR^n \times \square) \rightarrow L^2(\RR^n) $ is the adjoint of $\mathcal{I}_\ep$ given by   
$\mathcal{I}_\ep^*=\Gamma_\ep^{-1} U^{-1}
\, \chi^*\, \mathcal{F}\,\Gamma_\ep $  (where the adjoint $\chi^*$ of $\chi$ is the restriction from $\mathbb{R}^n$ to $\square^*$ in 
the first variable). Operator $\mathcal{I}_\ep^*$ is the continuous extension of 
\begin{equation}
\label{7.54-2}
\mathcal{I}_\ep^* u_0(x) \,\,=\,\,  \ep^{-n}\, \int_{\RR^n} 
u_0\left(s, \left\{ \frac{x}{\ep} \right\}\right) 
\,\,  
	{\rm Sinc}\left( \frac{x}{\ep}\,-\, \frac{s}{\ep}\right)
\,{\rm d}s\,, 
\end{equation}	
where $\{ p \}$ 
is the fractional part of $p\in\mathbb{R}^n$ 
($\{p\}:=p-m$ for the unique $m\in\mathbb{Z}^n$ such that $p-m\in [-1/2,1/2)^n\subset\square$). 
The ranges of $\mathcal{I}_\ep$ and $\mathcal{J}_\ep$ consist of all functions $f(x,y)\in L^2(\RR^n \times \square)$ whose Fourier transform in $x$ is supported in 
$[-\pi/\ep,\pi/\ep]^n$ for a.e. $y\in\Box$. Moreover, 
\begin{equation}
\label{calttstar}
\mathcal{I}_\ep^* \mathcal{I}_\ep \,=\,\mathcal{J}_\ep^* \mathcal{J}_\ep \,=\,  I, \quad \text{and} \quad 
\mathcal{I}_\ep \mathcal{I}_\ep^* \,=\, \mathcal{J}_\ep \mathcal{J}_\ep^* \,=\,\mathcal{S}_{\ep} 
\,\rightarrow I \text{ strongly},
\end{equation} 
where $\mathcal{S}_\ep$ is the smoothing operator as given by \eqref{7.22-2} (with $\chi$ replaced by the characteristic function of ${\square^*}$) 
applied to the first variable, i.e. 
\be
\label{se2sc}
  \mathcal{S}_\ep\,\,\,=\,\,\,\mathcal{F}^{-1}\,\chi_{\ep^{-1}\square^*}\,\mathcal{F}, 
\ee
where $\chi_{\ep^{-1}\square^*}$ is multiplication (in the first variable) by characteristic function of 
$\ep^{-1}\square^*$. 
\end{theorem}
\begin{proof}
For any $F\in\mathbb{R}^n$,  let $u_\ep=\left(\mathcal{L}_\ep+I\right)^{-1}F$ be 
the solution of  \eqref{dp.differentialPDEdp}. 
Set $g=U\Gamma_\ep F\in L^2\left(\square^*\times\square\right)$, 
and 
 observe via \eqref{fdp} and \eqref{ik3} that for  $u_{\ep,\t}=U\Gamma_\ep u_\ep$ is the solution 
of \eqref{dp.pde}, $u_{\ep,\t}=\mathcal{L}^{-1}_{\ep,\t}g$. 
Combining this all implies 
$
\mathcal{L}_{\ep,\t}^{-1} =  U \Gamma_\ep (\mathcal{L}_\ep+I)^{-1}  \Gamma_\ep^{-1} U^{-1}.
$
Due to the $L_2$-unitarity of $U$ and $\Gamma_\ep$, Theorem \ref{thm.bivariatehc} implies that 
\eqref{dpcompe3} holds with 
\be 
\label{jepdef}
\mathcal{J}_\ep\,\,=\,\,A_\ep U\Gamma_\ep\,\,=\,\, \Gamma_\ep^{-1}\mathcal{F}^{-1} \, \chi\,\mathcal{E}^{-1}\,U\, \Gamma_\ep\,, 
\ee
which is an $L^2$-isometry from  
$L^2(\RR^n)$ to $L^2\left(\RR^n \times \square\right)$ as a composition of $L^2$-norm preserving operators. 

As $\mathcal{E}$ is multiplication by $e^{-\i\t\cdot y}$ (or identity for $y\in\square\backslash B$), its 
inverse $\mathcal{E}^{-1}$ is simply the multiplication by $e^{\i\t\cdot y}$ or identity. 
Then, commuting it with (scaled) inverse Fourier transform, 
$\Gamma_\ep^{-1}\mathcal{F}^{-1} \, \chi\,\mathcal{E}^{-1}= T_\ep\Gamma_\ep^{-1}\mathcal{F}^{-1} \, \chi$ 
where $T_\ep f(x,y)=f(x+\ep y, y)$ (or identity). 
Clearly $T_\ep$ is unitary in $L^2\left(\RR^n \times \square\right)$, and from the above 
$\mathcal{J}_\ep=T_\ep\,\mathcal{I}_\ep$ where 
$\mathcal{I}_\ep:=\Gamma_\ep^{-1}\mathcal{F}^{-1} \, \chi\, U\, \Gamma_\ep$ is an $L^2$-isometry from  
$L^2(\RR^n)$ to $L^2(\RR^n \times \square)$. 
Hence, at a dense subspace, e.g. $C_0^\infty(\mathbb{R}^n)\ni F$, 
combining this with \eqref{Aep}, \eqref{gt1} 
and \eqref{gammaep} we obtain 
\[
\mathcal{I}_\ep F(x,y)\,:=\,T_\ep^{-1}A_\ep\, U\, \Gamma_\ep F(x,y)\,=\,\, (2\pi )^{-\,n}   \sum_{m\in \ZZ^n} F(\ep y + \ep m)  
	\int_{\square^*}   e^{\i\,  \t\, \cdot\, \left( \frac{x}{\ep}  - m - y\right)}\, {\rm d} \t, 
\]
which yields \eqref{7.54-1}. 
Similarly, combining \eqref{Aepstar} with \eqref{gt2} and \eqref{gammaep} gives 
\[
\mathcal{I}_\ep^* u_0(x) \,=\, \Gamma_\ep^{-1} U^{-1} A_\ep^*T_\ep u_0(x)\,=\,  (2\pi\ep)^{-\,n} \int_{\RR^n} 
u_0\left(s, \left\{ \tfrac{x}{\ep} \right\}\right) \left( 
\int_{\square^*}e^{\i\,\t\,\cdot\,\left(\frac{x}{\ep}- \frac{s}{\ep}\right)} \,  {\rm d}\t \right){\rm d}s, 
\]
yielding \eqref{7.54-2}. 
Finally, \eqref{calttstar} 
follows via \eqref{2ScInterp}, 
\eqref{abident} and \eqref{7.22-2}; 
and the strong convergence of $\mathcal{S}_\ep$ to the unity operator $I$ directly follows from \eqref{se2sc}. 
\end{proof}

We expect the two-scale approximations of type \eqref{dpcompe3} to be of a more general interest and wider 
application potential, 
so will discuss below some related aspects in more detail. 
Operator $\mathcal{J}_\ep$ plays in \eqref{dpcompe3} a key role of $L^2$-isometrically 
(and due to \eqref{calttstar} ``asymptotically unitarily'') 
converting, for a fixed $\ep>0$, any input function $F(x)$ from $L^2(\mathbb{R}^n)$ into corresponding two-scale 
function $\mathcal{J}_\ep F(x,y)$ in $L^2(\RR^n \times \square)$. 
The latter serves in turn as the input $g(x,y)$ for the two-scale limit problem \eqref{2sclimprob}, whose solution 
$u_0(x,y)=u(x)+v(x,y)$ is converted by the adjoint 
$\mathcal{J}_\ep^*$ back into a function of $x$. 
The whole point is that such a procedure delivers an approximate self-adjoint solution operator, which is the resolvent of the two-scale limit operator preceded by the 
projection operator $\mathcal{P}$ and flanked by the connecting operator $\mathcal{J}_\ep$ and its adjoint, 
delivering the operator-normed error estimate \eqref{dpcompe3}. 
Operator $\mathcal{J}_\ep$ is a composition of a problem specific (unitary) translation operator $T_\ep$ with  
a more generic $L^2$-isometric operator $\mathcal{I}_\ep$ defined by \eqref{2ScInterp}. 
Operator $\mathcal{I}_\ep$ was introduced, in an equivalent form, in \cite{Well2009} under the name of 
``periodic two-scale transform'' 
as a convenient tool for establishing various 
two-scale convergence and compactness 
results in periodic homogenisation. In our context here its additional power comes from the 
fact that, in combination 
with the above translation operator $T_\ep$, 
naturally emerging as a specialisation of our  generic approach, it 
is capable of {\it quantifying} for the two-scale 
convergence by 
providing two-scale type operator approximations with tight error bounds. 
\begin{remark}
\label{RemShann}
Interestingly, explicit representation \eqref{7.54-1} of  $\mathcal{I}_\ep$ appears to be a natural two-scale version of the classical  Whittaker–Shannon interpolation formula, see e.g. \cite{Higgins} for a 
review. 
In this respect, operator $\mathcal{I}_\ep$ can be viewed as a 
two-scale interpolation operator. 
Indeed for regular enough $F$, given $y\in\Box$, for every $x$ with the ``phase'' $y$ i.e. $x=\ep y+\ep l$ for $l\in \mathbb{Z}^n$ 
\eqref{7.54-1} implies $\mathcal{I}_\ep F(x,y)=F(x)$. So, for a chosen $y$,
$\mathcal{I}_\ep F(x,y)$ simply reads off the values of F at all the points with the phase $y$, i.e. 
on the shifted $\ep$-periodic lattice 
$\ep y+\ep\mathbb{Z}^n$, 
 smoothly interpolating  in between for 
other $x\in\mathbb{R}^n$.  
In particular, the following 
is 
implied by the classical 
Whittaker-Kotelnikov-Nyquist-Shannon sampling theorem (see e.g. \cite{Higgins}). -- If the right hand side $F$ is itself a two-scale function, i.e. 
$F_\ep(x)=\Phi(x,x/\ep)$ where $\Phi(x,y)$
is sufficiently regular, $\square$-periodic in $y$ and its Fourier transform in $x$ is uniformly for a.e. $y$ compactly supported in an origin-centred cube $Q$ of size $2R$, 
i.e. $Q=[-R,R]^n$, then for all 
 $0<\ep\le\pi R^{-1}$, $\big(\mathcal{I}_\ep F_\ep\big)(x,y)=\Phi(x, y)$.  
On the other hand it is seen 
from \eqref{2ScInterp} that, for any $F\in L^2\left(\RR^n\right)$, 
$\left(\mathcal{I}_\ep F\right)(x,y)$ 
automatically has the above property of uniformly compact support of the $x$-Fourier transforms with $R=\pi/\ep$. 
This property is inherited by the input $g(x,y)=\mathcal{J}_\ep F(x,y)=T_\ep\left(\mathcal{I}_\ep F\right)(x,y)$ 
of the two-scale limit problem \eqref{2sclimprob} and then in turn by 
its solution $u_0(x,y)$, and further by $u_{0\, \ep}(x,y):=T_\ep^{-1}u_0(x,y)=u_0(x-\ep y,y)$ 
(or $u_{0 \,\ep}(x,y)=u_{0}(x,y)$ for $y\notin B$ in the case of  
$\mathcal{E}_\t=\mathcal{E}_\t^{(2)}$, 
Remark \ref{ecalextns}). 
It then follows from noticing that  \eqref{7.54-2} is a convolution of $u_0$ (with respect to its first variable) with the rescaled Sinc-function whose 
Fourier transform is the characteristic function of $\ep^{-1}\square^*$,  
that 
$\mathcal{J}_\ep^* u_{0}(x)=\mathcal{I}_\ep^* u_{0 \,\ep}(x)=u_{0\, \ep}\big(x,\, \{x/\ep\}\big)= 
u_0\big(\ep [x/\ep],\, \{x/\ep\}\big)$ (where $[p]:=p-\{p\}$ is the entire part of $p$), or just 
$\mathcal{J}_\ep^* u_{0}(x)=u_{0}\big(x,\, \{x/\ep\}\big)$ for 
$\mathcal{E}_\t=\mathcal{E}_\t^{(2)}$ and 
$\{x/\ep\}\notin B$. 

The above two-scale interpolation and sampling theorem properties of $\mathcal{I}_\ep$ are in fact 
encoded in its operator representation \eqref{2ScInterp}. 
Indeed, 
for sufficiently regular $g\in L^2(\RR\times\square)$, let $\mathcal{R}_\ep$ denote the mapping 
$g(x,y)\mapsto g\big(x,\{x/\ep\}\big)$. Observe from the inversion formula \eqref{gt2} for Gelfand transform that 
$U^{-1}=\mathcal{R}_1\mathcal{F}^{-1}\chi$, implying $\mathcal{R}_1\mathcal{F}^{-1}\chi U=I$. 
Next, since $\Gamma_\ep^{-1}\mathcal{R}_1=\mathcal{R}_\ep\Gamma_\ep^{-1}$, in combination with \eqref{2ScInterp} 
this yields 
\be
\label{2scIntOp0}
\mathcal{R}_\ep\mathcal{I}_\ep\,\,=\,\,I,
\ee
 i.e. (for sufficiently regular $F\in L^2(\RR)$) 
$\mathcal{I}_\ep F\big(x,\{x/\ep\}\big)=F(x)$ which is the above  
two-scale interpolation property of $\mathcal{I}_\ep$: if 
$x=\ep l+\ep y$ for $l\in\ZZ$ and $y\in\square$, 
then  simply $\mathcal{I}_\ep F\big(x,y\big)=F(x)$. 
Further, combining \eqref{2scIntOp0} with the second identity in \eqref{calttstar}, one obtains 
$\mathcal{I}_\ep^*=\mathcal{R}_\ep\mathcal{S}_\ep$. Applying $\mathcal{I}_\ep$ to both sides and then 
using \eqref{calttstar} again results in  
$\mathcal{I}_\ep\mathcal{R}_\ep\mathcal{S}_\ep={S}_\ep$, which recovers the two-scale Whittaker-Shannon sampling theorem. 
Namely, if $F_\ep(x)=\Phi(x,\,x/\ep)$ (so $F_\ep=\mathcal{R}_\ep\Phi$)
where $\Phi(x,y)\in L^2\left(\RR\times\square\right)$ 
is $\square$-periodic in $y$ 
and for every $y$ its Fourier transform in $x$ is supported within a bounded cube $[-R,R]^n$, 
then 
for $\ep\le\pi/R$ simply 
$\mathcal{S}_\ep\Phi=\Phi$ and 
$\left(\mathcal{I}_\ep F_\ep\right)(x,y)=\Phi(x,y)$. 


\end{remark}
\begin{remark} 
\label{ShannVsUnf}
In contrast to the classical homogenisation (Example \ref{e.class} above), a 
connecting operator $\mathcal{J}_\ep$ 
is necessary in \eqref{dpcompe3} for recasting any input $F$ as 
a two-scale function $g(x,y)$, to serve 
as the input for the two-scale limit problem \eqref{2sclimprob}. 
Naively setting $g(x,y)=\left(\tilde{\mathcal{J}}_\ep F\right)(x,y):=F(x)$ for all $y$, 
 the adjoint operator 
is $\tilde{\mathcal{J}}_\ep^*u_0(x)=\int_\square u_0(x,y)dy$ so for fixed $F$ the self-adjointness preserving approximation 
$u^{\rm appr}_\ep=\tilde{\mathcal{J}}_\ep^*u_0$ with $u_0$ solving \eqref{2sclimprob} is $\ep$-independent.  
Hence $u^{\rm appr}_\ep$ cannot approximate the exact solutions $u_\ep$ when they are genuinely two-scale (i.e. when $u_0(x,y)$ solving 
\eqref{2sclimprob} with ``correct'' two-scale $g=\mathcal{J}_\ep F$ remains $y$-dependent, e.g. for $F$ with a 
compactly supported Fourier transform). 
By similar arguments, 
an even more naive output approximation $u^{\rm appr}_\ep=u_0(x,x/\ep)$ in conjunction with the input $g(x,y)=F(x)$, 
apart from losing the self-adjointness, could not deliver an estimate like \eqref{dpcompe3} either. 

It could be of interest to compare our 
connecting operator $\mathcal{J}_\ep$ 
with any 
other candidate 
operators of a similar nature. One of these is the periodic unfolding operator see e.g. \cite{CDG},  
which has in fact been successfully used for establishing operator-type error estimates, 
although in 
classical homogenisation problems, 
see e.g. \cite{griso}. 
Denoting by $\mathcal{T}_\ep: L^2(\RR^n) \rightarrow L^2(\RR^n \times \square)$ the $L^2$-isometric unfolding operator, for sufficiently regular $F$ we have 
$\left(\mathcal{T}_\ep F\right)(x,y):=F\big(\ep\,[x/\ep]\,+\,\ep y\big)$. 
This can be compared with our $\mathcal{J}_\ep=T_\ep\mathcal{I}_\ep$ when 
$\mathcal{E}_\t=\mathcal{E}_\t^{(1)}$ 
(Remark \ref{ecalextns}),  
in which case $T_\ep f(x,y)=f(x+\ep y,y)$ for all $y\in\square$ and so 
from \eqref{7.54-1}
\be
\label{I-ep-mod}
\left(\mathcal{J}_\ep F\right)(x,y) \,\,=\,\,  \sum_{m\in \ZZ^n} F\big(\ep y + \ep m\big)\,\,  
	{\rm Sinc}\left(\, \frac{x}{\ep}  \,-\, m\,\right), \ \ \ y\in \square. 
\ee
One can see that, for regular enough $F$, both $\mathcal{T}_\ep$ and 
$\mathcal{J}_\ep$ produce on the same $\ep$-periodic lattice 
$\ep\, \mathbb{Z}^n$ exactly the same values $F(x+\ep y)$, however interpolate between those in different ways. 
Namely, while $\mathcal{T}_\ep$ simply extends the latter value for the whole of the related $\ep$-cell $x\in \ep l+\ep\square$ in piecewise constant way, 
$\mathcal{J}_\ep$ 
smoothly interpolates between the above points according to \eqref{I-ep-mod}. 
We briefly discuss here some further similarities and differences between $\mathcal{J}_\ep$ 
and $\mathcal{T}_\ep$, 
postponing a more detailed discussion for another study.  
Notice that for regular enough $F$ the unfolding operator can be written in a form akin to \eqref{I-ep-mod}. Namely, 
	$\left(\mathcal{T}_\ep F\right)(x,y) \,\,=\,\,  \sum_{m\in \ZZ^n} F\big(\ep y + \ep m\big)\,\,  
	\chi_\square\left( \frac{x}{\ep}  - m\right)$,  
where $\chi_\square$ is the characteristic function of the periodicity cell $\,\square$. 
Comparing then with \eqref{jepdef}, one observes that $\mathcal{T}_\ep$ has the following operator form
\be
	\label{2ScInterpUF}
	\mathcal{T}_\ep\,\,\,=\,\,\Gamma_\ep^{-1}\mathcal{F}^{-1} \, {\rm Sn }\,\,\mathcal{E}^{-1}\, U\, \Gamma_\ep\,, 
	\ee
where ${\rm Sn}: L^2\left(\square^*\times\square\right)\rightarrow L^2\left(\RR^n\times\square\right)$, replacing the 
extension operator $\chi$ in \eqref{jepdef},  is the operator of 
``$\,\square^*$-periodisation'' in the first variable followed by multiplication by the ${\rm Sinc}$ function. Namely, 
\[ 
\big({\rm Sn}\, g\big)(\xi,y) \,\,=\,\,  
{\rm Sinc}\left(\tfrac{\xi}{2\pi}\right)\,g\left(2\pi\left\{\tfrac{\xi}{2\pi}\right\},\,y\right).
\] 
It can be seen that, $\forall x\in\mathbb{R}^n$, $\sum_{l\in \ZZ^n} {\rm Sinc }^2(x+l)=1$, which implies that operator 
${\rm Sn}$ is an $L^2$-isometry (and hence so is $\mathcal{T}_\ep$, as a composition \eqref{2ScInterpUF}); 
and in particular ${\rm Sn}^*\,{\rm Sn}=I$. 

As we have seen, 
our new two-scale 
connecting operator $\mathcal{J}_\ep$ 
delivers 
a desired approximation with an 
operator norm 
error estimate 
\eqref{dpcompe3}. 
However one can show from 
the above 
using the structure of the two-scale limit operator (and in fact of more general bivariate 
operators) 
that 
the approximations based on $\mathcal{J}_\ep$ and on the unfolding operator $\mathcal{T}_\ep$ 
(i.e. when  $\mathcal{J}_\ep$ is replaced in \eqref{dpcompe3} by $\mathcal{T}_\ep$) 
are $\ep^2$-close\footnote{In the notation of Section \ref{s:resolv}, the difference of the two approximating operators on 
the left hand side of \eqref{TvsIep2} is
$
\Gamma_\ep^{-1}U^{-1}\mathcal{E}\Big[
\chi^*\Gamma_\ep^{-1}\mathbb{L}^{-1}\mathcal{P}\Gamma_\ep\chi\,-\,
{\rm Sn}^*\Gamma_\ep^{-1}\mathbb{L}^{-1}\mathcal{P}\Gamma_\ep{\rm Sn}
\Big]\mathcal{E}^{-1}U\Gamma_\ep
$. 
 Operator $\mathbb{L}^{-1}\mathcal{P}$ is direct integral of $
\mathbb{L}_\xi^{-1}\mathcal{P}^0_{\mathcal{H}_0}$, 
$\xi\in\mathbb{R}^n$, 
and one can show from its special structure that its ``symbol'' stabilises for large $\xi$, namely 
$\mathbb{L}_\xi^{-1}\mathcal{P}^0_{\mathcal{H}_0}=A_0+R(\xi)$ where $A_0$ is $\xi$-independent and 
$\|R(\xi)\|_{\mathcal{H}\to\mathcal{H}}\le C/\left(1+|\xi|^2\right)$. 
Then the parts corresponding to $A_0$ are seen to cancel, and the remaining parts via some further estimates 
 yield \eqref{TvsIep2}.}, 
 i.e. 
\be
\label{TvsIep2}
\bigl\Vert\,\mathcal{J}_\ep^* \left(\mathcal{L}_0+I\right)^{-1} \mathcal{P} \mathcal{J}_\ep F \,-\,  
		\mathcal{T}_\ep^* \left(\mathcal{L}_0+I\right)^{-1} \mathcal{P} \mathcal{T}_\ep F \bigr\Vert_{L^2(\RR^n)}\,\,\,\le\,\,\,
		C\,\ep^2\,\|F\|_{L^2\left(\RR^n\right)}, \quad \forall F\in L^2\left(\RR^n\right). 
\ee
(We remark that the above estimate holds despite $\mathcal{J}_\ep$ and $\mathcal{T}_\ep$ not being 
$L_2$-close to each other.) 
Estimate \eqref{TvsIep2} 
 implies that both approximations give operator estimate \eqref{dpcompe3}, and the underlying reasoning suggests a 
possibility for constructing 
similar approximations based on other extension operators with properties similar to those of $\chi$ and ${\rm Sn}$ for a broader  
class of examples. 
Still, we believe that our new two-scale connecting operator $\mathcal{J}_\ep$ appears here most naturally. 
Indeed, 
the extension operator $\chi$ (being the prototype of $\mathcal{J}_\ep$) naturally appears in the abstract 
setting of Theorem \ref{thm.bivariate} for arbitrary $\Theta$, while there seem no natural prototypes for ${\rm Sn}$. 

Remark finally that, like the unfolding operator, the new two-scale connecting operator 
$\mathcal{J}_\ep$ provides an equivalence link 
between two-scale convergence \cite{Ng,All} and ``conventional'' convergence: 
one can show that $F_\ep\in L^2\left(\RR^n\right)$ weakly (resp strongly) two-scale converges to 
$f_0\in L^2\left(\RR^n\times \square\right)$ if and only if $\,\mathcal{J}_\ep F_\ep$ weakly (resp strongly) 
converges to $f_0$ in $L^2\left(\RR^n\times \square\right)$. 
\end{remark}
One potential disadvantage of 
$\mathcal{J}_\ep$ specifically for the choice  $\mathcal{E}_\t=\mathcal{E}_\t^{(1)}$ 
(Remark \ref{ecalextns}), suffered in fact also by 
$\mathcal{T}_\ep$,  
is that for a given $x$ e.g.  
$x\in \ep\,\mathbb{Z}^n$, even for smooth $F(x)$ it produces a discontinuity in $y$  
on the boundary of $\square$ in the $\square$-periodic extension 
of $(\mathcal{J}_\ep F)(x,y)=F(x+\ep y)$. 
Moreover, for the related solution $u_0(x,y)=u(x)+v(x,y)$  of the two-scale limit problem \eqref{2sclimprob},  
the resulting  approximation in \eqref{dpcompe3} 
$u_\ep^{\rm appr}(x)=\mathcal{J}_\ep^*u_0(x)=u_0\big(\ep[x/\ep],\,\left\{x/\ep\right\}\big)$ is not in $H^1$ 
due to the piecewise constant dependence on the first variable. 
This appears not to pose a problem for the $L^2$ estimates like \eqref{dpcompe3}, however would cause issues in 
adopting such $\mathcal{J}_\ep$ (as well as $\mathcal{T}_\ep$) for $H^1$ estimates like 
\eqref{dp.H1est}. 
Notice however that the two-scale interpolation operator $\mathcal{I}_\ep$ does preserve the $H^1$ property. 
Indeed, it follows from its definition \eqref{2ScInterp} and the key properties of the Floquet-Bloch-Gelfand transform $U$ 
that if $u\in H^1\left(\mathbb{R}^n\right)$ then $\mathcal{I}_\ep u\in H^1\left(\mathbb{R}^n\times \square_{\rm per}\right)$ 
with $\square_{\rm per}$ denoting the periodicity torus, 
and 
$\big(\mathcal{I}_\ep\nabla u\big)(x,y)=\left(\nabla_x +\ep^{-1}\nabla_y\right)\big(\mathcal{I}_\ep u\big)(x,y)$, see also \cite{Well2009}. 
So, if we choose 
$\mathcal{E}_\t=\mathcal{E}^{(2)}_\t$ 
in Remark \ref{ecalextns}, i.e. 
with no translation in $\mathcal{J}_\ep$ outside the inclusions, the resulting approximations $J_\ep^*u_0$ would remain in 
$H^1$ at least 
separately 
in the matrix phase $M_\ep =\RR^n\backslash 
\overline{ 
B_\ep}$ and in 
the inclusion phase 
$B_\ep=\bigcup_{m\in \ZZ^n} \ep (B+m)$. 
Notice that this is exactly what would be consistent with the reduced 
$H^1$ estimates like \eqref{dp.H1est}; 
cf. also Remark \ref{H1-2slp} below. 

Notice that, for both choices of $\mathcal{J}_\ep$, in the inclusion phase i.e. for $y\in B$ it has to be the same and as in
 \eqref{I-ep-mod}. 
As per 
Remark \ref{RemShann}, for regular enough 
inputs $F(x)$ 
for $x=x_l=\ep l$, $l\in \ZZ^n$, 
this yields 
$\big(\mathcal{J}_\ep F\big)(x,y)=F(x+\ep y)$. 
The $\ep y$ shift in the argument of $F$ appears natural, as 
in the two-scale limit problem \eqref{2sclimprob} 
$x$ and $y$ are regarded as independent variables and given $u(x)$ the 
equation for $v(x,y)$ would have to be solved on the inclusion $B$ for every fixed $x$ with the right hand side 
$\big(\mathcal{J}_\ep F\big)(x,y)$. 
The role of the shift is also exposed in 
the resulting approximation 
$u_\ep^{\rm appr}(x)=\mathcal{J}_\ep^* u_0(x)=u_0\big(\ep[x/\ep],\,\left\{x/\ep\right\}\big)$ where 
$u_0(x,y)=u(x)+v(x,y)$ is the solution of the two-scale limit 
problem \eqref{2sclimprob}. Notice that a priori estimates for \eqref{2sclimprob} provide no 
control for (supported on the inclusion $B$ only) $\nabla_x v(x,y)$, which issue is happily dealt with 
by ``freezing'' in $u_\ep^{\rm appr}$ 
the $x$-variable of $v$ 
 on every isolated inclusion to $\ep[x/\ep]$. 
\begin{remark}
\label{H1-2slp}
Using the above tools, 
one can see that the approximation  $u_\ep^{\rm appr}=u_\ep^{(0)}+v_\ep^{(0)}$ in Theorem \ref{IK77-2}, entering in particular 
the new $H^1$ estimate \eqref{dp.H1est}, can also be expressed in an operator form in term of a solution  to the two-scale limit 
problem (c.f. \eqref{g-theta} in the abstract setting). 
Namely, as can be seen via \eqref{dp.app1} and \eqref{dplimp1}, 
$u_\ep^{\rm appr}(x)=u(x)+v\big(\ep[x/\ep],\,x/\ep\big)$, where $u_0(x,y)=u(x)+v(x,y)$ solves 
$
Q_0\big(u+v,\,\phi\,+\,\psi\big)\,+\,\big(u+v,\,\phi\,+\,\psi\big)_{L^2\left(\mathbb{R}^n\times\square\right)} 
\,\,=\,\,\langle \,f,\,\phi\,+\,\psi\,\rangle.  
$
Here $Q_0$ is the two-scale form \eqref{q02sc}, and $f$ is anti-linear continuous functional on 
$\mathbb{H}_0=L^2\left(\mathbb{R}^n\right)\dot{+}L^2\left(\mathbb{R}^n\times B     \right)$ 
specified by 
$
\langle \,f,\,\phi\,+\,\psi\,\rangle\,=\,
\int_{\mathbb{R}^n}\int_\square \left(\mathcal{I}_\ep F\right)(x,y)
\overline{\big[
\,{\phi(x)}+
{\psi(x-\ep y,y)}\big]}{\rm d}x\,{\rm d}y\,$. 
Then
$ \langle \,f,\,\phi\,+\,\psi\,\rangle=\left(\mathcal{G}_\ep F, \,\phi\,+\,\psi\right)_{\mathbb{H}_0}$, 
with an explicit bounded (although, in contrast to $\mathcal{I}_\ep$ and $\mathcal{J}_\ep$, not isometric) operator 
$\mathcal{G}_\ep: L^2\left(\mathbb{R}^n\right) \to {\mathbb{H}_0}$. 
Namely, $\mathcal{G}_\ep =\widehat T_\ep^*\mathcal{P}\mathcal{I}_\ep$ where projector $\mathcal{P}$ is given 
by \eqref{p-cal-dp} and $\widehat T_\ep^*:\mathbb{H}_0\to\mathbb{H}_0$ is the 
adjoint to 
the translation operator applied only to the 
$\psi$-term: 
$\widehat T_\ep\big(\phi(x)+\psi(x,y)\big)= \phi(x)+\psi(x-\ep y, y)$. 
Via a more explicit calculation, 
for the resulting connecting operator $\mathcal{G}_\ep$, 
\begin{equation}
\label{geph1}
\left(\mathcal{G}_\ep F\right)(x,y)\,\,=\,\,
\chi_{\square\backslash B}(y)\,|\square\backslash B|^{-1}\left\{
\left(\mathcal{S}_\ep F\right)(x)\,-\,\int_B\mathcal{I}_\ep F\big(x+\ep y',y'\big){\rm d}y'\right\}\,\,+\,\,
\chi_B(y)\,\mathcal{I}_\ep F(x+\ep y,y), 
\end{equation}
where $\mathcal{S}_\ep$ is the smoothing operator defined by \eqref{7.22-2} with $\chi=\chi_{\square^*}$, 
and in fact $\left(\mathcal{S}_\ep F\right)(x)=\int_\square \mathcal{I}_\ep F(x,y){\rm d}y$. 
As a result, $u_0=\left(\mathcal{L}_0+I\right)^{-1}\mathcal{G}_\ep F$, i.e. is the solution to the two-scale 
limit problem \eqref{2sclimprob} with the right hand side $g=\mathcal{G}_\ep F$. 
Comparing \eqref{geph1} with the coupled system \eqref{2sclimprob}, we observe that the actual right hand sides in the latter are: 
$\int_\square g(x,y){\rm d}y=\left(\mathcal{S}_\ep F\right)(x)$ and (for $y\in B$) 
$g(x,y)=\mathcal{I}_\ep F(x+\ep y,y)$. 
(Similarly to Example \ref{e.class}, 
cf. \eqref{7.23-2}--\eqref{7.23-3}, one can probably remove  $\mathcal{S}_\ep$ in the above.) 
We finally observe that $u_\ep^{\rm appr}=\mathcal{G}_\ep^*u_0$ where $\mathcal{G}_\ep^*$ is the adjoint of $\mathcal{G}_\ep$. 
So the approximation  in Theorem \ref{IK77-2} is provided by a self-adjoint operator: 
$u_\ep^{(0)}+v_\ep^{(0)}=\mathcal{G}_\ep^*\left(\mathcal{L}_0+I\right)^{-1}\mathcal{G}_\ep F$. 
\end{remark}

\subsubsection{Error estimates on the rate of convergence of spectral characteristics} 
\label{specestdp}
The following important 
results on the rate of convergence 
of the spectrum hold by specialising our general Theorem \ref{bivariate.spec} to the present example. 
(Notice that the resolvent estimate \eqref{dpcompe3} is  insufficient for this in its own, 
as $\mathcal{J}_\ep$ are isometric but only ``asymptotically'' unitary, 
with the gap in effect closed in the general 
Theorem \ref{t.collectivespec}.) 
We emphasise that the following theorem and other results of this subsection hold, with obvious changes, 
for arbitrary complex Hermitian matrix coefficients as mentioned below \eqref{aepdp}. 
It seems for the chosen simplest model \eqref{aepdp}, as well as for some of its generalisations with real 
(possibly matrix) valued coefficients, the estimate \eqref{specest} below as well as some of the following results can be improved, which we postpone  
to a separate investigation. 
\begin{theorem}
\label{EVestDP}
For every real $b$ there exists a non-negative constant $C(b)$ 
growing at most quadratically as $b\to+\infty$, such that 
	for every interval $[a,b] \subset (-\infty,\infty)$ one has 
	\begin{equation}
	\label{specest}
	{\rm dist}_{[a,b]} \Big( {\rm Sp}\,  \mathcal{L}_\ep,{\rm Sp}\, \mathcal{L}_0\Big) \,\,\,\le \,\,\,C(b)\,\ep,  \ \ \ \ \forall \ \ 0<\ep<1. 
	\end{equation}
	Further, for the spectrum of the above two-scale limit operator $\mathcal{L}_0$ with infinitely many gaps, 
	\be
	\label{limspdp}
	{\rm Sp}\, \mathcal{L}_0 \,\,=\,\,
	\bigcup_{m=1}^\infty \big[\,\mu_m\,,\,\lambda_m\,\big]
	\,\,=\,\, \Big\{ \lambda \notin {\rm Sp}\, \left(-\Delta_{H^1_0(B)}\right) \,:\, \beta_{\rm B}(\lambda) \,\ge\, 0  \Big\} \,\cup\, 
	{\rm Sp}\,\left(-\Delta_{H^1_0(B)}\right).
	\ee
Here $-\Delta_{H^1_0(B)}$ is the Dirichlet Laplacian on the inclusion $B$ and 
$\lambda_m$ are its eigenvalues in ascending order; 
$\mu_m$ are eigenvalues, in ascending order, of ``electrostatic problem'': find $0\ne c_m+\psi_m\in 
\CC\dot{+}H_0^1(B)$ such that 
$\int_B\nabla\psi_m\cdot\nabla\tilde\psi=\mu_m\int_B\big(c_m+\psi_m\big)\big(\tilde c+\tilde\psi\big)$, 
$\forall\,  \tilde c+\tilde\psi\in \CC\dot{+}H_0^1(B)$, 
i.e. $\mu_1=0$ and for $m\ge 2$, $\mu_m>0$ are such that 
$-\Delta\psi_m=\mu_m\left(\psi_m-|B|^{-1}\int_B\psi_m\right)$ for $\psi_m\ne 0$. 
Further, 
$
\beta_{\rm B}(\lambda) := \lambda + \lambda^2 \int_B \Big( -\Delta_{H^1_0(B)} \,-\, \lambda\,\Big)^{-1} \mathbf{e} 
$ 
is the $\beta$-function associated with $B$ 
introduced by Zhikov, see e.g. \cite{Zhi2000,Zhi2005}. 
	In particular, when $(a,b)$ is a gap in  ${\rm Sp}\, \mathcal{L}_0$  then 
	$\big[a +\,C(b)\, \ep,\, b\,-\,C(b)\,\ep\big]$ is in a gap of  $\,{\rm Sp}\, \mathcal{L}_\ep$ when $\ep \,<\, (b-a)/(2C(b))$.
\end{theorem}
\begin{corollary}
\label{HempLien-rate}
For the Hempel-Lienau \cite{HeLi} operator $\mathcal{B}_\delta u=\,-\,\nabla_y\cdot\big(B_{\delta}(y)\nabla_y u\,\big)$, 
equivalent to $\mathcal{L}_\ep$ with $\delta=\ep^2$ 
and 
$B_\delta(y)=\delta^{-1}A_{\delta^{1/2}}(y)$ $\square$-periodic and equal to $1$ in the inclusions and 
$\delta^{-1}$ in the matrix, 
\be
\label{HempLien-est}
{\rm dist}_{[a,b]} \Big( {\rm Sp}\,  \mathcal{B}_\delta\,,\,{\rm Sp}\, \mathcal{L}_0\Big) \,\,\,\le \,\,\,
C(b)\,\delta^{1/2},  \ \ \ \ \forall \ \ 
0<\delta<1. 
\ee
\end{corollary}
Theorem \ref{EVestDP} 
follows from Theorem \ref{bivariate.spec} upon 
the following specialisations for the present example: 
$ 
 \overline{\bigcup_{\theta \in \Theta} {\rm Sp} \, \mathcal{L}_{\ep,\t}}\,=\,
{\rm Sp}\, \left(\mathcal{L}_\ep+I\right)$, 
$\mathcal{L}=\mathcal{L}_0+I$, 
$\mathbf{B}_\star \,=\, -\,\Delta_{H^1_0(B)}+ I$, 
$\mathbf{B}_0-I=\mathbb{L}_0-I$ is the electrostatic operator, 
$C(b)=C_{b-1}$, 
and 
$\beta_\lambda[\mathbf{e}] \, = \,\beta_{\rm B}(\lambda - 1)$, 
see \eqref{betaform}--\eqref{6.27-2}. 
Further, 
\eqref{iksign} specialises to\footnote{Remark that, for $B=B_a$ a ball of radius $a<1/2$, 
$\beta_B(\lambda)$ is found explicitly in terms of trigonometric or Bessel functions. 
In particular, for $n=3$, 
$\beta_B(\lambda)=\lambda\left(1-4\pi a^3/3\right)+4\pi a\left(1-a\lambda^{1/2}\mbox{cotan}\left(\lambda^{1/2}a\right) \right)$, see e.g. \cite{BaKaSm} p. 419. } 
\be
\label{beta-zhikov}
\beta_B(\lambda)\,\,=\,\,\lambda\,+\,\lambda^2\, \sum_{m=1}^\infty\,
\frac{\left\vert\left\langle\phi_m\right\rangle\right\vert^2}{\lambda_m\,-\,\lambda}, 
\ee
where 
$\left\langle\phi_m\right\rangle:=\int_B\phi_m(y){\rm d}y$
and $\phi_m$ are 
$L^2$-orthonormal 
eigenfunctions of  
$-\Delta_{H^1_0(B)}$ 
corresponding to $\lambda_m$. 
This implies, cf. \cite{Zhi2000,Zhi2005}, that the spectrum of the limit operator $\mathcal{L}_0$ typically has infinitely many gaps. 
\vspace{.04in} 

It was shown in \cite{HeLi} and \cite{Zhi2005} that the Floquet-Bloch spectrum of the original 
operator $\mathcal{L}_\ep$ or $\mathcal{B}_\delta$ converges to that of $\mathcal{L}_0$ in the sense of Hausdorff. 
The estimates \eqref{specest} or \eqref{HempLien-est} provide a 
result on the uniform rate of this convergence. 
They, as well as \eqref{dpcompe3}, can be compared with 
recent results of \cite{Lipton2017} and \cite{ChErKi,ChKiVeZu23}. 
In \cite{Lipton2017}, for scalar two- and three-dimensional cases, 
robust quantitative estimates on both band gap opening and passband persistence near $\lambda_m$ with $\left\langle\phi_m\right\rangle\ne 0$ were obtained for contrasts sufficiently high  in terms of the inclusions' shape and geometry.  
The method in \cite{Lipton2017} is based on decomposition of a solution operator with a subsequent analysis of 
related quasi-periodic resonances by tools of layer potential theory. 
The resulting estimates akin to \eqref{HempLien-est} appear sharper (of order $\delta$ rather than $\delta^{1/2}$, and with  
more explicit values for $C(b)$) although restricted to the vicinities of $\lambda_m$ with $\left\langle\phi_m\right\rangle\ne 0$. 
The tools of the potential theory rely in \cite{Lipton2017} on the coefficients being constant both inside and outside the inclusions 
(and on the inclusions' boundaries to have a higher regularity than Lipschitz). 
We re-emphasise that 
our approach applies without change to problems with no regularity assumption on the generally complex-valued Hermitian matrix coefficients and inclusions with Lipschitz boundaries (as well as, with minimal modifications, to vector problems like elasticity, see footnote above \eqref{dp.differentialPDEdp}). 
As a result, it is in particular capable of obtaining in all those cases uniform quantitative estimates on the rate of 
convergence of the Floquet-Bloch spectra like \eqref{HempLien-est}. 
In \cite{ChErKi,ChKiVeZu23}  
estimates similar to \eqref{dpcompe3} and \eqref{specest}, including some of improved order $\ep^2$, have been obtained in terms of certain $\ep$-dependent approximate operators $\mathcal{L}_\ep^{app}$, while our 
estimates are 
{\it directly in terms the $\ep$-independent two-scale limit operator} $\mathcal{L}_0$. Moreover, in contrast to \cite{ChErKi,ChKiVeZu23}, 
our approach again does not 
require any regularity restrictions on the coefficients (and is also less restrictive on the regularity of the 
inclusion boundaries). 
\begin{remark} 
\label{EFestDP}
The general spectral results of Section \ref{s:resolv} imply also certain 
estimates for convergence of dispersion relations $\lambda_{\ep,\t}^{(k)}$ and related eigenfunctions, 
in the present example of 
the Bloch waves. 
Not attempting here a 
detailed investigation, 
remark that Theorems \ref{splimSe.3} and \ref{ikthm2} in combination with general methods of 
e.g. \cite{VishLus} imply the following. Let $\xi\in\mathbb{R}^n$, $k\in\mathbb{N}$ and let $\lambda^{(k)}_\xi$ be (for simplicity) a simple eigenvalue of 
$\mathbb{L}_\xi-I$ with an associated eigenfunction $\psi_\xi^{(k)}\in \CC\dot{+}H^1_0(B)$. 
Then there exist $0<\ep_0\le 1$ and $\delta>0$ such that 
$\forall\, 0<\ep<\ep_0$ and for $\t=\ep\xi\in\square^*$ 
the eigenvalue $\lambda^{(k)}_{\ep,\t}$ of $\mathcal{L}_{\ep,\t}-I$ is single and the only one 
in the $\delta$-neighbourhood of $\lambda^{(k)}_\xi$ 
with an associated eigenfunction $\varphi^{(k)}_{\ep,\t}\in H^1_{per}(\square)$ such that 
\be
\label{eigfunest}
\left\vert\,\lambda^{(k)}_{\ep,\t}\,\,-\,\,\lambda_\xi^{(k)}\,\right\vert\,\,\le\,\,C\,\ep, \quad \quad 
\left\|\,\varphi^{(k)}_{\ep,\t}\,\,-\,\,\psi_\xi^{(k)}\,\right\|_{L^2(\square)}\,\,\le\,\,C\,\ep,
\end{equation}
with a constant $C>0$ independent of $\ep$ (and for the first inequality also of $\xi$, and so for the second 
one at least if $k=1$). 
As $\,e^{\i\t\cdot y}\varphi^{(k)}_{\ep,\t}(y)$ is a $\t$-quasiperiodic Bloch wave associated with the original (rescaled) operator, 
\eqref{eigfunest} implies two-scale ($x=\ep y$) approximation of the latter by 
$\,e^{\i\t\cdot y}\psi_{\t/\ep}^{(k)}(y)=e^{\i\xi\cdot x}\psi_{\xi}^{(k)}(x/\ep)$, where $\psi_{\xi}^{(k)}(y)$ is explicitly found from the two-scale limit 
problem. 
Namely, cf. Remark \ref{remdisprel}, either $\lambda_\xi^{(k)}\in {\rm Sp}\,\left(-\Delta_{H^1_0(B)}\right) $,  
or 
$\beta_B\left(\lambda_\xi^{(k)}\right)=A^{\rm hom}_{\rm pd}\xi\cdot\xi$ and 
$\psi_{\xi}^{(k)}=c\mathbf{e}+\lambda c \left( -\Delta_{H^1_0(B)} - \lambda\,\right)^{-1} \mathbf{e}$ 
for $c\in\CC\backslash\{0\}$. 
\end{remark}

The first uniform estimate in \eqref{eigfunest} implies also the following 
new 
asymptotics 
of the integrated density of states, in the notation of \cite{Friedlander} and \cite{Selden}, $m_\tau(\lambda)$ of operator 
$\mathcal{B}_\delta$ as $\tau=\delta^{-1}=\ep^{-2}\to\infty$, uniformly valid with exception of small neighbourhoods of 
the limit spectral bands 
(and for operators with general Hermitian matrix coefficients, see below 
\eqref{aepdp}): 
\begin{corollary}
\label{denstates}
Let $k=1$ or $k> 1$ and $\lambda_{k-1}<\mu_k<\lambda_k<\mu_{k+1}$, 
and let $C$ be the constant in the first inequality in \eqref{eigfunest}. 
Then, for 
$\mu_k+2C\tau^{-1/2}<\lambda<\lambda_k-2C\tau^{-1/2}$, $\beta_B(\lambda)> 0$ and
\be 
\label{densstatasym}
m_\tau(\lambda)\,\,=\,\,k\,-\,1\,+\,\,(2\pi)^{-n}\,\omega_n\tau^{\,-n/2}\frac{\,\,\,\left[\beta_B(\lambda)\right]^{n/2}}
{\left({\rm det} A^{\rm hom}_{\rm pd}\right)^{1/2}}
\Big\{\,1\,\,+\,\,R_\tau(\lambda)\Big\}, \ \ \mbox{ where } \ 
\ee
$\omega_n$ is the volume of the unit ball in $\mathbb{R}^n$, and 
\be
\label{densstatest}
\left\vert R_\tau(\lambda)\right\vert\,\,\le\,\,
C_k\,\frac{\tau^{-1/2}}
{\left(\lambda-\mu_k\right)
\left(\lambda_k-\lambda\right)} 
\ \ \mbox{with a constant 
$C_k$ independent of $\tau
\ge 1$ and $\lambda$}. 
\ee
\end{corollary} 
\begin{proof}
Let $\mu_k<\lambda<\lambda_k$. Then, with $\lambda_{0}:=0=\mu_1$, from e.g. \eqref{limspdp} and \eqref{beta-zhikov}, 
$\beta_B(\lambda)$ is smooth on 
$(\lambda_{k-1},\lambda_k)$, $\lambda_k$ is a single eigenvalue with $\langle\phi_k\rangle\ne 0$, 
$\beta_B'(\lambda)>0$, $\beta_B(\mu_k)=0$, and $\beta_B(\lambda)>0$. 
Let $C$ be the constant in \eqref{eigfunest}, $\lambda^{\rm h}_{\rm min}>0$ the minimal eigenvalue of 
$A^{\rm hom}_{\rm pd}$, and let 
$\ep<\ep_0:={\rm min}\left\{(\lambda_k-\lambda)/(2C), (\lambda-\mu_k)/(2C), \pi\left(\lambda^{\rm h}_{\rm min}\right)^{1/2}\beta_B^{-1/2}(\lambda),\,
\left(\mu_{k+1}-\lambda_k\right)/C,\,1\right\}$. 
Notice that, from variational arguments, $\lambda_{\ep,\t}^{(k-1)}\le\lambda_{k-1}<\mu_k$, 
and for $\ep<\ep_0$, $\lambda_{\ep,\t}^{(k+1)}\ge \lambda_{\t/\ep}^{(k+1)}-C\ep\ge\mu_{k+1}-C\ep>\lambda_k$. 
Therefore, from the definition of the integrated density of states for $\tau=\ep^{-2}$, 
$m_\tau(\lambda)=k-1+(2\pi)^{-n}{\rm meas}\,S_\ep(\lambda)$ where $S_\ep(\lambda):= 
\left\{\t\in\square^*\,\big\vert \,\lambda^{(k)}_{\ep,\t}\le \lambda\right\}$. 
Consider set $E_\ep(\lambda)=\left\{\t\in\mathbb{R}^n\big\vert \lambda^{(k)}_{\t/\ep}\le\lambda\right\}$. 
As, from \eqref{finallimitspectralproblem}, $\beta_B\left(\lambda_\xi^{(k)}\right)=A^{\rm hom}_{\rm pd}\xi\cdot\xi$, 
we observe that 
$E_\ep(\lambda)=\left\{\t\in\mathbb{R}^n\big\vert A^{\rm hom}_{\rm pd}\t\cdot\t \le\ep^2\beta_B(\lambda)\right\}$  
is an ellipsoid with   
${\rm meas}\,E_\ep(\lambda)=\omega_n\ep^n\left[\beta_B(\lambda)\right]^{n/2} 
\left({\rm det} A^{\rm hom}_{\rm pd}\right)^{-1/2}$, and  $
E_\ep(\lambda)\subset\square^*$ for $\ep<\ep_0$. 
Notice that if $\t\in E_\ep(\lambda-C\ep)\subset E_\ep(\lambda)$ then 
$\lambda^{(k)}_{\t/\ep}\le \lambda-C\ep$, and 
hence via \eqref{eigfunest}, 
$\lambda^{(k)}_{\ep,\t}\le \lambda^{(k)}_{\t/\ep}+C\ep\le\lambda$ and so $\t\in S_\ep(\lambda)$. 
Hence 
$E_\ep(\lambda-C\ep)\subset S_\ep(\lambda)$, and by a similar argument 
$S_\ep(\lambda)\subset E_\ep(\lambda+C\ep)$, 
and therefore  
\be 
\label{sepbeta}
\omega_n\ep^{n}\frac{\,\,\,\left[\beta_B(\lambda-C\ep)\right]^{n/2}}
{\left({\rm det} A^{\rm hom}_{\rm pd}\right)^{1/2}}\,\,\le\,\,\mbox{meas}\,S_\ep(\lambda)\,\,\le\,\,
\omega_n\ep^{n}\frac{\,\,\,\left[\beta_B(\lambda+C\ep)\right]^{n/2}}{\left({\rm det} A^{\rm hom}_{\rm pd}\right)^{1/2}}. 
\ee 
Now 
\[
\left[\frac{\beta_B(\lambda\pm C\ep)}{\beta_B(\lambda)}\right]^{n/2}\,\,=\,\,
1\,\,\pm\,\,\ep\, C\frac{n}{2}
\left[\frac{\beta_B\left(\lambda\pm \xi_\pm\right)}{\beta_B(\lambda)}\right]^{n/2}
\frac{\beta_B'\left( \lambda\pm \xi_\pm\right))}{\beta_B(\lambda\pm \xi_\pm)}\,, 
\]
where $0<\xi_\pm<C\ep$. 
It follows from $\beta_B(\mu_k)=0$, $\beta'_B(\mu_k)>0$ and  \eqref{beta-zhikov} that for some $c_1, c_2, c_3 
>0$ and all 
$\lambda\in\left(\mu_k,\lambda_k\right)$, 
$c_1\left(\lambda-\mu_k\right)/\left(\lambda_k-\lambda\right)\le\beta_B(\lambda)\le 
c_2\left(\lambda-\mu_k\right)/\left(\lambda_k-\lambda\right)$ and 
$\beta'_B(\lambda)/\beta_B(\lambda)\le 
c_3/\left(\left(\lambda-\mu_k\right)\left(\lambda_k-\lambda\right)\right)$. 
Hence $\beta_B(\lambda+\xi_+)\le\beta_B(\lambda+C\ep)\le 
c_2\left(\lambda+C\ep-\mu_k\right)/\left(\lambda_k-\lambda-C\ep\right)$ and 
${\beta_B'\left( \lambda\pm \xi_\pm\right))}/{\beta_B(\lambda\pm \xi_\pm)}\le 
c_3/\left(\left(\lambda\pm \xi_\pm-\mu_k\right)\left(\lambda_k-\lambda\mp\xi_\pm\right)\right)$. 
Notice that, for $\ep<\ep_0$, 
$\left(\lambda+C\ep-\mu_k\right)\le \tfrac{3}{2} \left(\lambda-\mu_k\right)$ and 
$\left(\lambda_k-\lambda-C\ep\right)\ge\tfrac{1}{2}\left(\lambda_k-\lambda\right)$. 
Hence $\beta_B(\lambda+\xi_+)/\beta_B(\lambda)\le 3c_2/c_1$, and clearly 
$\beta_B(\lambda-\xi_-)/\beta_B(\lambda)\le 1$. 
Similarly, as $0<\xi_\pm<C\ep$, 
$\left(\lambda\pm \xi_\pm-\mu_k\right)\left(\lambda_k-\lambda\mp\xi_\pm\right)\ge 
\left(\lambda-\mu_k\right)\left(\lambda_k-\lambda \right)/2$, 
and hence 
${\beta_B'\left( \lambda\pm \xi_\pm\right))}/{\beta_B(\lambda\pm \xi_\pm)}\le 
2c_3/\left(\left(\lambda -\mu_k\right)\left(\lambda_k-\lambda \right)\right)$ 
which implies \eqref{densstatest}. 

It remains to show \eqref{densstatest} for any $\ep=\tau^{-1/2}$ such that 
${\rm min}\left\{ \pi\left(\lambda^{\rm h}_{\rm min}\right)^{1/2}\beta_B^{-1/2}(\lambda),\, 
\left(\mu_{k+1}-\lambda_k\right)/C\right\} \le \ep < 
{\rm min}\left\{(\lambda_k-\lambda)/(2C), (\lambda-\mu_k)/(2C), 1 \right\}$. 
Denote by $c$ some constants independent of $\lambda$ and $\ep$ (which can change values).  
If $\ep\ge \pi\left(\lambda^{\rm h}_{\rm min}\right)^{1/2}\beta_B^{-1/2}(\lambda) $ then since $m_\tau(\lambda)\le c$, 
$1+R_\tau(\lambda)=c\big[m_\tau(\lambda)-k+1\big]\ep^{-n}\left[\beta_B(\lambda)\right]^{-n/2}\le 
c$ and hence $R_\tau(\lambda)\le c\le c\,\ep\left(\lambda-\mu_k\right)^{1/2}\left(\lambda_k-\lambda\right)^{-1/2}$ and \eqref{densstatest} follows. 
Similarly, if $\ep\ge(\mu_{k+1}-\lambda_k)/C$, then (recalling $\lambda-\mu_k\ge 2C\ep\ge c$), 
$\big[m_\tau(\lambda)-k+1\big]\ep^{-n}\left[\beta_B(\lambda)\right]^{-n/2}\le 
c\,\ep^{-n}\left(\lambda-\mu_k\right)^{-n/2}\left(\lambda_k-\lambda\right)^{n/2}\le c\,\ep^{-3n/2}\le c$ 
again implying \eqref{densstatest}. 
\end{proof}
Notice that \eqref{densstatasym}--\eqref{densstatest} provides a uniform asymptotic approximation to $m_\tau(\lambda)$ 
as long as $\lambda_k-\lambda\gg\tau^{-1/2}$ and $\lambda-\mu_k\gg\tau^{-1/2}$, i.e. with exception of 
small neighbourhoods of 
both ends of the limit passbands $\left[\mu_k,\lambda_k\right]$. 
On the other hand, developing similar arguments for more elementary estimate \eqref{ik889.0} rather than \eqref{eigfunest}, 
valid here for $\t\ne 0$ with $\lambda^{(k)}_{\ep,\t}$ replaced by $\lambda^{(k)}_{\ep,\t}+1$, $\mu_\t^{(k)}=\lambda_k+1$ 
 and $\nu=\gamma|\t|^2$, one arrives at a lower bound: 
$ 
m_\tau(\lambda)\,\le\,k\,-\,1\,+\,\,(2\pi)^{-n}\omega_n\gamma^{-n/2}(1+\lambda_k)^{n}\tau^{-n/2}\left(\lambda_k-\lambda\right)^{-n/2} 
$ 
for $\lambda_k-\lambda\ge \pi^{-2}\gamma^{-1}\left(1+\lambda_k\right)^2\tau^{-1}$, 
agreeing 
with \cite{Friedlander, Selden} and demonstrating concentration of the spectral density near $\lambda=\lambda_k$. 
When $\lambda-\mu_k=O\left(\tau^{-1/2}\right)$, the 
inequalities in \eqref{sepbeta} lead to upper and 
lower (for $\lambda>\mu_k+C\tau^{-1/2}$) bounds for $m_\tau(\lambda)$. 
\begin{remark}
\label{class-hom-validity}
Again without attempting here a detailed investigation, notice 
that the uniformity of the ``non-classical'' spectral approximations \eqref{eigfunest} provides 
also certain information 
on ranges of both validity and failure of 
classical homogenisation spectral approximations for the model \eqref{th-hc} with an increasingly high contrast 
$\tau=\delta^{-1}$. 
Indeed, for a fixed 
$\delta>0$, 
the first 
Floquet-Bloch 
eigenvalue $\Lambda^{(1)}_{\delta,\t}$ of operator $\mathcal{A}_{\delta}u=\delta\,\mathcal{B}_{\delta} u=-\,\nabla_y\cdot\big(a_{\delta}(y)\nabla_y u\,\big)$ 
where $a_\delta$ is $\square$-periodic with $a_\delta(y)=1$ outside inclusion $B$ and $a_\delta(y)=\delta$ in $B$, 
is known to be approximated for small $\t$ by $A^{\rm hom}_\delta\t\cdot\t$ 
where $A^{\rm hom}_\delta$ is the homogenised matrix for $\mathcal{A}_{\delta}$. 
Namely, see e.g. \cite{BiSu06} Proposition 4.2, 
\be 
\label{estclass}
\left\vert \Lambda^{(1)}_{\delta,\t}-A^{\rm hom}_\delta\t\cdot\t\right\vert\le C(\delta)|\t|^4,  
\ \ \ \forall\,\t\in\square^*.
\ee 
On the other hand, $\Lambda^{(1)}_{\delta,\t}=\delta\, \lambda_{\delta^{1/2},\t}^{(1)}$ is approximated 
via \eqref{eigfunest} differently,  
by $\delta\lambda_\xi^{(1)}$: 
$\left\vert\Lambda^{(1)}_{\delta,\t}-
\delta \lambda_{\t/\delta^{1/2}}^{(1)}\right\vert \le C \delta^{3/2}$. 
Now, as $\beta_B\left(\lambda_\xi^{(1)}\right)=A^{\rm hom}_{\rm pd}\xi\cdot\xi$ and 
$\beta_B(\lambda)=\lambda+O\left(\lambda^2\right)$ when $\lambda\to 0$, 
one can see that 
$c_1|\xi|^4/\left(1+|\xi|^2\right)\le \left\vert\lambda_\xi^{(1)}-A^{\rm hom}_{\rm pd}\xi\cdot\xi\right\vert  
\le c_2|\xi|^4/\left(1+|\xi|^2\right)$, for all $\xi\in\RR^n$. 
Finally, by a perturbation analysis in $\delta$, one can see that 
$\left\vert A^{\rm hom}_\delta\t\cdot\t-A^{\rm hom}_{\rm pd}\t\cdot\t\right\vert\le c_3\delta|\t|^2$. 
Combining the above, we obtain 
\be
\label{lamahom} 
c_4\left[\frac{|\t|^4}{\delta+|\t|^2}\,-\,\delta^{3/2}\,-\,\delta|\t|^2
\right]\,\,\le\,\,
\left\vert \Lambda^{(1)}_{\delta,\t}-A^{\rm hom}_\delta\t\cdot\t\right\vert\,\,\le\,\,
c_5\left[ \frac{|\t|^4}{\delta+|\t|^2}\,+\,\delta^{3/2}\,+\,\delta|\t|^2
\right], \ \ \forall \t\in\square^*, \ \forall 0<\delta<1. 
\ee
Now if $|\t|\sim\delta^{1/2}$ i.e. $\xi=\t\delta^{-1/2}\sim 1$, the above implies for small enough $\delta$ that 
$\left\vert \Lambda^{(1)}_{\delta,\t}-A^{\rm hom}_\delta\t\cdot\t\right\vert\ge c_6(\xi)\delta$  
and so the classical 
approximation $A^{\rm hom}_\delta\t\cdot\t=O(|\t|^2)=O(\delta)$ fails. 
In  fact, setting $|\t|=\delta^{1/2}$ in the left inequality in \eqref{lamahom} implies 
that $C(\delta)$ in \eqref{estclass} blows up as $\delta\to 0$, at least 
as $O\left(\delta^{-1}\right)$.

On the other hand,
if $|\t|\ll\delta^{1/2}$,  
the right inequality in \eqref{lamahom} ensures the relative smallness of the error 
in the classical approximation as long as 
$\delta^{3/4}\ll|\t|\ll\delta^{1/2}\ll 1$. 
Hence the classical approximation, while failing when $|\t|\sim\delta^{1/2}$, remains valid ``just before'' that i.e. as long as $|\t|\ll\delta^{1/2}$.
For the related approximating Bloch wave 
to be of a two-scale form in \eqref{th-hc} (with $x=\ep y$), 
$U_{\delta,\t}=e^{\i\t\cdot y}\psi_{\t/\delta^{1/2}}^{(1)}(y)=
e^{\i\ep^{-1}\t\cdot x}\psi_{\t/\delta^{1/2}}^{(1)}(x/\ep)$ 
one needs $k:=\ep^{-1}\t=O(1)$ i.e. $\ep\sim\t$, which implies that
$\ep^2\ll \delta\ll\ep^{4/3}$. 
In this case also $\psi_{\t/\delta^{1/2}}^{(1)}(y)\approx \psi_{0}^{(1)}(y)$ which is a constant. 
So, to the leading order, $U_{\delta,\t}\approx ce^{\i k\cdot x}$ with $c\in \CC$, which is the (homogenised) plane wave 
with $A^{\rm hom}_\delta k\cdot k\approx\Lambda^{(1)}_{\delta,\t}/\ep^2=\rho\omega^2$.  

Notice finally that when $|\t|\gg\delta^{1/2}$, in our approximation based on \eqref{eigfunest} $\xi=\t\delta^{-1/2}$ is large and 
it follows from \eqref{beta-zhikov} that 
$\lambda_\xi^{(1)}=\lambda_1-\lambda_1^2\left\vert\left\langle\phi_1\right\rangle\right\vert^2/\left(A^{\rm hom}_{\rm pd}\xi\cdot\xi\right)+O\left(|\xi|^{-4}\right)$. 
Then, via \eqref{eigfunest}, one can conclude that 
$\Lambda^{(1)}_{\delta,\t}=\delta\lambda_1-\delta^2 
\lambda_1^2\left\vert\left\langle\phi_1\right\rangle\right\vert^2/\left(A^{\rm hom}_{\rm pd}\t\cdot\t\right)
+O\left(\delta^3|\t|^{-4}+\delta^{3/2}\right)$ which provides a valid two-term approximation for 
$\Lambda^{(1)}_{\delta,\t}$ at least as long as $\delta^{1/2}\ll|\t|\ll\delta^{1/4}$. 
(In fact, the present argument applies for any $\Lambda^{(k)}_{\delta,\t}$ with associated $\left\langle\phi_k\right\rangle\ne 0$.) 
On the other hand, our estimate \eqref{ik889.0} implies that, for the leading-order approximation, 
$\left\vert \Lambda^{(1)}_{\delta,\t}-\delta\lambda_1\right\vert\le c_7\delta^2|\t|^{-2}$. 
When $\t$ is fixed, \cite{Ammari2006, Ammari2009} constructed for both scalar and elastic high-contrast inclusions
a full asymptotic expansion in small $\delta$ for (in our notation) 
$\Lambda^{(k)}_{\delta,\t}$, 
both for $\t\ne 0$ and $\t=0$ although with a non-uniformity near $\t=0$. Further \cite{Lipton2011, Lipton2022, Lipton2022b}, for 
scalar, elastic and full electromagnetic problems,  constructed convergent 
power series in $\delta$ for $\Lambda^{(k)}_{\delta,\t}$ with a radius of convergence uniform with respect to $\t\in\square^*$, although 
again with a loss of uniformity of the actual truncated approximations  near $\t=0$. 
In this respect, our uniform approximation 
\eqref{eigfunest} may be viewed as bridging and matching to the leading 
order 
the approximations of classical homogenisation 
near $\t=0$ and those in the above references for fixed $\t\ne 0$. 
More detailed analysis of this all 
may deserve a separate investigation. 

\end{remark} 

\subsection{Stiff periodic inclusion problem} 
\label{e:idp}
Here we provide an example for which the form $a_\t$ has $\t$-regular associated spaces $V_\theta$, that is the assumption \eqref{bddspec} of Theorem  \ref{thm:contV} of Section \ref{s:uniforma} holds. 
Such an example is the 
`inverted high-contrast' model: the case where the roles of the isolated inclusions and connected matrix sets $B$ and $\square \backslash B$ respectively in \eqref{aepdp} of Example \ref{e.dp} above are swapped. 
In fact, as we will see, this results in an approximation of the original problem, with error bounds, in terms of a limit problem which in contrast to the previous two examples is not 
homogenised or two-scale but is 
with an ``infinite'' contrast inclusions 
 (e.g. with rigid inclusions in case of linear elasticity). 
We will consider here a slightly less general elliptic systems\footnote{This could be routinely extended to the case of linear elasticity for example, by  appropriate modifications in the ellipticity conditions \eqref{ivass} and in the related extension operator, cf. Example \ref{e.pdelast} below.} with a fixed period and high 
contrast $\delta^{-1}$: find $u_\delta \in H^1(\RR^n;\,\CC^m)$, $n\ge 1$, $m\ge 1$, solving 
\begin{equation}
	\label{inv-hc-prob}
	-\,\nabla\cdot\Big(A^{(\delta)}\left(y\right)\nabla u_\delta(y) \Big) \,\,+\,\, u_\delta(y)\,\, =\,\, F(y), \quad 
	y \in \RR^n. 
	\end{equation}
	Here $A^{(\delta)}(y) \,=\, \delta^{-1}\chi_{B}(y)A_1(y)\, +\, \big(1- \chi_B(y)\big) A_2(y)$, $\chi_B$ the characteristic function of $B$,  
$F \in L^2\left(\RR^n;\,\CC^m\right)$; 
$A_1$ and $A_2$ are Hermitian $\square$-periodic tensor-valued bounded coefficients 
satisfying with some 
$\gamma_0\ge 1$, 
\begin{equation}\label{ivass}
\begin{aligned}
\gamma_0^{-1} \int_{ B} | \nabla \phi|^2 & \,\,\le\,\, \int_{B}  A_1 \nabla \phi : \overline{\nabla \phi}  \,\,&\le&\,\,\,\, \gamma_0 \int_{B} | \nabla \phi|^2, \quad &\forall\,\phi& \in H^1(B;\, \CC^m), \\
\gamma_0^{-1} \int_{\square \backslash B} | \nabla \phi|^2 & \,\,\le\,\, \int_{\square \backslash B}  A_2 \nabla \phi : \overline{\nabla \phi} \,\, &\le&\,\,
 \,\,\gamma_0 \int_{\square \backslash B} | \nabla \phi|^2,  \quad &\forall\,\phi& \in H^1(\square \backslash B;\, \CC^m). 
\end{aligned}
\end{equation}

In this setting, upon applying the 
Gelfand transform, \eqref{inv-hc-prob} reduces to \eqref{p1} with: $\ep=\delta^{1/2}$, 
$H = H^1_{per}\left(\square;\, \CC^m\right)$, 
$\Theta = \square^\star : = [-\pi,\pi]^n$, 
$$
\begin{aligned}
a_\t(u,\tilde u) = \int_{B} A_1 \big(\nabla u +\i\t \otimes u \big) : \overline{\big(\nabla \tilde u+\i\t\otimes\tilde u\big) }, & \ & 
b_\t(u,\tilde u) =  \int_{\square \backslash B} A_2 \big(\nabla u+\i\t\otimes u\big) : \overline{\big(\nabla\tilde u+\i\t\otimes\tilde u\big)} +  
\int_\square u\cdot\overline{\tilde u}. 
\end{aligned}
$$
Now, 
for each $\theta \in\square^\star$ and $u\in H$, one has
\[
a_\t[u] \,\,=\,\, \int_B A_1 \nabla \big( e^{\i \t \cdot y} u\big) : \overline{\nabla \big( e^{\i \t \cdot y} u\big) } \,\,\ge\,\, 
\gamma_0^{-1} \int_B \big\vert \nabla \left( e^{\i \t \cdot y} u \right) \big\vert^2
\]
whence, assuming for simplicity $B$ connected, 
\[
V_\theta \,=\, \Big\{ v \in H^1_{per}\big(\square;\,\CC^m\big) \, \,\Big\vert \, \, \text{$v(y) = e^{-i\t\cdot y} c$, $\,y \in B$, for some constant $c \in \CC^m$} \Big\},
\]
and $W_\t$ is the orthogonal complement of $V_\t$ in $H=H^1_{per} (\square ;\, \CC^m)$ with respect to the inner product $a_\t + b_\t$. 
Let us now show that \eqref{KA} holds uniformly on $\square^\star$, i.e. condition \eqref{bddspec} is satisfied. 
\begin{proposition}
	\label{vcontinverdp} There exists a constant $\nu >0$ independent of $\t\in\Theta=\square^\star$ such that
\[
\nu \big( \,a_\t[w] \,+\, b_\t[w]\,\big)\,\,\,\le\,\,\, a_\t[w], \qquad \forall w \in W_\t, \ \ \ \forall\,\t\in\Theta. 
\]	
\end{proposition}
\begin{proof}
Let $E:H^1(B) \rightarrow H_0^1(\square)$ be a Sobolev extension, cf. Proposition \ref{prp.zhiext}, 
and for any fixed $u \in H^1_{per}(\square;\,\CC^m)$ and $\t \in \square^*$, consider 
$v = u - e^{-\i\t\cdot y} E \big( e^{\i\t\cdot y} u - |B|^{-1}\int_B e^{\i\t\cdot y'} u(y') \, {\rm d}y'\big)$ 
with the extension $E$ acting component-wise. Clearly $v \in V_\t$ and we readily estimate 
\[
\|u\,-\,v\|_\t^2\,\,=\,\,
a_\t[u - v ]+b_\t[u-v] \,\,\le\,\, \gamma_0 \left\|  E \left( e^{\i\t\cdot y} u - |B|^{-1}
\int_B e^{\i\t\cdot y'} u(y') \, {\rm d}y'\right) \right\|_{H^1(\square)}^2  \,\,\,\le 
\]
\[
 \gamma_0 C_E^2 \left\|   e^{\i\t\cdot y} u - |B|^{-1}\int_B e^{\i\t\cdot y'} u(y') \, {\rm d}y' \right\|_{H^1(B)}^2\,\,
\le\,\, \gamma_0\, C_E^2\, C_B^2 \Big\|\,  \nabla \left(  e^{\i\t\cdot y} u\right)\, \Big\|_{L^2(B)}^2 \,\,\le\,\, \gamma_0^2\,\, C_E^2\,\, C_B^2\,\, a_\t[u], 
\]
where $C_E$ and $C_B$, respectively, are the $H^1$-operator norm of $E$ and the  Poincar\'{e}-Wirtinger  (Poincar\'{e} inequality with mean) constant for domain $B$. Hence,
 for $u=w\in W_\t$, $\|w\|_\t^2\,\le\, \|w-v\|_\t^2\,\le\, \gamma_0^2\, C_E^2\, C_B^2\, a_\t[w]$, and 
 the desired inequality holds with $\nu = \big(\gamma_0\, C_E\, C_B\big)^{-2}$.
\end{proof}
Hence, 
for these class of problems, the 
approximation 
is given by Theorem \ref{thm:contV}. 
Employing like in the previous examples the inverse Gelfand transform, 
the following result holds.
\begin{theorem}
	For any fixed $0<\delta
	<1$ let $F \in L^2\left(\RR^n;\,\CC^m\right)$ and $u_\delta \in H^1(\RR^n;\,\CC^m)$ solve  the elliptic PDE system 
	\eqref{inv-hc-prob}. 
	Consider $v(\t,\cdot) \in V_\theta$ the (unique) solution 
	to
	\[
	\int_{\square \backslash B} A_2(y) \big(\nabla v(\t,y)+\i\t\otimes v(\t,y)\big)  : \overline{\big(\nabla \phi(y)+\i\t\otimes \phi(y)\big)} \, {\rm d}y \,+\,  \int_\square v(\theta,y) \cdot \overline{\phi(y)}  \, {\rm d}y \,=\, \int_\square   U 
	F(\theta,y)\cdot \overline{\phi(y)} \, {\rm d}y, 
	\]
for all $\phi \in V_\theta$, a.e. $ \t \in \square^\star$. Then, for the approximation 
$u^{(0)}:= 
U^{-1} v$, 
inequality \eqref{errorcontinuouscase2} implies the following: 
\begin{equation}
\label{inv-dp-est}
\gamma_0^{-1} \int_{\RR^n}   \big|	\nabla u_{\delta} \,-\, \nabla u^{(0)} \big|^2  \,\,+\,\, 
\int_{\RR^n} \big|u_\delta \,-\, u^{(0)} \big|^2 \,\,\, \le\,\,\, \nu^{-2} \delta^2 \int_{\RR^n} |F|^2.
\end{equation}
	The limit  $u^{(0)}\in \Big\{u\in H^1\left(\RR^n\right): 
	\,\nabla u=0 \,\,\,{\rm in } \,\,B_1:= \bigcup_{l \in \ZZ^n} (B+l)\Big\}=:H^{(0)}$
	is the solution of the infinite-contrast stiff inclusion problem, cf. e.g. \cite{JKO} \S 3.2, in weak form: 
	\[
	\int_{\RR^n\backslash B_1} A_2(y)\,\nabla u^{(0)}\cdot\nabla \widetilde u\,+\,
	\int_{\RR^n}u^{(0)}\,\widetilde u\,\,=\,\,\int_{\RR^n}F\,\widetilde u, \ \ \ \ \ \ 
	\forall\, \widetilde u\in H^{(0)}. 
	\]
	As $\nabla u^{(0)}$
	vanishes in the inclusion set $B_1$, 
	one in particular has 
	$  \int_{B_1}   \left|	\nabla u_{\delta} \right|^2   \,\le\, \gamma_0 \nu^{-2} \delta^2 \int_{\RR^n} |F|^2$. 
\end{theorem}
To our 
knowledge, this 
result is not found in previous 
literature. 
It also 
implies 
an approximation with error estimates 
for the Floquet-Bloch spectrum of the high-contrast problem (see  Section \ref{s.spcontV}) in terms of that for the limit stiff problem. 
Estimate \eqref{inv-dp-est} can in fact be proved by applying Theorem \ref{thm:contV} {\it directly} to 
problem \eqref{inv-hc-prob} i.e. without applying Gelfand transform. 
Moreover, the periodicity assumption can be relaxed, requiring instead for the (for simplicity identical) inclusions to be uniformly separated.  
\begin{remark}
If, for $\ep=\delta^{1/2}$, one makes in \eqref{inv-hc-prob} change of variable $x=\ep y$, i.e. applies scaling transformation 
$\Gamma_\ep^{-1}$ 
then $\hat u_\ep:=\Gamma_{\ep}^{-1}u_{\ep^2}$ is the 
solution of the following ``inverted'' high-contrast $\ep$-periodic problem equivalent to \eqref{inv-hc-prob}: 
$-\,\nabla_x\cdot\big(A_\ep\left(x/\ep\right)\nabla_x \hat u_\ep \big) \,+\, \hat u_\ep\,\, =\,\hat F$, 
where $A_\ep(y) \,=\, \chi_{B}(y)A_1(y)\, +\, \ep^2\big(1- \chi_B(y)\big) A_2(y)$, 
and $\hat F=\Gamma_\ep^{-1} F$. 
Then \eqref{inv-dp-est} transforms into an equivalent estimate for $\hat u_\ep$, with 
corresponding approximation 
$\hat u^{(0)}_\ep=\Gamma_\ep^{-1} u^{(0)}$.  
In dimension $n=1$, the inverted high-contrast model is equivalent to the `direct' one. 
The quantitative homogenisation of the  scalar ($m=1$) one-dimensional case 
was studied, by different means,  in 
\cite{ChChCo, ChKi}. 
\end{remark} 


\subsection{Periodic inclusions with imperfect interfaces}
\label{sec:impint}
Here we give example of 
a model with `weakly bonded' imperfect interfaces (rather than with a high contrast). 
For homogenisation of problems with imperfect interfaces, including the spectral problems, in various asymptotic 
regimes see e.g. 
recent works \cite{Donato21}, \cite{Assier} and \cite{Avila.J-24} and further references therein. 
Here we apply our approach to obtain new operator error estimates for such kind of problems. 
As we will see, this fits our general scheme with the resulting two-scale limit problem of a ``two-phase'' 
macroscopic type. 
In this case, 
the space $V_\t$ has a removable singularity at the origin but $V_\t$ is not piece-wise constant, i.e. \eqref{contVs} holds but the conditions of Remark \ref{constV} do not\footnote{Another such example, in the context of linear elasticity, is in Section \ref{e.pdelast} below.}. 

Let 
the reference inclusion $B\subset\square$ be as in Example \ref{e.dp}, $B_\ep = \bigcup_{m\in\ZZ^n}\ep (B+m)$ the set 
of associated $\ep$-periodic inclusions 
and its complement $M_\ep :=\RR^n \backslash \overline{B_\ep}$ the connected matrix, and 
let $n_\ep$ be the 
unit normal to  the interface $I_\ep=\partial B_\ep$ exterior to the inclusions. For 
$0<\ep <1$, consider the resolvent problem: 
Given $F\in L^2(\RR^n)$, find in the matrix and in the inclusions  
 $u_1^\ep$ and $u_2^\ep$ respectively, such that 
\begin{equation}
\label{impintclass}
\begin{aligned}
-\, \Delta\, u_1^\ep \,+\, u_1^\ep \,\,=\,\, F \ \text{in $M_\ep$}, \quad  - \,\Delta\, u_2^\ep \,+\, u_2^\ep \,\,=\,\, F \ \text{in $B_\ep$}, 
\quad \partial_{n_\ep} u_1^\ep \,=\, \partial_{n_\ep} u_2^\ep \,=\, \ep\, \left(u_2^\ep \,-\, u_1^\ep\right) \ \text{on $I_\ep$},
\end{aligned}
\end{equation}
where $\partial_{n_\ep}$ denotes the normal derivative. 
The problem admits equivalent variational formulation: Find $u_\ep \in W_\ep$ where (with a slight abuse of notation) $W_\ep := L^2(\RR^n) \cap  H^1(M_\ep)  \cap H^1(B_\ep)$, such that
\begin{equation}\label{p.ic}
\int_{M_\ep} \nabla u_\ep \cdot \overline{\nabla \phi}\,\, + \int_{B_\ep} \nabla u_\ep \cdot \overline{\nabla \phi}\,\, + 
\ep \int_{I_\ep} \left[ u_\ep\right]_\ep \overline{\left[{\phi}\right]_\ep}\,\, + \int_{\RR^n}  u_\ep\, \overline{\phi} \,\,\,=\,\, 
\int_{\RR^n} F\, \overline{\phi}, \quad \forall \phi \in W_\ep. 
\end{equation}
In \eqref{p.ic} 
$[u]_\ep$ denotes the jump in $u$ across $I_\ep$, i.e. $[u]_\ep := T_\ep^+ u - T_\ep^- u$ where 
$T^+_\ep : H^1\left(M_\ep\right) \rightarrow L^2\left(I_\ep\right)$ and 
$T^-_\ep : H^1\left(B_\ep\right) \rightarrow L^2\left(I_\ep\right)$ are the trace operators. 

We take our usual approach and restate problem \eqref{p.ic} in the form \eqref{p1} via the transforms $\Gamma_\Ep$ and $U$ (see Example \ref{e.class}). 
Then $u_{\ep,\t} := U \Gamma_\ep u_\ep(\t,\cdot)$  is the solution to
	$\ep^{-2} a_\t \left(u_{\ep,\t}, \tilde{u}\right)\, +\, b_\t\left(u_{\ep,\t},\tilde{u}\right)  \,\,=\,\, \l f,\tilde{u} \r$, 
	$\forall \tilde{u} \in H$, a.e. $\theta \in \Theta$,
where  $u_{\ep,\t}\in H =  L^2(\Box) \cap H^1_{per}(\Box\backslash B) \cap H^1(B)$,  $\Theta := \Box^*$, $\l f, \tilde{u} \r := (U \Gamma_\ep F(\t,\cdot),\tilde{u})$ with $(\cdot,\cdot)$ the standard $L^2(\Box)$ inner product, 
\begin{equation}
\label{abimpint}
\begin{aligned}
a_\t[u] \,\,=\,\, \int_{\Box \backslash B} \big| (\nabla + \i \t ) u \big|^2 &\,+\,  
\int_{B} \big| (\nabla + \i \t ) u \big|^2, \quad \text{and } \quad b_\t[u] \,=\, b[u]\,\, :=\,\, \int_{\partial{B}} \big|\, [u]\, \big|^2\,+\, 
\int_\Box |u|^2,
\end{aligned}
\end{equation}
where  
 $[u]$ denotes the jump in $u$ across the interface $\partial B$.  

Let us 
show that all our general assumptions hold. It is clear, cf. \eqref{2.1classhom}, that the basic assumptions  \eqref{as.b1} and \eqref{ass.alip}  hold.  Next, 
denoting $\chi_B$ the characteristic function of the inclusion $B$, observe that 
\begin{equation}
\label{vwimpint}
V_\t \,=\, \left\{ 
\begin{array}{lr}
{\rm Span}\,  \left\{e^{-\i \t \cdot y}\chi_B\right\} & \t \neq 0, \\ 
\vspace{-.08in} \\ 
 {\rm Span}\, \left\{  \mathbf{e}, \, \chi_B \right\}  & \t = 0,
\end{array}
\right. \quad 
W_\t \,=\, \left\{ 
\begin{array}{lr}
\Big\{ w \in H \, \big| \, \int_{\partial B} [w] e^{\i \t \cdot y}  \,= \, \int_B we^{\i \t \cdot y}  \Big\}  & \t \neq 0, \\ 
\vspace{-.08in} \\
\Big\{ w \in H \, \big| \, \int_{\partial B} [w]  \,=\, \int_{ B} w \text{ and }  \int_{\Box }  w =0    \Big\}  & \t = 0.
\end{array}
\right.
\end{equation}
Therefore $V_\t$ has a discontinuity at $\t=0$ and 
varies with 
$\t$ in $\Box^* \backslash \{ 0 \}$. 
Now show 
\eqref{KA}--\eqref{H6}. 

\textbullet \, {\it Proof of \eqref{KA}.} 
Standard arguments show that
there exists a constant $C_{B}>0$ such that
\begin{equation}\label{ic.h1'e1}
b[\phi] \,\,\le\,\, C_{B} \left( a_0[\phi] \,+\, \bigg|  \int_{\Box \backslash B} \phi \bigg|^2 \,+\,  
\bigg|  \int_{B} \phi \,-\, \int_{\partial B} [\phi] \bigg|^2 \right), \quad \forall \phi \in L^2(\square) \cap H^1(
\square \backslash B) \cap H^1(B).
\end{equation}
Indeed, if $\phi_n$ with $b\left[\phi_n\right]=1$ are such that \eqref{ic.h1'e1} is violated with $C_B$ replaced by $n$, then $a_0\left[\phi_n\right]\to 0$ and so 
$\left\{\phi_n\right\}$ are bounded in $H^1(\square \backslash B)$ and $H^1(B)$. So, up to a subsequence, $\left\{\phi_n\right\}$ converges  
$H^1$-weakly and  $L^2$-strongly to some $\phi_0$. Then $b\left[\phi_0\right]= 1$  by the $L^2$-compactness of the trace operators, 
and by the weak lower-semicontinuity $a_0\left[\phi_0\right]= 0$ 
so $\phi_0\in V_0$ i.e. 
$\phi_0=c_1 +c_2\chi_B$. On the other hand, by the compactness, for $\phi=\phi_0$ both other terms on the right hand side of \eqref{ic.h1'e1} 
are zero, which implies $c_1=c_2=0$ i.e. 
$b\left[\phi_0\right]= 0$ which is a contradiction. 

We then show that \eqref{KA2} holds for $c(u,\tilde{u}) =\, C_{B}\, |\Box \backslash B| \int_{\Box \backslash B} u\,  \overline{\tilde{u}}$ and $C = C_{B}+1$. Indeed for fixed $\t \in \Theta$ and  $w \in W_\t$, see \eqref{vwimpint}, $\int_{B} w\,e^{\i \t \cdot y} \, - \int_{\partial B} \left[w\,e^{\i \t \cdot y}\right]  =0$ 
and so \eqref{ic.h1'e1} for $\phi = e^{\i \t \cdot y} w$ gives
\begin{equation}\label{ic.h1proof}
\begin{aligned}
b[w] & \,\,=\,\, b\left[e^{\i \t \cdot y} w\right]\,\,\le\,\, C_{B} \left( a_0\left[e^{\i \t \cdot y} w\right] \,+\, \bigg|  \int_{\Box \backslash B}  e^{\i \t \cdot y} w \bigg|^2\right) \,\,\le\,\,  C_{B}  \left( a_\t[w] \,+\,  |\Box \backslash  B|  \int_{\Box \backslash B} | w |^2 \right).
\end{aligned}
\end{equation}
Hence \eqref{KA2} holds (as form $c$ is clearly $\|\cdot\|_\t$-compact), and so \eqref{KA} holds too by Proposition \ref{prop.kaequiv}.

\textbullet \,  {\it Proof of \eqref{contVs}.} We set $V_\star={\rm Span}\, \left\{\chi_B\right\}$, and so 
$V^\star_\t : = {\rm Span}\, \left\{e^{-\i \t \cdot y}\chi_0\right\}$ for all $\t \in \Theta= \Box^*$. 
Then, 
for given $V^\star_{\t_1} \ni v_1 = c_1 e^{-\i \t_1 \cdot y}\chi_B $, $c_1\in \CC$,  set $v_2 =c_1 e^{-\i \t_2 \cdot y}\chi_B \in V_{\t_2}^\star$ and so 
$\left| v_1(y) - v_2(y) \right| \le\,\left(n^{1/2}/2\right) \left| \t_1 - \t_2\right|\, \left| v_1(y) \right|$, $y \in \Box$. Therefore, since for 
$i=1,2$, 
$a_{\t_i}\left[ v_i\right]=0$ and $\left[v_i\right] = \,-\,v_i$,   
\[
\bigl\| v_1 - v_2\bigr\|_{\t_2}^2 \,\,=\,\, \left|\t_1-\t_2\right|^2 \int_\square \left| v_1 \right|^2 \,\,+\, 
\int_{\partial{B}} \left| v_1 - v_2\right|^2 \,\,+\, \int_\square \left|v_1 - v_2\right|^2\,\,\,\le\,\, \left(\frac{n}{4}+1\right) \left|\t_1-\t_2\right|^2\, 
\left\| v_1 \right\|^2_{\t_1},
\]
and so \eqref{contVs} holds 
with $L_\star = \left({n}/{4}+1\right)^{1/2}$. Furthermore, 
one can naturally choose  $Z = {\rm Span}\, \{\mathbf{e}\}$ so that  \eqref{spaceZ} and \eqref{VZorth} 
hold with $K_Z=|B|/\big(|B|+|\partial B|\big)^{1/2} < 1$).


\textbullet \, Assumption \eqref{distance} 
follows by applying \eqref{ic.h1proof} and then \eqref{dpH1}, with e.g. 
$\gamma =\left(1+C_{B}\right)^{-1}\left(n\pi^2  + |\Box \backslash B| C_E^2\right)^{-1} $. 

\textbullet \, Assumption \eqref{H4} is obviously satisfied, see \eqref{abimpint}, with  $K_{a'}=1, K_{a''}=0$, and 
\begin{equation}
\label{impint.a'}
\begin{aligned}
& a'_{0}(v, u) \cdot \t\,\, :=\,\,
 \int_{\Box \backslash B}  \i\, \t v \cdot \overline{ \nabla u} \,+\,  \int_{B}  \i\, \t v \cdot \overline{ \nabla u},
\quad \text{and} \quad 
& a''_{0}\,\left(v, \tilde{v}\right)\,\t \cdot \t\,\,:=\,\, |\t|^2 \int_{\Box } v\, \overline{\tilde{v}}.
\end{aligned}
\end{equation}

\textbullet \, Assumption \eqref{H5} trivially holds with $L_b =0$ since $b_\t$ is independent of $\t$.  Furthermore, it is clear that
for $\mathcal{E}_\t$ defined as multiplication on $V_\star$ by $e^{-\i\, \t\cdot y}$ 
 \eqref{Eprop1} and \eqref{Eprop2} hold with $K_b = n^{1/2}/2$.  

\textbullet \,  Finally, \eqref{H6} also 
holds  with $\mathcal{H} = L^2(\square)$, $d_\t$ the standard $L^2$-inner product, 
$\mathcal{E}_\t$ 
multiplication by $e^{-\i \t\cdot y}$ on $L^2(\square)$\footnote{We choose this for simplicity: 
 other extensions are also possible, cf. Remark \ref{ecalextns}.}
and $K_e = n^{1/2}/2$. Notice that the above choice of 
$Z={\rm Span}\, (\mathbf{e})$ satisfies \eqref{zvbd}.

As all the main assumptions \eqref{KA}--\eqref{H6} hold, our general results are applicable to the present example, in particular Theorem \ref{thm.IKunifest2} and Theorem \ref{ikthm2}. We will illustrate below the spectral results related to the latter, leaving it to the reader specialising any of our other general results to the present setting. 

Notice first that for the solution $u_\ep$ to \eqref{impintclass}, equivalently \eqref{p.ic}, $u_\ep=\left(\mathcal{L}_\ep+I\right)^{-1}F$, where $\mathcal{L}_\ep$ is 
the self-adjoint operator in $L^2(\RR^n)$, with  standard inner product, which is generated by the form 
\[
Q_\ep\left(u,\,\tilde u\right)\,\,=\,\,\int_{M_\ep} \nabla u\cdot\overline{\nabla \tilde u} \,\,+\, 
\int_{B_\ep} \nabla u\cdot\overline{\nabla \tilde u}\,\,\,+\, \ep \int_{I_\ep} \left[ u_\ep\right]\,\overline{\left[ \tilde u_\ep\right]}, 
\]
with the form domain $W_\ep$. We are interested in  the spectrum  ${\rm Sp}\, \mathcal{L}_\ep$   of $\mathcal{L}_\ep$. 

  Upon  applying the unitary rescaling $\Gamma_\ep$ and the Gelfand transform $U$, it follows that the spectrum ${\rm Sp}\, \mathcal{L}_\ep$ is equal to 
	$\overline{\bigcup_{\t \in \square^*}\mathbf{L}_{\ep,\t} }$ where $\mathbf{L}_{\ep,\t} $ is  the self-adjoint operator generated in $L^2(\square)$ by the form
 \[
q_{\ep,\t}\left(u,\,\tilde u\right)\,\,=\,\,
 \ep^{-2} \left(  \int_{\Box \backslash B}  (\nabla + \i \t) u \cdot\overline{\left(\nabla+\i \t \right) \tilde u}\,\,\,  +\,  
\int_{B} (\nabla + \i \t) u \cdot\overline{\left(\nabla+\i \t \right) \tilde u} \right)\,\,\,+\,
\int_{\partial B}[u]\,\,\overline{\left[\tilde u\right]} 
 \]
 with the form domain $H = L^2(\square) \cap H^1_{per}( \square \backslash B) \cap H^1(B)$.  The spectrum of $\mathbf{L}_{\ep,\t}$ consists of countably many nonnegative real eigenvalues $\{ \lambda_{\ep,\t}^{(k)}\}_{k \in \NN}$ labelled in the increasing order accounting for their multiplicity. The functions $E^{(k)}_\ep(\t) : = \lambda_{\ep,\t}^{(k)}$, $\t\in\square^*$, are the spectral band functions of $\mathcal{L}_\ep$. 
 

Theorem \ref{ikthm2} (with $\mathcal{L}_{\ep,\t}=\mathbf{L}_{\ep,\t}+I$) 
provides the asymptotics of 
$E^{(k)}_\ep$  in terms of the eigenvalues of  $\mathbb{L}_{\t / \ep} $, which in turn describes  the approximation of ${\rm Sp} \, \mathcal{L}_\ep$ by 
 $\overline{\bigcup_{\xi \in \ep^{-1} \square^*} {\rm Sp}\, \left(\mathbb{L}_\xi-I\right)}$ or, via Corollary \ref{c.collspec}, by  $\overline{\bigcup_{\xi \in \RR^n } {\rm Sp}\, \left(\mathbb{L}_\xi-I\right)}$. Finally, we have the characterisation of $\overline{\bigcup_{\xi \in \RR^n } {\rm Sp}\, \mathbb{L}_\xi}$ given in Theorems 
\ref{limspecsimple} or \ref{thm.limspecrep} thus completing our asymptotic analysis with error 
estimates for  
${\rm Sp}\, \mathcal{L}_\ep$ and its spectral band functions. 

Now, we follow the above steps with more detail. First we need to specify the limit operator $\mathbb{L}_\xi$ given by  the form \eqref{Sform}. 
To  determine the homogenised form $a^{\rm h}_\t$, defined by \eqref{defhom.form}, we first need the corrector $N_\t=\t\cdot N$, as specified by \eqref{cell:prob2}, on $Z$. 
For this, one can see from \eqref{cell:prob2} and \eqref{impint.a'}, cf. \eqref{dp.cell1}, that (up to an element of $V_0$) 
$\mathbf{e} \mapsto  (N\mathbf{e})(y) =  \i\,\left(1-\chi_B\right) \ourN^{\rm pd}(y) - \i\, y \chi_B $ 
 where $\ourN^{\rm pd}$ solves \eqref{dp.correctorproblem}. As a result $a^{\rm h}_\t$, which is fully determined 
on $Z$ by $a^{\rm h}_\t [\mathbf{e}]$, 
is found, cf. \eqref{dp.a'2}, to be in the form 
$a^{\rm h}_\t [\mathbf{e}] = A^{\rm hom}_{\rm pd} \t \cdot \t$ where $A^{\rm hom}_{\rm pd}$ is the perforated domain homogenised matrix 
  given by \eqref{dp.coef} in Example \ref{e.dp}. 
Further,  
we have $V_\star + Z \,=\, \overline{V_\star + Z} \,=\, \mathcal{H}_0 = \big\{ c_1 \,+\,c_2\,\chi_B  \,\, \big| \,\, c_1,c_2 \in \CC \big\}$. 
Putting 
this together, 
form \eqref{Sform} specialises to 
\[
\begin{aligned}
& \mathbb{S}_\xi \big(c_1+c_2 \chi_B \,,\, \tilde{c}_1 + \tilde{c}_2 \chi_B\big) \,\,=\,\, 
M_\xi\left( \begin{matrix}c_1 \\ c_2  \end{matrix} \right) \cdot  \overline{\left( \begin{matrix} \tilde{c}_1 \\ \tilde{c}_2  \end{matrix} \right)}, \quad \text{ where } M_\xi \, =\,\left( \begin{matrix}
	A^{\rm hom}_{\rm pd} \xi \cdot \xi + 1 & |B| \\ |B| & |\partial  B| + |B| 
\end{matrix} \right), \\
&d_0\big(c_1+c_2 \chi_B ,\, \tilde{c}_1 + \tilde{c}_2 \chi_B\big) \,\,=\,\, D\left( \begin{matrix}c_1 \\ c_2  \end{matrix} \right) \cdot  \overline{\left( \begin{matrix} \tilde{c}_1 \\ \tilde{c}_2  \end{matrix} \right)}, \quad \text{ where } D \, =\,\left( \begin{matrix}
1 & |B| \\ |B| &  |B| 
\end{matrix} \right). 
\end{aligned}
 \]
As a result, see \eqref{lxispec}, the eigenvalues of $\mathbb{L}_\xi$ are the solutions of the generalised 
eigenvalue problem: 
$ 
 M_\xi c \,\,=\,\, \lambda\, D c$ 
for some non-trivial $c \in \CC^2$. 
Hence $\lambda$ are roots of the polynomial ${\rm det} \left( M_\xi -\lambda D\right)$,  
however we shall determine $\lambda$ by using the representation
\eqref{finallimitspectralproblem} with $\beta_\lambda$ given by \eqref{betaform}--\eqref{6.27-2}.  
We note that $V_\star = \overline{V_\star} = {\rm Span}\, \left\{ \chi_B \right\}$,  
$b_0[\chi_B] = |\partial B| + |B|$ and 
$d_0[\chi_B] = |B|$. Therefore, the operator $\mathbf{B}_\star$ (defined in \eqref{ikddd2}, i.e. as the operator in $\overline{V_\star}$ with inner product $d_0$ generated by $b_0$ with form domain $V_\star$) is simply the multiplication by $1 + |\partial B | / |B|$. 
Hence ${\rm Sp}\, \mathbf{B}_\star = \{ 1 + \mu_0 \}$, where $\mu_0 := |\partial B| / |B|$.
Further, $\overline{Z} = Z = {\rm Span}\, ( \mathbf{e} )$, $b_0[\mathbf{e}] =  d_0[\mathbf{e}] =1$, $\mathcal{P}^0_{\overline{V_\star}}\, \mathbf{e} =  \chi_B$,  
$(\mathbf{B}_\star-\lambda I)^{-1}$ is multiplication by $(1 + \mu_0 - \lambda)^{-1}$ 
and $\mathcal{P}^0_{\overline{Z}}\chi_B=|B|\mathbf{e}$. 
Thus, via \eqref{6.27-2}, $d_0\big(\beta(\lambda) \mathbf{e},\,\mathbf{e}\big) = \lambda + (\lambda-1)^2 (1 + \mu_0 - \lambda)^{-1}|B|$ and so 
(see \eqref{betaform}) 
\be
\label{fimpint}
\beta_\lambda[\mathbf{e}] \,\,=\,\, \Phi(\lambda - 1), \quad \text{for}\quad \Phi(\mu) \,\,:=\,\,
 \frac{\mu}{\mu_0 - \mu} \Big(\mu_ 0 \,\,-\,\, \mu \big(1 \,-\, |B|\big)\,\Big). 
\ee
Then, from \eqref{finallimitspectralproblem}, for each $\xi$ the eigenvalues $\lambda$ of $\mathbb{L}_\xi$ are the two solutions of  the dispersion relation
\begin{equation}
\label{22.12.20e1}
A^{\rm hom}_{\rm pd} \xi \cdot \xi  \,\,=\,\, \Phi(\lambda - 1),
\end{equation}
i.e. $\lambda - 1$ are the roots of quadratic polynomial 
$p(\mu) = |\square \backslash B| \mu^2 - \mu\big(A^{\rm hom} \xi \cdot \xi + \mu_0\big) + \mu_0 A^{\rm hom} \xi \cdot \xi$.  
(Note that for every $\xi\in\mathbb{R}^n$, $p(\mu)$ has two distinct nonnegative roots, and 
$p(\lambda-1) = |B|^{-1} {\rm det} \left( M_\xi -\lambda D\right)$.) 

Theorem \ref{ikthm2}, wherein 
$\mathcal{L}_{\ep,\t}= \mathbf{L}_{\ep,\t}+I$, now implies the following for the spectral bands of $\mathcal{L}_\ep$:
\begin{theorem}	
Let $\left\{ \lambda^{(k)}_{\ep,\theta}\right\}_{k\in \NN}$ be the eigenvalues of $\mathbf{L}_{\ep,\t}$, and $1 \le \Lambda^{(1)}_\xi < \Lambda^{(2)}_\xi$  the eigenvalues of $\mathbb{L}_\xi$ i.e. 
the roots of \eqref{22.12.20e1}. Then, there exists a positive constant $C$ independent of $\ep$, $\t$ and $k$ such that
 	\begin{gather}
	\label{est1.eigrotballs}
 	\Big| 1/\big(\lambda_{\ep,\theta}^{(k)}+1\big) \,-\,  1/\Lambda_{\theta / \ep}^{(k)}  \Big|\,\, \le\,\, C\ep, \ k = 1, 2,  \quad \text{and} \quad    
	\Big| 1/\big(\lambda_{\ep,\theta}^{(k)} +1\big)\Big| \,\,\le\,\, C\ep, \quad \forall  k \ge 3, \qquad \forall \t \in \square^*.
 	\end{gather}
 \end{theorem}

Now, we can apply Theorem \ref{limspecsimple} to approximate ${\rm Sp}\, \mathcal{L}_\ep$. 
 Indeed from the above, in the notation of Theorem \ref{limspecsimple},  
$N={\rm dim}\,V_\star=1$ and $\lambda_\star^{(1)}=1+\mu_0$. 
Further, for the two eigenvalues of $\mathbf{B}_0=\mathbb{L}_0$, setting $\xi=0$ in \eqref{22.12.20e1} 
we see from \eqref{fimpint} that $\lambda_0^{(1)}=1$ and 
$\lambda_0^{(2)}=1+\mu_0 / (1 - |B|) = 1+\mu_0 + \mu_1$ where $\mu_1 = |\partial B| / |\square \backslash B|>0$. 
Hence, from \eqref{valthm2},  
$\overline{\bigcup_{\xi \in \RR^n} {\rm Sp}\, ( \mathbb{L}_\xi - I)} \,=\,  \left[0,\,\mu_0\right] \,\cup\, 
\left[\mu_0+\mu_1,\,+\infty\right)$, 
%
i.e. $\left(\mu_0, \,\mu_0+\mu_1\right)$ is a gap in the limit collective spectrum $\overline{\bigcup_{\xi \in \RR^n} {\rm Sp}\, ( \mathbb{L}_\xi - I)}$. 
Combining this with 
Corollary \ref{c.collspec} 
provides the following results on the structure of ${\rm Sp}\, \mathcal{L}_\ep$:
\begin{theorem}\label{5.3.21e1}
	For every interval $ [a,b] \subset (-\infty,\infty)$  there exists $C(b)\geq 0$, such that
	\begin{flalign*}
		{\rm dist}_{[a,b]}\Big({\rm Sp}\, \mathcal{L}_{\ep}\,,\,[0,\mu_0] \,\cup\, [\mu_0+\mu_1,\,+\infty) \Big) \,\,\,\le\,\,\, C(b) \ep, \quad \forall\, 0<\ep<1.
	\end{flalign*}
In particular, if  $\ep < \mu_1/(2C(b))$ then 
$\bigl[\mu_0 +C(b) \ep,\,\mu_0+\mu_1-C(b)\ep\bigr]$  is in a gap in the spectrum ${\rm Sp}\, \mathcal{L}_{\ep}$. 
\end{theorem}
We finish this example by observing that, routinely specialising 
the constructions of Section \ref{s.bivariate} leads to the associated 
 bivariate operator 
of the form $\mathcal{L} = \mathcal{L}_0 + I$, where $\mathcal{L}_0$ is the operator in 
$L^2\left(\RR^n\right) \,\dot{+}\, L^2\big(\RR^n;\, {\rm Span}\,( \chi_B)\big)$, equipped with the standard $L^2\left(\RR^n \times \square\right)$ inner product,  generated by the form
\be
\label{2scformimpint}
\mathbb{Q}\big( u + v\chi_B,\, \tilde{u} + \tilde{v} \chi_B\big) \,\,\,: =\,
\int_{\RR^n} A^{\rm hom}_{\rm pd} \nabla u(x) \cdot\overline{ \nabla \tilde{u}(x)} \, {\rm d}x \,\,+\,\,  
|\partial B|\int_{\RR^n} v(x)\, \overline{\tilde{v}(x)}\, {\rm d}x,
\ee
with the form domain $H^1\left(\RR^n\right) \,\dot{+}\, L^2\left(\RR^n;\, {\rm Span}\,\left( \chi_B\right)\right)$. 
Adjusting the derivation leading to Theorem \ref{thm.2scOpRes} for the present example, we have the following result.

\begin{theorem}
	For $0<\ep<1$ one has, with an $\ep$-independent constant $C$,  
	\begin{equation}
		\label{dpcompe3-2}
		\left\|\, \left(\mathcal{L}_\ep+I\right)^{-1} \,\,-\,\,  
		\mathcal{J}_\ep^* \left(\mathcal{L}_0+I\right)^{-1} \mathcal{P} \mathcal{J}_\ep \,\right\|_{L^2(\RR^n) \rightarrow L^2(\RR^n)} \,\,\,\le\,\,\, C\, \ep,
	\end{equation}
where $\mathcal{J}_\ep=T_\ep\mathcal{I}_\ep$, $\mathcal{J}_\ep^*=\mathcal{I}_\ep^*T_\ep^{-1}$, with 
translation operator $T_\ep f(x,y)=f(x+\ep y, y)$, and  
the two-scale interpolation operator $\mathcal{I}_\ep$ and its adjoint $\mathcal{I}_\ep^*$ given by 
\eqref{2ScInterp}--\eqref{7.54-1} and \eqref{7.54-2} respectively;  
and \newline
$\mathcal{P}: L^2\left(\RR^n \times \square\right) \rightarrow L^2\left(\RR^n\right) \dot{+} L^2\big(\RR^n;\, {\rm Span}\,( \chi_B)\big)$ is the orthogonal projection. 
\end{theorem}.
\begin{remark}
Notice that, as follows from \eqref{2scformimpint}, for $g \in L^2(\RR^n \times \square)$, $(\mathcal{L}_0 + I)^{-1} \mathcal{P} g= u(x) + v(x) \chi_B(y)$ where $(u,v) \in H^1(\RR^n) \times L^2(\RR^n)$ solve the coupled system
\be
\label{2phlimsys}
\left\{\ \begin{aligned}
- \,\nabla\cdot\, A^{\rm hom}_{\rm dp} \nabla u(x)  \,+\, u(x)\, +\, |B| v(x) \,\,=\, \int_\square g(x,y) \, {\rm d}y, \quad x \in \RR^n ;\\
|B| u(x) \,+\, \big( |\partial B| \,+\, |B|\big) v(x)\,\, =\, \int_B g(x,y) \, {\rm d}y,  \quad x \in \RR^n.
\end{aligned} \right.
\ee
We remark 
that \eqref{2phlimsys} is the 
two-scale limit system for the original problem \eqref{impintclass}. 
Notice that it appears to be of a ``two-phase'' 
type: the limit behaviour is characterised by the pair of 
macroscopic functions $u(x)$ and $v(x)$, $x\in\mathbb{R}^n$, in the matrix and 
inclusion phases respectively. 
Furthermore, Theorem \ref{5.3.21e1} immediately implies (cf. Theorem \ref{bivariate.spec}) 
the following estimate on the closeness of the spectra of the original and the limit problems: 
	${\rm dist}_{[a,b]}\big({\rm Sp}\, \mathcal{L}_\ep \,, {\rm Sp}\, \mathcal{L}_0 \big) \,\,\le\,\, C(b) \ep$.
\end{remark}

\subsection{A problem with 
concentrated perturbations 
}
\label{sec:concpert}
In this section we demonstrate that the parameter $\t$ does not necessarily have to come from the 
Floquet-Bloch-Gelfand  transform. 
Still, in this example, hypotheses \eqref{KA}--\eqref{distance} are valid and an appropriate minor modification of 
the general scheme leads to meaningful approximations with error bounds.

Let $F\in L^2(\mathbb{R}^3)$, $0<\ep<1$, $0\le\delta\le 1/2$, and $B_r(x)$ denote the ball of radius $r$ centred at $x$, with $B_r(0)$ denoted by $B_r$. Consider the problem, in the weak form,
\begin{equation}
\label{IKs}
\left\{ \begin{aligned}
& \text{ Find $U_{\ep,\delta} \in H^1(\mathbb{R}^3)$ the solution to} \\
& \int_{\mathbb{R}^3} \nabla U_{\ep,\delta}\cdot \overline{\nabla\Phi} \,\,+\,\,\ep^{-2}\delta^{-1}\sum_{j\in \mathbb{Z}^3} \int_{B_{\delta\ep}(\ep j)}U_{\ep,\delta}\,\overline{ \Phi}\,\,+\,\int_{\mathbb{R}^3}  U_{\ep,\delta}\, \overline{\Phi} \,\,\, =\,\, \int_{\mathbb{R}^3} F\, \overline{\Phi} , \qquad \forall \Phi \in H^1(\mathbb{R}^3),
\end{aligned} \right.
\end{equation}
where  
for 
$\delta=0$ the singular term (i.e. the second term in 
\eqref{IKs}) is regarded absent. 
Related problems with different scalings for ``concentrated perturbations'' 
were considered, for example,  in \cite{GoNa92,Na93}.

Our aim is to construct, for small $\ep$, approximations to the solution $U_{\ep,\delta}$ which would be uniform in $\delta$. 
The idea here is to reduce problem \eqref{IKs} to the general form \eqref{p1} by regarding $\delta$ as another component in the abstract parameter $\t$. 
Namely, let 
$\theta=(k,\delta)
\in \Theta=\square^\star\times[0,1/2]\subset\mathbb{R}^4$, where 
$k
\in \square^\star=[-\pi,\pi]^3$ is the usual Floquet-Bloch quasiperiodicity parameter. 
Then, 
as in the preceding examples, after rescaling and application of Gelfand transform,  we arrive at 
equivalent problem: 
\begin{equation}
\label{IKs4}
\left\{ \begin{aligned}
& \text{ Find $u_{\ep,\theta} \in H^1_{per}(\Box)$, $\Box=[-1/2,1/2]^3$, the solution to} \\
& \ep^{-2}\int_{\Box} (\nabla+ik) u_{\ep,\theta}\cdot \overline{(\nabla+ik)\phi} \,\,\,+\,\ep^{-2}\delta^{-1} \int_{B_{\delta}}u_{\ep,\theta}\,\overline{ \phi}
\,\,\,+\,\int_{\Box}  u_{\ep,\theta} \,\overline{\phi} \,\,\, =\,\, \int_{\Box} g_{\ep,k} \,\overline{\phi} , \quad \forall \phi \in H^1_{per}(\Box), 
\end{aligned} \right.
\end{equation}
where 
$g_{\ep,k} = {U} \Gamma_{\ep}F(k ,\cdot)$.  
Thus 
\eqref{IKs4} is of the form \eqref{p1} with
$H=H^1_{per}(\Box)$, 
$ \langle f,\tilde{u}\rangle= \int_{\Box} g_{\ep,k} \overline{\tilde{u}}$,
\begin{equation}
\label{conc-mass-ab}
a_\theta\left(u,\tilde{u}\right)\,\,=\,\,\int_{\Box} (\nabla+ik) u\cdot \overline{(\nabla+ik)\tilde{u}} \,\,+\,\,
\delta^{-1} \int_{B_{\delta}}u\,\overline{\tilde{u}}, \quad \text{ and} \quad b_\t\left(u,\tilde{u}\right) \,=\, \int_{\square} u \,\overline{\tilde{u}}.
\end{equation}
To check \eqref{as.b1} notice that, by  H\"older inequality and standard Sobolev embeddings, with some $c_0 >0$  
\begin{equation}\label{GNSineq}
\delta^{-1}\int_{B_\delta }|\phi|^2\,\,\le\,\, \delta^{-1} \Big(\int_\square |\phi|^6\Big)^{1/3} \left| B_\delta\right|^{2/3} \,\,=\,\, 
\left(\tfrac{4}{3}\pi\right)^{2/3} \,\delta\, \|\phi\|_{L^6(\square)}^2\,\,\le\,\, c_0\, \delta\, \|\phi\|_{H^1(\square)}^2, \quad 
\forall\delta>0, \ \forall \phi \in H^1_{per}(\square).
\end{equation}
This, together with the arguments as in Section \ref{e.class}, cf. \eqref{2.1classhom}, 
implies that 
\eqref{as.b1} holds. Further, as $\delta\le |\t|$ inequality \eqref{GNSineq} also implies that $a_\t$ is  Lipschitz in $\t$ at the origin, i.e. 
\eqref{ass.alip} is satisfied 
for $\t_1 = 0$, $\t_2\in\Theta$. 
Notice however that \eqref{ass.alip} does not hold globally on $\Theta$ (it can be shown by estimates similar to \eqref{GNSineq} that $a_\t$ is 
merely $\tfrac{2}{3}$-H\"{o}lder continuous in $\delta$ and hence in $\t$), and that \eqref{H4} fails to hold for similar reasons. 
Still, we can proceed here with our general method in its relevant parts. 

First notice that as follows from \eqref{conc-mass-ab} the  spaces $V_\theta$ and $W_\theta$ are  as in the classical setting (Example \ref{e.class}, 
see \eqref{vwclasshom}, with $\mathbf{e}$ denoting the identical unity function):
\[
\begin{aligned}
V_\theta = \left\{
\begin{array}{lr}
\{ 0 \}, & \theta \neq 0, \\[5pt]
{\rm Span} (\mathbf{e}), & \theta = 0,
\end{array} 
\right. & \qquad & W_\theta = \left\{
\begin{array}{lr}
H^1_{per}(\Box), & \theta \neq 0, \\[5pt]
H^1_{per, 0} : =\Big\{ u\in H^1_{per}(\Box) \,\,\, \Big\vert \,\, \int_\Box u = 0 \Big\}, & \theta = 0.
\end{array} 
\right.
\end{aligned}
\]
Observe next that the key condition \eqref{KA}  holds: this can be seen by noting \eqref{KA2} is obviously valid with $C=1$ and $c[u]=\int_{\square} |u|^2$. 
Hypothesis \eqref{contVs} is trivially satisfied (with $V^\star_\t=\{0\}$ and $L_\star=0$), and 
we will prove at the end of the subsection that \eqref{distance} also holds. 

In applying our abstract results as based on hypotheses \eqref{KA}--\eqref{distance}, certain care needs to be exercised as some of these results may be based on 
global version of \eqref{ass.alip} while in the present example the latter is 
assured when $\t_1=0$ (in fact when $\t_1=(k_1,\delta_1)$ with $\delta_1=0$). 
Fortunately, thanks to Remark \ref{rem.s3.1},  
Theorem \ref{lmm1} is applicable 
and states that $u_{\ep,\t}$ is approximated  when $|\t| < \nu_0 / (2L_a)$ by $\M v_0$, where $v_0 \in V_0$ is the solution to \eqref{w1prob0}. 
On the other hand, 
 due to \eqref{distance},  for each $r>0$ \eqref{bddspec} is satisfied on $\Theta_r := \{ \t \in \Theta : | \t | \ge r\}$ with $\nu = \gamma r^2$. Therefore,  Theorem \ref{thm:contV} applies when $\t \in \Theta_r$  (see Remark \ref{rem3.2}), and states in this setting that the solution $u_{\ep,\t}$ to 
\eqref{IKs4} satisfies \eqref{errorcontinuouscase} and \eqref{errorcontinuouscase2} with $v_\t=0$ and $\nu = \gamma r^2$. 

Specialising \eqref{IliaN2} and \eqref{IliaN} to the present example, $v_0=z_{\ep,\t}\,\mathbf{e}$ for some $z_{\ep,\t}\in\CC$, 
$\M v_0=v_0+\N v_0 = z_{\ep,\t}(\mathbf{e}+\ourN_\t) $ where 
$\ourN_\t:=\N\mathbf{e} \in H^1_{per,0}$ is the unique solution to 
\begin{equation}
	\label{IKs15}
	a_\t( \ourN_\t, \widetilde{w}_0)
	\,=\,-\,\,\delta^{-1} \int_{B_{\delta}}\overline{\widetilde{w}_0}, \qquad \forall \widetilde{w}_0 \in H^1_{per,0}. 
\end{equation} 
Equation \eqref{w1prob0} reduces then to an algebraic equation for $z_{\ep,\t}$ as follows. Setting $\tilde v=\tilde z\mathbf{e}$, $\tilde z\in\CC$, 
and e.g. using \eqref{amnorth}, one obtains: 
$
a_\t\left(\M v_0,\M\tilde v\right)\,=\,\Big( |k|^2\,+\,4\pi\delta^2/3 \,-\, a_\t[\ourN_\t]\Big)z_{\ep,\t}\,\overline{\tilde z}$, 
$\,\,\, b_\t\left(\M v_0,\M\tilde v\right)\,=\,\Big( 1\,+\, \|\ourN_\t\|^2_{L^2(\square)}\Big)z_{\ep,\t}\,\overline{\tilde z}$, 
$\,\,\, \left\langle f, \M\tilde v\right\rangle\,=\,  \left(\int_{\Box} g_{\ep,k}(y) \overline{\left(1+\ourN_\t(y)\right)} {d} y\right)\overline{\tilde z}$. 
Then \eqref{w1prob0} results in  
\begin{equation}\label{SVz}
	z_{\ep,\t} \,\,=\,\, \frac{  \int_{\Box}  g_{\ep,k}(y) \overline{(1+\ourN_\t(y))} \, {\rm d} y}
	{\ep^{-2} \Big( |k|^2+4\pi\delta^2/3 - a_\t[\ourN_\t]\Big) \,+\, 
	1\,+\, \|\ourN_\t\|^2_{L^2(\square)}}\,, 
\end{equation}
where  $\ourN_\t \in H^1_{per,0}$ is the solution to the ``cell problem'' \eqref{IKs15}. 
(Notice that the bracketed term in the denominator coincides with $a_\t\left[\M\mathbf{e}\right]$ which via \eqref{distance} is bounded from below 
by e.g. $\gamma|\t|^2$.) 
 
The above approximates $u_{\ep,\t}$ when $\t=(k,\delta) \in \Theta,$ $|\t| < \nu_0/(2L_a)$. 
Namely, for 
\[
A_{\ep,\t}\big(u,\tilde u\big)\,:=\,\ep^{-2}a_\theta\left(u,\tilde{u}\right)+b_\t\left(u,\tilde{u}\right)\,=\,
\ep^{-2}\left(
\int_{\Box} (\nabla+ik) u\cdot \overline{(\nabla+ik)\tilde{u}} \,+\,
\delta^{-1}\int_{B_{\delta}}u\,\overline{\tilde{u}}\right)\,+\, \int_{\square} u \,\overline{\tilde{u}}, 
\]
\eqref{ik43000}--\eqref{s4ep2bd} imply: 
\begin{equation}
\label{thm42strange}
A_{\ep,\t}\Big[u_{\ep,\t}\,-\,z_{\ep,\t}(\mathbf{e}+\ourN_\t)  \Big]\,\le\,c_1\ep^2\left\Vert g_{\ep,k}\right\Vert_{L^2(\Box)}^2, \quad \text{and} \quad 
\big\Vert u_{\ep,\theta} \,-\,z_{\ep,\t}(\mathbf{e}+\ourN_\t)\big\Vert_{L^2(\square)} \,\le\,   c_1 \ep^2 
\left\Vert g_{\ep,k}\right\Vert_{L^2(\Box)}, \ \ \forall \, |\t|<\,\tfrac{\nu_0}{2L_a},
\end{equation} 
with a constant $c_1>0$ independent of $\ep$, $\t=(k,\delta)$ and $F$ 
(recall $g_{\ep,k}:=U\Gamma_\ep F(k,\cdot)$). 
In principle this, together with  
\eqref{errorcontinuouscase}--\eqref{errorcontinuouscase2} for $|\t|\ge r_0 = \nu_0/(2L_a)$,  
 after the inverse Gelfand and scaling transforms provides us with an approximation to the 
solution $U_{\ep,\t}$ of \eqref{IKs} for small $\ep$, uniform with respect to both $F$ and $\delta$. 
This is quite inexplicit 
as requires in particular solving the cell problem \eqref{IKs15} for ranges of $k$ and $\delta$. 
However, we can construct a more explicit further approximation of \eqref{SVz} as follows.  

First notice that $\ourN_\t$ is small for small $\delta$. 
Indeed, by \eqref{IKs15} with $\tilde w_0=\ourN_\t$,  $a_\t\left[\ourN_\t\right] \,=\,\big| \delta^{-1}\int_{B_{\delta}}  \ourN_\t \big|$. 
On the other hand, arguing similarly to \eqref{GNSineq} and employing 
 the Poincar\'{e}-Wirtinger inequality we observe that 
\begin{equation}\label{SVPoinMean}
\Big| \delta^{-1}\int_{B_{\delta}}  \phi_0 \Big| \,\,\le\,\, c_2 \,\delta^{3/2}\, \left\| \nabla \phi_0\right\|_{L^2(\square)}, \quad \ \ \ \forall \phi_0 \in H^1(\square), \,\, \int_\square \phi_0 =0,
\end{equation}
for some $c_2 > 0$. Thus, by sequentially using \eqref{SVPoinMean}, \eqref{coercv0t} and \eqref{conc-mass-ab}  one has  
$a_\t\left[\ourN_\t\right] \,\le\, 2\,c_2^2\,\nu_0^{-1} \delta^3 $, and recalling \eqref{C2} also 
$\|\ourN_\t\|^2_{L^2(\square)}  \,\le\, 4\,c_2^2\,K^2\nu_0^{-2} \delta^3$.

So estimating, in terms of both 
$|\t|=\left(\delta^2+|k|^2\right)^{1/2}$ and $\ep$, the error of neglecting in \eqref{SVz} all the terms containing $\ourN_\t$,  
and recalling \eqref{distance} for bounding from below the denominator of \eqref{SVz}, 
we obtain  
\begin{equation}\label{SVz2}
\Big| z_{\ep,\t} \,-\, c_{\ep,\t}\Big| \,\,\le\,\, c_3\, 
\left[
\frac{\ep^{-2}|\t|^3}{\left(\ep^{-2}|\t|^2+1\right)^2}
\,+\,
\frac{
|\t|^{3/2}}{\ep^{-2}|\t|^2+1}
\right]
\left\| g_{\ep,k} \right\|_{L^2(\square)}, \quad \text{where}  \quad 	
c_{\ep,\t}\,\, =\,\, \frac{  \int_{\Box} g_{\ep,k}(y) \, {\rm d} y}{\ep^{-2} \big( |k|^2+4\pi\delta^2/3  \big) \,+ \,1}\,, 
\end{equation}
 and $c_3$ is some positive constant independent of $\ep$, $\delta$, $k$ and $F$. 
Now replace in \eqref{thm42strange} the approximation $z_{\ep,\t}(\mathbf{e}+\ourN_\t)$ by $c_{\ep,\t}\mathbf{e}$. As a result, for the first estimate, 
\begin{equation}
\label{innerest}
A_{\ep,\t}\big[u_{\ep,\t}\,-\,c_{\ep,\t}\mathbf{e}  \big]\,\,\le\,\,\,3\, A_{\ep,\t}\big[u_{\ep,\t}\,-\,z_{\ep,\t}(\mathbf{e}+\ourN_\t)  \big]\,+\,
3 \,A_{\ep,\t}\big[\left(z_{\ep,\t}-c_{\ep,\t}\right)\mathbf{e}  \big]\,+\,
3 \,A_{\ep,\t}\big[z_{\ep,\t}\ourN_\t  \big]. 
\end{equation}
With the first term on the right hand side bounded by \eqref{thm42strange}, for the second term via \eqref{SVz2} 
\[
A_{\ep,\t}\big[\left(z_{\ep,\t}-c_{\ep,\t}\right)\mathbf{e}  \big]\,\,\,=\,\,\,
\big\vert z_{\ep,\t}-c_{\ep,\t}\big\vert^2\left[\ep^{-2}\left(|k|^2+\frac{4}{3}\pi\delta^2\right)+1\right]\,\,\,\le
 \ \ \ \ \ \ \ \ \ \ \ \ \ \ \ \ \ \ \ \ \ \ \ \ \ \ \ \ \ \ \ \ \quad \ \ \ \ 
\]
\[
\ \ \ \ \ \ \ \
\frac{8}{3}\pi c_3^2\,
\left[
\frac{\ep^{-4}|\t|^6}{\left(\ep^{-2}|\t|^2+1\right)^3}
\,\,+\,\,
\frac{|\t|^3}{\ep^{-2}|\t|^2+1}
\right]
\left\| g_{\ep,k} \right\|^2_{L^2(\square)}
\,\,\le\,\,
c_4\ep^2\left\| g_{\ep,k} \right\|^2_{L^2(\square)}, 
\]
with a constant $c_4$ 
independent of $\ep$, $\t$ and $F$. 
(In the last inequality we used that for $0\le t:=|\t|/\ep<+\infty$, $t^6/(1+t^2)^3<1$, $t^2/(1+t^2)<1$, and 
that $|\t|$ is bounded.) 
Finally, for the last term in \eqref{innerest}, via \eqref{SVz} together with the above estimates for $a_\t\left[\ourN_\t\right]$ and 
$\left\|\ourN_\t\right\|_{L^2(\square)}$, 
\[
A_{\ep,\t}\big[z_{\ep,\t}\ourN_\t  \big]\,=\,|z_{\ep,\t}|^2\left(\ep^{-2}a_\t\left[\ourN_\t\right]+\left\|\ourN_\t\right\|^2_{L^2(\square)}\right)\,\le\,
c_5\,\frac{\left\| g_{\ep,k} \right\|^2_{L^2(\square)}}{\Big(\ep^{-2}|\t|^2\,\,+\,\,1\Big)^2}\left(\ep^{-2}\delta^3+\delta^3\right)\,\le\,
c_6\,\ep\left\| g_{\ep,k} \right\|^2_{L^2(\square)},
\]
with 
constants $c_5$ and $c_6$ independent of $\ep$, $\t$ and $F$ (having used in the last inequality the boundedness of $t^3/(1+t^2)^2$, 
$0\le t<+\infty$). 
Combining the above we bound \eqref{innerest}, for $|\t|<r_0:=\nu_0/(2L_a)$, by a constant times $\ep\left\| g_{\ep,k} \right\|^2_{L^2(\square)}$. 
On the other hand, for $\t\ge r_0$, it immediately follows from \eqref{SVz2} 
that $\big|c_{\ep,\t}\big| \le  r_0^{-2} \ep^2 \| g_{\ep,k} \|_{L^2(\square)}$ and as a result 
$A_{\ep,\t}\left[c_{\ep,\t}\mathbf{e}\right]=\left|c_{\ep,\t}\right|^2\left(\ep^{-2}a_{\t}\left[\mathbf{e}\right]+1\right)$ is bounded 
by a constant times $\ep^2\left\| g_{\ep,k} \right\|^2_{L^2(\square)}$. 
Also, for $|\t|\ge r_0$,  by 
\eqref{errorcontinuouscase} and 
\eqref{distance} $ A_{\ep,\t}\left[u_{\ep,\t}\right] \le \gamma^{-1} r_0^{-2}  \ep^2 \|  g_{\ep,k} \|_{L^2(\square)}^2$. 
As a result, the left hand side of \eqref{innerest} is bounded by a constant times $\ep^2\left\| g_{\ep,k} \right\|^2_{L^2(\square)}$ for $|\t|\ge r_0$. 

Repeating the above arguments for the second estimate in \eqref{thm42strange} with the approximation $z_{\ep,\t}(\mathbf{e}+\ourN_\t)$ again replaced 
by $c_{\ep,\t}\mathbf{e}$, we observe that the corresponding estimate is dominated by the term analogous to the second term on the right hand side of \eqref{innerest}. 
Namely, recalling \eqref{SVz2}, 
\[
\big\|\left(z_{\ep,\t}-c_{\ep,\t}\right)\mathbf{e}  \big\|_{L^2(\square)}\,\,=\,\,
\big\vert z_{\ep,\t}-c_{\ep,\t}\big\vert\,\,\le\,\,
c_3\, 
\left[  
\frac{\ep^{-2}|\t|^3}{\left(\ep^{-2}|\t|^2+1\right)^2}
\,+\,
\frac{
|\t|^{3/2}}{\ep^{-2}|\t|^2+1}
\right]
\left\| g_{\ep,k} \right\|_{L^2(\square)}
\,\,\le\,\,
c_7\,\ep\left\| g_{\ep,k} \right\|_{L^2(\square)}, 
\]
with $c_7>0$ independent of $\ep$, $\t$ and $F$. 

Combining all the above estimates we deduce that 
 \[ 
\ep^{-2}\left(  \int_\square \big\vert (\nabla+ \i k) \left(u_{\ep,\theta} -c_{\ep,\t}\mathbf{e} \right) \big\vert^2 \,+\, 
\delta^{-1}\int_{B_\delta} \big\vert u_{\ep,\theta} -c_{\ep,\t}\mathbf{e} \big\vert^2\right)  \,\, +\,\, 
\big\Vert u_{\ep,\theta} -c_{\ep,\t}\mathbf{e}\big\Vert_{L^2(\square)}^2 \,\,\,\le\,\,\,   c_8\, \ep\,
\left\Vert g_{\ep,k}\right\Vert_{L^2(\Box)}^2, 
\]
\begin{equation}
\label{eststgt}
 \text{and} \hspace{3cm} \big\Vert u_{\ep,\theta} \,\,-\,\,c_{\ep,\t}\mathbf{e}\big\Vert_{L^2(\square)} \,\,\,\le\,\,\,   c_8\, \ep 
\left\Vert g_{\ep,k}\right\Vert_{L^2(\Box)}, \quad \quad g_{\ep,k}=U\Gamma_\ep F(k,\cdot), \quad \quad \forall \t \in \Theta, \hspace{4cm}
\end{equation} 
for some constant $c_8>0$ independent of $\ep$, $\t=(k,\delta)$ and $F$. 
Comparing the above estimates with \eqref{thm42strange}, we observe that replacing the approximation $z_{\ep,\t}(\mathbf{e}+\ourN_\t)$ by the simplified 
one $c_{\ep,\t}\mathbf{e}$ results in a ``one power of $\ep$'' loss in the accuracy. 
In terms of Section \ref{sect4.1}, $c_{\ep,\t}\mathbf{e}$ can be seen to solve a modification of problem 
\eqref{w1prob0} with operator $\M$ replaced by identity. 
The point is that the above established smallness of the corrector $\ourN_\t$ in $\delta$ 
(of order $\delta^{3/2}$) allows to achieve \eqref{eststgt}. 

Now, arguing as in Example \ref{e.class} (cf. \eqref{SVz2} with \eqref{classicalzsol} leading to \eqref{Sep} ), we deduce that the 
approximation ${U}_{\ep,\delta}^{(0)} = \Gamma_\ep^{-1} U^{-1} c_{\ep,\t}\mathbf{e}$ to the exact solution 
${U}_{\ep,\delta} = \Gamma_\ep^{-1} U^{-1} u_{\ep,\t}$ of 
\eqref{IKs} 
solves 
 \begin{equation}\label{SVlim}
 \left(-\,\Delta\,\, +\,\frac{4}{3}\pi\,\frac{\delta^2}{\ep^2}\,\,+\,1\right){U}_{\ep,\delta}^{(0)}\,\,=\,\, \mathcal{S}_\ep F, \quad  \text{in} \ \mathbb{R}^3, 
 \end{equation}
with $\mathcal{S}_\ep$ given by \eqref{7.22-2} where $\chi$ stands for the characteristic function of $\square^*$. 
Further, estimates \eqref{eststgt} imply similar estimates for ${U}_{\ep,\delta}^{(0)}$, cf. \eqref{IKH1est}--\eqref{IKL2est}, 
in particular 
 \begin{equation}\label{SVLimest2}
\begin{aligned}
\left\Vert U_{\ep,\delta}  \,-\,   U_{\ep,\delta}^{(0)} \right\Vert_{H^1(\RR^3)}  \,\,\le\,\,   c_8\, \ep^{1/2}\,
\|F\|_{L^2(\RR^3)}, \quad \text{and} \quad 
\left\Vert U_{\ep,\delta}  \,-\,  
U_{\ep,\delta}^{(0)} \right\Vert_{L^2(\RR^3)}  \,\,\le\,\,   c_8\, \ep\,\|F\|_{L^2(\RR^3)}. 
\end{aligned}
\end{equation}
Finally, as \eqref{7.22-2} implies (cf. \eqref{7.23-3} leading to \eqref{7.23-2}) that 
$\| \mathcal{S}_\ep F -F \|_{H^{-1}(\RR^3)} \le  \ep \pi^{-1} \| F\|_{L^2(\RR^3)}$, it follows that the estimates analogous to \eqref{SVLimest2} remain valid if 
 $  \mathcal{S}_\ep$ is removed in \eqref{SVlim}. We collect all of the above observations to state the following theorem.   
\begin{theorem}
	\label{IKVolcompareclassicalh1}
Let $U_{\ep,\delta}$ solve \eqref{IKs} and, for each $\alpha \in [0,\infty)$, let  $U_{\alpha}\in H^1(\mathbb{R}^3)$ solve 
\[\left(-\,\Delta \,\,+\,\,\frac{4}{3}\pi\, \alpha^2\,\,+1\,\,\right)U_{\alpha}\,\,=\,\,F \quad \text{in} \ \mathbb{R}^3. 
\]
Then there exists a positive constant $c$ independent of $\ep$, $\delta$ and $F$ such that 
\begin{equation}\label{SVfinalEst}
\left\Vert U_{\ep,\delta} \,-\, U_{\delta/\ep} \right\Vert_{H^1(\mathbb{R}^3)}\,\, \le\,\,c\, \ep^{1/2}\,
 \Vert F \Vert_{L^2(\mathbb{R}^3)},\quad \text{and} \quad   
\left\Vert U_{\ep,\delta} \,-\, U_{\delta/\ep} \right\Vert_{L_2(\mathbb{R}^3)}\,\, \le\,\,c\, \ep\,\Vert F \Vert_{L^2(\mathbb{R}^3)}, 
\end{equation}
for all $0<\ep<1$, 
$0 \le \delta \le 1  / 2$ and $F\in L^2\left(\RR^3\right)$.
\end{theorem}
\begin{remark}
Notice that  estimates \eqref{SVfinalEst}, uniform in both $\ep$ and $\delta$, hold in particular for the `critical scaling' $\delta = O(\ep)$. 
For example, for 
$\delta=\ep$  Theorem \ref{IKVolcompareclassicalh1} states that $u_\ep$, the solution to the concentrated perturbation problem 
\[
 \int_{\mathbb{R}^3} \nabla u_{\ep}\cdot \overline{\nabla\phi} \,\,+\,\,\ep^{-3}\sum_{j\in \mathbb{Z}^3} \int_{B_{\ep^2}(\ep j)}u_{\ep}\overline{ \phi}
\,\,+\,\,\int_{\mathbb{R}^3}  u_{\ep}\overline{\phi} \,\, =\,\, \int_{\mathbb{R}^3} F \overline{\phi} , \qquad \forall \phi \in H^1(\mathbb{R}^3),
\] 
i.e. with $\ep$-periodic inclusions of size $\ep^2$ and of ``density'' $\ep^{-3}$, 
is approximated with operator-type error estimates 
by $u_0 $ the solution to the $\ep$-independent 
averaged 
problem
$ 
\big(-\,\Delta \,\,+\,\, \mu \,\,+\,\,1\big)u_0\,\,=\,\,F$ in 
$\mathbb{R}^3$, 
where $\mu = 4\pi / 3$. 
Indeed, \eqref{SVfinalEst} gives 
$ 
\Vert u_{\ep} \,-\, u_0 \Vert_{H^1(\mathbb{R}^3)} \,\,\le\,\, c\,\ep^{1/2} \Vert F \Vert_{L^2(\mathbb{R}^3)}$, 
$\Vert u_{\ep} \,-\, u_0 \Vert_{L_2(\mathbb{R}^3)} \,\,\le\,\, c\,\ep \Vert F \Vert_{L^2(\mathbb{R}^3)}$. 
\end{remark}
We conclude this example with the proof of hypothesis \eqref{distance}. 
\begin{proof}
For proving \eqref{distance} it is sufficient to show that 
\begin{equation*}
\| (\nabla + \i k) u\|_{L^2(\square)}^2 \,\,+\,\, \delta^{-1} \int_{B_{\delta}} | u|^2 \,\, \ge\,\, 
\tilde\gamma \, \bigl(|k|^2 \,+\,\delta^2\bigr)\| u \|^2_{L^2(\square)}, \quad \forall u \in H^1_{per}(\square),\, \ \  \forall \,(k,\delta) \in \Theta,
\end{equation*}
for some $\tilde\gamma>0$. Clearly, cf. e.g. \eqref{IKaest} and from the triangle inequality, 
\[
\| (\nabla + \i k) u\|_{L^2(\square)} \,\ge \,  |k|\, \| u \|_{L^2(\square)}, \quad \text{and}\quad  
\| (\nabla + \i k) u \|_{L^2(\square)} \, +\, |k|\| u \|_{L^2(\square)}  \,\ge \,\| \nabla u \|_{L^2(\square)} , \quad   \forall u \in H^1_{per}(\square), \, \  \forall k \in \square^*.
\]
Combining these implies $3\,\| (\nabla + \i k) u\|_{L^2(\square)}\,\ge\, |k|\, \| u \|_{L^2(\square)}+\| \nabla u \|_{L^2(\square)}$, and 
so it suffices to show that 
\begin{equation}\label{SV.bddspec}
	\| \nabla u\|_{L^2(\square)}^2 \,+\, \delta^{-1} \int_{B_{\delta}} | u|^2 \,\, \ge \,\,c\,\delta^2\,\| u \|^2_{L^2(\square)}, \quad \forall u \in H^1_{per}(\square),\quad  \forall \delta \in (0, 1/2],
\end{equation}
with some 
$c>0$. 

To prove \eqref{SV.bddspec}, 
for $u = c+u_0 $  with $c =\int_\square u$ and so $\int_\square u_0 =0$, using \eqref{SVPoinMean} we obtain  
\begin{flalign*}
&\delta^{-1}\int_{B_{\delta}} | u|^2 \,\,\ge\,\, \frac{4}{3} \pi\, \delta^2|c|^2 \,\,+\,\,
 2\,\delta^{-1}\, {\rm Re} \int_{B_{\delta}}  u_0 \,\overline{c} \,\,  \ge\,\, 
\frac{4}{3} \pi \delta^2|c|^2   -\, 2\,c_2\delta^{3/2}|c| \, \| \nabla u \|_{L^2(\square)}\,\,\ge \\ 
\quad & \frac{2}{3} \pi \delta^2|c|^2  \, -\, \frac{3\,c_2^2}{2\pi} \,\delta\,  \| \nabla u \|_{L^2(\square)}^2  
\,\,=\,\, \frac{2}{3} \pi \delta^2\left(|c|^2 \,+\, \| \nabla u \|_{L^2(\square)}^2\right) \,-\,  
\left( \frac{2}{3} \pi \delta^2 \,+\, \frac{3\,c_2^2}{2\pi} \delta \right) \| \nabla u \|_{L^2(\square)}^2. 
\end{flalign*}
This implies  (as $\delta \le 1/2$) that 
\[
\left(\frac{\pi}{6} \,+\, \frac{3\,c_2^2}{4\,\pi}\right)  \| \nabla u \|_{L^2(\square)}^2 \,\,+\,\, 
\delta^{-1} \int_{B_{\delta}} | u|^2 \,\,\,\ge\,\,\, \frac{2}{3}\, \pi\, \delta^2
\left(\left\vert\int_\square u\right\vert^2 \,+\, \| \nabla u \|_{L^2(\square)}^2\right), \quad \forall u \in H^1(\square).
\]  
Then, after application of the Poincar\'{e}-Wirtinger inequality, one arrives at \eqref{SV.bddspec}. 
\end{proof}

\subsection{An example  with a `partial' high-contrast}
\label{e.pdelast}
Here we consider an 
example of a high-contrast linear elasticity problem with a 
`partial degeneracy' in the inclusions. 
Consider the following resolvent problem:
\begin{equation}
\label{elastres}
\left\{ 
\begin{aligned}
& \text{Find $u_\ep \in \left[H^1(\RR^3)\right]^3$ such that} \\
& -\nabla\cdot\, \sigma_\ep(u_\ep) \,+\, u_\ep \,\,=\,\, F \in [L^2(\RR^3)]^3,
\end{aligned}
\right.
\end{equation}
where the matrix is assumed stiff but the $\ep$-periodic inclusions are stiff in compression but soft in shear. 
Namely, for 
 the stress-strain constitutive relation, 
\[
\sigma_\ep(u_\ep)(x) \,\,=\,\, \lambda\left(\tfrac{x}{\ep}\right) \big(\nabla\cdot u_\ep\big) I \,\, +\,\, 2\mu_\ep\left(\tfrac{x}{\ep}\right) e(u_\ep), \ 
x \in \RR^3, \qquad e(u)= \tfrac{1}{2}\big( \nabla u + \nabla u^T \big),
\]
with $\square$-periodic Lam\'{e} coefficients ($\square=[-1/2,1/2]^3$) of the form
\[
\begin{aligned}
\lambda(y) \,\,=\,\, \left\{ \begin{array}{lr}
\lambda_1(y), & y \in \square \backslash B, \\[5pt]
\lambda_2(y), & y \in B,
\end{array} \right. & \qquad &\mu_\ep(y) = \left\{ \begin{array}{lr}
\mu_1(y), & y \in \square \backslash B, \\[5pt]
\ep^2 \mu_2(y), & y \in B.
\end{array} \right.
\end{aligned}
\]
Here as before  the reference inclusion set $B$ is assumed to have Lipschitz boundary,  
$\overline{B} \subset (-\tfrac{1}{2},\tfrac{1}{2})^n$ and 
so the periodic matrix 
$\left(\square \backslash \overline{B}\right)+\mathbb{Z}^3$ is connected, 
and the measurable coefficients $\lambda_i$, $\mu_i$, $i=1,2$, are 
uniformly positive and bounded. 
Now, we proceed as in the above examples, to find that $u_{\ep,\theta}
= U\Gamma_\ep u_\ep(\theta,\cdot)$ solves \eqref{p1} where: 
$H = \left[H^1_{per}(\square)\right]^3$, $\Theta = [-\pi,\pi]^3$, 
$\l f, \tilde{u} \r = \int_\square  U \Gamma_\ep F(\theta,y) \cdot\overline{\tilde{u}(y)} \, {\rm d} y$, and
\begin{eqnarray}
\label{atpd}
	a_\t\big[u\big]  &=& \int_{\square \backslash B} \Big(\lambda_1\big| \nabla\cdot  u + \i\, \t \cdot u\big|^2 \,+\, 2\mu_1 \big| e(u) + \i \,\t \odot u\big|^2 \Big) 
	\,\,+\,\, \int_{B}\lambda_2\, \big| \nabla\cdot u + \i\, \t \cdot u\big|^2,   \\
	\label{btpd}
	b_\t\big[u\big] &=&  \int_{B}  2\mu_2 \big| e(u) + \i\, \t \odot u\big|^2 \,\,+\, \int_\square |u|^2, \quad \quad 
	\t \odot u : = \tfrac{1}{2}\big(\t \otimes u + u \otimes \t\big).
\end{eqnarray}
Let us now check 
the main abstract assumptions. 
We begin by recalling the elasticity theory variant of extension Proposition \ref{prp.zhiext}, whose proof we shall provide for the reader's convenience.

\begin{proposition}\label{prp.zhiextelast}
There exists an extension operator $E : \left[H^1(\Box \backslash B)\right]^3 \rightarrow \left[H^1(\Box)\right]^3$ 
such that: 
$Eu|_{\Box \backslash B} = u$ 
and for some constant $C_E>0$ independent of $u$, 
 $\| Eu \|_{\left[H^1(\Box)\right]^3} \,\,\le\,\, C_E \| u \|_{\left[H^1(\Box \backslash B)\right]^3}$ and 
\begin{equation}\label{ZhExtensionelast}
	\int_\Box \big| e(Eu)\big|^2 \,\,\,\le\,\,\, C_E^2 \int_{\Box \backslash B} \big| e(u)\big|^2. 
\end{equation}	
\end{proposition}
\begin{proof}
For fixed $u \in \left[H^1(\square \backslash B)\right]^3$  let $Ru=c + d \times y$ 
be the $[H^1(\square \backslash B)]^3$-projection of $u$ onto 
the subspace $\mathcal{R}=\left\{\tilde c + \tilde d \times y\,\big|\,\tilde c, \tilde d \in \CC^3\right\}$ of rigid body motions of $\square \backslash B$ 
(so in particular $e(Ru)=0$ and $\|Ru\|_{\left[H^1(\square \backslash B)\right]^3}\le \|u\|_{\left[H^1(\square \backslash B)\right]^3}$). 
Recall Korn inequality in the following form:
\begin{equation}
\label{kornineq}
\big\| u -\, Ru \big\|_{\left[H^1(\square \backslash B)\right]^3} \,\,\,\le\,\,\, 
C_K\, \big\| \,e(u)\, \big\|_{\left[L^2(\square \backslash B)\right]^{3\times3}}
\end{equation} 
for some $C_K>0$ independent of $u$. Let $P:[H^1(\square \backslash B)]^3 \rightarrow [H^1(B)]^3$ be the standard Sobolev extension 
(applied component-wise), i.e. $P u = u$ in $\square \backslash B$ and  there exists $C_P>0$ such that 
\begin{equation}\label{15.04.21}
	\| Pu \|_{\left[H^1(\square)\right]^3} \,\,\,\le\,\,\, C_P\, \| u \|_{ \left[H^1(\square \backslash B)\right]^3}, \qquad \forall 
	u \in \left[H^1(\square \backslash B)\right]^3.
\end{equation}
We construct $E$ as follows: $Eu := Ru + P(u - Ru)$, where $Ru=c + d \times y$ for all $y\in\square$. 
As $\mathcal{R}$ is finite-dimensional and a direct sum of `translational' ($d=0$) and 
`rotational' ($c=0$) subspaces, one can see that 
$c_1\|Ru\|_{\left[H^1(\square)\right]^3}\le |c|^2+|d|^2\le c_2 \|Ru\|_{\left[H^1(\square \backslash B)\right]^3}$ with positive constants $c_1$ and $c_2$ 
independent of $u$. 
It is then straightforward to check via \eqref{15.04.21} and \eqref{kornineq} 
that all the stated properties of $E$ hold.
\end{proof}
Now, the above uniform positivity assumptions on the Lam\'{e} coefficients imply that
\[
a_\t[u]\,\,\, \ge\,\,\, C\, 
\left( \int_\square \big| \nabla\cdot u + \i\, \t \cdot u \big|^2 \,\,+\, 
\int_{\square \backslash B} \big| e(u) +\, \i \,\t\odot u\big|^2 \right), \quad \forall u \in [H^1_{per}(\square)]^3,
\]
for some positive constant $C$. Next, by Proposition \ref{prp.zhiextelast}, and arguing as in \eqref{dpH1}, we obtain 
\begin{flalign*}
\int_{\square \backslash B} \big| e(u) +\, \i\, \t\odot u\big|^2 & \,\,=\, 
\int_{\square\backslash B} \big| e\left(e^{\i \t\cdot y} u\right)\big|^2 \,\,\ge\,\, 
C_E^{-2}  \int_{\square} \big| e\left(E\left(e^{\i \t\cdot y} u\right)\right)\big|^2 \,\,\ge\,\,
\frac{1}{2} C_E^{-2}  \int_{\square} \big| \nabla \left(E\left(e^{\i \t\cdot y} u\right)\right)\big|^2 \\
& \ge\,\,\frac{1}{2} C_E^{-2}\,|\t|^2  \int_{\square} \big| E\left(e^{\i \t\cdot y} u\right)\big|^2 \,\,\ge\,\,
\frac{1}{2} C_E^{-2}\,\,|\t|^2  \int_{\square \backslash B} |  u|^2,
\end{flalign*}
where the second and third inequalities  follow from 
decomposing $\square$-periodic $e^{-\i \t\cdot y}E\left(e^{\i \t\cdot y} u\right)$ into Fourier series. 
Consequently, one has 
\begin{equation}
\label{PDelast:nondeg}
a_\t[u] \,\,\,\ge\,\,\, C\, \left( \int_\square \big| \nabla\cdot u + \i\, \t \cdot u \big|^2 \,\,+\,\,
\frac{1}{2}\, C_E^{-2}\, |\t|^2 \int_{\square \backslash B} | u|^2 \right), \quad \forall u \in \left[H^1_{per}(\square)\right]^3.
\end{equation}
Thus, from \eqref{atpd} and \eqref{PDelast:nondeg}, the space $V_\theta$  
is
\begin{equation}
	\label{aaspaceV}
	\begin{aligned}
		V_\theta = \left\{ \begin{array}{lr}
			\left\{ v \in \left[H^1_0(B)\right]^3 \,\,\, \big| \, \text{ $\nabla\cdot v + \i\, \t \cdot v=0$ in $B$}\right\}, & \theta \neq 0, \\[5pt]
			\left\{ v \in \left[H^1_{per }(\square)\right]^3 \, \, \big| \, \text{$v$ constant in $\square \backslash B$ and 
			$\nabla\cdot v=0$ in $B$}\right\}, & \theta = 0.
		\end{array}
		\right.
	\end{aligned}
\end{equation}
As before, we regard $H^1_0(B)$ as a subspace of $H^1_{per}(\square)$ by extending by zero into $\square \backslash B$. 
For  proving the key spectral gap condition \eqref{KA}  we  shall be using the following `Sobolev Modification' lemma.
\begin{proposition}\label{SobModprop} 
	There exists linear operator $M : \left[H^1(\square)\right]^3 \rightarrow \left[H^1(\square)\right]^3$ and a positive constant $C_M$ such that
	\begin{equation*}
		\begin{aligned}
			& (i)\ \ \text{ $Mu = u$ in $\square \backslash B$;} \\
			& (ii)\ \ \text{$\nabla\cdot Mu = \nabla\cdot u$ in $B$;} \\
			& (iii)\ \ \Vert Mu \Vert_{\left[H^1(\square)\right]^3}^2 \,\,\le\,\, 
			C_M\,  \Big( \| u \|_{\left[L^2(\square \backslash B)\right]^3}^2 \,+\,  \Vert e(u) \Vert_{\left[L^2(\square \backslash B)\right]^{3\times 3}}^2\,+\,
			\| \nabla\cdot u\|_{L^2(B)}^2 \Big),  \quad \forall u \in \left[H^1(\square)\right]^3.
		\end{aligned}
	\end{equation*}
\end{proposition}
\begin{proof}

It is well know that the divergence operator  ${\rm div}:\, [H^1_0(B)]^3 \rightarrow L^2_0(B) : = \{ f \in [L^2(B)]^3 \, | \, \int_{B} f =0 \}$ is surjective,  see for example \cite[ Chapter 1, Section 2.1]{Lady}. Moreover, 
\begin{equation}
	\label{divinverse}
	\left\{ \begin{aligned}
		& \text{there exists a linear map $U:L^2_0(B) \rightarrow [H^1_0(B)]^3$ such that 
			for each $f \in L^2_0(B)$ 
			}
		\\
		& \nabla\cdot\, Uf = f, \quad \text{and} \quad \Vert Uf \Vert_{\left[H^1(B)\right]^3} \,\,\le\,\, C_B \Vert f \Vert_{\left[L^2(B)\right]^3} \quad \text{for some $C_B>0$ independent of $f$.}
	\end{aligned} \right.
\end{equation}
Let us now construct $M$. For fixed $u \in \left[H^1(\square )\right]^3$, 
$u-Pu \in [H^1_0(B)]^3$ for $P$ as in \eqref{15.04.21} and so $\nabla\cdot( u-Pu) \in L^2_0(B)$.
Now let $f =\nabla\cdot( u -Pu)$ 
and set $M u=  Uf + Pu$ (where $Uf$ is continuously extended by zero into $\square \backslash B$). Now, the desired properties (i) and (ii) immediately follow by construction. For property (iii), 
via \eqref{15.04.21}, \eqref{divinverse} and inequality 
$\Vert \nabla\cdot (Pu) \Vert_{L^2(B)}\le\sqrt{3} \Vert \nabla (Pu) \Vert_{L^2(B)}$ 
 we obtain  
\begin{flalign*}
	\Vert M u\Vert_{\left[H^1(\square)\right]^3}\,\,\, & \le\,\,\,  
	\Vert Pu \Vert_{\left[H^1(\square)\right]^3}\,+\, \Vert Uf \Vert_{\left[H^1(B)\right]^3}\,\, \le\,\, 
	C_P\,\Vert u \Vert_{\left[H^1(\square \backslash B)\right]^3} \,+\, C_B\, \Vert \nabla\cdot (u - Pu) \Vert_{L^2(B)} \\
	& \le\,\,\, C_P\left(1 + \sqrt{3}C_B\right) \Vert u \Vert_{\left[H^1(\square \backslash B)\right]^3} \,\,+\, C_B\, \Vert \nabla\cdot\, u  \Vert_{L^2(B)}. 
\end{flalign*}
Then  (iii)  follows from 
the Korn's second inequality 
in $\square \backslash B$.
\end{proof}

Now,  one can readily check that the main assumptions of the article hold by arguing as in the previous examples (in particular Example \ref{e.dp}).  We sketch the details below.

First, \eqref{as.b1} and \eqref{ass.alip} hold by essentially the same arguments as in the previous examples, cf. e.g.  \eqref{2.1classhom}.

\textbullet\, {\it Proof of \eqref{KA}}. 
We argue that \eqref{KA2} holds  with $c[w]$ chosen as a constant times $\int_{\square \backslash B} |w|^2$ (which is $\|\cdot\|_\t$-compact via Korn inequality). 
Indeed, for given $u \in \left[H^1_{per}(\square)\right]^3$, it follows by properties of $M$ that $v : = u - e^{-\i \t\cdot y}M(e^{\i \t \cdot y}u)$ belongs to $V_\t$. 
Hence, for any $w\in W_\t$ and $v$ chosen as above for $u=w$, we have $\|w\|_\t\le\|w-v\|_\t$ and arguing similarly to \eqref{h1prdp} we readily see via  Proposition \ref{SobModprop} that
\eqref{KA2} holds. 

\textbullet\, {\it Proof of \eqref{contVs}}. It is straightforward to show that \eqref{contVs} holds for 
\begin{equation}\label{PDvstar}
V^\star_\t \,\,: =\,\, \left\{  v \in \left[H^1_0(B)\right]^3 \,\,\, \big| \,\,\, \nabla\cdot\, v + \i\, \t \cdot v = 0 \text{ in } B\right\}.
\end{equation}
 Indeed, for each $v_1 \in V^\star_{\t_1}$  it is sufficient to consider $v_2 = e^{\i (\t_1 - \t_2) \cdot y}\, v_1\in V^\star_{\t_2}$.

\textbullet \, {\it Proof of \eqref{distance}} follows from combining \eqref{KA2} where  $c[w] = k\int_{\square \backslash B} |w|^2$ with some $k>0$, 
and \eqref{PDelast:nondeg}. 

\textbullet \, {\it Proof of \eqref{H4}}. This is immediate  in the present setting with
\[
a_0'(v,u) \cdot \t \,\,=\,\, 
\int_{\square \backslash B} \lambda_1 \i\,  \t \cdot v\,\, \overline{\nabla\cdot\, u} \,\,\,+  
\int_{\square \backslash B} 2\mu_1 \i\, \t \odot v : \overline{e(u)} \,\,\,+ 
\int_B \lambda_2\, \i\, \t \cdot v\,\,\overline{\nabla\cdot\, u} 
\]
and
\[a_0''(v,v) \t \cdot \t \,\,=\,\, 
\int_{\square \backslash B} \lambda_1 |\t \cdot v|^2 \,\,\,+ \int_{\square \backslash B}  2\mu_1 | \t \odot v|^2 \,\,\,+ \int_B \lambda_2 |\t \cdot v|^2. 
\]

Next we can choose the defect space $Z$  to be the $3$-dimensional vector space of constant functions. 
Indeed, from \eqref{aaspaceV} and \eqref{PDvstar},  
with $V_\star:=V^\star_0$ \eqref{spaceZ} clearly holds, and \eqref{VZorth} holds with $K_Z=0$: 
\be
\label{kzpd}
\left(v_\star, z\right)_0\,\,=\,b_0\left(v_\star, z\right)\,\,=\,\int_B\,v_\star\cdot \overline{z}\,\,=
\,\int_B\,\nabla\cdot\big(\left(\overline{z}\cdot y\right)v_\star\big)\,\,=\,\,0, \quad \forall \,v_\star\in V_\star, \,\,\forall\,z\in Z. 
\ee
It then routinely follows via a derivation similar to that leading to \eqref{dp.ahom} that 
\[
a^{\rm h}_\t(z,\tilde{z})\,\, =\,\, A^{\rm hom}_p z \odot \t \,:\, \overline{ \tilde{z} \odot \t}, \quad \,\, \forall z,\, \tilde{z} \in Z , \quad \forall \t \in \RR^3,
\] 
where $A^{\rm hom}_{p}$ is the homogenised tensor for the natural analog 
of perforated elastic domain  for the present example. 
Namely, $A^{\rm hom}_{p}$ correspond to the periodic matrix-inclusion composite with $\lambda=\lambda_1$ and $\mu=\mu_1$ in the matrix $\square\backslash B$, and 
  with $\lambda=\lambda_2$ but 
zero shear modulus $\mu=0$ in the inclusion $B$. 

\textbullet\, The proof of \eqref{H5} is immediate from \eqref{btpd};   
and \eqref{H6} holds for $\mathcal{H} = \left[L^2(\square)\right]^3$, $d_\t$  the standard ($\t$-independent) 
$[L^2(\square)]^3$ inner product and $\mathcal{E}_\t$ (for example)  multiplication by $e^{-\i\t\cdot y}$. In this setting, we observe that the bivariate operator (see Section \ref{s.bivariate}) is the (shifted for the identity operator) two-scale homogenised limit operator, found in \cite{Co}, and therefore its spectrum is the semi-axis $[1,+\infty)$.
\vspace{.07in}

In the rest of this section, 
we shall specify the 
approximation given by general Theorem \ref{thm.IKunifest2} 
and provide some new results 
for the present example. Estimate \eqref{IKfinal3-2} 
in particular implies
		\[
	\left\| u_{\ep,\t} \,-\, \left(z + e^{-\i \t \cdot y} v\right) \right\|_{\left[L^2(\square)\right]^3} \,\,\,\le\,\,\,  
	C_{10}^{1/2}\, \ep\, \| U\Gamma_\ep F(\t, \cdot) \|_{\left[L^2(\square)\right]^3}\,,
	\]
	where $z+v \in Z \,\dot{+}\, V_\star=V_0$ solves  \eqref{IKz3prob88}. Therefore, exploiting as in the earlier examples the $L^2$-unitarity of the scaling and 
	Gelfand transforms, one has 	
	\[
	\big\| u_{\ep} \,-\,\left( \mathbf{u}_\ep  + \mathbf{v}_\ep\right) \big\|_{[L^2(\RR^3)]^3} \,\,\,\le\,\,\, C_{10}^{1/2}\, \ep\, \| F\|_{\left[L^2(\RR^3)\right]^3}\,,
	\]
	where $\mathbf{u}_\ep = \Gamma_\ep^{-1} U^{-1} z $ and $\mathbf{v}_\ep = \Gamma_\ep^{-1} U^{-1} e^{-\i\t\cdot y} v$. Let us determine the problems that $\mathbf{u}_\ep$ and $\mathbf{v}_\ep$ solve. 	First note, see \eqref{kzpd}, that $V_\star$ and $Z$ are  orthogonal with respect to $b_0$ and so problem \eqref{IKz3prob88} decouples. 
	Thus, routinely specialising to the present example,  $z \in Z$ solves
\begin{equation}\label{pdzprob}
A^{\rm hom}_p  z \odot \tfrac{\t}{\ep} \,:\, \overline{\tilde{z} \odot \tfrac{\t}{\ep}} \,\,+\,\, 
z \cdot \overline{\tilde{z}} \,\,\,=\,\,\, 
\big(U \Gamma_\ep F(\t,\cdot)\,,\, \tilde{z}\big)_{\left[L^2(\square)\right]^3}, \quad \forall \tilde{z} \in Z,
\end{equation}
and $v \in V_\star$ solves
\begin{equation}\label{pdVprob}
	\int_B 2\mu_2\, e(v) : \overline{e(\tilde{v})} \,\,+\, \int_B v \cdot \overline{\tilde{v}}  \,\,\,=\,\, 
	\int_B e^{\i \t \cdot y} U \Gamma_\ep F(\t,y)\cdot\overline{\tilde{v}(y)} \, {\rm d}y,  \quad \forall \tilde{v} \in V_\star.
\end{equation}
Now, similarly to Example \ref{e.class}, we take the inverse Gelfand and scaling transforms in  \eqref{pdzprob} to find  that 
$\mathbf{u}_\ep \in \left[H^1(\RR^3)\right]^3$ solves, cf. \eqref{Sep}, 
$ 
-\, \nabla\cdot\, \sigma_0(\mathbf{u}_\ep) \,\,+\,\, \mathbf{u}_\ep\,\,\,=\,\,\, \mathcal{S}_\ep F$  
in $\mathbb{R}^3$, 
$\sigma_0(u) = A^{\rm hom}_p e(u)$,  
for the smoothing operator $\mathcal{S}_\ep = \mathcal{F}^{-1} \chi\big(\tfrac{\cdot}{\ep}\big) \mathcal{F}$, where 
$\chi$ is the characteristic function of $\square^*$. Furthermore, as we saw in Example \ref{e.class}, 
cf. \eqref{7.23-2}--\eqref{7.23-3}, 
one can remove $\mathcal{S}_\ep$ for $F\in [L^2(\RR^3)]^3$. Namely, one has \newline 
$ 
\big\| \mathbf{u}_\ep - u_0 \big\|_{\left[L^2(\RR^3)\right]^3} \,\le\, \ep \pi^{-1} \gamma_0 \| F\|_{\left[L^2(\RR^3)\right]^3}$, 
where  $u_0 \in [H^1(\RR^3)]^3$ solves the homogenised equation 
\begin{equation}\label{pdhomp}
- \,\nabla\cdot\, \sigma_0(u_0)\,\, +\,\,u_0\,\,\,=\,\,\, F \quad \text{in $\mathbb{R}^3$},
\end{equation}
and $\gamma_0>0$ is a strong ellipticity constant of $A^{\rm hom}_p$ (i.e. 
$ A^{\rm hom}_p z \odot \xi : \overline{ {z} \odot \xi}\ge\gamma_0^{-1}|z|\,|\xi|$, $\forall z, \xi\in \CC^3$). 

Let us now turn to $\mathbf{v}_\ep$. By \eqref{pdVprob} and the properties of the Gelfand transform we conclude that $\mathbf{v}_\ep$ belongs to 
$V_\ep : = \left\{ u \in \left[H^1(\RR^3)\right]^3 \,\, \big| \, u = 0 \text{ in } \RR^3 \backslash B_\ep \text{ and } 
\nabla\cdot\, u = 0 \text{ in } B_\ep \right\}$, $B_\ep = \bigcup_{m \in \mathbb{Z}^3} \ep( B+m)$.
Moreover, 
re-writing  \eqref{pdVprob} in terms of $\mathcal{V}=e^{-\i\t\cdot y}v$ and $\widetilde{\mathcal{V}}=e^{-\i\t\cdot y}\tilde{v}$ 
and then applying the inverse transforms $U^{-1}$ and $\Gamma_\ep^{-1}$, we 
deduce that $\mathbf{v}_\ep$ solves
\begin{equation}\label{pdVprob2}
	\ep^2\int_{B_\ep} 2\,\mu_2\left(\tfrac{x}{\ep}\right)\, e(\mathbf{v}_\ep) \,:\, \overline{e(\tilde{v})}\,\,\, +\, 
	\int_{B_\ep} \mathbf{v}_\ep \cdot \overline{\tilde{\mathbf{v}}}  \,\,\,=\,\, 
	\int_{B_\ep} F\cdot \overline{\tilde{\mathbf{v}}}   \quad \forall\, \tilde{\mathbf{v}} \in V_\ep,
\end{equation}
which is nothing but a collection of Stokes problems on each inclusion $\ep(B+m)$ of $B_\ep$.
In general,  $\mathbf{v}_\ep$ is not negligible and  the solution $u_\ep$ to \eqref{elastres} is approximated  up to leading order by $u_0 + \mathbf{v}_\ep$. 
However, if $F$ does not rapidly vary over $B_\ep$ (for simplicity if $\| F\|_{[H^1(F_0^\ep)]^3}$ is  bounded) we can see that
 $\mathbf{v}_\ep$ is $\ep$ small in $L^2$-norm. Indeed as, cf. \eqref{kzpd}, 
$\int_{\ep(B+m)} \tilde{v} = 0$ for each $m$, one has
 \[
 \big\| \mathbf{v}_\ep \big\|_{\left[L^2(\ep(B+m))\right]^3} \,\,\,\le\,\,\, 
\left\| \, F \,-\, \frac{1}{|\ep(B+m)|}\int_{\ep(B+m)}F\,  \right\|_{\big[L^2(\ep(B+m))\big]^3}\,\,\,\le\,\,\,   
\ep\, C_{B}\, \big\| \nabla F \big\|_{\big[L^2(\ep(B+m))\big]^{3\times 3}} 
 \]
 where $C_{B}$ is the Poincar\'{e}-Wirtinger constant for $B$.

Putting all this together gives the following approximation results.
\begin{theorem}\label{PDapprox}
For $F \in \left[L^2(\RR^3)\right]^3$, the solution $u_\ep$ to   \eqref{elastres}, the solution  $u_0$ to the  
homogenised system \eqref{pdhomp} 
and the solution $\mathbf{v}_\ep$ to the inclusion Stokes problems \eqref{pdVprob2} satisfy the estimate
\[
\big\|u_\ep \,-\, \left(u_0+\mathbf{v}_\ep\right)\big\|_{\left[L^2(\RR^3)\right]^3} \,\,\,\le\,\,\, C\, \ep\, \big\| \,F\, \big\|_{\left[L^2(\RR^3)\right]^3}  
\]
for the constant $C = {C_{10}^{1/2}} + \pi^{-1} \gamma_0$ 
independent of $\ep$ and $F$.
If additionally $F \in \left[H^1(\RR^3)\right]^3$, then 
\[
\big\|u_\ep \,-\, u_0\,\big\|_{\left[L^2(\RR^3)\right]^3} \,\,\,\le\,\,\, C\, \ep\, \big\| \,F\, \big\|_{\left[H^1(\RR^3)\right]^3}
\]
with $C = {C_{10}^{1/2}} + \pi^{-1} \gamma_0 + C_{B}$ independent of $\ep$ and $F$.
\end{theorem}

\subsection{Schr\"{o}dinger equation with a `strong' periodic magnetic field}
\label{magnschrod}
Up until now, in the examples, whenever $\t$ was the quasi-periodicity parameter the space $V_\theta$ has either  been continuous in $\t$ or possessed an isolated discontinuity at the origin $\theta =0$. 
Here, we give a simple one-dimensional example demonstrating that  for certain physically motivated models 
it is possible to have isolated discontinuities appear at non-zero points in the $\theta$-space. 
For a given periodic magnetic field $A:C^1_{per}[0,1] \rightarrow \RR$, $\int_0^1 A(y)\,{\rm d}y \notin 2\pi \mathbb{Z}$,  a uniformly positive periodic potential $V \in L_{per}^\infty(0,1)$ and $F \in L^2(\mathbb{R})$ 
consider the solution $u_\ep$ to the one-dimensional Magnetic-Schr\"{o}dinger equation
\[
 -\,\,  \left( \frac{d}{d x} \,-\,  \frac{\i}{\ep} A\left(\frac{x}{\ep}\right)\, \right)^2 u_\ep \,\,+\,\, 
V\left(\frac{x}{\ep}\right)\,u_\ep \,\,\,=\,\,\, F, 
\qquad x \in \RR,
\] 
with small parameter $0< \ep <1$. After the spatial rescaling $\Gamma_\ep$ and the Gelfand transform $U$, we find that $u_{\ep,\t}:= U \Gamma_\ep u_\ep(\t,\cdot)$ is the $[0,1]$-periodic solution to 
\[
- \,\, \ep^{-2}\left( \frac{d}{d y} \,+\, \i\, \t \, -\,  \i\, A(y)\right)^2 u_{\ep,\t}(y) \,\,+\,\, V(y)\,u_{\ep,\t}(y) \,\,\,=\,\,\, U\Gamma_\ep F(\t,y), \quad 
a.e. \,  \ \t \in \Theta=[-\pi, \pi].
\]
The variational  form 
of this problem is of type \eqref{p1}  for $H = H^1_{per}(0,1)$, 
$\l f, \tilde{u} \r = \int_0^1 U \Gamma_\ep F(\t,y)\,\overline{\tilde{u}(y)} \, {\rm d}y$, 
%
%
$$
a_\t[u] \,\,= \int_0^1 \Big|u'(y)\,+\, \i\, \t u(y) \,-\, \i\, A(y) u(y)\, \Big|^2\,{\rm d}y, \qquad 
b_\t[u] \,\equiv\, b_0[u]\,\, =\,\, \int_0^1 V(y)\, \big|u(y)\big|^2\,{\rm d}y, \qquad u \in H^1_{per}(0,1).
$$
Upon observing that 
$$
a_\t[u] \,\,=\,\, \int_0^1 \left| \frac{d}{dy}\left(u(y)\exp \left[\i \t y -\i \int_0^yA(y'){\rm d}y'\right] \right)\right|^2 \, {\rm d}y, 
$$
we can readily see that
\[
\begin{aligned}
V_\theta \,\,=\,\, \left\{ 
\begin{array}{lr}
\{ 0 \} & \theta \neq \theta_0, \\[5pt]
\Big\{ c\,\exp\Big[ -\,\i\,\Big(\,\t_0 y -\int_0^yA(y'){\rm d}y'\,\Big)\Big] \,\, \Big| \,\, c \in \CC\, \Big\} & \theta = \theta_0,
\end{array}
\right.
\end{aligned}
\]
where the unique $\t_0\in (-\pi,\pi]\backslash \{0\} $ equals $\int_0^1 A(y){\rm d}y$ up to an integer multiple of $2\pi$. 
That is, $V_\theta$ is 
trivial except its 
isolated discontinuity point $\t_0$ determined by the mean value of the magnetic field $A$. Clearly, $\t_0$ can take any value in $\Theta$.
Moreover, one can readily check that the assumptions \eqref{KA2}-\eqref{H6} hold with straightforward details left to the reader.

\subsection{A non-local example/differential-difference equation}
\label{sec:nonloc}
Here, we provide an example where the dependence of $a_\t$ on the Floquet-Bloch parameter $\t$ does not have to be quadratic (and not even polynomial). 
We consider for this a simple model of a nonlocal operator. For recent results on operator estimates in homogenisation of other classes of nonlocal operators see \cite{PSSZh-2023}.

Let $u_\ep \in H^1(\RR)$ be the solution to one-dimensional problem 
\begin{equation}
\label{nonlocalp}
	\begin{aligned}
	\int_{\RR} A\left(\tfrac{x}{\ep}\right)  u'_\ep\,\,  \overline{\tilde{u}'}\,\,\, +\,\, \ep^{-2}	
	\int_{\RR} D\left(\tfrac{x}{\ep}\right) \big(  u_\ep(x + \ep) - u_\ep(x) \big)\,
	\overline{\big(  \tilde{u}(x + \ep) - \tilde{u}(x) \big)} \, {\rm d} x \,\,\,+\, 	
	\int_{\RR}  u_\ep\,\overline{\tilde{u}}\,\,\, =\, 	\int_{\RR} F\,\overline{\tilde{u}},  \\
	\forall \tilde{u} \in H^1(\RR),
	\end{aligned}
\end{equation}
for given $F\in L^2(\RR)$ and uniformly positive bounded $1$-periodic functions $A(y)$ and $D(y)$.  
We apply the usual scaling and Gelfand transforms, noticing via \eqref{gt1} that for a function with a shifted argument $\hat u(x)=u(x + \ep)$ one has 
$U \Gamma_\ep \hat u (\t,y)=e^{i\t}U \Gamma_\ep u (\t,y)$. 
As a result, 
we determine that $u_{\ep,\t} : = U \Gamma_\ep u_\ep (\t,\cdot) \in H^1_{per}(\square)$ solves a problem of the form \eqref{p1}
for $H = H^1_{per}(\square)$, $\square = [0,1]$, $\Theta = [-\pi,\pi]$, 
$\langle f , \tilde{u} \rangle =	\int_{\square} U \Gamma_\ep F(\t ,y)\,\overline{\tilde{u}}(y) \, {\rm d}y$, with the sesquilinear forms
\begin{equation}\label{nonlocalforms}
	\begin{aligned}
a_\t(u,\,\tilde u)\,\, =\,\, 	\int_{\square} A(y) \big(u'(y) +\, \i\, \t u(y)\big) \, \overline{ \big( \tilde u'(y) +\, \i\, \t \tilde u(y)\big)}\,{\rm d}y\,\, +  
\int_\square	 D(y)  \big|1 - e^{\i \t}\big|^2 u(y)\,\overline{\tilde u(y)}  \, {\rm d} y\,, \\ 
\quad \text{and} \quad 
b_\t(u,\,\tilde u)\,\, =\,\, 	\int_{\square}  u(y)\,\overline{\tilde u(y)}\,{\rm d}y.
	\end{aligned}
\end{equation}
Here, as in the classical homogenisation Example \ref{e.class}, we have 
\[
\begin{aligned}
	V_\theta = \left\{
	\begin{array}{lr}
		\{ 0 \}, & \theta \neq 0, \\[5pt]
		{\rm Span} (\mathbf{e}), & \theta = 0,
	\end{array} 
	\right. & \qquad & 
	W_\theta = \left\{
	\begin{array}{lr}
		H^1_{per}(\Box), & \theta \neq 0, \\[5pt]
		H^1_{per, 0} : =\left\{ u\in H^1_{per}(\Box) \,\, \big| \,\, \int_\Box u = 0 \right\}, & \theta = 0.
	\end{array} 
	\right.,
\end{aligned}
\]
Further, we can show 
with the 
same reasoning as in Example \ref{e.class} 
that \eqref{KA2}, \eqref{contVs} and \eqref{distance} all hold. In particular, we have $V_\star = \{ 0 \}$ and $Z = {\rm Span} ( \bf{e} )$. Next, 
since $1-e^{\i\t}=-\i\,\t+O(\t^2)$ as $\t\to 0$,   
it is clear that 
 \eqref{H4} holds with  
the forms 
	\[
	a_0'(v,u) \cdot \t \,\,=\,\, \i \int_\square A\,\t v \, \overline{ u'}, \qquad 
	a_0''(v,\tilde{v})\t \cdot \t \,\,=\,\,|\t|^2 \int_\square A\,  v \,\overline{\tilde{v}} \,\,+\,\, |\t|^2 \int_\square D\, v\, \overline{\tilde{v}} , \quad 
	\text{	$v, \tilde{v} \in V_0$, $u \in H^1_{per}(\square)$.}
	\]
As a result, Theorem \ref{thm.maindiscthm} is applicable. Arguing again as  in Example \ref{e.class}, 
since $a_0(u,\tilde u)=\int_\square A u'\,\overline{\tilde u'}$ 
we readily verify that for the corrector $N_\t$ defined by \eqref{cell:IKprob22}, $N_\t \bf{e}=\i\,\t\ourN$ where 
$\ourN$ is the solution to the classical corrector problem \eqref{IKclasscor} (for $n=1$), and  applying \eqref{ahom7.1} 
\[
a^h_{\xi}[\mathbf{e}] \,\,=\,\, \left(\,\left\langle A^{-1} \right\rangle^{-1}\,+\,\, \langle D \rangle\, \right) | \xi|^2,  \qquad \xi \in \RR,    
\]
where 
$\langle h \rangle := \int_\square h(y)\,{\rm d}y$. 
Following further the pattern of Example \ref{e.class}, we observe that the 
solution $u_\ep$ to the original problem \eqref{nonlocalp} is approximated 
in terms of the following homogenised problem, cf. \eqref{homeq}: 
\[
		-\,\,\left(\left\langle A^{-1} \right\rangle^{-1}\,\,+\,\, \langle D\rangle \right)u''\,\, +\,\, u \,\,=\,\,F. 
\]
As a result, estimates directly analogous to those in Proposition \ref{homeqest} hold, with one further refinement namely with the possibility of 
removing in the present example the smoothing operator $\mathcal{S}_\ep$ in the analogue of \eqref{IKH1est...}. 
(This follows from noticing  that in the present one-dimensional case $N'(y)$ is bounded, and then establishing the $L^2$-smallness of 
$N'(x/\ep)\big(\,(\mathcal{S}_\ep-I)u\big)'(x)$ via an argument similar to \eqref{7.23-3}.) 
Consequently,  
 the following estimates are satisfied, with some constant $C$ independent of $\ep$ and $F$: 
	\begin{gather*}
		\left\Vert u_\ep \,-\,\, \Big(u\,+\, \ep\ourN\left(\tfrac{\cdot}{\ep}\right)  u'\Big)\,  \right\Vert_{H^1(\mathbb{R})} \,\,\le\,\, \ep\, C\, 
		\big\Vert\, F\, \big\Vert_{L^2(\mathbb{R})}, \qquad \text{ and } \qquad 
		\big\Vert \,u_\ep \,-\, u\,  \big\Vert_{L^2(\mathbb{R})} \,\,\le\,\, \ep\, C\,\big\Vert \,F\,\big\Vert_{L^2(\mathbb{R})}. 
	\end{gather*}

\subsection[Difference equation]{A difference equation\footnote{The authors are grateful to Prof. Igor $\rm{Vel\check{c}i\acute{c}}$ (University of Zagreb) for bringing this example to their attention. }}
\label{sec:nonloc2}
Here we provide an example where our general theory remains applicable while the defect space $Z$ is infinite dimensional. 
Consistently with Proposition \ref{prop:Zfinite}, this can only happen when the weaker hypothesis 
\eqref{KA} holds in 
the absence of its stronger version \eqref{KA2}. 
We consider the  problem resembling \eqref{nonlocalp} but with the term involving the derivatives 
dropped. Namely, let  $u_\ep \in L^2(\RR)$ be the solution to the difference equation 
\begin{equation}
\label{nonlocalp2}
	\ep^{-2}	
	\int_{\RR} D\left(\tfrac{x}{\ep}\right) \big(  u_\ep(x + \ep) - u_\ep(x) \big)\,
	\overline{\big(  \tilde{u}(x + \ep) - \tilde{u}(x) \big)} \, {\rm d} x \,\,\,+\, 	
	\int_{\RR}  u_\ep\,\overline{\tilde{u}}\,\,\, =\, 	\int_{\RR} F\,\,\overline{\tilde{u}}, \,\, \,\,\,
	\forall \tilde{u} \in L^2(\RR),
\end{equation}
for given $F\in L^2(\RR)$ and   $D\in L^\infty_{per}(0,1)$ satisfying  
\[
D(y) \,\,\,\ge\,\,\, m, \quad \,\,\, \forall y\in\RR, 
 \]
 with some positive constant $m$.  Arguing as in the last example we arrive at 
 a problem of the form \eqref{p1}:
\begin{equation}\label{nonlocal3}
\ep^{-2} a_\t\left(u_{\ep,\theta} ,\tilde{u}\right) \,\,+\,\, b_\t \left(u_{\ep,\theta},\tilde{u}\right) \,\,\,=\,\,\, \l f,\tilde{u}\r, \quad \forall \tilde{u} \in H, 
\end{equation}
for $H = L^2(\square)$, $\square = [0,1]$, $\Theta = [-\pi,\pi]$, 
$\langle f , \tilde{u} \rangle =	\int_{\square} U \Gamma_\ep F(\t ,y)\,\overline{\tilde{u}}(y) \, {\rm d}y$, with the sesquilinear forms
$$
a_\t(u,\,\tilde u)\,\, = 
\int_\square	 D(y)  \big|1 - e^{\i \t}\big|^2 u(y)\,\overline{\tilde u(y)}  \, {\rm d} y=4\sin^2(\t/2)\int_\square	 D(y) \,  u(y)\,\overline{\tilde u(y)}  \, {\rm d} y\,, $$ 
$$ \text{and} \quad 
b_\t(u,\,\tilde u)\,\, =\,\, 	\int_{\square}  u(y)\,\overline{\tilde u(y)}\,{\rm d}y.
$$
Hence we have 
\[
\begin{aligned}
	V_\theta = \left\{
	\begin{array}{lr}
		\{ 0 \}, & \theta \neq 0, \\[5pt]
		L^2(\square), & \theta = 0,
	\end{array} 
	\right. & \qquad & 
	W_\theta = \left\{
	\begin{array}{lr}
		L^2(\square), & \theta \neq 0, \\[5pt]
		\{ 0 \}, & \theta = 0.
	\end{array} 
	\right.
\end{aligned}
\]
Obviously \eqref{KA} holds, with for example $\nu_\t= \frac{4m}{4m+1}\sin^2(\t/2)$ for $\t\neq 0$ and $\nu_0$ 
any positive number. Further \eqref{contVs} trivially holds with $V_\star = \{ 0 \}$, and so  $Z=L^2(\Box)$. 
Elementary estimates then show that \eqref{distance} holds with $\gamma=\frac{4m}{\pi^2(4m+1)}$, as well as   
 \eqref{H4} holds with  
	\[
	a_0'\left(u,\tilde{u}\right) \cdot \t \,\,=\,\,0, \qquad 
	a_0''(u,\tilde{u})\t \cdot \t \,\,=\,\,|\t|^2 \int_\square D\, u\, \overline{\tilde{u}} , \quad 
	\text{	$u,\, \tilde{u}  \in L^2(\square)$.}
	\]
As a result, $N_\t$ is zero (since $W_0=0$) and   Theorem \ref{thm.maindiscthm} is applicable with 
\[
a^{\rm h}_{\t}\left(z,\tilde{z}\right) \,\,=\,\, | \t|^2 \int_\Box D(y)\, z(y)\,\overline{\tilde{z}(y)}\,\,{\rm d} y  ,  \qquad \forall z,\,\tilde{z}\in L^2(\Box). 
\]
Notice next 
that \eqref{H5} is obviously valid, since $b_\t$ does not depend on $\t$. 
Moreover, as $V_\t^\star=V_\star=\{0\}$ are trivial so are $\mathcal{E}_\t:V_\star\rightarrow V_\t^\star$, and 
Theorem \ref{thm.IKunifest} coincides with Theorem \ref{thm.maindiscthm}.

We next adjust this example to the framework of Section \ref{s:resolv}.  
First we observe that the right hand side of \eqref{nonlocal3} has the form as in \eqref{ik3}, with 
$\mathcal{H}=L^2(\Box)$, $d_\t(u,\tilde{u})=\int_\Box u\,\overline{\tilde{u}}$ and $g=U \Gamma_\ep F(\t,\cdot)$. 
At the beginning of Section \ref{s:resolv},
we made an assumption of 
compactness of embedding of $H$ into $\mathcal{H}$ for the purpose of investigation of the spectra. 
In the present example however $H$ coincides with $\mathcal{H}$, 
and therefore there is no  compactness of the embedding. 
Nevertheless hypothesis \eqref{H6} trivially holds with $\mathcal{E}_\t$ chosen to be the identity operator, 
and inspecting the proof of Theorem \ref{p.unitaryequiv} shows that it does not actually require on its own the 
embedding compactness. 
As a result, 
the theorem 
takes the form of the following inequality 
	\begin{equation}\label{splimSe.56}
\left\|\, \mathcal{L}_{\ep,\t}^{-1} \,\, -\,\, 
 \mathbb{L}_{\t / \ep}^{-1}\,\right\|_{L^2(\Box) \rightarrow L^2(\Box)} \,\,\,\le\,\,\, C \ep,\, \ \ \  
0<\ep<1, \ \ \forall\,
\t \in [-\pi,\pi],
\end{equation}
where  $\mathcal{L}_{\ep,\t}$ is the self-adjoint operator in $L^2(\Box)$,  generated by the form on the left hand side of \eqref{nonlocal3} and 
$\mathbb{L}_{\xi}$ is the self-adjoint operator generated by the form
$$    
| \xi|^2 \int_\Box D(y) z(y)\overline{\tilde{z}(y)}\,{\rm d} y  \,\,+\,\,
\int_\Box  z(y)\,\overline{\tilde{z}(y)}\,{\rm d} y ,  \qquad 
z,\,\tilde{z}\in L^2(\Box),\,\,\,\, \xi \in \RR. 
$$
The simplification in \eqref{splimSe.56} compared to Theorem \ref{p.unitaryequiv} is due to the choice of 
$\mathcal{E}_\t$ 
the identity operator, 
and to $\mathcal{H}_0:=\overline{V_0}=L^2(\Box)$ (and hence  the projector 
in Theorem \ref{p.unitaryequiv} is also the identity.) 

Now Theorem \ref{thm.bivariate} takes the form\footnote{Theorem \ref{thm.bivariate} relies on Appendix B, which in turn uses the assumption of compactness of embedding. In particular this assumption allowed us to pick up a special basis in both $V_0$ and 
$\mathcal{H}_0$, see the proof of Proposition \ref{propb1}.  However the results of Appendix B are still valid in the present setup 
(in fact their proofs are simpler) since $Z=V_0=\mathcal{H}_0=\mathcal{H}$, $b_0=d_0$ and any basis in $V_0$ will serve the purpose.}
	\[
 \left\|\,\mathcal{L}_{\ep,\t}^{-1} \,\,-\,\, 
\big(A_\ep^* \mathcal{L}^{-1} A_\ep\big)(\t)\,\right\|_{L^2(\Box) \rightarrow L^2(\Box)}
\,\,\, \le\,\,  C\,\ep\,  ,	\quad  \quad 
for \,\,a.e. 
\ \t \in [-\pi,\pi]. 
	\]
Here $A_\ep^* \mathcal{L}^{-1} A_\ep$ is 
a self-adjoint 
operator in $L^2\left(\mathbb{R}; L^2(\Box)\right)$ decomposable into a direct integral in $\t$; the bivariate (two-scale) limit operator 
$\mathcal{L}$ is an unbounded self-adjoint operator in $L^2\left(\mathbb{R}\times\Box\right)$  generated by
\be
\label{2sc-NL}  
\int_{\mathbb{R}} \int_{\Box} D(y)\,
\tfrac{\partial u}{\partial x} (x,y)\,
\overline{\tfrac{\partial \tilde u}{\partial x}(x,y)}\,{\rm d} y\,{\rm d} x \,\, +\,
\int_{\mathbb{R}} \int_{\Box}  u (x,y)\,\overline{\tilde{u}(x,y)}\,{\rm d} y\,{\rm d} x \, ,  \qquad 
u,\,\tilde{u}\in H^1\left(\mathbb{R};L^2(\Box)\right)\, ;
\ee 
$A_\ep : L^2\big(\Theta \,; L^2(\Box)\big) \rightarrow L^2\big(\mathbb{R}\, ; L^2(\Box)\big) $ is 
as in Theorem \ref{thm.bivariate} (for $n=1$ and $\mathcal{E}=I$), i.e. 
$A_\ep : = \Gamma_\ep^{-1}\mathcal{F}^{-1} \, \chi $ where 
$\chi: L^2\big(\Theta \,;\, L^2(\Box)\big) \rightarrow  L^2\big(\mathbb{R} \,;\, L^2(\Box)\big)  $ 
is the 
extension by zero outside $\Theta$. 
We notice that, while the original problem \eqref{nonlocalp2} is nonlocal,  the limit problem \eqref{2sc-NL} is local (although two-scale).

Further, following the pattern of Subsection \ref{interpol77} we observe that the analog of Theorem \ref{thm.2scOpRes} for the self-adjoint operator $\mathcal{L_\ep}$ in $L^2(\mathbb{R})$ 
corresponding to the original problem, i.e. 
generated by the form on the left hand side of   \eqref{nonlocalp2}, is:  
\begin{theorem}\label{thm.2scOpResnonloc}
	For $0<\ep<1$ one has 
	\begin{equation}
		\label{dpcompe3nl}
		\bigl\Vert \,\mathcal{L}_\ep^{-1} \,\,-\,\,  
		\mathcal{I}_\ep^* \mathcal{L}^{-1}  \mathcal{I}_\ep\, \bigr\Vert_{L^2(\mathbb{R}) \rightarrow L^2(\mathbb{R})} \,\,\le\,\, C\, \ep
	\end{equation}
for some positive constant $C$ independent of $\ep$.	Here  $\mathcal{I}_\ep: L^2(\mathbb{R}) \rightarrow L^2(\mathbb{R} \times \square)$ is the 
two-scale interpolation operator defined (for $n=1$) by \eqref{2ScInterp}, 
	which is 
	an $L^2$-isometry and the continuous extension of \eqref{7.54-1} for $n=1$; 
$\mathcal{I}_\ep^*: L^2(\mathbb{R} \times \square) \rightarrow L^2(\mathbb{R}) $ is the adjoint of $\mathcal{I}_\ep$ given by  
\eqref{7.54-2} for $n=1$. 
 
\end{theorem}
\begin{remark}\label{lastrem} 
Comparing \eqref{dpcompe3} with \eqref{dpcompe3nl}, we observe that the structure of the approximating operator in the latter  corresponds to that in the former with 
$\mathcal{J}_\ep=\mathcal{I}_\ep$, i.e. with no translation operator $T_\ep$ (or by formally 
assigning it to be unity). This is an implications of our ability to choose in the present example $\mathcal{E}_\t$ as the identity operator. 
(That is indeed in contrast with \eqref{dpcompe3}, where $\mathcal{E}_\t$ could not have been taken the identity operator in the 
inclusions for the genuine $\t$-dependence of $b_\theta$ in \eqref{abcdp}.) 

Estimate \eqref{dpcompe3nl}  allows a particularly simple specialisation 
for 
 a class of two-scale right hand sides $F_\ep(x)=\Phi(x,\,x/\ep)$. 
Namely, let $F_\ep(x)=\Phi(x,\,x/\ep)$ 
where $\Phi(x,y)\in L^2\left(\RR\times\square\right)$ is $\square$-periodic in $y$ 
and for every $y$ its Fourier transform in $x$ is supported within a bounded segment $[-R,R]$, cf. Remark \ref{RemShann}. 
Hence, by the Nyquist-Shannon sampling theorem,   
for $\ep\le\pi/R$ simply $\mathcal{I}_\ep F_\ep(x)=\Phi(x,y)$, and 
as $\mathcal{L}^{-1}$ commutes with 
$\mathcal{S}_\ep$, for the solution 
$u=\mathcal{L}^{-1}\Phi$ of the two-scale limit problem $\mathcal{I}_\ep^*u(x)=u\big(x,\,x/\ep\big)$. 
As a result, for the exact solution $u_\ep=\mathcal{L}_\ep^{-1}F_\ep$, \eqref{dpcompe3nl} yields 
\[
\left\Vert \,u_\ep\,\,-\,u\left(x,\,\tfrac{x}{\ep}\right)\,\right\Vert_{L^2(\RR)}\,\,\,\le\,\,\,C\,\ep 
\left\Vert\Phi\left(x,\,\tfrac{x}{\ep}\right)\right\Vert_{L^2(\RR)}, 
\]
with a constant $C$ independent of $\ep$ and $\Phi$. 

\end{remark}

\section*{Declarations}
There are no competing interests.

\section*{Data Availability}  
 No datasets were generated or analysed during the current study

\appendix


\section*{Appendix A}\label{appa}
\setcounter{section}{1}
\renewcommand{\theequation}{\Alph{section}}
\setcounter{equation}{0}
\setcounter{theorem}{0}
\renewcommand{\theequation}{\mbox{\thesection.\arabic{equation}}}
We prove here Lemma \ref{propeth}, i.e. the existence of $\mathcal{E}_\t$ satisfying \eqref{Eprop1} and \eqref{Eprop2} (or equivalently \eqref{Eprop3}). We shall prove it in the following form:
\begin{proposition}\label{prop.Eexists} 
	Assume \eqref{contVs} and \eqref{H5}. Then, there exists a bijection $\mathcal{E}_\t : V_\star \rightarrow V_\t^\star$ such that
	\begin{gather}
	b_\t(  \mathcal{E}_\t  v_\star, \mathcal{E}_\t \tilde{v}_\star)=b_0(v_\star,\tilde{v}_\star), \quad \forall v_\star,\tilde{v}_\star\in V_\star; \label{Eunitary} \\
	\text{and there exists a constant $K'_b \ge 0$ such that } \nonumber  \hspace{\linewidth}\\
	\| \mathcal{E}_\t v_\star - v_\star \|_\t  \le K'_b |\t| \| v_\star\|_0, \quad  \forall v_\star \in V_\star. \label{Elip}
	\end{gather}
\end{proposition}
Notice that combining \eqref{H5} with \eqref{Elip} implies \eqref{Eprop2} and \eqref{Eprop3} 
(for $K_b = L_b + K K'_b$). 

Before proving this proposition, let us first demonstrate  that under \eqref{contVs} the subspace $V^\star_{\t_1}$ is isomorphic to $V_{\t_2}^\star$ when $\t_1$ and $\t_2$ are close. 
\begin{proposition}\label{PVbi}
	Assume \eqref{contVs}. Then, for $\t_1,\t_2\in\Theta$, $P({\t_1,\t_2}): V_{\t_1}^\star \rightarrow V_{\t_2}^\star$, given by $v \mapsto P_{V_{\t_2}^\star}  v$, is a bijection when $K L_\star |\t_1 - \t_2| < 1$.
\end{proposition}
\begin{proof} Fix $\t_1, \t_2 \in \Theta$, $K L_\star |\t_1 - \t_2| <1$, and 	$v_1  \in V_{\t_1}^\star$. 
By \eqref{contVs2} and \eqref{as.b1} we have  $  \| P_{W_{\t_2}^\star} v_1 \|_{\t_2} \le K L_\star |\t_1-\t_2|\, \| v_1 \|_{\t_2}$  and so 
	\begin{equation}\label{22.09.20e1}
	\| P_{V_{\t_2}^\star} v_1 \|_{\t_2}^2 \,\,=\,\, \| v_1 \|_{\t_2}^2  - \| P_{W_{\t_2}^\star} v_1 \|_{\t_2}^2  
	\,\,\ge\,\, \left(1 \,-\, \left( K L_\star |\t_1 - \t_2| \right)^2\right) \| v_1 \|_{\t_2}^2. 
	\end{equation}
This implies $P({\t_1,\t_2})$ is injective and has a closed range. It remains to prove that $ P_{V_{\t_2}^\star} V_{\t_1}^\star$ is not a proper subset of $V_{\t_2}^\star$. Suppose there exists $0 \neq v \in V_{\t_2}^\star$ such that $ v$ is orthogonal to  $ P_{V_{\t_2}^\star} V_{\t_1}^\star$ with respect to $(\cdot,\cdot)_{\t_2}$.  
Then, 
 $\big(v, P_{V_{\t_1}^\star} v\big)_{\t_2} = \big(v, P_{V_{\t_2}^\star}P_{V_{\t_1}^\star} v\big)_{\t_2} = 0$.
 Consequently,  we compute 
	\[
	\| v\|_{\t_2}^2 \,=\, \big(v , P_{V_{\t_1}^\star} v\big)_{\t_2} + \big( v, P_{W_{\t_1}^\star} v\big)_{\t_2} 
\, =\, 
	\big( v, P_{W_{\t_1}^\star} v\big)_{\t_2} \,\,\le \,\, \| v \|_{\t_2} \| P_{W_{\t_1}^\star} v\|_{\t_2} \, 
	\le \, K L_\star |\t_1-\t_2| \| v \|_{\t_2}^2,
	\]
where we have used  \eqref{as.b1} and \eqref{contVs2} in the last inequality.
This leads to the contradiction $\| v\|_{\t_2} = 0$ for $K L_\star |\t_1 - \t_2|  <1$. Hence $P_{V_{\t_2}^\star} V_{\t_1}^\star = V_{\t_2}^\star$. 
\end{proof}

\begin{proof}[Proof of Proposition \ref{prop.Eexists}]
	As $H$ is separable, the dimension (i.e. any basis) of $V_\t^\star$ is at most countable. Moreover, it
follows from Proposition \ref{PVbi} that the dimension of $V_\t^\star$ is independent of $\t$ for close enough $\t_1$ and $\t_2$. Since $\Theta$ is assumed connected, $V_\t^\star$ and $V_\star$ are isomorphic for any $\t\in\Theta$, and in particular one can always find a $\mathcal{E}_\t$ which satisfies \eqref{Eunitary}. Moreover it is clear that, for any such $\mathcal{E}_\t$, and for any chosen $r>0$ \eqref{Elip} holds for $|\t| \ge r >0$ 
(with $K_b'$ replaced by $(1+K)/r$).  As such, we need only establishing \eqref{Elip} for the case 
$K L_\star |\t| <1/\sqrt{2}$, for which we construct below $\mathcal{E}_\t$ in a particular way.

	Consider the Hilbert spaces $(V_\star,b_0)$ and $(V_\t^\star, b_\t)$, and let $Q_\t : (V_\t^\star, b_\t) \rightarrow (V_\star,b_0)$ be the  inverse of $P(0,\t)$, which is bounded (see \eqref{22.09.20e1}),
	and let $Q_\t^* : (V_\star, b_0) \rightarrow (V_\t^\star,b_\t)$ be the adjoint of $Q_\t$.  
	Noticing that $Q_\t^*Q_\t: (V_\t^\star, b_\t) \rightarrow (V_\t^\star, b_\t)$ is bounded, self-adjoint and non-negative, 
	set $\mathcal{E}_\t = (Q_\t^* Q_\t)^{1/2} P(0,\t)$ and note that \eqref{Eunitary} holds: 
	\[
	b_\t(\mathcal{E}_\t v_\star,  \mathcal{E}_\t \tilde v_\star)= (\mathcal{E}_\t v_\star,  \mathcal{E}_\t \tilde v_\star)_\t= 
	\big((Q_\t^* Q_\t)^{1/2} P(0,\t) v_\star,\,  (Q_\t^* Q_\t)^{1/2} P(0,\t) \tilde v_\star\big)_\t \ \,= \,
	\]
	\[
	\ \ \ \ 
	\big((Q_\t^* Q_\t) P(0,\t) v_\star, \,   P(0,\t) \tilde v_\star\big)_\t\,=\, 
	\big(Q_\t P(0,\t) v_\star,\,  Q_\t  P(0,\t) \tilde v_\star\big)_0 = 
	( v_\star,   \tilde v_\star)_0\,=\,b_0( v_\star,   \tilde v_\star). 
	\]
	
	Let us prove \eqref{Elip}.	For arbitrary $v_\star \in V_\star$, $v_\t \in V_\t$, recall that 
	$\|v_\t\|_\t^2 = b_\t[v_\t]$, $\| v_\star \|_0^2 = b_0[v_\star]$. Now, since $(Q_\t^* Q_\t)^{1/2}$ is non-negative on $(V_\t^\star, b_\t)$ one has $\| v_\t \|_\t \le\| (I + (Q_\t^* Q_\t)^{1/2}) v_\t\|_\t$, which (upon setting $v_\t = P_{V^\star_\t}v_\star - \mathcal{E}_\t v_\star$) gives
\begin{flalign*}
\| P_{V^\star_\t}v_\star - \mathcal{E}_\t v_\star\|_\t  \le 
\| (I+(Q_\t^* Q_\t)^{1/2})\left(P_{V^\star_\t}v_\star - \mathcal{E}_\t v_\star\right) \|_\t	= 
	\| (I-Q_\t^* Q_\t) P_{V^\star_\t} v_\star \|_\t	=\| P_{V^\star_\t} v_\star - Q_\t^* v_\star  \|_\t,
	 \, \forall v_\star \in V_\star.
\end{flalign*}	
Combining the above inequality with \eqref{contVs2} gives
\begin{flalign}
		\label{9.03.21}
	\| v_\star - \mathcal{E}_\t v_\star\|_\t   
	= \| P_{W^\star_\t} v_\star +\left(P_{V^\star_\t} v_\star - \mathcal{E}_\t v_\star\right)\|_\t 
	\le L_\star |\t| \, \|v_\star\|_0 +  \| P_{V^\star_\t} v_\star  - Q_\t^* v_\star \|_\t,
	\quad \forall v_\star \in V_\star.
\end{flalign}	
It remains to estimate the difference $ P_{V^\star_\t}-Q^*_\t$ on $V_\star$. 
For this note that, for any $v_\t\in V_\t^*$, 
\begin{flalign*}
b_\t\big(P_{V^\star_\t}v_\star - Q^*_\t v_\star, v_\t \big) & \,=\,\, b_\t(v_\star, v_\t) \,-\, b_0 \left(v_\star,Q_\t v_\t \right) \\ 
&=\,\, b_\t\left(v_\star,v_\t - Q_\t v_\t\right) \,+\, b_\t\left(v_\star, Q_\t v_\t \right) \, -\,  
b_0\left(v_\star, Q_\t v_\t \right), \quad \forall v_\star \in V_\star, \, \forall v_\t \in V^\star_\t. 
\end{flalign*}
To estimate  the first term on the right, we use $P_{V_\t^\star} Q_\t v_\t=v_\t$ implying  $	Q_\t v_\t  - v_\t = P_{W_\t^\star} Q_\t v_\t$, and \eqref{contVs2}. To bound the difference of the remaining two terms  we use \eqref{H5}. Hence, one has
$
|b_\t(P_{V^\star_\t}v_\star - Q^*_\t v_\star, v_\t )| \le \left(K^2 L_\star + L_b \right) |\t|  
\| v_\star \|_0 \| Q_\t v_\t\|_\t$, $\forall v_\star \in V_\star$, $\forall v_\t \in V^\star_\t
$. 
Now the last inequality and the bound $\| Q_\t v_\t \|_\t \le \sqrt{2} \| v_\t \|_\t$ (see \eqref{22.09.20e1} for $KL_\star |\t| \le 1/\sqrt{2}$)  give 
$
 \| P_{V^\star_\t} v_\star  - Q_\t^* v_\star \|_\t \le \sqrt{2}\left(K^2 L_\star + L_b \right)|\t|\, \| v_\star \|_0$, 
$\forall v_\star \in V_\star
$,  
which along with \eqref{9.03.21} implies  \eqref{Elip} with $K_b'=\sqrt{2}\left(K^2 L_\star + L_b \right)$. 
\end{proof}

\section*{Appendix B}\label{appb}
\setcounter{section}{2}
\setcounter{equation}{0}
\setcounter{theorem}{0}
\renewcommand{\theequation}{\mbox{\thesection.\arabic{equation}}}

We provide here some basic facts from theory of Bochner spaces, see e.g. \cite{ReeSim1,Hytonen}, as relevant and specialised 
to our setting in Section \ref{s.bivariate}, as well as justify some accompanying facts specific to our context. The latter follow quite standard arguments, but are still sketched here for the reader's convenience. 

Let $\mathcal{H}$ be a separable complex Hilbert space with inner product $(\cdot,\cdot)$ and associated 
norm $\|\cdot\|$. 
Bochner space $L^2(\RR^n;\mathcal{H})=:\mathbb{H}$ consists of all (Lebesgue measure zero equivalence classes of) weakly-measurable\footnote{i.e. $\forall\tilde u\in \mathcal{H}$, 
$\mathbb{R}^n\ni \xi \mapsto \big(u(\xi),\tilde u\big)\in \mathbb{C}$ is 
Lebesgue measurable} 
 maps 
$u:\RR^n\rightarrow \mathcal{H}$, 
such that 
$\|u\|^2_\mathbb{H}\,:=\,\int_{\RR^n}\|u(\xi)\|^2{\rm d}\xi<\infty$. 
The latter defines the norm $\|u\|_\mathbb{H}$ in $\mathbb{H}$. 
With associated inner product 
$
(u,\tilde u)_\mathbb{H}\,\,:=\,\int_{\RR^n}
\big(\,u(\xi)\,,\,\tilde u(\xi)\,\big)\,{\rm d}\xi,
$ 
 $\mathbb{H}$ is known to become a (separable) Hilbert space. 
Similarly are defined 
$L^2(\Theta;\mathcal{H})$ for any measurable $\Theta\subset\RR^n$, with induced Lebesgue measure. 
Weighted Bochner spaces like $L^2\left(\mathbb{R}^n,\langle\xi\rangle^2{\rm d}\xi;\, \mathcal{H}\right)$ 
are defined in obvious way. 

\vspace{.04in}

Let us now show that the domain $\mathbb{D}$ of the form on the left-hand side of 
\eqref{Lproblint} is dense in $L^2(\RR^n;\mathcal{H}_0)$.
\begin{proposition}
\label{propb1}
$\mathbb{D}\,:=\,L^2\left(\RR^n; \left(V_\star, b_0\right)\right)\dot{+}
L^2\left(\RR^n, {\langle\xi\rangle^2}{\rm d}\xi;\, \left(Z, b_0\right)\right)$ is dense in 
$\mathbb{H}_0:=L^2(\RR^n;\left(\mathcal{H}_0, d_0\right))$. 
\end{proposition} 
\begin{proof}
Show first that $\mathbb{V}_0:=L^2\left(\RR^n;\left(V_0, b_0\right)\right)$, where $V_0=V_\star\dot{+}Z$,  
is dense in $\mathbb{H}_0$. 
Both $V_0$ and $\mathcal{H}_0=\left(\overline{V_0}, d_0\right)$ are separable, and $V_0$ is compactly embedded into 
$\mathcal{H}_0$. 
Let $\lambda_0^{(k)}$, $k=1,2,...$, be the eigenvalues of $\mathbb{L}_0$, i.e. of $\mathbb{L}_\xi$ for $\xi=0$. 
Since the domain of $\mathbb{L}_0$ is in $V_0$, the associated eigenfunctions 
$\psi^{(k)}\in V_0\subset\mathcal{H}_0$ 
can be chosen to form 
an orthogonal basis in both $V_0$ and $\mathcal{H}_0$. 
Let $u\in\mathbb{H}_0$. Then, decomposing along this basis, for a.e. $\xi\in\RR^n$, 
$u(\xi)=\sum_{k=1}^\infty c_k(\xi)\psi^{(k)}$. 
For the truncated sums, $u^N(\xi):=\sum_{k=1}^N c_k(\xi)\psi^{(k)}$, clearly $u^N\to u$ in $\mathbb{H}_0$. 
On the other hand, since $b_0\left(\psi^{(k)},\psi^{(l)}\right)=
\delta_{kl}\lambda_0^{(k)}d_0\left[\psi^{(k)}\right]$ 
(where $\delta_{kl}$ denotes the 
Kroneker symbol), $u^N\in \mathbb{V}_0$ 
for any finite $N$.  
Hence $\mathbb{V}_0$ is dense in $\mathbb{H}_0$. 
Next we argue that $L^2\left(\RR^n,\langle\xi\rangle^2{\rm d}\xi;\, V_0\right)$ is in turn dense in $\mathbb{V}_0$. 
Indeed, 
given $u\in\mathbb{V}_0$ with associated $u(\xi)\in V_0$ for a.e. $\xi$, we construct $u^N$ by  
setting $u^N(\xi):=\chi_{B_N}(\xi)u(\xi)$, where $\chi_{B_N}$ is the characteristic function of ball of radius 
$N$ in $\RR^n$ centered at the origin. Then 
$u^N\in L^2\left(\RR^n,\langle\xi\rangle^2{\rm d}\xi;\, V_0\right)$ and $u^N\to u$ in $\mathbb{V}_0$ as $N\to\infty$, 
hence the stated density. 
Combining the above two density statements, we conclude that 
$L^2\left(\RR^n,\langle\xi\rangle^2{\rm d}\xi; V_0\right)$ is dense in $\mathbb{H}_0$. 
Finally, we 
observe that $L^2\left(\RR^n,\langle\xi\rangle^2{\rm d}\xi; V_0\right)$ is a subset of 
$L^2\left(\RR^n; V_\star\right)\dot{+}L^2\left(\RR^n, {\langle\xi\rangle^2}{\rm d}\xi; Z\right)$. 
Hence the latter is dense in $\mathbb{H}_0$, as required. 
\end{proof}

\begin{proposition}
\label{propb2}
The form $\mathbb{A}$ defined by the left-hand side of \eqref{Lproblint} with domain $\mathbb{D}$ is closed. 
\end{proposition} 
\begin{proof}
Let $\left\{u_m\right\}_{m=1}^\infty\subset\mathbb{D}$, 
$u_m(\xi)=v_m(\xi)+z_m(\xi)$, be a Cauchy sequence with respect to $\mathbb{A}$, 
i.e. $\mathbb{A}\left[u_m-u_l\right]\to 0$ as $m,l\to\infty$. 
It follows from \eqref{ahcoercive} and \eqref{kappa0} that 
\[
\mathbb{A}\left[u_m-u_l\right]\,\ge\, 
\nu_\star\int_{\mathbb{R}^n}\,|\xi|^2\,b_0\big[z_m(\xi)-z_l(\xi)\big]{\rm d}\xi\,\,+\,
\left(1-K_Z\right)\int_{\mathbb{R}^n}
\Big(\,b_0\big[z_m(\xi)-z_l(\xi)\big]\,+\,b_0\big[v_m(\xi)-v_l(\xi)\big]\Big)\,{\rm d}\xi.
\]
Hence $\left\{v_m\right\}$ and $\left\{z_m\right\}$ are 
 Cauchy sequences in, respectively, $L^2\left(\mathbb{R}^n; V_\star\right)$ and 
$L^2\left(\mathbb{R}^n, \langle\xi\rangle^2; Z\right)$. 
From the basic theory of Bochner spaces both of these spaces are 
 complete, and hence there exist $v\in L^2\left(\mathbb{R}^n; V_\star\right)$ and 
$z\in L^2\left(\mathbb{R}^n, \langle\xi\rangle^2; Z\right)$ such that 
$
\int_{\mathbb{R}^n}
\left\{\left(1+|\xi|^2\right)b_0\left[z_m(\xi)-z(\xi)\right]+b_0\left[v_m(\xi)-v_l(\xi)\right]\right\}{\rm d}\xi\rightarrow 0$.  
For $u:=v+z\in \mathbb{D}$ this 
implies that $\mathbb{A}\left[u_m-u\right]\to 0$, which 
completes the proof. 
\end{proof} 

\begin{proposition}
\label{propb3}
Let $h\in L^2\left(\Theta;\mathcal{H}_0\right)$ and regard it as an element of 
$L^2\left(\mathbb{R}^n;\mathcal{H}_0\right)$ by setting $h(\t)=0$ for $\t\notin\Theta$.  
Let $0<\ep<1$ and set $\xi=\t/\ep$. Then, for a.e. $\xi\in\mathbb{R}^n$, the unique solutions 
$v(\xi)+z(\xi)$ to \eqref{Lproblint} and \eqref{Linvprobl} coincide. 
\end{proposition} 
\begin{proof}
Let $v+z\in\mathbb{D}$ be the solution to \eqref{Lproblint}. 
Let $\tilde v_j+\tilde z_j\in V_\star\dot{+}Z$, $j=1,2,...$, form a dense set in 
$\left(V_\star\dot{+}Z,\, b_0\right)$. For any $j$, consider arbitrary 
$\varphi(\xi)\in C_0^\infty\left(\mathbb{R}^n\right)$ and set 
$\tilde v(\xi)=\tilde v_j\varphi(\xi)$ and $\tilde z(\xi)=\tilde z_j\varphi(\xi)$. Then from \eqref{Lproblint} 
\begin{equation}
\label{b3pf1}
\int_{\RR^n}\Big[\,a^{\rm h}_{\xi}\big(z(\xi),\tilde z_j\big) + 
b_0\big(v(\xi)+z(\xi),\, \tilde v_j+\tilde z_j\big)\,\,-\,\,
d_0\big(\,h(\ep\xi), \tilde v_j+\tilde z_j\big)\,\Big]\,\overline{\varphi(\xi)}\,{\rm d}\xi\,\,
=\,\,0, \ \ \ 
\forall \varphi\in C_0^\infty\left(\mathbb{R}^n\right). 
\end{equation}
Because of the density of 
$C_0^\infty\left(\mathbb{R}^n\right)$, 
the square bracket in \eqref{b3pf1} must vanish for a.e. $\xi\in\mathbb{R}^n$, for all 
$j\ge 1$. 
Finally, because of the density of $\left\{v_j+z_j\right\}$ in $V_\star\dot{+}Z$ (and hence 
also in $\mathcal{H}_0$), 
the above square bracket 
must vanish for a.e. $\xi\in\mathbb{R}^n$ for 
all $\tilde v+\tilde z\in V_\star\dot{+}Z$ which is equivalent to \eqref{Linvprobl}. 
The converse statement 
trivially follows from the uniqueness of the solutions to \eqref{Linvprobl} and \eqref{Lproblint}. 
\end{proof} 

We next recall the notion of a direct integral of operators and specialise it to the context of 
Section \ref{s.bivariate}. 
\begin{definition} \label{defb4} 
Let $\left(\mathcal{H}_0, d_0\right)$ be a complex Hilbert space and $\mathbb{L}$ be a self-adjoint operator in Bochner space $\mathbb{H}_0:=L^2\big(\mathbb{R}^n;\mathcal{H}_0\big)$. Let $\mathbb{L}_\xi$, 
$\xi\in\,\mathbb{R}^n$, be a family of self-adjoint operators in $\mathcal{H}_0$ with their 
spectra (say) contained in $[1,+\infty)$ and which are 
weakly-measurable\footnote{i.e. $\forall g,\tilde u\in \mathcal{H}_0$,  
$\xi\mapsto d_0\left(\mathbb{L}_\xi^{-1}g,\tilde u\right)$ is Lebesgue-measurable 
as a map from $\mathbb{R}^n$ to $\mathbb{C}$} in $\xi$. 
We say that $\mathbb{L}$ is a direct integral of $\mathbb{L}_\xi$ over $\xi\in\mathbb{R}^n$, 
denoted $\mathbb{L}=\int_{\RR^n}^\oplus \mathbb{L}_\xi \, {\rm d}\xi$, and $\mathbb{L}_\xi$ 
are fibers of $\mathbb{L}$, if 
\begin{enumerate}
\item[(i)] $u\in \mathbb{H}_0$ is in the domain ${\rm dom}\,\mathbb{L}$ of $\mathbb{L}$,  if 
and only if $u(\xi)\in {\rm dom}\,\mathbb{L}_\xi$ for a.e. $\xi\in\mathbb{R}^n$ and 
$\int_{\mathbb{R}^n}d_0\big[\mathbb{L}_\xi u(\xi)\big]\,{\rm d}\xi<+\infty$; 
\item[(ii)] $\forall u\in {\rm dom}\,\mathbb{L}$, 
$\big(\mathbb{L}u\big)(\xi)\,=\,\mathbb{L}_\xi \big(u(\xi)\big), \ \ \ \mbox{ for a.e. } \ 
\xi\in\mathbb{R}^n$. 
\end{enumerate}
\end{definition} 
\begin{proposition}
\label{propb5}
For $\mathbb{L}_\xi$, $\xi\in\mathbb{R}^n$, and $\mathbb{L}$ as defined in Sections 
\ref{s.spbt} and \ref{s.bivariate} respectively, 
\[
\mathbb{L}\,=\,\int_{\RR^n}^\oplus \mathbb{L}_\xi \, {\rm d}\xi, \ \ \ \mbox{and } \ \ 
\mathbb{L}^{-1}\,=\,\int_{\RR^n}^\oplus \mathbb{L}_\xi^{-1}\, {\rm d}\xi. 
\]
\end{proposition} 
\begin{proof}
First, $\mathbb{L}_\xi$ are weakly-measurable as for any $g\in \mathcal{H}_0$, 
$d_0\left[ \mathbb{L}_\xi^{-1}g\right]$ is continuous in $\xi$. 
As a brief sketch for proving the latter, consider $\xi_1,\xi_2\in\RR^n$ with associated $u_j=\mathbb{L}_{\xi_j}^{-1}g$, 
$j=1,2$.  Then, via \eqref{Sform},  $\mathbb{S}_{\xi_j}(u_j, u_1-u_2)=d_0(g,u_1-u_2)$, $j=1,2$. 
Hence, subtracting, $S_{\xi_1}[u_1-u_2]=S_{\xi_2}(u_2,u_1-u_2)-S_{\xi_1}(u_2,u_1-u_2)$. 
When $\xi_2\to\xi_1$, the latter difference form becomes small, see \eqref{Sform}. 
Hence, via standard arguments, $d_0[u_1-u_2]\le S_{\xi_1}[u_1-u_2]\to 0$ as $\xi_2\to\xi_1$, as required. 
The rest of the proof essentially follows that of Proposition \ref{propb3}. 
By definition, $v+z\in \mathbb{D}$ is in ${\rm dom}\,\mathbb{L}$ if there exists 
$H\in\mathbb{H}_0$ such that 
\begin{equation}
\label{b4pf1}
\mathbb{A}\big(v+z,\,\tilde v +\tilde z\big)\,\,=\,\,
\int_{\mathbb{R}^n}\,d_0\big(H(\xi),\,\tilde v +\tilde z\big) \, {\rm d}\xi, \ \ 
\ \forall \tilde v +\tilde z\in \mathbb{D},  
\end{equation}
where $\mathbb{A}$ is the form on the left-hand side of \eqref{Lproblint}. 
Arguing then as in the proof of Proposition \ref{propb3}, we conclude that for a.e. 
$\xi\in\mathbb{R}^n$, $\forall\, \tilde v+\tilde z\in V_\star\dot{+}Z$, 
$\,\,\mathbb{S}_\xi\big(v(\xi)+z(\xi),\,\tilde v +\tilde z\big)=
d_0\big(H(\xi),\tilde v +\tilde z\big)$.  
The latter implies both conditions in Definition \ref{defb4}, so $\mathbb{L}$ is 
the direct integral of $\mathbb{L}_\xi$. 
For the inverses, condition $(i)$ in Definition \ref{defb4} trivially follows from the 
well-posedness of \eqref{Lproblint}. 
Also, if $H\in\mathbb{H}_0$ and $ u=\mathbb{L}^{-1}H$ then $(ii)$ implies 
$\mathbb{L}_\xi^{-1}\big(H(\xi)\big)=u(\xi)= 
\big(\mathbb{L}^{-1}H \big)(\xi)$ for  a.e. 
$\xi\in\mathbb{R}^n$, as required. 
\end{proof}
Fourier transform $\mathcal{F}$ is known to be a well-defined unitary operator in a Bochner space 
$L^2\left(\mathbb{R}^n; \mathcal{H}\right)$, together with its inverse $\mathcal{F}^{-1}$: 
\begin{definition}
\label{defb6}
Given $u\in L^2\left(\mathbb{R}^n; \mathcal{H}\right)=:\mathbb{H}$, $\mathcal{F}u=:\hat u$ and 
$\mathcal{F}^{-1}u=:\check u$ are such elements of $\mathbb{H}$ that, 
 respectively for a.e. $\xi\in\mathbb{R}^n$ and a.e. $x\in\mathbb{R}^n$, 
\begin{equation}
\label{ftdef}
\big(\hat u(\xi),\tilde u\big)\,=\,
(2\pi)^{-n/2}\int_{\mathbb{R}^n}e^{-{\rm i}x\cdot\xi}\big(u(x),\tilde u\big){\rm d}x, \ \ \ 
\big(\check u(x),\tilde u\big)\,=\,
(2\pi)^{-n/2}\int_{\mathbb{R}^n}e^{{\rm i}x\cdot\xi}\big(u(\xi),\tilde u\big){\rm d}\xi, \ \ \ 
\forall \tilde u\in \mathcal{H}. 
\end{equation} 
[The above integrals denote conventional (inverse) Fourier transforms 
in $L^2\left(\mathbb{R}^n\right)\ni \big(u(\cdot),\tilde u\big)$.] 
It is straightforward to check that 
the above $\hat u$ and $\check u$ are (uniquely) well-defined, 
and $\mathcal{F}$ and $\mathcal{F}^{-1}$ are unitary in $\mathbb{H}$ inverses of each other, 
with Plancherel theorem held: 
\begin{equation}
\label{planch}
\big( u,\,\tilde u\big)_{\mathbb{H}}\,\,=\,\,\big(\mathcal{F} u,\,\mathcal{F}\tilde u\big)_{\mathbb{H}}\,\,=\,\,
\big(\mathcal{F}^{-1} u,\,\mathcal{F}^{-1}\tilde u\big)_{\mathbb{H}}, \ \ \ \forall u,\tilde u\in\mathbb{H}. 
\end{equation}
\end{definition}
If $\mathcal{H}_0$ is a closed subspace of $\mathcal{H}$, then 
$\mathbb{H}_0:= L^2\left(\mathbb{R}^n; \left(\mathcal{H}_0, (\cdot,\cdot)\,\right)\right)$ 
is invariant under $\mathcal{F}$ whose restriction to $\mathbb{H}_0$ coincides with the Fourier transform directly defined on  
$\mathbb{H}_0$.

Bochner Sobolev space $H^1\left(\mathbb{R}^n; \mathcal{H}\right)$ can be defined in two equivalent ways, 
see e.g. \cite{ReeSim1} and \cite{Hytonen}: via generalised derivatives or via Fourier transform. 
Adopting the former, 
\begin{definition}
\label{defb7}
It is said that $u\in L^2\left(\mathbb{R}^n; \mathcal{H}\right)$ has a (first order) $L^2-$generalised derivative 
$\partial_{x_j}u\in L^2\left(\mathbb{R}^n; \mathcal{H}\right)$, $j=1,2,...,n$, if for all 
$\varphi\in C_0^\infty\left(\mathbb{R}^n\right)$ and all $\tilde u\in \mathcal{H}$, 
$
\int_{\mathbb{R}^n}\big( u(x),\tilde u\big)\partial_{x_j}\varphi(x){\rm d}x = 
-\int_{\mathbb{R}^n}\big( \partial_{x_j}u(x),\tilde u\big)\varphi(x){\rm d}x  
$. 
$H^1\left(\mathbb{R}^n; \mathcal{H}\right)$ is the space of all $u\in L^2\left(\mathbb{R}^n; \mathcal{H}\right)$ 
having all the first-order $L^2-$generalised derivatives. 
\end{definition}
It is known that $H^1\left(\mathbb{R}^n; \mathcal{H}\right)$ is a separable Hilbert space with inner 
product 
\[
\big(u,\,\tilde u\big)_{H^1\left(\mathbb{R}^n; \mathcal{H}\right)}\,\,:=\,\int_{\RR^n}\Big[\, 
\big(u(x),\tilde u(x)\,\big)\,+\,
\sum_{j=1}^n\Big(\,\partial_{x_j}u(x)\,,\,\partial_{x_j}\tilde u(x)\Big)\,\Big]\,{\rm d}x.
\]
Finally, $u\in H^1\left(\mathbb{R}^n; \mathcal{H}\right)$ if and only if 
$u\in L^2\left(\mathbb{R}^n; \mathcal{H}\right)$ and 
$\mathcal{F}u\in L^2\Big(\mathbb{R}^n,\langle\xi\rangle^2{\rm d}\xi;\, \mathcal{H}\Big)$. 
For $u\in H^1\left(\mathbb{R}^n; \mathcal{H}\right)$, the gradient 
$\nabla u \in \Big(L^2\left(\mathbb{R}^n; H\right)\Big)^n$ is defined in a standard way, 
and 
$\mathcal{F}(\nabla u)(\xi)\,=\,{\rm i}\xi\,\mathcal{F}(u)(\xi)$. 

\begin{lemma}
\label{lemft}
Form $Q$ given by \eqref{Q}  on domain $\check{\mathbb{D}}=
H^1\left(\RR^n ; \left(Z,b_0\right)\right) \dot{+} 
L^2\left(\RR^n ;\left(V_\star, b_0\right)\right)$ determines a self-adjoint operator $\mathcal{L}$ in Hilbert space 
$\mathbb{H}_0=L^2\left(\RR^n; \mathcal{H}_0\right)$, $\mathcal{H}_0=\left(\overline{V_0},\, d_0\right)$, 
$V_0=V_\star\dot{+}Z$. In fact, 
$\mathcal{L}=\mathcal{F}^{-1}\mathbb{L}\,\mathcal{F}$, 
where $\mathcal{F}$ is the Fourier transform in $\mathbb{H}_0$. 
\end{lemma}
\begin{proof}
Since $\mathcal{F}$ is a unitary operator in $\mathbb{H}_0$, see \eqref{planch}, it suffices to 
show that $\check{\mathbb{D}}=\mathcal{F}^{-1}\mathbb{D}$ and 
\begin{equation}
\label{ftforms}
Q\big(u+v,\tilde{u} + \tilde{v}\big) \,\,=\,\,\mathbb{A}\big(\,\mathcal{F}^{-1}(u+v),\, 
\mathcal{F}^{-1}\left(\tilde{u} + \tilde{v}\right)\big), \quad \forall\,\, u+v,\,\, \tilde u+\tilde v\in \mathbb{D}, 
\end{equation} 
where $\mathbb{A}$ is the form on the left-hand side of \eqref{Lproblint}. 
\vspace{.04in}

1. Notice that by properties of the Fourier transform 
 $\check{\mathbb{D}}=\mathcal{F}_b^{-1}\mathbb{D}$
where $\mathcal{F}_b$ is the Fourier transform, according to Definition \ref{defb6}, for Bochner space 
$\mathbb{V}_0:=L^2\left(\mathbb{R}^n; \left(V_0, b_0\right)\right)$, 
$V_0=V_\star\dot{+}Z$. 
So the first assertion follows as soon as we show that $\mathcal{F}_bv=\mathcal{F}v$, $\forall v\in\mathbb{V}_0$. 
From \eqref{ik2}, 
for Hilbert space $\mathcal{V}_0:=(V_0,b_0)$, there exists a bounded linear map 
$T:\mathcal{V} \rightarrow \mathcal{V}$ such that 
$d_0(v,\tilde v)=b_0(Tv,\tilde v)=b_0(v,T^*\tilde v)$, $\forall v,\tilde v\in V_0$, where $T^*$ is the 
adjoint of $T$ in $\mathcal{V}_0$. 
From \eqref{ftdef}, 
for a.e. $\xi$,  
$\big(\mathcal{F} u(\xi), \tilde u\big)=\mathcal{F}_c \big(u(\cdot), \tilde u\big)(\xi)$ where 
$\mathcal{F}_c$ is the conventional Fourier transform in $L^2\left(\mathbb{R}^n\right)$. 
So, for any 
$v\in\mathbb{V}_0\subset\mathbb{H}_0$ and $v'\in V_0\subset\mathcal{H}_0$, for a.e. $\xi$, 
\[
d_0\big(\mathcal{F}v(\xi),v'\big)=\mathcal{F}_c d_0\big(v(\cdot),v'\big)(\xi)= 
\mathcal{F}_c b_0\big(v(\cdot),T^*v'\big)(\xi)= 
b_0\big(\mathcal{F}_b v(\xi),T^*v'\big)=
d_0\big(\mathcal{F}_bv(\xi),v'\big), \ \ \forall v'\in V_0. 
\]
As $V_0$ is dense in $\mathcal{H}_0$, the latter implies $\mathcal{F}_bv=\mathcal{F}v$ for a.e. $\xi$, as required. 
\vspace{.04in}

2. For proving \eqref{ftforms}, 
it would suffice to show that a variant 
of Plancherel theorem holds for all the forms entering \eqref{ahgrad}--\eqref{Lproblint}. 
Namely, $\forall\, v+z$, $\tilde v+\tilde z\in \mathbb{V}_0$, 
\begin{equation}
\label{plancher}
\int_{\mathbb{R}^n} \mathfrak{b}\big(v(\xi)+z(\xi),\, \tilde v(\xi)+\tilde z(\xi)\big)\,{\rm d}\xi\,\,=\,\,
\int_{\mathbb{R}^n} \mathfrak{b}\big(\mathcal{F}^{-1}v(x)+\mathcal{F}^{-1}z(x),\, 
\mathcal{F}^{-1}\tilde v(x)+\mathcal{F}^{-1}\tilde z(x)\big)\,{\rm d} x, 
\end{equation} 
for $\mathfrak{b}=
b_0$ and $\mathfrak{b}=a^h_{jk}$, $j,k=1,...,n$. 
(For 
$\mathfrak{b}=a^h_{jk}$ one has to set 
$v=\tilde v=0$.) 
For $\mathfrak{b}=b_0$, by \eqref{planch}, 
$ 
\int_{\mathbb{R}^n} {b}_0\big(v(\xi)+z(\xi),\, \tilde v(\xi)+\tilde z(\xi)\big)\,{\rm d}\xi\,\,=\,\,
\int_{\mathbb{R}^n} {b}_0\big(\mathcal{F}_b^{-1}v(x)+\mathcal{F}_b^{-1}z(x),\, 
\mathcal{F}_b^{-1}\tilde v(x)+\mathcal{F}_b^{-1}\tilde z(x)\big)\,{\rm d} x. 
$ 
Hence \eqref{plancher} immediately follows from $\mathcal{F}_bv=\mathcal{F}v$, $\forall v\in\mathbb{V}_0$, shown above.  
\vspace{.04in}

3. To prove \eqref{plancher} for $\mathfrak{b}=a^h_{jk}$ for any fixed $1\le j,k\le n$, we notice first that the form 
$a^h_{jk}$ is bounded in terms of $b_0$: 
with some 
$C>0$, 
$ 
\left\vert a^h_{jk}(z,\tilde z)\right\vert \le C\, b_0^{1/2}[z]\, b_0^{1/2}[\tilde z]$, 
$\forall z,\tilde z\in Z$. 
(This follows e.g. from \eqref{ahij}, \eqref{H4} and \eqref{Nbdd}.) 
Therefore, 
for $\mathcal{Z}:=(Z,b_0)$, there exists a 
bounded linear map $T_{jk}:\mathcal{Z} \rightarrow \mathcal{Z}$ such that 
$a^h_{jk}(z,\tilde z)=b_0(T_{jk}z,\tilde z)$, $\forall z,\tilde z\in Z$. 
Let now  $z,\tilde z\in \mathbb{Z}:=L^2\left(\mathbb{R}^n;\mathcal{Z}\right)$. Then, using the boundedness of 
$a^h_{jk}$ and $T_{jk}$, 
the isometry of $\mathcal{F}_b$, and the above established identity $\mathcal{F}_b z= \mathcal{F} z$, $\forall z\in \mathbb{Z}\subset\mathbb{V}_0\subset\mathbb{H}_0$, 
\begin{equation}
\label{vs1}
\int_{\mathbb{R}^n} a^h_{jk}\big(z(x),\tilde z(x)\big)\,{\rm d}x\,\,=\,\,
\int_{\mathbb{R}^n} b_0\big(T_{jk}z(x),\tilde z(x)\big)\,{\rm d}x\,\,=\,\, 
\int_{\mathbb{R}^n} b_0\Big(\mathcal{F}\left(T_{jk}z\right)(\xi),\,\mathcal{F}z(\xi)\Big)\,{\rm d}\xi. 
\end{equation}
Notice 
that $\mathcal{F} z(\xi)= \mathcal{F}_b z(\xi)\in Z$ for a.e. $\xi$ (as 
follows 
from 
\eqref{ftdef}). 
Next show that, for a.e. $\xi$, 
$\mathcal{F}\left(T_{jk}z\right)(\xi)= T_{jk}\left(\mathcal{F} z(\xi)\right)$. 
Applying \eqref{ftdef} 
for $\mathcal{F}_b$, $u=T_{jk}z$, and 
$\tilde u=z'\in Z$ 
implies 
$b_0\big(\mathcal{F}_bT_{jk} z(\xi), z'\big)=\mathcal{F}_c b_0\big(T_{jk} z(\cdot), z'\big)(\xi)$.  
Then, with $T^*_{jk}$ denoting the adjoint of $T_{jk}$, 
applying \eqref{ftdef} again, 
\begin{equation}
\label{ttstar}
b_0\big(\mathcal{F}T_{jk} z(\xi), z'\big)\,\,=\,\,\mathcal{F}_c b_0\big( z(\cdot), T_{jk}^*z'\big)(\xi)\,\,=\,\, 
b_0\left(\mathcal{F} z(\xi), T_{jk}^*z'\right)\,\,=\,\,b_0\big(T_{jk}\mathcal{F} z(\xi), z'\big). 
\end{equation}
Since the above holds for arbitrary $z'\in Z$, we conclude that 
$\mathcal{F}\left(T_{jk} z\right)(\xi)= T_{jk}\left(\mathcal{F} z(\xi)\right)$, as desired. 
Employing the latter in \eqref{vs1} and then using again the definition of $T_{jk}$: 
$ 
\int_{\mathbb{R}^n} a^h_{jk}\big(z(x),\tilde z(x)\big)\,{\rm d}x\,=\,
\int_{\mathbb{R}^n} b_0\Big(T_{jk}\mathcal{F}z(\xi),\,\mathcal{F}\tilde z(\xi)\Big)\,{\rm d}\xi\,=\,
\int_{\mathbb{R}^n} a^h_{jk}\big(\mathcal{F}z(\xi),\mathcal{F}\tilde z(\xi)\big)\,{\rm d}\xi, 
$ 
as required.  
\end{proof} 

\begin{proposition}
\label{propb9}
Let $\mathcal{P}_{\mathcal{H}_0}^0:(\mathcal{H},d_0) \rightarrow \mathcal{H}_0$ and  
$\mathcal{P}:L^2\big(\RR^n ; (\mathcal{H},d_0)\big) \rightarrow L^2(\RR^n ; \mathcal{H}_0)$ be the orthogonal projections 
on the respective subspaces. Then
\begin{equation}
\label{b9statm}
\mathcal{P}_{\mathcal{H}_0}^0f(\xi) = \big(\mathcal{F}  \mathcal{P}\mathcal{F}^{-1} f \big)(\xi) \quad for\,\, a.e.\  \xi, \quad f \in L^2(\RR^n; (\mathcal{H},d_0)).
\end{equation} 
\end{proposition}
\begin{proof}
Given $f\in \mathbb{H}:=L^2\big(\RR^n ; (\mathcal{H},d_0)\big)$ let $g=\mathcal{F}^{-1}f\in 
\mathbb{H}$ and notice first that, for a.e. $x$, $(\mathcal{P}g)(x)=\mathcal{P}_{\mathcal{H}_0}^0g(x)$. 
Indeed, $h=\mathcal{P}g\in L^2(\RR^n ; \mathcal{H}_0)=:\mathbb{H}_0$ is such that, for any $\tilde h\in \mathbb{H}_0$, 
$
\int_{\mathbb{R}^n} d_0\big(g(x)-h(x),\,\tilde h(x)\big){\rm d}x=0,  
$
and the latter obviously holds for $h(x)=\mathcal{P}_{\mathcal{H}_0}^0g(x)$ for a.e. $x$. 
So \eqref{b9statm} is equivalent to 
$\mathcal{P}_{\mathcal{H}_0}^0\mathcal{F}g(\xi) = \big(\mathcal{F}  \mathcal{P}_{\mathcal{H}_0}^0 g(\cdot) \big)(\xi)$, 
for a.e. $\xi$. The latter can be proved by the argument identical to \eqref{ttstar}, 
with $T_{jk}$ replaced by (self-adjoint) $\mathcal{P}_{\mathcal{H}_0}^0$, $b_0$ by $d_0$, 
$z$ by $g$ and $z'$ by $g'\in\mathcal{H}$. 
\end{proof}



\end{document}